\documentclass[reqno]{amsart}

\usepackage{amsmath, amsthm, amsfonts, amssymb}
\usepackage[shortlabels]{enumitem}
\usepackage{multirow}
\usepackage{hyperref}
\usepackage{dsfont} 
\usepackage{ulem} 
\usepackage{xcolor}
\usepackage{tikz, pgfplots}
\usetikzlibrary{patterns}
\usetikzlibrary{decorations.pathreplacing}

\pgfplotsset{compat=newest}
\usepackage[nocompress]{cite}

\theoremstyle{plain}
  \newtheorem{theorem}{Theorem}[section]
  
  \newtheorem{lemma}[theorem]{Lemma}
  \newtheorem{corollary}[theorem]{Corollary}

\theoremstyle{definition}
  
  \newtheorem{example}[theorem]{Example}
  \newtheorem{assumption}[theorem]{Assumption}

\theoremstyle{remark}
  \newtheorem{remark}[theorem]{Remark}

\newcommand{\set}[1]{\left\{#1\right\}} 
\newcommand{\abs}[1]{\left\vert#1\right\vert} 
\newcommand{\ind}[1]{\ensuremath{{\mathds 1}_{{#1}}}} 
\newcommand{\sind}[1]{\ensuremath{\mathds{1}_{\set{#1}}}} 

\DeclareMathOperator{\cov}{Cov}
\DeclareMathOperator{\var}{Var}

\newcommand{\E}{\mathbb{E}}
\newcommand{\C}{\mathbb{C}}

\newcommand{\Z}{\mathbb{Z}}
\renewcommand{\P}{\mathbb{P}}
\renewcommand\Re{\operatorname{Re}}
\renewcommand\Im{\operatorname{Im}}
\newcommand{\eps}{\varepsilon}

\newcommand{\da}{\downarrow}
\newcommand{\ua}{\uparrow}

\begin{document}

\title[Gauss--Lucas theorem refinement for random polynomials]{An asymptotic refinement of the Gauss--Lucas Theorem for random polynomials with i.i.d.\ roots}

\author{Sean O'Rourke}
\address{Department of Mathematics\\ University of Colorado\\ Campus Box 395\\ Boulder, CO 80309-0395\\USA}
\email{sean.d.orourke@colorado.edu}
\thanks{S. O'Rourke has been partially supported in part by NSF CAREER grant DMS-2143142.}

\author{Noah Williams}
\address{Department of Mathematical Sciences\\ Appalachian State University\\ASU Box 32092\\Boone, NC 28608\\USA}
\email{williamsnn@appstate.edu}


\begin{abstract}
If $p:\mathbb{C} \to \mathbb{C}$ is a non-constant polynomial, the Gauss--Lucas theorem asserts that its critical points are contained in the convex hull of its roots.  We consider the case when $p$ is a random polynomial of degree $n$ with roots chosen independently from a radially symmetric, compactly supported probability measure $\mu$ in the complex plane.  We show that the largest (in magnitude) critical points are closely paired with the largest roots of $p$.  This allows us to compute the asymptotic fluctuations of the largest critical points as the degree $n$ tends to infinity.  We show that the limiting distribution of the fluctuations is described by either a Gaussian distribution or a heavy-tailed stable distribution, depending on the behavior of $\mu$ near the edge of its support.  As a corollary, we obtain an asymptotic refinement to the Gauss--Lucas theorem for random polynomials.
\end{abstract}


\maketitle
\tableofcontents
\section{Introduction}
Let $p: \mathbb{C} \to \mathbb{C}$ be a polynomial in a single complex variable.  The \emph{critical points} of $p$ are the roots of its derivative $p'$.  The Gauss--Lucas theorem states that if $p$ has degree $n > 1$, then the $n-1$ critical points of $p$ lie inside the convex hull formed from the $n$ roots of $p$.  This result has been refined and extended in various ways; see, for instance, \cite{MR1506859,MR19157,MR21148,MR2063118,MR1452801,MR239061,MR2398412,MR2159699,MR14508,MR1473453,MR2000158,MR4130852,1706.05410,MR2082007,MR3556367,SS,MR3992156,MR3367639} and references therein.

In this note, we focus on a class of random polynomials with independent and identically distributed (i.i.d.) roots in the complex plane.  Specifically, we consider polynomials $p_n: \mathbb{C} \to \mathbb{C}$ of the form 
\begin{equation} \label{def:pn}
	p_n(z) := \prod_{i=1}^n (z - X_i), 
\end{equation}
where $X_1, \ldots, X_n$ are i.i.d. draws from a probability measure $\mu$ in the complex plane.  Many results are known for the roots of the derivatives of this model; see, for example, \cite{PR,K,KS,H3,OW2,OW,TRR,BLR,S,MR4669281,MR4242313,MR3318313,MR3698743,MR4711583,MR4762159,MR4741259,2312.14883,2404.12472,2307.11935} and references therein.  
For example, it is known that the empirical measure constructed from the critical points of $p$ converges in probability to $\mu$ as $n$ tends to infinity.  This was first established by Pemantle and Rivin \cite{PR}, under some technical assumptions, and the general case was proved later by Kabluchko \cite{K}.  

More recently, it was shown in \cite{KS,OW2} that there is a close pairing between the roots and critical points.  For example, if $W_1$ represents the critical point closest to $X_1$, the results in \cite{KS,OW2} describe the fluctuations for the distance between $W_1$ and $X_1$ as $n$ tends to infinity.  

In this paper, instead of considering an arbitrary root $X_i$, we consider the order statistics $X^{\ua}_{(i)}$, $|X^{\ua}_{(1)}| \leq \cdots \leq |X^{\ua}_{(n)}|$, and we study the behavior of the critical points closest to the largest in magnitude roots.  Since the convex hull can often be determined by only a few of the largest points, one motivation for studying the largest critical points is to better understand the behavior of the convex hull formed from the critical points.  As a corollary of our main results, for instance, we obtain an asymptotic refinement of the Gauss--Lucas theorem, which allows us to compare the diameter of the convex hull formed from the critical points with the diameter of the convex hull formed from the roots.  
We provide a detailed discussion of our results and their connections to the works cited above in Section \ref{sec:results}.


\subsection{Outline} 
The paper is organized as follows.  We present our main results and their connections to the previous works cited above in Section \ref{sec:results}. The proofs are presented in Sections \ref{sec:proofsPos} and \ref{sec:proofsNeg}.  The appendix contains a number of auxiliary results.


\subsection{Notation} 
Throughout the paper, we use asymptotic notation (e.g., $O, o, \ll$) under the assumption that $n \to \infty$.  We use $a_n=O(b_n)$, $b_n=\Omega(a_n)$, $a_n \ll b_n$, or $b_n \gg a_n$ if there is a constant $C >0$, independent of $n$, such that $n>C$ implies the estimate $|a_n| \leq C |b_n|$. If $C$ depends on a parameter, e.g., $C = C_k$, we indicate this with subscripts, e.g., $a_n = O_k(b_n)$. The notation $a_n = \Theta(b_n)$ means $a_n \ll b_n \ll a_n$, and we write $a_n = o(b_n)$ or $b_n = \omega(a_n)$ if $|a_n| \leq c_n\abs{b_n}$ for some sequence $c_n$ that goes to zero as $n \to \infty$.

As we work in the complex setting, we use $\sqrt{-1}$ for the imaginary unit to reserve $i$ as an index. Unless otherwise specified, the function $\arg:\C\setminus\set{0} \to(-\pi, \pi]$ returns the principal value of the argument of a nonzero complex number. We denote the Lebesgue measure on $\C$ with $\lambda$. Many of our results concern order statistics of the complex-valued random sample $X_1, \ldots, X_n$, which we precisely define by ordering $\mathbb{C}$ in ``spiral'' fashion as follows. Set $0 \preceq z$ for any $z \in \mathbb{C}$, and for nonzero $z,w\in \mathbb{C}$, define $z \preceq w$ when either of the following hold
\begin{itemize}
	\item $\abs{z} < \abs{w}$;
	\item $\abs{z} = \abs{w}$, and $\arg{z} \leq \arg{w}$.
\end{itemize} We use $X^\ua_{(k)}$ to mean the $k$th smallest element of the list $X_1, \ldots, X_n$ ordered with respect to $\preceq$: i.e.\ $X^\ua_{(1)} \preceq X^\ua_{(2)} \preceq \cdots \preceq X^\ua_{(n)}$. Similarly, we define $X^\da_{(k)}$ to be the $k$th largest element from among $X_1, \ldots X_n$ so that $X^\da_{(n)} \preceq \cdots \preceq X^\da_{(2)} \preceq X^\da_{(1)}$. (For completeness, in the case where several $X_i$ coincide, list them in order of increasing index before determining which is the $k$th largest or smallest.)

We use the following probability and set-theoretic notation. We say an event $E_n$ (which depends on $n$) holds with \textit{high probability} if $\lim_{n\to \infty} \P(E_n) = 1$. For an event $E$, the random variable $\ind{E}$ is the indicator function of $E$, and for a square integrable random variable $\xi$, $\var(\xi) := \E\abs{\xi-\E[\xi]}^2$ is its variance.  For those $z\in \C$ where the following integral exists, we define the Cauchy--Stieltjes transform of a probability measure $\mu$ on $\C$ to be
\begin{equation} \label{def:st}
m_\mu(z) := \int_\C\frac{1}{z-x}\,d\mu(x).
\end{equation}
(By Fubini's theorem, it follows that $m_\mu$ is defined and finite for $\lambda$-almost every $z \in \mathbb{C}$.) For a set $S$, $\#S$ and $\abs{S}$ denote the cardinality of $S$, and $S^c$ is the complement of $S$.  We denote the discrete interval of length $n$ with $[n]:= \set{1, 2, \ldots, n}$, and we define the open ball of radius $r > 0$ centered at $x \in \mathbb{C}$ via $B(x,r):=\set{z\in \C: \abs{z-x} < r}.$


\subsection*{Acknowledgements}
S.\ O'Rourke thanks Andrew Campbell and Stefan Steinerberger for useful references and discussions. The authors also thank the anonymous referees for their helpful and detailed feedback.


\section{Main results} \label{sec:results}

Let $\mu$ be a probability measure on $\mathbb{C}$.  We will consider the critical points of the polynomial $p_n$, defined in \eqref{def:pn}, where $X_1, X_2, \ldots, X_n$ are i.i.d. random variables with distribution $\mu$.  We make the following assumptions concerning $\mu$.  

\begin{assumption}\label{ass:ComplexMu}
	Suppose $\mu$ is a radially symmetric distribution supported on the unit disk centered at the origin in the complex plane that satisfies the following properties:
	\begin{enumerate}[(\thetheorem.i)]
		\item \label{ass:density} there exists $\eps \in (0, 1]$ so that in the annulus 
		\begin{equation} \label{def:Aeps}
		\mathbb{A}_\eps := \set{z \in \C: 1-\eps \leq \abs{z} \leq 1}
		\end{equation}
		$\mu$ has a density $f_\mu(z)$ with respect to the Lebesgue measure on $\C$;
		
		\item \label{ass:radial} there are constants $c_\mu, C_\mu >0$, and $\alpha > -1$  so that the radial density \[f_R(r):= 2\pi r f_\mu(re^{\sqrt{-1}\cdot\theta}),\ 1-\eps \leq r \leq 1\] satisfies
		\[
		c_\mu \leq \frac{f_R(r)}{(1-r)^\alpha} \leq C_\mu.
		\]
	\end{enumerate}
\end{assumption}

\begin{remark}
When \ref{ass:density} and \ref{ass:radial} hold for a radially symmetric distribution $\mu$, we use $F_R(r):=\mu\left(\set{z \in \mathbb{C}: \abs{z} \leq r}\right)$ to denote the radial cumulative distribution function (c.d.f.) of $\mu$. In this case, for $1-\eps \leq r \leq 1$, we have
\[
\frac{c_\mu}{\alpha+1}(1-r)^{\alpha+1} \leq \int_r^1f_R(t)\,dt \leq \frac{C_\mu}{\alpha+1}(1-r)^{\alpha+1},
\]
so 
\begin{equation}\label{eqn:RadCdfBds}
	1- \frac{C_\mu}{\alpha + 1}(1-r)^{\alpha + 1} \leq F_R(r) \leq 1-\frac{c_\mu}{\alpha+1}(1-r)^{\alpha + 1}.
\end{equation} 
Via Lemma \ref{lem:StielCalc}, we also have
\begin{equation}\label{eqn:RadStiel}
m_\mu(z) = \frac{F_R(\abs{z})}{z},\ \text{for $z \in \mathbb{C}\setminus\set{0}$},
\end{equation}
where $m_\mu$ is defined in \eqref{def:st}. 
\end{remark}

\begin{example}
 We note that for any $\alpha > -1$, if $\mu$ is a radially symmetric distribution with radial density $f_R(r) = (\alpha +1)(1-r)^\alpha$, $r\in [0,1)$, then $\mu$ satisfies Assumption \ref{ass:ComplexMu}. 
\end{example}

A few remarks concerning Assumption \ref{ass:ComplexMu} are in order.  First, we have assumed $\mu$ is supported on the unit disk centered at the origin.  One can easily consider cases where $\mu$ is supported on any other disk by simply scaling and shifting.  Second, the radial symmetry assumption is likely not required.  We include this assumption since it simplifies\footnote{Consider e.g.\ the proofs of Lemmas \ref{lem:fewRts}, \ref{lem:GnSmall}, \ref{lem:heavyTail}, \ref{lem:moments} and of Corollary \ref{cor:CLT}, where we integrate using polar coordinates, and the statements and proofs of Lemmas \ref{lem:sepRts} and \ref{lem:sepRtsNeg} that allow us to more easily establish the radial separation of the largest $X^\da_{(i)}$ that we use to achieve results like Equations \eqref{eqn:maxPair}, \eqref{eqn:maxPairSharper}, \eqref{eqn:maxPairNeg}, and \eqref{eqn:maxPairSharperNeg}.} our (already technical) proofs. For example, Lemma \ref{lem:StielCalc} in Appendix \ref{sec:appendix} allows us to easily compute the Cauchy--Stieltjes transform $m_{\mu}$ for radially symmetric measures $\mu$, and we use this explicit formula, \eqref{eqn:RadStiel}, in our proofs of Lemmas \ref{lem:manyPair} and \ref{lem:manyPairConcNeg} and of Corollary \ref{cor:CSnice}. The radial symmetry also gives us independence of the angular parts of the (dependent) order statistics $X^\da_{(i)}$, a fact we rely on when we prove the fluctuations results, Theorems \ref{thm:maxPair} and \ref{thm:fluctNeg} (see Lemmas \ref{lem:CLT} and \ref{lem:LLNCLTNeg} and their proofs). Finally, we remark that the exponent $\alpha$ determines the behavior of the largest (in magnitude) roots of $p_n$.  For positive values of $\alpha$, the largest roots are farther from the edge of the disk than when $\alpha$ is zero or negative.  This dependence on $\alpha$ can be seen in Theorem \ref{thm:gauss} below, which summarizes our main results. We note that the condition $-0.095<\alpha < 0$, which also appears in Theorems \ref{thm:manyPairNeg} and \ref{thm:fluctNeg}, is an artifact of our proof.  It can likely be improved slightly by optimizing some of the parameters in Table \ref{table:params} and some of our proofs (see e.g.\ Lemma \ref{lem:badEventSmallNeg}, which requires the strictest bound on $\alpha < 0$).  We conjecture that our main results should hold as long as $-1 < \alpha < 0$, however, proving such a thing will likely require new methods.  

\begin{theorem}[User-Friendly Main Result]\label{thm:gauss}
	Suppose $X_1, X_2, \ldots$ are i.i.d.\ draws from a distribution $\mu$ that satisfies Assumption \ref{ass:ComplexMu} with $\alpha > -0.095$. With probability $1-o(1)$, all critical points of the polynomial $p_n(z) := \prod_{j=1}^n(z-X_j)$ lie within a disk centered at the origin of radius at most
	\begin{equation}\label{eqn:asymRefine}
		\abs{X^\da_{(1)}}(1-n^{-1}) + o(n^{-1}) < 1-n^{-1},
	\end{equation}
	and if we define 
	\[
	\mathfrak{a}_n :=\begin{cases}
		n^{\frac{3}{2}} & \text{if $\alpha > 0$}\\
		\frac{n^{\frac{3}{2}}}{\log{n}}& \text{if $\alpha = 0$}\\[6pt]
		n^{\frac{3+2\alpha}{2 + \alpha}} & \text{if $-0.095 < \alpha < 0$},
	\end{cases}
	\] then, provided $\lim_{r \to 1^-}\frac{f_R(r)}{(1-r)^\alpha}$ exists when $\alpha \leq 0$, we have
	\begin{equation}\label{eqn:asymRefineConvergence}
		\frac{\mathfrak{a}_n}{e^{\sqrt{-1}\arg(X^\da_{(1)})}}\cdot\left(W^\da_{(1)} - X^\da_{(1)}(1-n^{-1})\right) \to \begin{cases}  N_1 &\text{if $\alpha \geq 0$}\\  \mathcal{H}_{2+\alpha}&\text{if $-0.095 < \alpha < 0$},
		\end{cases}
	\end{equation}
	in distribution as $n\to \infty$, where $W^\da_{(1)}$ and $X^\da_{(1)}$ denote the largest critical point and root of $p_n$ in magnitude, $ N_1$ has a complex Gaussian distribution with mean zero and covariance structure given by \eqref{eqn:CovStructurePos} below, and $\mathcal{H}_{2+\alpha}$ is the complex-valued $(2+\alpha)$-stable random variable described in Theorem \ref{thm:fluctNeg}. 
\end{theorem}

This theorem immediately follows from Theorems \ref{thm:maxPair}, \ref{thm:manyPairNeg}, and \ref{thm:fluctNeg} below. Equation \eqref{eqn:asymRefine} could be thought of as an asymptotic refinement of the Gauss--Lucas theorem.
Indeed, the Gauss--Lucas theorem asserts that the critical points of a polynomial are always contained in the convex hull of its roots, so in particular, the critical points of $p_n(z)$ lie in  a disk of radius $\abs{X^\da_{(1)}}$, centered at the origin.
Theorem \ref{thm:gauss} shows that, with high probability, the critical points are contained in a disk of radius $\abs{X^\da_{(1)}}(1-n^{-1}) + o(n^{-1})$, centered at the origin. 
In other words, for random polynomials, when compared to the roots, the critical points move inwards by a factor of $n^{-1}$.  
In fact, we conjecture that each derivative should move the roots inward by a factor of $n^{-1}$, so that the zeros of the $k$-th derivative will be (roughly) a distance of $k/n$ from the original roots of the polynomial. 

Equation \eqref{eqn:asymRefineConvergence} details the asymptotic fluctuations of the largest critical point about its predicted location from \eqref{eqn:asymRefine}. Note the very different behavior in the fluctuations based on whether $\alpha \geq 0$ or $\alpha < 0$. Our main results below give a more detailed picture and several generalizations.

\begin{theorem}[Pairing when $\alpha \geq 0$]\label{thm:manyPair}
Suppose $X_1, X_2, \ldots$ are i.i.d.\ draws from a distribution $\mu$ that satisfies Assumption \ref{ass:ComplexMu} with $\alpha \geq 0$, fix $\delta \in \left(\frac{1}{4\alpha +3}, \frac{1}{\alpha +1}\right)$, and suppose $c_n =\omega(1)$ is a positive sequence satisfying $\log(c_n) = o(\log{n})$. Then there is a constant $C> 0$ so that with probability $1-O(c_n^{-1})$, the following are true statements concerning the polynomial  $p_n(z) := \prod_{j=1}^n(z-X_j)$ and its roots and critical points that lie in the annulus
\[
\mathcal{A}_n := \set{z\in \C : 1-\frac{1}{n^\delta} \leq \abs{z} \leq 1}:
\]
\begin{enumerate}[(\thetheorem.i)]	
	\item \label{item:enoughCpts} There are at least $\frac{n^{1-\delta(\alpha+1)}}{c_n\log{n}}$ critical points of $p_n$ in $\mathcal{A}_n$.
	\item \label{item:rtsToCp} Within distance $n^{-(5+\delta)/6}$ of each root $X_i \in \mathcal{A}_n$, $i \in [n]$, there is precisely one critical point $W_i^{(n)}$ of $p_n$, and these critical points satisfy
	\begin{equation}\label{eqn:rtsToCptsBd}
	\max_{i: X_i \in \mathcal{A}_n}\abs{W_i^{(n)} - X_i + \frac{1}{n} \frac{1}{\frac{1}{n-1}\sum_{j\neq i}\frac{1}{X_i -X_j}}} < \frac{1}{n^{(3-\delta)/2}}
	\end{equation}
	and
	\begin{equation}\label{eqn:rtsToCptsBd2}
		\max_{i: X_i \in \mathcal{A}_n}\abs{W_i^{(n)} - X_i\left(1-n^{-1}\right)}< \frac{C}{nc_n}.
	\end{equation}
	\item \label{item:iota} There is an injection $\iota$ from the set of critical points of $p_n$ that lie in $\mathcal{A}_n$ to the set of indices $[n]$ of the roots of $p_n$, so that each critical point $W \in \mathcal{A}_n$ corresponds to precisely one root $X_{\iota(W)} \in \mathcal{A}_n$ according to the relationship
	\begin{equation}\label{eqn:cptsToRtsBd}
	 	W = X_{\iota(W)}\left(1-n^{-1} + o(n^{-1})\right).
	\end{equation}
	More specifically, we have
	\begin{equation}\label{eqn:cptsToRtsUniformBd}
	\max_{{\rm c.p.}\; W \in \mathcal{A}_n}\abs{W - X_{\iota(W)}\left(1-n^{-1}\right)}< \frac{C}{nc_n},
	\end{equation}
	so the asymptotic notation in \eqref{eqn:cptsToRtsBd} is uniform over all critical points $W \in \mathcal{A}_n$.
\end{enumerate}
\end{theorem}
We note that for $\alpha \geq 0$, Theorem \ref{thm:manyPair} gives detailed information about pairings between the largest $n^{1-\delta(\alpha+1)}$ roots and critical points of $p_n$ (up to slowly growing factors involving $\log{n}$), where $1/(4\alpha+3) < \delta < 1/(\alpha+1)$. Using the lower bound on $\delta$, in the case $\alpha = 0$, we can control around $n^{2/3}$ root-and-critical-point pairs, and for large $\alpha > 0$, we can control nearly $n^{3/4}$ pairings. While the upper bound on $\delta$ is sharp in the sense that we expect the largest roots to be within a distance $n^{-1/(\alpha+1)}$ of the unit circle (see e.g.\ Lemma \ref{lem:fewRts} for details), the lower bound restriction, that $\delta > 1/(4\alpha+3)$, is likely an artifact of our method of proof in Section \ref{sec:proofsPos} below, and this lower bound could almost certainly be improved. It is challenging to do so, however, because widening the annulus $\mathcal{A}_n$ means, among other things, that one has to establish (simultaneous) regularity of the sums $\frac{1}{n-1}\sum_{j\neq i}\frac{1}{X_i-X_j}$ for an increasingly large number of points $X_i$. We direct the reader to the arguments in Section \ref{sec:proofsPos} for details.

For $\alpha \geq 0$, the next theorem concerns the joint behavior of the largest (in magnitude) $\ell_n = o(\sqrt[16]{\log{n}})$ critical points of $p_n$, including the joint fluctuations of the largest $L$ critical points, where $L$ is any fixed positive integer. Notice that the behavior of these critical points is slightly different in the regimes $\alpha = 0$ and $\alpha > 0$ because the variables $\frac{1}{z-X_j}$, $\abs{z} = 1$, have up to (but not including) $2 + \alpha$ moments. Consequently, the $\alpha = 0$ case requires a heavy-tailed central limit theorem involving the extra scaling factor $a_n = \sqrt{\log{n}}$ that appears in Theorem \ref{thm:maxPair} below. 
\begin{theorem}[Joint fluctuations of largest critical points when $\alpha \geq 0$]\label{thm:maxPair}
	Suppose $X_1, X_2, \ldots$ are i.i.d.\ draws from a distribution $\mu$ that satisfies Assumption \ref{ass:ComplexMu} with $\alpha\geq 0$, let $\ell_n =\omega(1)$ be a sequence of positive integers with $\ell_n = o(\sqrt[16]{\log{n}})$, and in order of decreasing magnitude, denote the largest $\ell_n$ critical points of $p_n := \prod_{j=1}^n(z-X_j)$ by $W^\da_{(1)}, \ldots, W^\da_{(\ell_n)}$.  Then, there is a constant $C> 0$ so that with probability $1-O(\ell_n^{-1})$,
	\begin{equation}\label{eqn:maxPair}
		\max_{1\leq i \leq \ell_n}\abs{W^\da_{(i)} - X^\da_{(i)}(1-n^{-1})} < \frac{C}{n\ell_n^4},
	\end{equation}
	and 
	\begin{equation}\label{eqn:maxPairSharper}
		\max_{1\leq i \leq \ell_n}\abs{W^\da_{(i)} - X^\da_{(i)} + \frac{1}{n} \frac{1}{\frac{1}{n-1}\sum_{\substack{j=1\\ j\neq i}}^n\frac{1}{X^\da_{(i)} -X^\da_{(j)}}}} < C\ell_n^8\cdot n^{-\frac{3}{2}-\frac{\alpha^2}{2(1+\alpha)^2}},
	\end{equation}
	and for each $i \in [\ell_n]$, $W^\da_{(i)}$ is the unique critical point of $p_n$ that is within a distance $n^{-11/12}$ of $X^\da_{(i)}$. In addition, provided $\lim_{\abs{z} \to 1^-} f_\mu(z) = f_\mu(1)$ in the case $\alpha =0$, if $L$ is any fixed positive integer and
	\[
	a_n := \begin{cases}
		\sqrt{\log{n}} & \text{if $\alpha = 0$,}\\1 &\text{if $\alpha > 0$},
	\end{cases}
	\]
	we have
	\begin{equation}\label{eqn:maxConv}
		\left(\frac{n^{3/2}}{a_ne^{\sqrt{-1}\arg(X^\da_{(i)})}}\left(W^\da_{(i)} - X^\da_{(i)}(1-n^{-1})\right)\right)_{i=1}^L \longrightarrow( N_1, \ldots,  N_L)
	\end{equation}
	in distribution as $n\to \infty$, where each coordinate $ N_i$, $1 \leq i \leq L$ has a complex Gaussian marginal distribution with mean zero and covariance structure given by
	\begin{equation}\label{eqn:CovStructurePos}
		\begin{aligned}
		\var\left(\Re( N_i)\right) &=
			\begin{cases}
				 \frac{\pi f_\mu(1)}{4} & \text{if $\alpha =0$,}\\
				\var\left(\Re\left(\frac{X_1}{1-X_1}\right)\right) & \text{if $\alpha > 0$};
			\end{cases}\\
		\var\left(\Im( N_i)\right) &= 
			\begin{cases}
				\frac{\pi f_\mu(1)}{4} & \text{if $\alpha =0$,}\\
				\var\left(\Im\left(\frac{X_1}{1-X_1}\right)\right) & \text{if $\alpha > 0$;}
			\end{cases}\\
		\cov\left(\Re( N_i),\Im( N_i)\right) &= 0. 
		\end{aligned}
	\end{equation}
	When $\alpha = 0$, the vector $( N_1, \ldots,  N_L)$ has i.i.d.\ coordinates.  When $\alpha > 0$, the joint distribution of $( N_1, \ldots,  N_L)$ can be described by a compound Gaussian distribution; see Lemma \ref{lem:CLT} for details.   
\end{theorem}

Even though there are differences between the $\alpha = 0$ and $\alpha > 0$ cases, the fluctuations obtained in Theorem \ref{thm:maxPair} for the largest critical point are always Gaussian.  This contrasts with the behavior of the largest root, whose asymptotic fluctuations are described by a max-stable distribution.  
We now present our results for the $\alpha < 0$ case, where we see significant differences compared to the $\alpha \geq 0$ cases.

\begin{theorem}[Pairing when $\alpha < 0$]\label{thm:manyPairNeg}
	Suppose $X_1, X_2, \ldots$ are i.i.d.\ draws from a distribution $\mu$ that satisfies Assumption \ref{ass:ComplexMu} with $-0.095<\alpha < 0$, let $\ell_n =\omega(1)$ be a sequence of positive integers with $\ell_n < \sqrt{\log{n}}$, and fix $\delta \in (0, -\alpha)$. Then, there is a constant $C> 0$ so that with probability $1-O(\ell_n^{-1})$, the following are true statements concerning the polynomial $p_n := \prod_{j=1}^n(z-X_j)$ and its roots and critical points that lie in the annulus
	\[
	\mathcal{A}_n^- := \set{z\in \mathbb{C} : 1-\frac{\ell_n}{n} \leq \abs{z} \leq 1}:
	\]	
	\begin{enumerate}[(\thetheorem.i)]	
		\item \label{item:enoughCptsNeg} There are at least $n^\delta$ critical points of $p_n$ in $\mathcal{A}_n^-$.
		\item \label{item:rtsToCpNeg} Within distance $3/n$ of each root $X_i \in \mathcal{A}_n^-$, $i \in [n]$, there is precisely one critical point $W_i^{(n)}$ of $p_n$, and these critical points satisfy
		\begin{equation}\label{eqn:rtsToCptsBdNeg}
			\max_{i: X_i \in \mathcal{A}_n^-}\abs{W_i^{(n)} - X_i + \frac{1}{n} \frac{1}{\frac{1}{n-1}\sum_{j\neq i}\frac{1}{X_i -X_j}}} <C\ell_nn^{-\frac{5\alpha +3}{2(\alpha+1)}}
		\end{equation}
		and
		\begin{equation}\label{eqn:rtsToCptsBd2Neg}
			\max_{i: X_i \in \mathcal{A}_n^-}\abs{W_i^{(n)} - X_i\left(1-n^{-1}\right)}< C\left(\frac{1}{\ell_n^4n}
			\right)^{1/(\alpha +1)}.
		\end{equation}
		\item \label{item:iotaNeg} There is an injection $\iota$ from the set of critical points of $p_n$ that lie in $\mathcal{A}_n^-$ to the set of indices $[n]$ of the roots of $p_n$, so that each critical point $W \in \mathcal{A}_n^-$ corresponds to precisely one root $X_{\iota(W)} \in \mathcal{A}_n^-$ according to the relationship
		\begin{equation}\label{eqn:cptsToRtsBdNeg}
			W = X_{\iota(W)}\left(1-n^{-1} + o(n^{-1})\right).
		\end{equation}
		More specifically, we have
		\begin{equation}\label{eqn:cptsToRtsUniformBdNeg}
			\max_{{\rm c.p.}\; W \in \mathcal{A}_n^-}\abs{W - X_{\iota(W)}\left(1-n^{-1}\right)}< C\left(\frac{1}{\ell_n^4n}
			\right)^{1/(\alpha +1)},
		\end{equation}
		so the asymptotic notation in \eqref{eqn:cptsToRtsBdNeg} is uniform over critical points $W \in \mathcal{A}_n^-$.
		
		\item \label{item:maxPairNeg} The largest $\ell_n$ roots of $p_n$ satisfy
		\begin{equation}\label{eqn:maxPairRtsInNeg}
			1-\left(\frac{\ell_n^2}{n}\right)^{\frac{1}{\alpha +1}} \leq \abs{X^{\da}_{(i)}} \leq 1 - \left(\frac{1}{n\ell_n}\right)^{\frac{1}{\alpha +1}},\ 1 \leq i \leq \ell_n,
		\end{equation} and, denoting the largest $\ell_n$ critical points of $p_n$ in order of decreasing magnitude by $W^\da_{(1)}, \ldots, W^\da_{(\ell_n)}$, we have
		\begin{equation}\label{eqn:maxPairNeg}
			\max_{1\leq i \leq \ell_n}\abs{W^\da_{(i)} - X^\da_{(i)}(1-n^{-1})} < C\left(\frac{1}{\ell_n^4n}
			\right)^{1/(\alpha +1)};
		\end{equation} 
		\begin{equation}\label{eqn:maxPairSharperNeg}
			\max_{1\leq i \leq \ell_n}\abs{W^\da_{(i)} - X^\da_{(i)} + \frac{1}{n} \frac{1}{\frac{1}{n-1}\sum_{\substack{j=1\\ j\neq i}}^n\frac{1}{X^\da_{(i)} -X^\da_{(j)}}}} < C\ell_nn^{-\frac{5\alpha +3}{2(\alpha+1)}};
		\end{equation}
		and for each $i \in \set{1, \ldots, \ell_n}$, $W^\da_{(i)}$ is the unique critical point of $p_n$ that is within a distance $3/n$ of $X^\da_{(i)}$.		
	\end{enumerate}
\end{theorem}

\begin{theorem}[Fluctuations of largest critical point when $\alpha < 0$]\label{thm:fluctNeg}
	Suppose $X_1, X_2, \ldots$ are i.i.d.\ draws from a distribution $\mu$ that satisfies Assumption \ref{ass:ComplexMu} with $-0.095<\alpha < 0$ and associated radial density $f_R(r)$, and denote the largest in magnitude critical point of $p_n := \prod_{j=1}^n(z-X_j)$ by $W^\da_{(1)}$. Then, provided $\lim_{r \to 1^-}\frac{f_R(r)}{(1-r)^\alpha}$ exists, we have
	\begin{equation}\label{eqn:maxConvNeg}
		\frac{n^{\frac{3+2\alpha}{2 + \alpha}}}{e^{\sqrt{-1}\arg(X^\da_{(1)})}}\left(W^\da_{(1)} - X^\da_{(1)}(1-n^{-1})\right) \longrightarrow \mathcal{H}_{2+\alpha}
	\end{equation}
	in distribution as $n\to \infty$, where $\mathcal{H}_{2+\alpha}$ is a complex-valued random variable, which, by identifying $\mathbb{C}$ with $\mathbb{R}^2$, is a multivariate $(2+\alpha)$-stable random vector with spectral measure described by the right-hand side of \eqref{eqn:tailAngle}. 
\end{theorem}

We make a number of comments concerning these results.  In contrast to Theorem \ref{thm:maxPair}, the asymptotic fluctuations in the case when $\alpha < 0$ are described by heavy-tailed stable distributions, and we refer the reader to \cite{MR1280932,MR150795,MR270403,MR1652283,MR4230105,MR62975} and references therein for more details concerning stable distributions.  One way to view the behavior of these two cases is the following: the asymptotic fluctuations of the difference $W^\da_{(1)} - X^\da_{(1)}$ are always described by stable distributions (either Gaussian or heavy-tailed), while the behavior of the largest root is always described by a max-stable distribution.  This might be surprising at first, but our proof shows how we can express the difference $W^\da_{(1)} - X^\da_{(1)}$ as a sum of (nearly-)independent random variables, up to some negligible errors.  
In fact, Theorems \ref{thm:maxPair} and \ref{thm:manyPairNeg} show how the difference $W^\da_{(1)} - X^\da_{(1)}$ roughly looks like $1/\xi_n$, where $\xi_n$ behaves like a sum of (nearly-)independent random variables.  This intuition can also explain the scaling factor in Theorem \ref{thm:fluctNeg}.  In this case, we can write the exponent as
\[ \frac{3 + 2\alpha}{2 + \alpha} = 2 - \frac{1}{2 + \alpha}. \]
If $\xi_n$ is a random variable with mean on the order of $n$ and fluctuations on the order of $n^{1/(2 + \alpha)}$, as is standard for heavy-tailed distributions (which is what arises when $\alpha < 0$), then $1/\xi_n$ has fluctuations on the order of $\E[\xi_n]^2 / n^{1/(2 + \alpha)}$, which is precisely the scaling exponent we find.  The scaling in Theorem \ref{thm:maxPair} can be similarly explained since, when $\alpha > 0$, $\xi_n$ has mean on the order of $n$, but fluctuations on the order of $n^{1/2}$ as the fluctuations follow from the classical central limit theorem.  The $\alpha = 0$ case is similar, but an additional $\sqrt{\log n}$ factor arises.  

As already discussed above, there are many results concerning the critical points of random polynomials with independent roots.  Many of the early results focus on the global behavior of the critical points; see, for example, \cite{K,PR,S,O} and some of the other references cited above.  In recent years, the pairing between critical points and roots has been studied more extensively \cite{H1,H2,H3,OW,OW2,KS}.  
The main results in this paper are somewhat similar as they also describe the pairing between critical points and roots.  For example, the results of \cite{KS,OW2} describe the fluctuations between a fixed (or conditioned on) root and the nearest critical points.  In \cite{OW2}, this pairing is quantified to show that the Wasserstein distance between the empirical distributions of roots and critical points is $O(n^{-1})$, up to logarithmic corrections.  
While these results focus on the pairing phenomenon, similar to the pairing described in our main results above, they only allow for the study of a ``typical'' or ``bulk'' root and its nearest critical point; we cannot deduce behavior about the fluctuations of the extreme critical points from them.  

In contrast, our main results instead focus on the pairing between the largest roots and the nearby critical points.  We are not aware of other results which deal with the pairing between the extreme roots and critical points for polynomials with independent roots in the same way.  

Our method to proving the main results above is based loosely on the approach from \cite{OW2} (see also \cite{KS}).
However, the methods developed in \cite{OW2} are designed for studying a fixed root and the nearest critical point, not the largest roots and critical points.  
As such, our approach requires significant technical changes and several new obstacles must be overcome.  
The starting point for our proof is a deterministic result from \cite{OW2} (see Theorem \ref{thm:determ} below).
This result shows that, under some technical conditions, the roots and critical points (in certain regions of the unit disk) pair very closely.  
Unfortunately, showing that the technical conditions are satisfied for the largest ordered roots and critical points requires overcoming significant new technical hurdles which were not present in \cite{OW2}.  
In addition, since the ordered roots are no longer independent, even showing that the sums of such values converge to limiting distributions requires a more delicate analysis.  
Due to the number of technical obstacles that must be overcome, we have not tried to optimize our approach to the widest range of values for $\alpha$, nor have we attempted to generalize our results beyond the rotationally symmetric case. 

Lastly, our main results show that there is different behavior between the fluctuations of the critical points and the fluctuations of the roots; similar differences between the second-order behavior of critical points and roots has been observed in other models before.  For example, the results in \cite{feldheim2024hyperuniformitynonhyperuniformityzerosgaussian} for certain models of Gaussian analytic functions show that the critical points exhibit non-hyperuniformity, while the roots display hyperuniform behavior.  

The rest of the paper is devoted to the proofs of our main results.  Theorems \ref{thm:manyPair} and \ref{thm:maxPair} are proven in Section \ref{sec:proofsPos}, and  Theorems \ref{thm:manyPairNeg} and \ref{thm:fluctNeg} are established in Section \ref{sec:proofsNeg}.


\section{Proof of results when $\alpha \geq 0$} \label{sec:proofsPos}


\subsection{Heuristic overview}\label{sec:heuristicPos} In this subsection, we attempt to give the reader some big-picture insight into the estimates that feature prominently in our proofs of Theorems \ref{thm:manyPair} and \ref{thm:maxPair} governing the case when $\alpha \geq 0$. Subsection \ref{sec:detailsPos} below contains a more detailed introduction to our plan of attack that requires some technical notation we wish to avoid at present.

In the case where $\alpha \geq 0$, one can use Theorem \ref{thm:manyPair} to determine, with high probability, the precise locations of nearly $n^{\frac{3\alpha + 2}{4\alpha + 3}}$ of the largest critical points of $p_n(z) = \prod_{j=1}^n(z-X_j)$ in relation to the largest roots of $p_n$. It seems natural (to the authors at least) to present these results by describing the behavior of roots and critical points of $p_n$ that lie within an annulus $\mathcal{A}_n \subset \mathbb{A}_\eps$  at the edge of the unit disk. However, to include as many as $n^{\frac{3\alpha + 2}{4\alpha + 3}}$ roots of $p_n$, such an annulus must be wide enough that the one-to-one correspondence $\iota$ mentioned in \ref{item:iota} is not a bijection. Indeed, with high probability, there are more roots than critical points in $\mathcal{A}_n$ because those roots of smallest magnitude (i.e.\ the ones nearest the ``inner'' edge of $\mathcal{A}_n$) are paired with critical points that have strictly smaller magnitude lying just across the boundary of $\mathcal{A}_n$ (see for example the root and critical point pair straddling the boundary of $\mathcal{A}_n$ at the very top of Figure \ref{fig:paramPic}). For this reason, we have chosen to present the main conclusions of Theorem \ref{thm:manyPair} above in two parts: \ref{item:iota} that pairs critical points to roots and \ref{item:rtsToCp} that pairs roots to critical points. Accordingly, our proofs rely on two different deterministic statements,  namely Equation \eqref{eqn:GLPair} and Theorem \ref{thm:determ}, that we introduce now.

To establish that large critical points pair to large roots, we will begin with the well-known fact that if $W$ is a critical point of $p_n$, then either $W$ is a root of $p_n$ having multiplicity greater than one, or $0= \sum_{j=1}^n\frac{1}{W-X_j}$.\footnote{Indeed, ${p'(W)}/{p(W)} = 0$ implies $p'(W)=0$. See e.g.\ \S1.2 of \cite[pp.\ 6--7]{M} for more details.} When $X_1, X_2, \ldots$ are i.i.d.\ draws from a distribution satisfying Assumption \ref{ass:ComplexMu}, the $X_i$ lying in $\mathbb{A}_\eps$ are almost surely distinct (such a distribution restricted to $\mathbb{A}_\eps$ is absolutely continuous with respect to the Lebesgue measure on $\C$), so we have
\begin{equation}\label{eqn:GLPair}
W - X_i = - \frac{1}{\sum_{j\neq i}\frac{1}{W-X_j}}\ \text{for $W \in \mathbb{A}_\eps$ with $p'_n(W) = 0$, $i \in [n]$}.
\end{equation}
We will control the right side of \eqref{eqn:GLPair} for critical points $W$ that are near the unit circle and their closest neighbors $X_i$, $1\leq i \leq n$. The dependence between the critical points and roots of $p_n(z)$ poses obvious challenges to analyzing the sums $\sum_{j\neq i} \frac{1}{W-X_j}$, but a law of large numbers heuristic suggests approximating these with $n\cdot m_\mu(W)$, which is $O(n)$ with high probability since the Cauchy--Stieltjes transform $m_\mu(z)$ is bounded and Lipschitz continuous on $\mathbb{A}_\eps$ (see Corollary \ref{cor:CSnice}). After justifying the validity of this intuition, we will have established that with high probability, $W-X_i = O(1/n)$. (Note: our analysis will give the slightly sharper asymptotics \eqref{eqn:cptsToRtsBd} and \eqref{eqn:cptsToRtsUniformBd}.)

In order to show that $p_n$ has a critical point near each of its largest roots, we will appeal to the following theorem for deterministic polynomials from \cite{OW2}. Roughly speaking, this result says that any root $\xi$ of a complex, degree $n+1$ polynomial that is sufficiently isolated from the other roots $\zeta_1, \ldots, \zeta_n$ must pair to a nearby critical point $w_\xi^{(n)}$ of the polynomial.
\begin{theorem}[Theorem 3.1 in \cite{OW2}]\label{thm:determ}
	Suppose $\xi \in \C$, $\vec{\zeta} = (\zeta_1, \ldots, \zeta_n)$ is a vector of complex numbers, and $C_1$, $C_2$, $k_{\rm Lip}$ are values for which the following three conditions hold:
	\begin{enumerate}[(\thetheorem.i)]
		\item \label{item:det1}$C_1 \leq \abs{\frac{1}{n}\sum_{j=1}^n\frac{1}{\xi-\zeta_j}}\leq C_2$;
		\item \label{item:det2} The function $z\mapsto \frac{1}{n}\sum_{j=1}^n\frac{1}{z-\zeta_j}$ is Lipschitz continuous with constant $k_{\rm Lip}$ on the set $\set{z\in\C:\abs{z-\xi}\leq \frac{2}{C_1n}}$;
		\item \label{item:det3} $\displaystyle\min_{1\leq j\leq n}\abs{\xi-\zeta_j} > \frac{3}{C_1n}$.
	\end{enumerate}
	Then, if $C > 0$ and $n \in \mathbb{N}$ satisfy
	\[
	C > \frac{8(1+2C_2^2)}{C_1^3}\ \text{and}\ n> 4C_2\max\set{\frac{1}{C_1},\ C(k_{\rm Lip} + 1)},
	\]
	the polynomial $q_n(z):=(z-\xi)\prod_{j=1}^n(z-\zeta_j)$ has exactly one critical point, $w_\xi^{(n)}$, that is within a distance of $\frac{3}{2C_1n}$ of $\xi$, and
	\[
	\abs{w_\xi^{(n)} - \xi + \frac{1}{n+1}\frac{1}{\frac{1}{n}\sum_{j=1}^n\frac{1}{\xi-\zeta_j}}} < \frac{C(k_{\rm Lip} + 1)}{n^2}.
	\]
\end{theorem}
We intend to apply Theorem \ref{thm:determ} several times at once, with each of the largest roots of $p_n$ taking turns playing the role of $\xi$. In each instance, we will verify condition \ref{item:det1} by approximating sums of the form $\frac{1}{n}\sum_{j\neq i}\frac{1}{X_i-X_j}$ with $m_\mu(X_i)$ for $X_i$ near the unit circle (this estimate seems reasonable in view of the law of large numbers). Condition \ref{item:det3} will follow from Assumption \ref{ass:ComplexMu}, which guarantees that the largest roots are isolated from the rest of the $X_i$ with high probability. We defer our discussion of condition \ref{item:det2} to Subsection \ref{sec:detailsPos} because of the technical details involved.
 
We conclude this subsection with some notes about the proof of Theorem \ref{thm:maxPair}. This result concerns the largest roots of $p_n$ where, with high probability, the radial separation between these $X_i$ is asymptotically wider than the margin of error we predict for the locations of the largest critical points of $p_n$ in \eqref{eqn:maxPair} and \eqref{eqn:maxPairSharper}. Consequently, the one-to-one correspondence $\iota$ from \ref{item:iota} gives an order-preserving bijection between the largest $o(\sqrt[16]{\log{n}})$ roots and critical points of $p_n$  (i.e., the critical point of largest magnitude pairs with the root of largest magnitude, the critical point of second largest magnitude pairs with the root of second largest magnitude, and so on). For a fixed natural number $L$, we can then use \eqref{eqn:maxPairSharper} to obtain the joint fluctuations of the largest $L$ critical points about their predicted locations by analyzing the sums $\frac{1}{n-1}\sum_{j\neq i}\frac{1}{X^{\da}_{(i)}-X^\da_{(j)}}$ for $ 1 \leq i \leq L$. The dependence between $X^\da_{(1)},\ldots, X^\da_{(n)}$ inhibits the direct application of a central limit theorem, but, since the largest $X^\da_{(i)}$ are so close to the unit circle, we can replace each of these with the values $e^{\sqrt{-1}\arg(X^\da_{(i)})}$ on the unit circle that have matching arguments. The radial symmetry of the distribution $\mu$ guarantees that for $1 \leq i \leq L$, these unit vectors are i.i.d.\ and jointly independent from the $\sigma$-algebra generated by the rest of the roots. Thus, modulo some technical details, as $n \to \infty$, the sums $\frac{1}{n-1}\sum_{j\neq i}\frac{1}{X^\da_{(i)}-X^\da_{(j)}}$ behave in distribution like the sums $\frac{1}{n}\sum_{j=1}^n\frac{1}{U_i - X_j}$, where $U_1, \ldots, U_L$ are i.i.d.\ draws from the unit circle that are independent from $X_1, X_2, \ldots$. In the case $\alpha > 0$, we conclude our analysis with standard central limit theorem arguments. When $\alpha = 0$, we appeal to Corollary \ref{cor:CLT}, a version of the central limit theorem for heavy-tailed random variables.

In the next subsection, we introduce some technical notation and a more detailed outline of our proofs of Theorems \ref{thm:manyPair} and \ref{thm:maxPair}. 
 
 
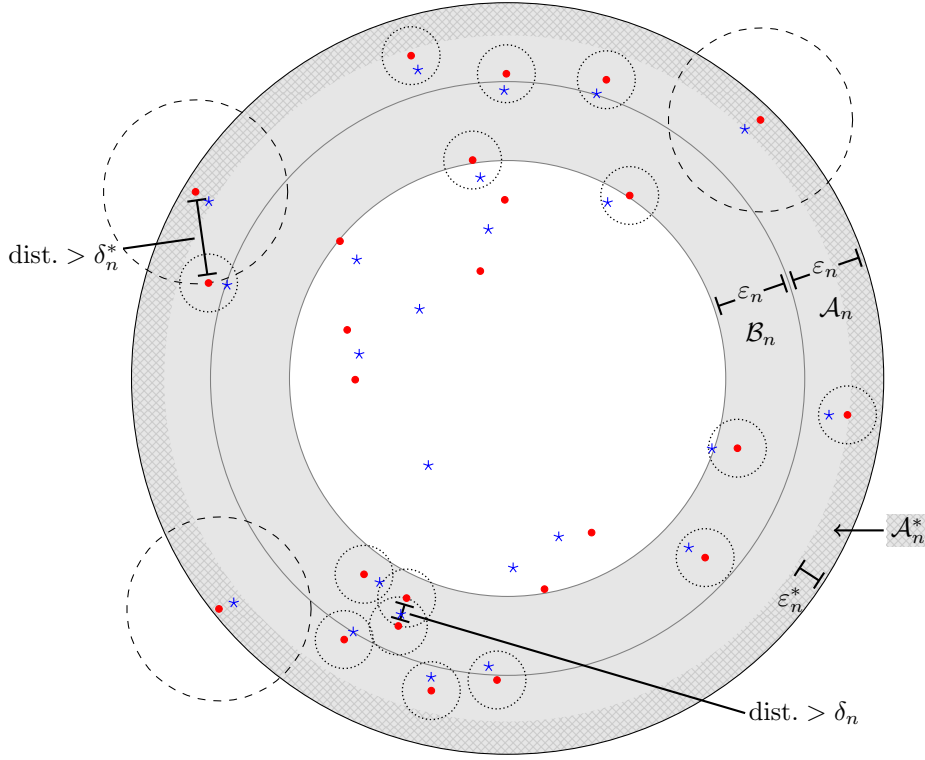
\begin{figure}
 	\begin{center}
 		
 		
 		\pgfmathsetmacro{\myEps}{.21}
 		\pgfmathsetmacro{\myEpsStar}{.087}
 		\pgfmathsetmacro{\medRad}{.077}
 		\pgfmathsetmacro{\maxRad}{.245}
 		
 		
 		\pgfplotstableread{datasetRts.txt}\rtsList  
 		\pgfplotstableread{datasetCpts.txt}\cptsList
 		\pgfplotstableread{datasetMedRts.txt}\medRtsList 
 		\pgfplotstableread{datasetMaxRts.txt}\maxRtsList 
 		\pgfplotstablegetrowsof{\rtsList} 
 		\pgfmathsetmacro{\numRts}{\pgfplotsretval} 
 		\pgfplotstablegetrowsof{\medRtsList} 
 		\pgfmathsetmacro{\numMed}{\pgfplotsretval-1} 
 		\pgfplotstablegetrowsof{\maxRtsList} 
 		\pgfmathsetmacro{\numMax}{\pgfplotsretval-1} 

 		\begin{tikzpicture} 			
 			\begin{axis}[axis lines = none, x=5cm, y=5cm,xmin=-1.8, xmax = 1.8]
 				\fill[color=gray!20] (axis cs: 0,0) circle[radius =1]; 
 				\draw[pattern = crosshatch, pattern color = gray!40] (axis cs: 0,0) circle[radius =1];
 				\fill[color = gray!20] (axis cs: 0,0) circle[radius = {1-\myEpsStar}]; 
 				\draw[color = gray] (axis cs: 0,0) circle[radius = {1-\myEps}]; 
 				\draw[color = gray, fill = white] (axis cs: 0,0) circle[radius = {1-(2*\myEps)}]; 
 				
 				\addplot[only marks, color = red, mark size = 1.25pt] table[x=Re, y=Im] from \rtsList;
 				\addplot[only marks, color = blue, mark size = 2pt, mark = star] table[x=Re, y=Im] from \cptsList;
 				
 				\foreach \i in {0,1,...,\numMed}{
 					\pgfplotstablegetelem{\i}{[index]0}\of\medRtsList
 					\pgfmathsetmacro{\rtx}{\pgfplotsretval} 
 					\pgfplotstablegetelem{\i}{[index]1}\of\medRtsList
 					\pgfmathsetmacro{\rty}{\pgfplotsretval} 
 					\addplot[domain = 0:2*pi,samples = 50, densely dotted, line width = 0.5pt]({\rtx+\medRad*cos(deg(x))},{\rty+\medRad*sin(deg(x))});
 				}
 				
 				\foreach \i in {0,1,...,\numMax}{
 					\pgfplotstablegetelem{\i}{[index]0}\of\maxRtsList
 					\pgfmathsetmacro{\rtx}{\pgfplotsretval} 
 					\pgfplotstablegetelem{\i}{[index]1}\of\maxRtsList
 					\pgfmathsetmacro{\rty}{\pgfplotsretval} 
 					\addplot[domain = 0:2*pi,samples = 50, dashed]({\rtx+\maxRad*cos(deg(x))},{\rty+\maxRad*sin(deg(x))});
 				}

 				\pgfmathsetmacro{\epsAngle}{19}

 				\draw ({cos(\epsAngle-9)*(1-1.5*\myEps)}, {sin(\epsAngle-9)*(1-1.5*\myEps)}) node {$\mathcal{B}_n$};
 				\draw ({cos(\epsAngle-7)*(1-0.5*\myEps)}, {sin(\epsAngle-7)*(1-0.5*\myEps)}) node  {$\mathcal{A}_n$};

 				\draw[|-, thick] ({cos(\epsAngle)*(1.01-2*\myEps)}, {sin(\epsAngle)*(1.01-2*\myEps)}) -- ({cos(\epsAngle)*(.955-1.5*\myEps)}, {sin(\epsAngle)*(.955-1.5*\myEps)});
 				\draw[-|, thick] ({cos(\epsAngle)*(1.035-1.5*\myEps)}, {sin(\epsAngle)*(1.035-1.5*\myEps)}) -- ({cos(\epsAngle)*(0.99-\myEps)}, {sin(\epsAngle)*(0.99-\myEps)});
 				\draw ({cos(\epsAngle)*(1-1.5*\myEps)}, {sin(\epsAngle)*(1-1.5*\myEps)}) node{$\eps_n$};
 				
 				\draw[|-, thick] ({cos(\epsAngle)*(1.01-\myEps)}, {sin(\epsAngle)*(1.01-\myEps)}) -- ({cos(\epsAngle)*(.955-0.5*\myEps)}, {sin(\epsAngle)*(.955-0.5*\myEps)});
 				\draw[-|, thick] ({cos(\epsAngle)*(1.035-.5*\myEps)}, {sin(\epsAngle)*(1.035-.5*\myEps)})-- ({cos(\epsAngle)*0.99}, {sin(\epsAngle)*0.99});			
 				\draw ({cos(\epsAngle)*(1-0.5*\myEps)}, {sin(\epsAngle)*(1-0.5*\myEps)}) node {$\eps_n$};

 				\pgfmathsetmacro{\epsStarAngle}{-38}
 				
 				\draw[<-, thick] ({cos(\epsStarAngle+13)*(1-0.5*\myEpsStar)}, {sin(\epsStarAngle+13)*(1-0.5*\myEpsStar)}) -- (1,{sin(\epsStarAngle+13)*(1-0.5*\myEpsStar)});
 				\draw ({cos(\epsStarAngle+13)*(1-0.5*\myEpsStar)+.14},{sin(\epsStarAngle+13)*(1-0.5*\myEpsStar)}) node[right, inner sep = 1.5pt, fill =gray!20] {\phantom{$\mathcal{A}_n^*$}} node[right, inner sep = 1.5pt, pattern color =gray!40, pattern =crosshatch] {\phantom{$\mathcal{A}_n^*$}} node[right, inner sep = 1.5pt] {$\mathcal{A}_n^*$};

 				\draw[|-|, thick] ({cos(\epsStarAngle+5)*(1.01-\myEpsStar)}, {sin(\epsStarAngle+5)*(1.01-\myEpsStar)}) -- ({cos(\epsStarAngle+5)*0.99}, {sin(\epsStarAngle+5)*0.99});		
 				\draw ({cos(\epsStarAngle)*(1-0.5*\myEpsStar)}, {sin(\epsStarAngle)*(1-0.5*\myEpsStar)}) node {$\eps_n^*$};

 				\pgfmathsetmacro{\cosAngStar}{1/sqrt(1 + ((0.4967 - 0.2541)/(-0.8297 + 0.7952))*((0.4967 - 0.2541)/(-0.8297 + 0.7952)))}
 				\pgfmathsetmacro{\sinAngStar}{(0.4967 - 0.2541)/(-0.8297 + 0.7952)/sqrt(1 + ((0.4967 - 0.2541)/(-0.8297 + 0.7952))*((0.4967 - 0.2541)/(-0.8297 + 0.7952)))}
 				
 				\draw[|-|, thick] ({-0.8297 + 0.02*\cosAngStar}, {0.4967 + 0.02*\sinAngStar}) -- ({-0.8297 + (\maxRad-0.02)*\cosAngStar}, {0.4967 + (\maxRad-0.02)*\sinAngStar});
 				
 				\draw[thick] ({-0.8297 + (\maxRad/2)*\cosAngStar + .02*\sinAngStar}, {0.4967 + (\maxRad/2)*\sinAngStar - .02*\cosAngStar}) -- ({-0.8297 + (\maxRad/2)*\cosAngStar + 0.37*\sinAngStar}, {0.4967 + (\maxRad/2)*\sinAngStar - 0.37*\cosAngStar}) node [fill=white, inner sep = 1pt] {$ \text{dist.} > \delta^*_n$};
 				
 				\pgfmathsetmacro{\cosAng}{1/sqrt(1 + ((-0.6586 + 0.5844)/(-0.2906 + 0.2687))*((-0.6586 + 0.5844)/(-0.2906 + 0.2687)))}
 				\pgfmathsetmacro{\sinAng}{(-0.6586 + 0.5844)/(-0.2906 + 0.2687)/sqrt(1 + ((-0.6586 + 0.5844)/(-0.2906 + 0.2687))*((-0.6586 + 0.5844)/(-0.2906 + 0.2687)))}
 				
 				\draw[|-|, thick] ({-0.2906 + 0.015*\cosAng}, {-0.6586 + 0.015*\sinAng}) -- ({-0.2906 + (\medRad-.015)*\cosAng}, {-0.6586 + (\medRad-.015)*\sinAng});
 				
 				\draw[thick] ({(-0.2906 + - 0.2687)/2 + .02*\sinAng},{(-0.6586 + -0.5844)/2 -.02*\cosAng}) -- ({(-0.2906 + - 0.2687)/2 + .95*\sinAng},{(-0.6586 + -0.5844)/2 - .95*\cosAng }) node [fill=white, inner sep = 1pt, right] {$\text{dist.} > \delta_n$};
 			\end{axis}			
 		\end{tikzpicture}
 		\caption{This graphic depicts parameters relevant to our proofs of Theorems \ref{thm:manyPair} and \ref{thm:maxPair} in the context of simulated data. The red dots and blue stars mark the respective locations of the roots and critical points of a random degree-$25$ polynomial whose i.i.d.\ roots were chosen uniformly at random from the unit disk. The large, dashed circles have radius greater than $\delta_n^*$, and the smaller, dotted circles have radius greater than $\delta_n$. See also Table \ref{table:params}. Note that this simulation result was chosen to illustrate behavior that may not occur with high probability until $n$ is much larger than $25$. \label{fig:paramPic}}
 	\end{center}
\end{figure}

\subsection{Detailed overview to motivate definitions of key parameters and events}\label{sec:detailsPos}
Much of our heuristic intuition in the previous subsection is predicated on precise control of sums that have the form $\sum_{j}\frac{1}{z - X_j}$ for different quantities $z\in \mathbb{C}$ near the unit circle, some of which depend on the $X_j$. In order to contend with all of these sums at once, we analyze behavior that is uniform in $z$ of a discrete analog to the Cauchy--Stieltjes transform, namely \begin{equation}\label{eqn:discreteCS} 
\overline{M}_\mu(z) := \frac{1}{n}\sum_{j\neq i_z}\frac{1}{z-X_j},
\end{equation}
where for each $z\in \mathbb{C}$, the random index \begin{equation}\label{eqn:indexMap}
	i_z := \min\set{i\in [n]: \abs{z-X_i} = \min_{j\in[n]}\abs{z-X_j}} 
\end{equation}
specifies a root from among $X_1, \ldots, X_n$ that is closest to $z$. We will show that on the complement of some ``bad'' events having asymptotically negligible probability as $n \to \infty$, $\overline{M}_\mu(z)$  is Lipschitz continuous near the unit circle, with an implied constant that grows with $n$ in a way we can manage (this is the content of Lemma \ref{lem:nearLip}). Furthermore, we will use Markov's inequality and the union bound to establish that with high probability, $\overline{M}_\mu(z)$ is near $m_\mu(z)$ on a deterministic (growing in $n$) net of points near the unit circle (see Lemma \ref{lem:netProb}). These two results will have the following important consequences that, up to some technical bookkeeping, imply the conclusions of Theorem \ref{thm:manyPair}.
\begin{itemize}
	\item On the complement of the ``bad'' events, we will be able to approximate $\overline{M}_\mu(z)$ with $m_\mu(z)$ for \textit{any} $z$ near the unit circle due to the Lipschitz continuity of $\overline{M}_\mu(z)$, the near Lipschitz continuity of $m_\mu$ (see Corollary \ref{cor:CSnice}), and proximity of these two functions on the deterministic net. This is the content of Lemma (\ref{lem:nearCS})
	
	\item We will be able to verify that off of the ``bad'' events, conditions \ref{item:det1} and \ref{item:det2} of Theorem \ref{thm:determ} hold when we consider, simultaneously, each of the largest $X_1,\ldots, X_n$ in the role of $\xi$.	
	
\end{itemize}
Observe that $z \mapsto \overline{M}_\mu(z)$ behaves more erratically in regions where several $X_j$ are clumped together because in these spots, the denominators of the terms $(z-X_j)^{-1}$ have the potential to be quite large. For this reason, we track the Lipschitz constant associated with $\overline{M}_\mu(z)$ in terms of parameters $\delta_n, \delta_n^* >0$ (defined precisely below, see also Figure \ref{fig:paramPic} and Table \ref{table:params}) related to the minimum spacing between the largest in magnitude $X_j$, $1 \leq j \leq n$. We also consider how the Lipschitz constant depends on $\eps_n, \eps_n^* >0$ (defined precisely below, see also Figure \ref{fig:paramPic} and Table \ref{table:params}), parameters that control the width of the annuli near the edge of the unit disk on which we're analyzing $z\mapsto \overline{M}_\mu(z)$. Here, we are motivated by the fact that with high probability, there are only $O(n\eps_n^{\alpha +1})$ many roots $X_j$, $1 \leq j \leq n$ in an annulus of width $\eps_n > 0$ near the edge of the unit disk (see Lemma \ref{lem:fewRts}), so for increasingly small $\eps_n$, the roots that do exist in this region are increasingly isolated. When proving the fluctuations result Theorem \ref{thm:maxPair}, we will analyze the separation of $X_j$ in two regimes, those within $\eps_n$ of the unit circle and those within a smaller distance $\eps^*_n$ of the unit circle, in order to get a sharper estimate on the Lipschitz constant associated with $z \mapsto \overline{M}_\mu(z)$.

In the next subsection, we give precise definitions of the parameters $\delta_n$, $\delta_n^*$, $\eps_n$, and $\eps^*$, and the ``bad'' events, off of which the largest in magnitude $X_j$, $1\leq j \leq n$ are sufficiently isolated to guarantee that $z \mapsto \overline{M}_\mu(z)$ is Lipschitz continuous with a constant that does not grow too fast with $n$. In Subsection \ref{sec:CSBehaved}, we will find bounds on the Lipschitz constant associated with $z\mapsto\overline{M}_\mu(z)$ and establish that off of the ``bad'' events, $\overline{M}_\mu(z)$ is close to $m_\mu(z)$ on $\mathcal{A}_n$. Then, we will prove Theorem \ref{thm:manyPair} in Subsections \ref{sec:manyPairConcHold}, \ref{sec:badEventsSmall}, and \ref{sec:manyPairPf} by showing its conclusions hold off of the ``bad'' events, finding bounds on the probabilities of the ``bad'' events, and showing that for optimized choices of the parameters $\eps_n$, $\delta_n$, the probabilities of the ``bad'' events tend to zero as $n \to \infty$. We will conclude our justification of the $\alpha \geq 0$ results by proving Theorem \ref{thm:maxPair} in Subsection \ref{sec:maxPairPf}.


\begin{table}
	\caption{This table gives explicit definitions, in terms of $\alpha$ and $n$, for the parameters we use in our proofs. Heuristically speaking, $\eps_n, \eps_n^*$ control the widths of the annuli where we expect to see the largest roots (we expect the largest root $X^{\da}_{(1)}$ to lie within a distance $\eps_n^*$ of the unit circle), and $\delta_n, \delta_n^*$, respectively, control the separation distance between the roots within these annuli (see also Figure \ref{fig:paramPic}). The sequences $l_n$ and $\ell_n$ control the number of roots for which we guarantee a unique critical point pair (in the case of Theorem \ref{thm:manyPairNeg}, we can guarantee pairing for more than $\ell_n$ roots but for $\ell_n$-many, we have the sharper asymptotics \eqref{eqn:maxPairNeg} and \eqref{eqn:maxPairSharperNeg}). Finally,  the slowly growing sequence $c_n$ is a tuning parameter that allows us to correct for powers of $\log{n}$ in our asymptotic analysis.\label{table:params}}
	
	\renewcommand{\arraystretch}{1.65}
	\begin{tabular}{|c|c|c|c|c|c|c|}
		\cline{4-7}
		\multicolumn{3}{c|}{} & \multicolumn{4}{|c|}{\textbf{Values by Subsection}}\\ \cline{4-7}
		\multicolumn{3}{c|}{} & Subsection \mbox{\ref{sec:manyPairConcHold}\strut} & Subsection \mbox{\ref{sec:maxPairPf}\strut} & Subsection \mbox{\ref{sec:paramsNeg}\strut}& Subsection \mbox{\ref{sec:fluctNegProof}\strut} \\[-6pt]
		\multicolumn{3}{c|}{} & (Thm.\ \mbox{\ref{thm:manyPair}\strut} Pf.) & (Thm.\ \mbox{\ref{thm:maxPair}\strut} Pf.) & (Thm.\ \mbox{\ref{thm:manyPairNeg}\strut} Pf.) & (Thm.\ \mbox{\ref{thm:fluctNeg}\strut} Pf.) \\ \cline{4-7}
		\multicolumn{3}{c|}{} &\multicolumn{2}{|c|}{$\alpha \geq 0$} & \multicolumn{2}{|c|}{$-0.095 < \alpha < 0$} \\ \hline
		\multirow[c]{8}{*}{\rotatebox[origin=c]{90}{\textbf{Parameters}\hspace*{.65cm}}}&\multirow[c]{3}{*}{\rotatebox[origin=c]{90}{$\omega(1)$\hspace{2mm}}} & $l_n$   & $\frac{n^{1-\delta(\alpha+1)}}{c_n\log{n}}$ & - & - & - \\
		&& $\ell_n$ & - & $o(\sqrt[16]{\log{n}})$ & $\leq \sqrt{\log{n}}$ & -\\
		&& $c_n$   & $e^{o(\log{n})}$ & $\ell_n^4$ & - & $<\sqrt{\log{n}}$ \\\cline{2-7}&&&&&&\\[-16pt]
		&\multirow[c]{4}{*}{\rotatebox[origin=c]{90}{$o(1)$\hspace*{.5cm}}}& $\eps_n$    & $\frac{1}{n^\delta}$ & $\frac{1}{\ell_n^4}\!\left(\!\frac{1}{n^{\frac{1+2\alpha}{2+2\alpha}}}\!\right)^{\!\frac{1}{\alpha+1}}$ & $\left(\frac{1}{n}\right)^{\frac{1}{2(\alpha+1)}}$ & $\left(\frac{1}{n}\right)^{\frac{(\alpha+1)^2}{(1-\alpha)(\alpha+2)}}$\\[6pt]
		&& $\delta_n$  & $\frac{(c_n^7\log{n})^{3/4}}{n^{3/4+\delta/4}}$ & $\left(\!\frac{1}{n^{\frac{3+4\alpha}{4+4\alpha}}}\!\right)^{\!\frac{1}{\alpha+1}}$ & $\log{n}\left(\frac{1}{n}\right)^{\frac{\alpha+3}{4(\alpha+1)}}$ & $\frac{c_n}{n}$\\[6pt]
		&& $\delta_n^*$&-& $ \frac{1}{\ell_n^3}\left(\frac{1}{\sqrt{n}}\right)^{\frac{1}{\alpha+1}}$ &-& $\left(\frac{1}{c_n^3n}\right)^{\frac{1}{\alpha+2}}$\\[3pt]
		&& $\eps_n^*$  &-& $\left(\frac{\ell_n^2}{n}\right)^{\frac{1}{\alpha+1}}$ & $\left(\frac{\ell_n^2}{n}\right)^{\frac{1}{\alpha+1}}$ & $\left(\frac{c_n^2}{n}\right)^{\frac{1}{\alpha+1}}$\\[3pt]\hline
		\multicolumn{3}{c|}{}& \!$\delta \in \big(\frac{1}{4\alpha+3},\frac{1}{\alpha+1}\big)$\!  & - & - & -\\\cline{4-7}
	\end{tabular}
\end{table}

\subsection{Important parameters and events}\label{sec:param}

Following the conventions in \cite{OW,OW2}, we identify a collection of ``bad'' events on whose complement our main results hold, and we show that these ``bad'' events have negligible probability as $n$ grows to infinity. For the reasons mentioned in the previous subsection, we allow the ``bad'' events to depend on the parameters $\eps_n, \eps_n^*, \delta_n, \delta_n^* \in (0, \eps)$ and $c_n, \ell_n \in \mathbb{Z}^+$ that we optimize in Subsections \ref{sec:manyPairConcHold}, \ref{sec:badEventsSmall}, \ref{sec:manyPairPf}, and \ref{sec:maxPairPf} below. The lengths $\eps_n$, $\eps_n^*$ are the widths of the annuli on which we seek to describe extremal root and critical point pairing, and $\delta_n$, $\delta_n^*$ are lower bounds on the distances between the roots of $p_n$ in these annuli and the rest of the $X_j$, $1\leq j \leq n$ (see Table \ref{table:params} and, for a graphical depiction, see Figure \ref{fig:paramPic}).  The positive sequences $c_n$ and $\ell_n$, defined among the hypotheses of our main theorems above, are slowly growing integer sequences that we use to sharpen a number of asymptotic inequalities. We have chosen to allow some of the parameters and definitions in this section to permit $\alpha < 0$. This will allow us to recycle some of the $\alpha \geq 0$ arguments when we prove the $\alpha < 0$ results in Section \ref{sec:proofsNeg} below.

Define the annuli $\mathcal{A}_n$ and $\mathcal{A}^*_n$, where we will analyze root and critical point pairs mentioned in Theorems \ref{thm:manyPair} and \ref{thm:maxPair}, respectively:
\begin{align*}
\mathcal{A}_n &:= \set{z\in \C: 1 - \eps_n \leq \abs{z} \leq 1},\\
\mathcal{A}^*_n &:=\set{z\in \C: 1 - \eps^*_n \leq \abs{z} \leq 1};
\intertext{and the annulus $\mathcal{B}_n$, in which we will keep track of the separation of $X_j$ for our estimation of the Lipschitz constant associated with $z \mapsto \overline{M}_\mu(z)$:}
\mathcal{B}_n &:= \set{z\in \C: 1 - 2\eps_n \leq \abs{z} < 1-\eps_n}.
\end{align*}
See also Figure \ref{fig:paramPic} and Table \ref{table:params}.

As foreshadowed in Subsection \ref{sec:detailsPos}, we will show that off of some ``bad'' events defined below, $\overline{M}_\mu(z)$ is close to $m_\mu(z)$ on a net of deterministic points $\mathcal{N}_n$ of $\mathcal{A}_n$ that grows with $n$. To this end, use e.g.\ Lemma \ref{lem:epsNet} to find nets $\mathcal{N}_n$ of $\mathcal{A}_n$, depending on $n$, which satisfy\footnote{Note: We will choose $c_n > 1$ and $\delta_n < \eps_n$, so that $\eps_n^{1/3}\delta_n^{2/3}c_n^{-2}$ is less than $\eps_n = 1-(1-\eps_n)$ and the hypotheses of Lemma \ref{lem:epsNet} are met.}
\begin{equation}
	\max_{z\in \mathcal{A}_n}\left(\min_{w\in \mathcal{N}_n}\abs{z-w}\right) \leq \frac{\eps_n^{{1}/{3}}\delta_n^{{2}/{3}}}{c_n^2}\quad \text{and}\quad \abs{\mathcal{N}_n} \ll \frac{c_n^4}{\eps_n^{{2}/{3}}\delta_n^{{4}/{3}}}\cdot 1 \cdot \eps_n = \frac{c_n^4\eps_n^{{1}/{3}}}{\delta_n^{{4}/{3}}}.
	\label{eqn:netParams}
\end{equation}

We now specify the ``bad'' events off of which $\overline{M}_\mu(z)$ is close to $m_\mu(z)$ and $z\mapsto \overline{M}_\mu(z)$ has a manageable Lipschitz constant:
{\allowdisplaybreaks
\begin{align*}
	E_n &:= \set{\#\set{i \in [n]: X_i \in \mathcal{A}_n\cup \mathcal{B}_n} > 3C_\mu n(2\eps_n)^{\alpha+1}}\\ 
		&\phantom{:={}}\textit{``There are too many $X_i$ in $\mathcal{A}_n \cup \mathcal{B}_n$''}\\[1ex]		
	F_n &:= \set{\exists\ i,j\in [n],\ i\neq j: X_i \in \mathcal{A}_n \cup \mathcal{B}_n,\ \abs{X_i-X_j} \leq \delta_n}\\
		&\phantom{:={}}\textit{``The $X_i$ in $\mathcal{A}_n \cup \mathcal{B}_n$ aren't sufficiently isolated.''}\\[1ex]
	F^*_n &:= \set{\exists\ i,j\in [n],\ i\neq j: X_i \in \mathcal{A}^*_n,\ \abs{X_{i}-X_{j}} \leq \delta^*_n}\\
		&\phantom{:={}}\cup \begin{cases} \emptyset & \text{if $\alpha > 0$,}\\
			\begin{aligned}
			\Big\{&\exists\ i\in [n]: X_i \in \mathcal{A}^*_n,\\
			&\#\set{j\in [n],\ j \neq i: \abs{X_i-X_j} < \eps_n} > c_n^3n\eps_n^{\alpha+2} \Big\}
			\end{aligned}  &\text{if $\alpha \leq 0$.}\\
		\end{cases}\\
		&\phantom{:={}}\textit{``The largest roots (those in $\mathcal{A}_n^*$) aren't sufficiently isolated.''}\\[1ex]
	G_n &:= \set{\sum_{i=1}^n\frac{1}{(1-\frac{3}{2}\eps_n -\abs{X_i})^2}\cdot \sind{\abs{X_i}< 1 -2\eps_n}\geq nc_n\cdot \max\set{\eps_n^{\alpha-1}, \eps_n^{-1}}}\\
		&\phantom{:={}}\begin{minipage}{.85\linewidth}\textit{``The  $X_i \notin \mathcal{A}_n \cup \mathcal{B}_n$ are too big (too close to the $X_j \in \mathcal{A}_n$) on average.''}\end{minipage}\\[1ex]
	H_n &:= \set{\max_{z\in \mathcal{N}_n}\abs{\frac{1}{n}\sum_{i=1}^n\frac{1}{z-X_i}\sind{\abs{z-X_i}> \frac{\delta_n}{2}}-m_\mu(z)}\geq \frac{1}{c_n}}\\
		&\phantom{:={}}\begin{minipage}{.85\linewidth}\textit{``$\overline{M}_\mu(z)$ does not uniformly concentrate around $m_\mu(z)$ at the rate $1/c_n$ for points in the net $\mathcal{N}_n$.''}\end{minipage}
\end{align*}
}

In the next subsection, we show that for a wide range of values of the parameters $\eps_n$, $\eps_n^*$, $\delta_n$, $\delta_n^*$, on the complement of $E_n \cup F_n \cup F_n^*\cup G_n \cup H_n$, the discrete approximation $\overline{M}_\mu(z)$ to the Cauchy--Stieltjes transform $m_\mu(z)$ is Lipschitz continuous and close to $m_\mu(z)$ on the annulus $\mathcal{A}_n$.


\subsection{On the complement of the ``bad'' events, $\overline{M}_\mu(z)$ is Lipschitz continuous and well approximated by $m_\mu(z)$ near the unit circle.}\label{sec:CSBehaved}
	As discussed in Subsection \ref{sec:detailsPos}, we are motivated by \eqref{eqn:GLPair} and conditions \ref{item:det1}, \ref{item:det2} of Theorem \ref{thm:determ} to show that $\overline{M}_\mu(z)$ is  Lipschitz continuous on $\mathcal{A}_n$ and may be approximated (uniformly for $z \in \mathcal{A}_n$) by $m_\mu(z)$. Lemma \ref{lem:nearLip} establishes upper bounds in terms of $\eps_n$, $\delta_n$, $\delta_n^*$ for the Lipschitz constants associated with $\overline{M}_\mu(z)$ for $z \in \mathcal{A}_n$ and $z \in \mathcal{A}_n^*$. The proof leverages the heuristic illustrated in Figure \ref{fig:paramPic}, namely that $X_i$ that lie increasingly near the edge of the unit disk are more likely to be isolated from other $X_i$ (by virtue of the small area of the annuli we are considering and the fact that $\mu$ has a density on $\mathbb{A}_\eps$). We achieve the sharpest bounds (for use in the proof of Theorem \ref{thm:maxPair} in Section \ref{sec:maxPairPf}) by considering $X_i$ in several regimes ($\abs{X_i} < 1-2\eps_n$, $X_i \in \mathcal{A}_n \cup \mathcal{B}_n$, and $X_i \in \mathcal{A}^*_n$). 
	
	Lemma \ref{lem:nearCS} gives an upper bound, in terms of $\eps_n$, $\delta_n$, $c_n$ and uniform over $z \in \mathcal{A}_n$, for the approximation error when estimating $\overline{M}_\mu(z)$ with $m_\mu(z)$. In the proof, we interpolate between $\overline{M}_\mu(z)$ and $m_\mu(z)$ using $\overline{M}_\mu(\xi_z)$ and $m_\mu(\xi_z)$, where $\xi_z$ is a point in the net $\mathcal{N}_n$ near to $z$. On the complement of $H_n$, $\overline{M}_\mu(\xi_z)$ and $m_\mu(\xi_z)$ are close to each other, and the Lipschitz continuity of $\overline{M}_\mu(z)$ (see Lemma \ref{lem:nearLip}) and of $m_\mu(z)$ (see Corollary \ref{cor:CSnice}) justify approximating $\overline{M}_\mu(z)$ and $m_\mu(z)$ with $\overline{M}_\mu(\xi_z)$ and $m_\mu(\xi_z)$, respectively.
	
	\begin{lemma}\label{lem:nearLip}
		Suppose $c_n > 0$, $\frac{1}{n} < \delta_n  \leq \frac{\eps_n}{2} < \frac{1}{4}$, and $n \geq 3$.	On the complement of $E_n\cup F_n\cup G_n$, for all $z,w\in \mathcal{A}_n\cup B(0,1)^c$, the differences $\abs{\overline{M}_\mu(z) - \overline{M}_\mu(w)}$ are finite, and 
		\begin{equation}
			\abs{\overline{M}_\mu(z) - \overline{M}_\mu(w)} < \abs{z-w}\left(c_n\cdot\max\set{\eps_n^{\alpha-1}, \eps_n^{-1}} + \frac{172\log{n}}{n\delta_n^2} \right) + \frac{8}{n\delta_n}.
			\label{eqn:nearLip}
		\end{equation}	
		The following special cases yield sharper bounds:
		\begin{enumerate}[(\thetheorem.i)]
			\item On the complement of $E_n\cup F_n \cup G_n$, when $z,w \in B(X_j, \frac{\delta_n}{2})$ for some $X_j \in \mathcal{A}_n$, $j\in [n]$, then \eqref{eqn:nearLip} holds and the last term, $\frac{8}{n\delta_n}$, can be replaced with $0$. (Note: In this case, it's possible to have $z,w \in B(X_j, \frac{\delta_n}{2}) \cap \mathcal{B}_n$.) \label{it:spCase1}
			\item Suppose $\alpha \leq 0$ and $\frac{1}{n}<\delta_n \leq \delta^*_n < 1$. On the complement of $E_n\cup F_n \cup F_n^* \cup G_n$, when $z, w\in \mathcal{A}_n \cup B(0,1)^c$ are in the open ball $B(X_j, \frac{\delta_n}{2})$ for some $X_j \in \mathcal{A}_n^*$, $j \in [n]$, then we have
			\begin{equation}
			\abs{\overline{M}_\mu(z) - \overline{M}_\mu(w)} < \abs{z-w}\left(\left(c_n+24C_\mu\right)\eps_n^{\alpha-1} + \frac{4c_n^3\eps_n^{\alpha + 2}}{\left(\delta^*_n\right)^2}\right).
			\label{eqn:spCase2}
			\end{equation}
			\label{it:spCase2}
		\end{enumerate}
	\end{lemma}
	\begin{proof}
		We work on the complement of the event $E_n\cup F_n\cup G_n$, and our strategy will be to break the sums comprising $\overline{M}_\mu(z) - \overline{M}_\mu(w)$ into pieces based on the proximity of each $X_i$ to $z$ and $w$.	First, notice that on the complement of $F_n$, all $X_i \in \mathcal{A}_n \cup \mathcal{B}_n$ are $\delta_n$-separated, so either $z$ with $\abs{z} > 1-\frac{3}{2}\eps_n$ is farther than $\delta_n/2$ from every $X_i$, $i \in [n]$, or there is precisely one $X_i$, namely $X_{i_z}$, for which $\abs{X_{i_z} - z} \leq \delta_n/2$. (Recall [see \eqref{eqn:indexMap}] that for $z\in \mathbb{C}$, the random index $i_z$ specifies a particular $X_{i_z}$ that is closest to $z$.) It follows that for $z$ with $\abs{z} > 1-\frac{3}{2}\eps_n$,
		\begin{equation}\label{eqn:onlyOneClose}
			\abs{\frac{1}{n}\sum_{j\neq i_z}\frac{1}{z-X_j}-\frac{1}{n}\sum_{i=1}^n\frac{1}{z-X_i}\sind{\abs{z-X_i}> \frac{\delta_n}{2}}} = \abs{\frac{1}{n}\cdot\frac{\sind{\abs{z-X_{i_z}}>\frac{\delta_n}{2}}}{z-X_{i_z}}}.
		\end{equation}
		Identical reasoning shows that \eqref{eqn:onlyOneClose} holds with $w$, $\abs{w} > 1-\frac{3}{2}\eps_n$ in place of $z$, so via the triangle inequality, for $z,w$ satisfying $\min\set{\abs{z},\abs{w}} > 1-\frac{3}{2}\eps_n$,
		\begin{align*}
			&\abs{\overline{M}_\mu(z) - \overline{M}_\mu(w)}\\
			&\quad\leq \abs{\frac{1}{n}\sum_{i=1}^n\frac{\sind{\abs{z-X_i}> \frac{\delta_n}{2}}}{z-X_i} - \frac{1}{n}\sum_{i=1}^n\frac{\sind{\abs{w-X_i}> \frac{\delta_n}{2}}}{w-X_i}} + \begin{cases}0&\text{cases \ref{it:spCase1}, \ref{it:spCase2}},\\
			 \frac{4}{n\delta_n}& \text{else}.\end{cases}
		\end{align*}
		
		We now consider the contributions of \[\frac{1}{z-X_i}\sind{\abs{z-X_i}> \frac{\delta_n}{2}} - \frac{1}{w-X_i}\sind{\abs{w-X_i}> \frac{\delta_n}{2}},\ i\in[n],\] based on four regimes that correspond to the quantities $S_n^{\rm I}$, $S_n^{\rm II}$, $S_n^{\rm III}$, and $S_n^{\rm IV}$ defined precisely in the argument below. These regimes have the following intuitive significance:
		\begin{itemize}
			\item \textbf{Regime I: $X_i \notin \mathcal{A}_n\cup \mathcal{B}_n$.} There is at least an $\eps_n$ distance between each of these $X_i$ and both of $z$ and $w$. This is uniform over $z,w$, and we can use moment methods to control this contribution by showing that ``bad'' event $G_n$ is negligible in the limit. (See Lemma \ref{lem:GnSmall} for details.)
			
			\item \textbf{Regime II: $X_i \in \mathcal{A}_n \cup \mathcal{B}_n$, but $\delta_n/2$-far from both $z$ and $w$.}  These $X_i$ could be quite close to $z$ and/or $w$, and we use that the $X_i \in \mathcal{A}_n \cup \mathcal{B}_n$ are $\delta_n$-separated and a geometric ``ball-packing'' argument to overestimate how many $X_i$ can be at distances $\delta_n/2, \delta_n, 2\delta_n, \ldots$, from $z$ and $w$. Since this argument only relies on the fact that off of $F_n$, the $X_i \in \mathcal{A}_n \cup \mathcal{B}_n$ are $\delta_n$-separated (and that $1/n < \delta_n \leq \eps_n/2$), it gives a bound uniform over $z$ and $w$. (For details on how to control the probability of the ``bad'' event $F_n$, see Lemmas \ref{lem:sepRts} and \ref{lem:sepRtsNeg}.) We note that in the proof of special case \ref{it:spCase2}, we use a slightly different method to bound these $X_i$ by also working on the complement of $F_n^*$, which gives us a bit more control of how the $X_i \in \mathcal{A}_n^*$ are separated.
			
			\item \textbf{Regimes III and IV: $X_i$ are within $\delta_n/2$ of precisely one of $w$ or $z$ and at least $d_n/2$-far from the other.} These $X_i$ yield virtually trivial contributions because, on the complement of $F_n$, the $\delta_n$-separation of the $X_i\in \mathcal{A}_n\cup \mathcal{B}_n$ forces at most one such $X_i$ for each of $w$ and $z$. For special cases \ref{it:spCase1} and \ref{it:spCase2}, these regimes are not present.
		\end{itemize}
		By assumption, $\delta_n \leq \eps_n/2$, so for $z,w$ satisfying $\min\set{\abs{z},\abs{w}} > 1-\frac{3}{2}\eps_n$ and any $X_i \notin \mathcal{A}_n \cup \mathcal{B}_n$,
		\[
		\min\set{\abs{z-X_i}, \abs{w-X_i}} > 
		\frac{\eps_n}{2} > \frac{\delta_n}{2},\] and after several more applications of the triangle inequality,
		\begin{equation}
		\abs{\overline{M}_\mu(z) - \overline{M}_\mu(w)} \leq \abs{z-w}\left(\abs{S^{\rm I}_n} + \abs{S^{\rm II}_n}\right) + \abs{S^{\rm III}_n} + \abs{S^{\rm IV}_n} + \begin{cases}
			0 &\text{\ref{it:spCase1}, \ref{it:spCase2}},\\\frac{4}{n\delta_n}& \text{else},\end{cases}
		\label{eqn:DnSplit}
		\end{equation}
		where
		\begin{align*}
			S^{\rm I}_n &:= \frac{1}{n}\sum_{i=1}^n\frac{1}{(z-X_i)(w-X_i)}\cdot\sind{X_j \notin \mathcal{A}_n \cup \mathcal{B}_n},\\
			S^{\rm II}_n &:= \frac{1}{n}\sum_{i=1}^n\frac{\sind{\abs{z-X_i}>\frac{\delta_n}{2}}\cdot\sind{\abs{w-X_i}>\frac{\delta_n}{2}}}{(z-X_i)(w-X_i)}\cdot\sind{X_i \in \mathcal{A}_n \cup \mathcal{B}_n},\\
			S^{\rm III}_n &:= \frac{1}{n}\sum_{i=1}^n\frac{\sind{\abs{z-X_i}>\frac{\delta_n}{2}}\cdot\sind{\abs{w-X_i}\leq\frac{\delta_n}{2}}}{(z-X_i)}\cdot\sind{X_i \in \mathcal{A}_n \cup \mathcal{B}_n},\\
			S^{\rm IV}_n &:= \frac{1}{n}\sum_{i=1}^n\frac{\sind{\abs{w-X_i}>\frac{\delta_n}{2}}\cdot\sind{\abs{z-X_i}\leq\frac{\delta_n}{2}}}{(w-X_i)}\cdot\sind{X_i \in \mathcal{A}_n \cup \mathcal{B}_n}.
		\end{align*}
		On the complement of $G_n$, $S^{\rm I}_n$ is bounded by $c_n\cdot\max\set{\eps_n^{\alpha-1}, \eps_n^{-1}}$, and on the complement of $F_n$, each of the sums $S^{\rm III}_n, S^{\rm IV}_n$ has at most one non-zero term bounded by $\frac{2}{\delta_n}$. These last two sums are $0$ in special cases \ref{it:spCase1} and \ref{it:spCase2}. Consequently, on the complement of $G_n \cup F_n$ and for $z,w$ satisfying $\min\set{\abs{z},\abs{w}} > 1-\frac{3}{2}\eps_n$, Inequality \eqref{eqn:DnSplit} implies
		\begin{equation}
		\abs{\overline{M}_\mu(z) - \overline{M}_\mu(w)} \leq \abs{z-w}\left(c_n\max\set{\eps_n^{\alpha-1}, \eps_n^{-1}} + \abs{S^{\rm II}_n}\right) + \begin{cases}0&\text{\ref{it:spCase1}, \ref{it:spCase2}},\\\frac{8}{n\delta_n}& \text{else}.\end{cases}
		\label{eqn:DnSimple}
		\end{equation}
	
		We conclude the proof by finding two different upper bounds for $\abs{S_n^{\rm II}}$, and the first of these will imply special case \ref{it:spCase2}. On the complement of $F_n^*$,  whenever $z,w \in \mathcal{A}_n \cup B(0,1)^c$ are within a distance $\frac{\delta_n}{2}$ of some $X_j \in \mathcal{A}_n^*$, we have the following for $i \neq j$:
		\begin{equation}\label{eqn:specialTrivialBd}
			\min\set{\abs{z-X_i}, \abs{w-X_i}} \geq \abs{X_i - X_j} - \frac{\delta_n}{2} \geq \begin{cases}
			\frac{\eps_n}{2} &\text{if $\abs{X_i - X_j} \geq \eps_n$,}\\
			\frac{\delta_n^*}{2} & \text{if $\abs{X_i - X_j} < \eps_n$.}
		\end{cases}
		\end{equation}
		When $\alpha \leq 0$, on the complement of $E_n \cup F_n^*$, the sum $S_n^{\rm II}$ has at most $3C_\mu n (2\eps_n)^{\alpha +1 }$ non-zero terms and at most $c_n^3n\eps_n^{\alpha+2}$ of these correspond to $X_i$ within a distance $\eps_n$ of $X_j$ (recall that $F_n^*$ is defined differently for $\alpha \leq 0$ than for $\alpha > 0$). Thus, applying \eqref{eqn:specialTrivialBd} to the summands of $S_n^{\rm II}$ establishes that when $\alpha \leq 0$, on the complement of $E_n \cup F_n^*$, any $z,w \in \mathcal{A}_n \cup B(0,1)^c$ within a distance $\frac{\delta_n}{2}$ of some $X_j \in \mathcal{A}_n^*$ satisfy
		\begin{align*}
		\abs{S_n^{\rm II}} &\leq 3C_\mu n (2\eps_n)^{\alpha +1 }\cdot \frac{1}{n}\cdot \frac{4}{\left(\eps_n\right)^2} + c_n^3n\eps_n^{\alpha+2}\cdot \frac{1}{n}\cdot \frac{4}{\left(\delta^*_n\right)^2}\\
		&= 48C_\mu(2\eps_n)^{\alpha - 1} + \frac{4c_n^3\eps_n^{\alpha + 2}}{\left(\delta^*_n\right)^2},
		\end{align*}
		and Inequality \eqref{eqn:spCase2} follows from \eqref{eqn:DnSimple}.
		
		To prove the main result and special case \ref{it:spCase1}, it remains to show that on the complement of $E_n \cup F_n \cup G_n$,
		\begin{equation}
			\abs{S^{\rm II}_n} < \frac{172\log{n}}{n\delta_n^2},\ \text{for $z,w$ with  $\min\set{\abs{z},\abs{w}} > 1-\frac{3}{2}\eps_n$,}
			\label{eqn:outerDisksLip}
		\end{equation}
		which we will accomplish by using the fact that on $F_n^c$, the $X_i \in \mathcal{A}_n\cup \mathcal{B}_n$ are $\delta_n$-separated. To that end, fix $z \in \mathbb{C}$, and compare the areas of the disjoint disks
		\[
		\set{\xi \in \C: \abs{\xi-X_i} \leq \frac{\delta_n}{2}},\ 1\leq i \leq n,\ X_i \in \mathcal{A}_n\cup \mathcal{B}_n
		\]
		and the annuli
		\[
		\set{\xi \in \C: (j-1)\delta_n \leq \abs{\xi-z} < (j+2)\delta_n},\ j\in \Z^+,\ 0 \leq j \leq \left\lceil 2/\delta_n+1\right\rceil.
		\] 
		(Note that these sets are disks when $j=0,1$.) The area comparison shows that at most $16$ of the $X_i\in \mathcal{A}_n \cup \mathcal{B}_n$ satisfy $\delta_n/2 < \abs{X_i -z} \leq \delta_n$, and for each $j$, $1 \leq j \leq \left\lceil 2/\delta_n+1\right\rceil$, there can be at most $4(6j+3)$ of the $X_i\in \mathcal{A}_n\cup \mathcal{B}_n$ satisfying
		\[
		j\delta_n\leq\abs{X_i-z}\leq (j+1)\delta_n.
		\]
		It follows that on the complement of $F_n$, for $n \geq 3$, $\delta_n > 1/n$, and $z\in \mathbb{C}$, 
		\begin{align*}
			{\frac{1}{n}\sum_{i=1}^n\frac{\sind{\abs{z-X_i}>\frac{\delta_n}{2}}}{\abs{z-X_i}^2}\cdot\sind{X_i \in \mathcal{A}_n \cup \mathcal{B}_n}} &\leq\frac{1}{n}\cdot 16\cdot\frac{4}{\delta_n^2}+\frac{1}{n}\sum_{j=1}^{\left\lceil {2}/{\delta_n}+1\right\rceil}4(6j+3)\frac{1}{j^2\delta_n^2}\\
			&\leq  \frac{1}{n\delta_n^2}\left(64 + 36\sum_{j=1}^{3n}\frac{1}{j}\right)\\
			&< \frac{1}{n\delta_n^2}\left(64 + 36\left(1+ \int_1^{3n}\frac{1}{x}\,dx\right)\right)\\
			&< \frac{172}{n\delta_n^2}\log(n).
		\end{align*}
		(We used that $n \geq 3$ and $\delta_n > 1/n$ to overestimate $\left\lceil {2}/{\delta_n}+1\right\rceil$ with $3n$.) Notice that this bound is uniform over all $z \in \mathbb{C}$ and also holds with $w\in \mathbb{C}$ in place of $z$. Hence, the discrete Cauchy--Schwarz inequality implies $\abs{S^{\rm II}_n} \leq \frac{172\log(n)}{n\delta_n^2}$, and \eqref{eqn:outerDisksLip} then follows from \eqref{eqn:DnSimple}.
	\end{proof}

	Lemma \ref{lem:nearCS} below will allow us to use the Cauchy--Stieltjes transform $m_\mu(z)$ to uniformly approximate $z \mapsto \overline{M}_\mu(z) = \frac{1}{n}\sum_{j\neq i_z}\frac{1}{z-X_j}$ on $\mathcal{A}_n$. (Recall [see \eqref{eqn:indexMap}] that for $z\in \mathbb{C}$, the random index $i_z$ specifies a particular $X_{i_z}$ that is closest to $z$.) Notice that for a fixed $z\in \C$, if we include the ``missing'' summand corresponding to $i_z$, the discrete expression above is equal in expectation to $m_\mu(z)$, so the proposed approximation seems intuitive. As mentioned at the beginning of the current subsection, our proof of Lemma \ref{lem:nearCS} relies on the bounds we just established for the Lipschitz constant associated with $z \mapsto \overline{M}_\mu(z)$ in addition to the Lipschitz continuity of $m_\mu(z)$ on $\mathbb{A}_\eps$. This last fact is the content of Corollary \ref{cor:CSnice}, which we include with proof in the appendix.

\begin{lemma}\label{lem:nearCS}
	Suppose $\eps_n$, $\delta_n$, and $c_n$ are positive sequences satisfying $2\delta_n < \eps_n < \frac{\eps}{2}$ and $1< c_n< n\delta_n$. For large $n$, on the complement of $E_n \cup F_n \cup G_n \cup H_n$,
	\begin{equation}
		\sup_{z\in \mathcal{A}_n}\abs{\overline{M}_\mu(z)-m_\mu(z)} <  \frac{14}{c_n} + \frac{172\log(n)\eps_n^{1/3}}{nc_n^2\delta_n^{4/3}}.
	\label{eqn:nearCS}
	\end{equation}
\end{lemma}
\begin{proof}
	Suppose the complement of $E_n \cup F_n \cup G_n \cup H_n$ occurs, and fix $z\in \mathcal{A}_n$. The hypotheses  $2\delta_n < \eps_n$ and $c_n > 1$ guarantee that the net $\mathcal{N}_n$ satisfies \eqref{eqn:netParams}, and since we assumed $\eps_n < \eps/2$, we know that $\mathcal{N}_n \subset \mathcal{A}_n \subset \mathbb{A}_{{\eps}/{2}} = \set{\xi \in \C: 1-\frac{\eps}{2} \leq \abs{\xi} \leq 1}$. Consequently, on the complement of $H_n$, there is a $\xi_z \in \mathcal{N}_n \subset \mathcal{A}_n$ for which 
	\[
	\abs{z-\xi_z} \leq \frac{\eps_n^{1/3}\delta_n^{2/3}}{c_n^2}\quad \text{and}\quad \abs{\frac{1}{n}\sum_{i=1}^n\frac{1}{\xi_z-X_i}\sind{\abs{\xi_z-X_i}> \frac{\delta_n}{2}}-m_\mu(\xi_z)}< \frac{1}{c_n}.
	\]
	Via \eqref{eqn:onlyOneClose} with $\xi_z$ in place of $z$, the second of these inequalities implies
	\[
	\abs{\frac{1}{n}\sum_{j\neq i_{\xi_z}}\frac{1}{\xi_z-X_j}-m_\mu(\xi_z)}< \frac{1}{c_n} + \frac{2}{n\delta_n} < \frac{3}{c_n},
	\]
	and so by the triangle inequality, \eqref{eqn:nearLip} from Lemma \ref{lem:nearLip}, and Corollary \ref{cor:CSnice}, we have
	\begin{align*}
		&\abs{\overline{M}_\mu(z)-m_\mu(z)}\\
		&\qquad\leq \abs{\frac{1}{n}\sum_{j\neq i_z}\frac{1}{z-X_j}-\frac{1}{n}\sum_{j\neq i_{\xi_z}}\frac{1}{\xi_z-X_j}} + \frac{3}{c_n} + \abs{m_\mu(\xi_z)-m_\mu(z)} \\
		&\qquad\leq\abs{z-\xi_z}\left(\frac{c_n}{\eps_n} + \frac{172\log{n}}{n\delta_n^2}\right) + \frac{8}{n\delta_n} + \frac{3}{c_n} + \kappa_\mu\abs{z-\xi_z},
	\end{align*}
	where $\kappa_\mu >0$ is a constant. 
	In view of the fact that $\abs{z-\xi_z} \leq \frac{\eps_n^{1/3}\delta_n^{2/3}}{c_n^2}$, we have 
	\[
	\abs{\overline{M}_\mu(z)-m_\mu(z)} < \left(\frac{\delta_n}{\eps_n}\right)^{2/3}\cdot\frac{1}{c_n} + \frac{172\log(n)\eps_n^{1/3}}{nc_n^2\delta_n^{4/3}} + \frac{11}{c_n} + \frac{\kappa_\mu\eps_n^{1/3}\delta_n^{2/3}}{c_n^2}.
	\]
	Inequality \eqref{eqn:nearCS} follows for large $n$ since $\delta_n <\eps_n < 1$ and the bound is uniform over $z \in \mathcal{A}_n$. (Note: The term $11/c_n$  is an upper bound for $8/(n\delta_n) + 3/c_n$ due to the assumption $c_n < n\delta_n$.)
\end{proof}

In the next subsection, we will use Lemmas \ref{lem:nearLip} and \ref{lem:nearCS} that we just proved to show that the conclusions of Theorem \ref{thm:manyPair} hold on the complement of the ``bad'' events defined in Subsection \ref{sec:param}. We will then complete the proof of Theorem \ref{thm:manyPair} in Subsection \ref{sec:manyPairPf} by showing that the probabilities of the ``bad'' events tend to $0$ as $n\to \infty$. The proof of Theorem \ref{thm:maxPair} will follow in Subsection \ref{sec:maxPairPf}.


\subsection{The conclusions of Theorem \ref{thm:manyPair} hold off of the ``bad'' events}\label{sec:manyPairConcHold}

In this subsection, we optimize the sequences $\eps_n$ and $\delta_n$ in terms of $n$ and $c_n$ to prove that the conclusions of Theorem \ref{thm:manyPair} hold on the complement of 
\[
E_n \cup F_n \cup G_n \cup H_n \cup \set{\abs{X^{\da}_{\left(\frac{n^{1-\delta(\alpha+1)}}{c_n\log{n}}\right)}} \leq 1-\frac{\eps_n}{2}}.
\] (Note that the last ``bad'' event in the union is needed to establish conclusion \ref{item:enoughCpts}.) Then, after establishing some upper bounds on the probabilities of these ``bad'' events in Subsection \ref{sec:badEventsSmall}, we will finish the proof of Theorem \ref{thm:manyPair} in Subsection \ref{sec:manyPairPf} by showing that these probabilities are $O(c_n^{-1})$.

\begin{lemma}\label{lem:manyPair}
	As in the hypothesis of Theorem \ref{thm:manyPair}, suppose $X_1, X_2, \ldots$ are i.i.d.\ draws from a distribution $\mu$ that satisfies Assumption \ref{ass:ComplexMu} with $\alpha \geq 0$, fix $\delta \in \left(\frac{1}{4\alpha + 3}, \frac{1}{\alpha +1}\right)$, and let $c_n = \omega(1)$ be a positive sequence such that $\log(c_n) = o(\log{n})$. If we define
	\[
	\eps_n := \frac{1}{n^\delta},\ \delta_n:=\frac{\left(c_n^7\log{n}\right)^{3/4}}{n^{3/4 + \delta/4}},\ \text{and}\ l_n := \frac{n^{1-\delta(\alpha+1)}}{c_n\log{n}}
	\] 
	then there is a positive constant $C$ such that for large $n$, on the complement of $E_n \cup F_n \cup G_n \cup H_n \cup \{|X^{\da}_{(l_n)}| \leq 1-\eps_n/2\}$, the conclusions of Theorem \ref{thm:manyPair} hold.
\end{lemma}
\begin{proof}
	Define $c_n$, $\eps_n$, $\delta_n$, and $l_n$ as in the hypothesis (note that $c_n > 1$,  $1/n < \delta_n < \frac{\eps_n}{2} < \frac{1}{2}$ and $\eps_n < \eps$ for large $n$), and assume the complement of $E_n \cup F_n \cup G_n \cup H_n\cup \{|X^{\da}_{(l_n)}| \leq 1-\eps_n/2\}$ occurs. 
	
	First, we argue that conclusion \ref{item:enoughCpts} of Theorem \ref{thm:manyPair} follows from the triangle inequality and the yet-to-be proved conclusion \ref{item:rtsToCp}.  To see that this is true, observe that on the complement of $\{|X^{\da}_{(l_n)}| \leq 1-\eps_n/2\}$, $p_n$ has at least $l_n$ roots $X_i$ satisfying $\abs{X_i} > 1-\frac{1}{2n^\delta}$, so the associated critical points $W_i^{(n)}$ promised by conclusion \ref{item:rtsToCp} lie in $\mathcal{A}_n$. Indeed, if $D$ is any positive constant, for large $n$, we have
	\[
	\frac{1}{2n^\delta} > \frac{2}{n^{1/(\alpha + 1)}} \geq \frac{2}{n} > \frac{1}{n} + \frac{D}{nc_n},
	\]
	so by \eqref{eqn:rtsToCptsBd2} and the triangle inequality, we have that for large $n$, on the complement of $E_n \cup F_n \cup G_n \cup H_n \cup \{|X^{\da}_{(l_n)}| \leq 1-\eps_n/2\}$, there are at least $l_n$ critical points in $\mathcal{A}_n = \set{z \in \mathbb{C}: 1-\frac{1}{n^\delta}\leq \abs{z} \leq 1}$ as is desired.

	 To establish conclusion \ref{item:rtsToCp} of Theorem \ref{thm:manyPair}, we begin by applying Theorem \ref{thm:determ} twice for each $X_i \in \mathcal{A}_n$. The first application will show that $p_n$ has precisely one critical point within a distance $n^{-(5+\delta)/6}$ of each $X_i \in \mathcal{A}_n$, and the second application will establish sharper control of the locations of these critical points. To start, for each $X_i \in \mathcal{A}_n$, $i\in [n]$, set 
	\[
	\xi_i := X_i,\ \vec{\zeta}_i := (X_1, \ldots, X_{i-1}, X_{i+1},\ldots, X_{n}),
	\]
	and apply Theorem \ref{thm:determ} with $n-1$ in place of $n$ and
	\[
	C_1 := \frac{3n}{2(n-1)n^{(1-\delta)/6}} = o(1),\ C_2 := 2,\ k_{\rm Lip} := \frac{c_n}{\eps_n}+\frac{172\log{n}}{n\delta_n^2} = o(n^{(1+\delta)/2}).
	\]
	We argue that each condition required in Theorem \ref{thm:determ} is satisfied for these choices, which are uniform in $i$. For large $n$, condition \ref{item:det1} holds because\footnote{Recall the radial symmetry of $\mu$, and see Assumption \ref{ass:ComplexMu} and in particular, \eqref{eqn:RadCdfBds}, \eqref{eqn:RadStiel}.}
	\begin{equation}
		1-\frac{C_\mu}{\alpha +1}\eps_n^{\alpha + 1} \leq \frac{F_R(1-\eps_n)}{1}\leq \inf_{z\in \mathcal{A}_n}\abs{m_\mu(z)} \leq \sup_{z\in \mathcal{A}_n}\abs{m_\mu(z)} \leq \frac{F_R(1)}{1-\eps_n}
		\label{eqn:mmubd}
	\end{equation}
	and via Lemma \ref{lem:nearCS}
	\begin{equation}
		\sup_{z\in \mathcal{A}_n}\abs{\frac{1}{n}\sum_{j\neq i_z}\frac{1}{z-X_j}-m_\mu(z)} < \frac{186}{c_n},
		\label{eqn:llnbd}
	\end{equation}
	so we can obtain $C_1 \leq \abs{\frac{1}{n-1}\sum_{j\neq i}\frac{1}{X_i - X_j}} \leq C_2$ by applying the triangle inequality and replacing $z$ with $X_i$ in \eqref{eqn:mmubd}, \eqref{eqn:llnbd}. Condition \ref{item:det2} follows from special case \ref{it:spCase1} of Lemma \ref{lem:nearLip} since $\frac{2}{C_1n} = o\left(\frac{\delta_n}{2}\right)$, and condition \ref{item:det3} always holds on the complement of  $F_n$. Provided that $n$ is also large enough to guarantee 
	\[
	n-1 > 8\max\set{\frac{2n^{(1-\delta)/6}}{3},\ 8(1+8)\cdot \frac{2^3}{3^3}n^{(1-\delta)/2}\cdot \left(n^{(1+\delta)/2}+ 1\right)},
	\]
	Theorem \ref{thm:determ} guarantees that each $\xi_i = X_i \in \mathcal{A}_n$ has exactly one critical point $W_i^{(n)}$ within a distance of $\frac{3}{2C_1(n-1)} = n^{-(5+\delta)/6}$. We now apply Theorem \ref{thm:determ} a second time for each $X_i \in \mathcal{A}_n$, $i \in [n]$, but (using \eqref{eqn:mmubd} and \eqref{eqn:llnbd}) we re-define $C_1:= \frac{1}{2}$
	to obtain the sharper bound
	\[
	\abs{W_{i}^{(n)} - X_i + \frac{1}{n}\frac{1}{\frac{1}{n-1}\sum_{j\neq i}\frac{1}{X_i-X_j}}} < \frac{\widetilde{C}\cdot (k_{\rm Lip}+1)}{n^2} = o\left(\frac{1}{n^{(3-\delta)/2}}\right),
	\]
	where $\widetilde{C} > 0$ is an absolute constant. The desired inequality \eqref{eqn:rtsToCptsBd} follows. To see that inequality \eqref{eqn:rtsToCptsBd2} is also true, observe that for large $n$, on the complement of $E_n\cup F_n \cup G_n \cup H_n$, we have that whenever $X_i \in \mathcal{A}_n$,
	\begin{equation}
		\begin{aligned}
			\abs{\frac{1}{\frac{1}{n-1}\sum_{j\neq i} \frac{1}{X_i - X_j}}-X_i} &= \frac{\abs{X_i}}{\abs{\frac{1}{n-1}\sum_{j\neq i} \frac{1}{X_i - X_j}}}\cdot \abs{\frac{1}{n-1}\sum_{j\neq i} \frac{1}{X_i - X_j}-\frac{1}{X_i}}\\
			&<2 \left(\abs{m_\mu(X_i) - \frac{1}{X_i}} + \frac{\abs{\sum_{j\neq i} \frac{1}{X_i - X_j}}}{n(n-1)} + \frac{186}{c_n}\right)\\
			&< 2\left(\frac{1-F_R(\abs{X_i})}{\abs{X_i}} + \frac{187}{c_n}\right)\\
			&\ll \eps_n^{\alpha + 1} + \frac{1}{c_n} \ll \frac{1}{c_n},
		\end{aligned}
		\label{eqn:CSnearX}
	\end{equation}
	where we have used \eqref{eqn:mmubd}, \eqref{eqn:llnbd}, and Assumption \ref{ass:ComplexMu} (see in particular \eqref{eqn:RadCdfBds}, \eqref{eqn:RadStiel}), and the implied constants are independent of $i$. We have established conclusion \ref{item:rtsToCp} of Theorem \ref{thm:manyPair} for large $n$ on the complement of $E_n \cup F_n \cup G_n \cup H_n$.

	We now work toward conclusion \ref{item:iota} of Theorem \ref{thm:manyPair}. Define the map \[\iota: \set{W\in \mathcal{A}_n: p_n'(W) = 0} \to [n]\] to be the restriction of the map $i_z: \C \to [n]$ from  \eqref{eqn:indexMap} to the set of critical points of $p_n$ that lie in $\mathcal{A}_n$. In particular, for each critical point $W \in \mathcal{A}_n$ of $p_n$, define 
	\[
	\iota(W) := \min\set{i \in [n] : \abs{W-X_i} = \min_{j\in[n]}\abs{W-X_j}},
	\]
	which specifies the index of a root $X_i$, $i\in [n]$, which is as close as possible to $W$. We seek to show that for large $n$, on the complement of $E_n \cup F_n \cup G_n \cup H_n$, $\iota$ is an injection and that each pair $(W, X_{\iota(W)})$, $W\in \mathcal{A}_n$, satisfies \eqref{eqn:cptsToRtsBd} and \eqref{eqn:cptsToRtsUniformBd}. Via  \eqref{eqn:GLPair}, we have almost surely
	\begin{equation}\label{eqn:GLIotaPair}
		W = X_{\iota(W)} - \frac{1}{n}\cdot\frac{1}{\frac{1}{n}\sum_{j\neq \iota(W)}\frac{1}{W-X_j}}\ \text{for each $W$ s.t. $p_n'(W) = 0$},
	\end{equation}
	and on the complement of the ``bad'' event $E_n \cup F_n \cup G_n \cup H_n$, \eqref{eqn:llnbd} implies 
	\begin{equation}\label{eqn:WNearM}
		\max_{W\in \mathcal{A}_n\: :\: p'(W) = 0}\abs{\frac{1}{n}\sum_{j\neq \iota(W)}\frac{1}{W-X_j}-m_\mu(W)} < \frac{186}{c_n}.
	\end{equation}
	In view of the triangle inequality and \eqref{eqn:mmubd}, \eqref{eqn:GLIotaPair}, and \eqref{eqn:WNearM}, for large $n$, on the complement of $E_n \cup F_n \cup G_n \cup H_n$, 
	\begin{equation}
		\max_{W\in \mathcal{A}_n\: :\: p'(W) = 0}\abs{W-X_{\iota(W)}} < \frac{2}{n}.
	\end{equation}
	It follows that for large $n$, on the complement of $E_n \cup F_n \cup G_n \cup H_n$, each critical point $W \in \mathcal{A}_n$ satisfies
	\[
	\abs{W - X_{\iota(W)}\left(1-\frac{1}{n}\right)} \leq \abs{W-X_{\iota(W)} + \frac{W}{n}} + \frac{2}{n^2},
	\]
	and by a string of inequalities nearly identical to \eqref{eqn:CSnearX}, we can replace $W$ with $\left(\frac{1}{n}\sum_{j \neq \iota(W)}\frac{1}{W-X_j}\right)^{-1}$ to obtain 
	\[
	\abs{W - X_{\iota(W)}\left(1-\frac{1}{n}\right)} < \abs{W -X_{\iota(W)} + \frac{1}{n}\cdot\frac{1}{\frac{1}{n}\sum_{j\neq \iota(W)}\frac{1}{W-X_j}}} + \frac{C'}{nc_n} + \frac{2}{n^2},
	\]
	where $C'>0$ is an absolute constant. The first quantity on the right is almost surely zero by \eqref{eqn:GLIotaPair}, so we have proved \eqref{eqn:cptsToRtsUniformBd} for a suitable choice of $C>0$ and large $n$ on the complement of $E_n \cup F_n \cup G_n \cup H_n$. To see that for large $n$, on the complement of $E_n \cup F_n \cup G_n \cup H_n$,  $\iota(W)$ is an injection, note by \eqref{eqn:cptsToRtsUniformBd} that $\abs{W} < \abs{X_{\iota(W)}}$ for each $W \in \mathcal{A}_n$, so $X_{\iota(W)} \in \mathcal{A}_n$. By the recently established conclusion \ref{item:rtsToCp} of Theorem \ref{thm:manyPair}, such a root $X_{\iota(W)} \in \mathcal{A}_n$ has precisely one critical point $W$ within a distance $n^{-(5+\delta)/6}\gg 1/n$,  so in view of \eqref{eqn:cptsToRtsUniformBd}, it is impossible for two critical points to pair via $\iota$ to the same root. Thus $\iota$ is an injection, and we have established conclusion \ref{item:iota} of Theorem \ref{thm:manyPair} for large $n$ on the complement of $E_n \cup F_n \cup G_n \cup H_n$.
\end{proof}


\subsection{The ``bad'' events are asymptotically unlikely}\label{sec:badEventsSmall}

In this subsection, we prove some upper bounds on the probabilities of the ``bad'' events off of which the conclusions of our main theorems hold. Per our discussion in the detailed overview Subsection \ref{sec:detailsPos}, the lemmas below primarily serve to establish that with high probability, the annuli $\mathcal{A}_n$, $\mathcal{B}_n$ and $\mathcal{A}_n^*$ contain roughly the expected number of $X_i$, $1 \leq i \leq n$, and that these roots are sufficiently separated from the other roots (by distances $\delta_n$ and $\delta_n^*$, respectively). Lemma \ref{lem:GnSmall} shows that the $X_i \notin \mathcal{A}_n \cup \mathcal{B}_n$ are dispersed enough to have a manageable contribution to the Lipschitz constant associated with $z\mapsto \overline{M}_\mu(z)$, and Lemma \ref{lem:netProb} shows that $\overline{M}_\mu(z)$ is close to $m_\mu(z)$ on the net $\mathcal{N}_n$ (see also Subsection \ref{sec:CSBehaved} above). We conclude this subsection with Lemma \ref{lem:maxCDF}, which guarantees that $p_n$ does in fact have extremal roots (and hence critical points) in the annuli of interest at the edge of the unit disk. 

We note that to avoid repeating certain arguments in Section \ref{sec:proofsNeg} below, we have written most of the lemmas in the current subsection so that they also apply to situations when $\alpha < 0$. 

\begin{lemma}[Controlling $\P(E_n)$ and the number of roots within annuli in $\mathbb{A}_\eps$]\label{lem:fewRts}
	Suppose $\alpha > -1$, and let $\gamma_n \in (0, \eps)$ be a positive sequence. Then,
	\begin{equation}
		\frac{c_\mu}{\alpha + 1}\gamma_n^{\alpha +1} \leq \P\left(1-\gamma_n \leq \abs{X_1} \leq 1\right) \leq \frac{C_\mu}{\alpha + 1}\gamma_n^{\alpha+1},\ \text{for $n \in \Z^+$},
		\label{eqn:massOnEdge}
	\end{equation}
	and if $C > \frac{2C_\mu}{\alpha+1}$, 
	\begin{equation}
		\P\left(\#\set{i\in [n]: 1-\gamma_n \leq \abs{X_i} \leq 1} > Cn\gamma_n^{\alpha + 1}\right) \leq e^{-(C-2C_\mu/(\alpha+1))n\gamma_n^{\alpha+1}}.
		\label{eqn:annulusNumBd}
	\end{equation}
\end{lemma}
\begin{proof}
	Define the random variable
	\[
	N_n := \#\set{i\in [n]: 1-\gamma_n \leq \abs{X_i} \leq 1},
	\]
	which has a binomial distribution with parameters $n$ and
	\begin{align*}
		\rho_n &:= \P\left(1-\gamma_n \leq \abs{X_1} \leq 1\right) = \int_{1-\gamma_n}^1f_R(r)\,dr = \int_{1-\gamma_n}^1(1-r)^\alpha\cdot \frac{f_R(r)}{(1-r)^\alpha}\,dr.
	\end{align*}
	Inequality \eqref{eqn:massOnEdge} follows since $c_\mu \leq \frac{f_R(r)}{(1-r)^\alpha} \leq C_\mu$ for $1-\eps \leq r \leq 1$. Via Markov's inequality and the upper bound in \eqref{eqn:massOnEdge}, we have
	\[
	\P(N_n > Cn\gamma_n^{\alpha+1}) \leq \frac{\E\left[e^{N_n}\right]}{e^{Cn\gamma_n^{\alpha+1}}} = \frac{\left(1-\rho_n + \rho_ne^1\right)^n}{e^{Cn\gamma_n^{\alpha+1}}} < e^{-Cn\gamma_n^{\alpha+1}}\left(1 + \frac{2C_\mu\gamma_n^{\alpha+1}}{\alpha + 1}\right)^n,
	\]
	where we computed the expectation by evaluating the binomial moment-generating function at $t=1$ (i.e., we have employed a Chernoff bound).\footnote{We note that this inequality holds for \textit{all} positive integers $n$ and is not an asymptotic result. In particular $\E[e^{N_n}] = (1-\rho_n +\rho_ne^1)^n = (1+(e-1)\rho_n)^n \leq (1+2\rho_n)^n$, so applying the upper bound from \eqref{eqn:massOnEdge} gives $\E[e^{N_n}]\leq (1+\frac{2C_\mu\gamma_n^{\alpha+1}}{\alpha+1})^n$.} Continuing from above using the log-exponential trick (and $\log(1+x) \leq x$ for $x \geq 0$) yields
	\[
	\P(N_n > Cn\gamma_n^{\alpha+1}) \leq e^{-Cn\gamma_n^{\alpha+1} + \frac{2C_\mu}{\alpha+1} n\gamma_n^{\alpha+1}},
	\]
	which completes the proof. (We note that \eqref{eqn:annulusNumBd} is only useful when $C>\frac{2C_\mu}{\alpha+1}$.)
\end{proof}

\begin{lemma}[The largest roots are separated with high probability, $\alpha \geq 0$ case]\label{lem:sepRts}
	Fix $\alpha \geq 0$, suppose $\gamma_n, d_n \in (0,\eps/2)$ are sequences satisfying $n\gamma_n^{\alpha+1} = \omega(1)$ and $n^2\gamma_n^{2\alpha + 1}d_n^2 = o(1)$, and define
	\[
	A_n := \set{z\in \C: 1-\gamma_n \leq \abs{z} \leq 1}\subset \mathbb{A}_\eps.
	\]
	There is a constant $c>0$ such that for large $n$, 
	\begin{equation}
		\P\left(\exists i,j \in [n], i\neq j: X_i\in A_n, \abs{X_i-X_j}\leq d_n\right)\ll\max\set{n^2\gamma_n^{2\alpha+1}d_n^2, e^{-cn\gamma_n^{\alpha+1}}}.
		\label{eqn:circSep}
	\end{equation}
	We note that the first term in the maximum bounds the separation probability in the case where $A_n$ contains $O(n\gamma_n^{\alpha+1})$ roots, as is expected, and the second term in the maximum bounds the probability that $A_n$ contains more than  $O(n\gamma_n^{\alpha+1})$ roots (see also Lemma \ref{lem:fewRts}).  To control the radial separation between the largest roots of $p_n$, let $\gamma_n, R_n \in (0,\eps/2)$ be sequences satisfying $n\gamma_n^{\alpha +1}=\omega(1)$ and $n^2\gamma_n^{2\alpha+1}R_n = o(1)$. There is a constant $c>0$ such that for large $n$,
	\begin{equation}
		\P\left(\exists i,j \in [n], i\neq j: X_i\in A_n, \abs{\abs{X_i}-\abs{X_j}}\leq R_n\right)\ll\max\set{n^2\gamma_n^{2\alpha+1}R_n, e^{-cn\gamma_n^{\alpha+1}}}.
		\label{eqn:radSep}
	\end{equation}
\end{lemma}

\begin{proof}
	We prove \eqref{eqn:circSep} and omit the proof of \eqref{eqn:radSep} since both results can be justified using nearly identical arguments. To that end, we will show that the event $E^{\rm sep}_n$, that every $X_i$ in $A_n$ is at least $d_n$-far from all other $X_j$, tends to 1 as $n\to \infty$. Set
	\[
	{\rm Sep}_n:= \min\set{\abs{X_i - X_j} : i,j \in [n], i \neq j,\ X_i \in A_n},
	\]
	whenever this minimum is properly defined, and let
	\[
	N_n := \#\set{i\in [n]: X_i \in A_n}.
	\]
	We have
	\[
	E^{\rm sep}_n = \set{N_n=0}\cup \set{N_n\geq 1,\ {\rm Sep}_n > d_n}.
	\]
	For conciseness, define
	\[
	M_n:=\sup_{z\in A_n}\abs{\pi\cdot f_\mu(z)} = O(\gamma_n^\alpha)
	\quad \text{and}\quad \rho_n := \mu(A_n) = \Theta\left(\gamma_n^{\alpha + 1}\right)
	\]
	(see Assumption \ref{ass:ComplexMu} for the first asymptotic and Lemma \ref{lem:fewRts} for the second), and let $D > 0$ be a large constant to be specified later. We begin by conditioning on $N_n\sim{\rm Binomial}(n, \rho_n)$ to obtain
	\begin{equation}
		\P\left(E^{\rm sep}_n\right) \geq \P\left(N_n=0\right) + \sum_{k = 1}^{\lfloor Dn\rho_n\rfloor}\P\left({\rm Sep}_n > d_n,\ N_n = k\right).
		\label{eqn:binomSplit}
	\end{equation}
	The $X_i$ are i.i.d., so for large $n$ and $1 \leq k\leq \lfloor Dn\rho_n\rfloor$, we have 
	\begin{align}
		&\P\left({\rm Sep}_n > d_n,\ N_n = k\right)\notag\\
		&\quad=\binom{n}{k}\cdot\P\left(X_1, \ldots, X_k \in A_n,\ X_{k+1} \ldots, X_n \notin A_n,\ {\rm Sep}_n > d_n\right)\notag\\
		&\quad\geq \binom{n}{k}\cdot\rho_n\cdot\left(\rho_n - 1\cdot M_nd_n^2\right) \cdots \left(\rho_n-(k-1)M_nd_n^2\right)\cdot (1-\rho_n-kM_nd_n^2)^{n-k}\label{eqn:avoidRts}\\
		&\quad\geq \binom{n}{k}\left(\rho_n - Dn\rho_nM_nd_n^2\right)^k (1-\rho_n-Dn\rho_nM_nd_n^2)^{n-k}\notag\\
		&\quad= \binom{n}{k}\left(1 - DnM_nd_n^2\right)^k\left(\frac{1-\rho_n-Dn\rho_nM_nd_n^2}{1-\rho_n}\right)^{n-k}\rho_n^k (1-\rho_n)^{n-k},\notag
	\end{align}
	where the first inequality follows from successively conditioning on the relative locations of $X_1, \ldots, X_n$, so that each one is at least a distance of $d_n$ from the other $X_i$ that are in $A_n$. (Note that $\rho_n > Dn\rho_nM_nd_n^2$ for large $n$ because $n\gamma_n^\alpha d_n^2=o(1)$ by our assumptions on $\gamma_n$ and $d_n$ in the hypothesis.) Substituting this inequality into \eqref{eqn:binomSplit} above yields that for large $n$,
	\begin{align}
		\P(E^{\rm sep}_n)&\geq (1-\rho_n)^n \notag\\
			&\quad+ \sum_{k = 1}^{\lfloor Dn\rho_n\rfloor}\binom{n}{k}\left(1 - DnM_nd_n^2\right)^k\left(1-\frac{Dn\rho_nM_nd_n^2}{1-\rho_n}\right)^{n-k}\rho_n^k (1-\rho_n)^{n-k} \notag\\
		&\geq\left(1 - DnM_nd_n^2\right)^{Dn\rho_n}\left(1-\frac{Dn\rho_nM_nd_n^2}{1-\rho_n}\right)^n\cdot\sum_{k = 0}^{\lfloor Dn\rho_n\rfloor}\binom{n}{k}\rho_n^k(1-\rho_n)^{n-k} \notag\\
		&\geq \exp\left(-4/(1-\rho_n)D^2n^2M_n \rho_nd_n^2\right)\cdot\P(N_{n}\leq \lfloor Dn\rho_n\rfloor)\label{eqn:binTrunc}\\
		&\geq \exp\left(-4/(1-\rho_n)D^2n^2M_n\rho_nd_n^2\right)\cdot\P(N_{n}\leq (D-1)n\rho_n)\label{eqn:npnBig}\\
		&\geq \exp\left(-C_1n^2\gamma_n^{2\alpha + 1}d_n^2\right)\left(1-\exp\left(-c_2n\gamma_n^{\alpha + 1}\right)\right),\label{eqn:expProd}
	\end{align}
	where we have used the  log-exponential trick and
	\begin{equation}
		\log(1-x)\geq \frac{-x}{\sqrt{1-x}},\ \text{for $x \in (0,1)$}
		\label{eqn:lowLogBd}
	\end{equation}
	to obtain \eqref{eqn:binTrunc}, inequality \eqref{eqn:npnBig} follows from the hypothesis $n\gamma_n^{\alpha+1} = \omega(1)$, and we have applied Lemma \ref{lem:fewRts} to achieve \eqref{eqn:expProd} for appropriate constants $C_1, c_2 > 0$. We conclude \eqref{eqn:circSep} by using $e^{-x} \geq 1-x$ to approximate from below the first exponential factor in \eqref{eqn:expProd}. (Note that to prove \eqref{eqn:radSep}, one could follow the same argument we just made after replacing $d_n^2$ with $R_n$ and $M_n$ with $\sup_{z \in A_n}\abs{4\pi\cdot f_\mu(z)}$, so that \eqref{eqn:avoidRts} is a lower bound on the probability that $X_i \in A_n$ are radially separated from all other $X_j$, $j\in [n]$.)
\end{proof}

\begin{lemma}\label{lem:GnSmall} Suppose $\eps_n \in (0,\eps/3)$ and $c_n$ are positive sequences. Then, 
	\[
	\P\left(\sum_{i=1}^n\frac{1}{(1-\frac{3}{2}\eps_n-\abs{X_i})^2}\cdot \sind{\abs{X_i}< 1 -2\eps_n}\geq nc_nL_n\right) \ll \frac{1}{c_n},
	\]
	where
	\begin{equation}
		L_n := \begin{cases}
			\eps_n^{\alpha-1}& \text{if $-1 < \alpha \leq 0$}\\
			\eps_n^{\alpha-1}\log(\eps_n^{-1}) & \text{if $0 < \alpha < 1$}\\
			\log(\eps_n^{-1}) & \text{if $\alpha \geq 1$}.
		\end{cases}
		\label{eqn:Lndef}
	\end{equation}
\end{lemma}
\begin{proof}
	For any sequence $L_n > 0$, Markov's inequality implies
	\begin{equation}
		\begin{aligned}
			&\P\left(\sum_{i=1}^n\frac{1}{(1-\frac{3}{2}\eps_n-\abs{X_i})^2}\cdot \sind{\abs{X_i}< 1 -2\eps_n}\geq nc_nL_n\right)\\
			&\qquad\leq \frac{1}{nc_nL_n}\cdot n\cdot\left(\frac{1}{(\eps-\frac{3}{2}\eps_n)^2} + \E\left[\frac{1}{(1-\frac{3}{2}\eps_n-\abs{X_1})^2}\cdot \sind{1-\eps < \abs{X_1}< 1-2\eps_n}\right]\right)\\
			&\qquad\leq  \frac{1}{c_nL_n}\left(\frac{1}{(\eps-\frac{3}{2}\eps_n)^2} +\int_{1-\eps}^{1-2\eps_n}\frac{1}{(1-\frac{3}{2}\eps_n-r)^2}\cdot\frac{f_R(r)}{2\pi r}\,dr\right),
		\end{aligned}
		\label{eqn:Gbd}
	\end{equation}
	where we have split expectation into two parts based on the size of $\abs{X_1}$ and used that $\mu$ has a radial density $f_R(r)$ on $\mathbb{A}_\eps = \set{z\in \C: 1-\eps \leq \abs{z} \leq 1}$. To find an upper bound on the integral, first over-approximate $\frac{f_R(r)}{r}$ with $(1-\eps)^{-1}C_\mu(1-r)^\alpha$, and then employ the substitution $u = 1-\frac{3}{2}\eps_n-r$ to obtain
	\begin{align*}
		\int_{1-\eps}^{1-2\eps_n}\frac{1}{(1-\frac{3}{2}\eps_n-r)^2}\cdot\frac{f_R(r)}{2\pi r}\,dr &\leq \frac{C_\mu}{2\pi(1-\eps)}\int_{1-\eps}^{1-2\eps_n}\frac{(1-r)^\alpha}{(1-\frac{3}{2}\eps_n-r)^2}\,dr\\
		&=\frac{C_\mu}{2\pi(1-\eps)}\int_{\frac{\eps_n}{2}}^{\eps-\frac{3}{2}\eps_n}\frac{(u+\frac{3}{2}\eps_n)^\alpha}{u^2}\,du.
	\end{align*}
	We now consider several cases based on the value of $\alpha$.

\textbf{Case I: $-1 < \alpha \leq 0$. } In this case, $(u+\frac{3}{2}\eps_n)^\alpha \leq u^\alpha$ (we have $u, \eps_n > 0$), so 
	\[
	\int_{\frac{\eps_n}{2}}^{\eps-\frac{3}{2}\eps_n}\frac{(u+\frac{3}{2}\eps_n)^\alpha}{u^2}\,du \leq \int_{\frac{\eps_n}{2}}^{\eps-\frac{3}{2}\eps_n}\frac{u^\alpha}{u^2}\,du = \frac{\left(\eps-\frac{3}{2}\eps_n\right)^{\alpha-1} - \left(\frac{\eps_n}{2}\right)^{\alpha-1}}{\alpha-1} < 4\eps_n^{\alpha-1},
	\]
which combines with \eqref{eqn:Gbd} and the definition $L_n:=\eps_n^{\alpha-1}$ for $-1 < \alpha \leq 0$ to give
\begin{align*}
&\P\left(\sum_{i=1}^n\frac{1}{(1-\frac{3}{2}\eps_n-\abs{X_i})^2}\cdot \sind{\abs{X_i}< 1 -2\eps_n}\geq nc_nL_n\right)\\
&\qquad\leq \frac{1}{c_nL_n}\left(\frac{1}{(\eps-\frac{1}{2}\eps)^2} + 4\eps_n^{\alpha-1}\right) = \frac{1}{c_n\eps_n^{\alpha-1}}\left(\frac{4}{\eps^2} + 4\eps_n^{\alpha-1}\right) \ll \frac{1}{c_n}.
\end{align*}

\textbf{Case II: $\alpha \geq 1$.} In this case, we can replace the numerator of the integrand with $u + \frac{3}{2}\eps_n$ (note $u + \frac{3}{2}\eps_n \leq \eps \leq 1$) to obtain
	\begin{align*}
	\int_{\frac{\eps_n}{2}}^{\eps-\frac{3}{2}\eps_n}\frac{(u+\frac{3}{2}\eps_n)^\alpha}{u^2}\,du &\leq	\int_{\frac{\eps_n}{2}}^{\eps-\frac{3}{2}\eps_n}\frac{u+\frac{3}{2}\eps_n}{u^2}\,du = \log\abs{\frac{\eps-\frac{3}{2}\eps_n}{\frac{\eps_n}{2}}} + 3 - \frac{\frac{3}{2}\eps_n}{\eps-\frac{3}{2}\eps_n}\\ &< 4+\log(\eps_n^{-1}),
	\end{align*}
which combines with \eqref{eqn:Gbd} and the definition $L_n:=\log(\eps_n^{-1})$ for $\alpha \geq 1$ to give
\begin{align*}
&\P\left(\sum_{i=1}^n\frac{1}{(1-\frac{3}{2}\eps_n-\abs{X_i})^2}\cdot \sind{\abs{X_i}< 1 -2\eps_n}\geq nc_nL_n\right)\\
&\leq \frac{1}{c_nL_n}\left(\frac{1}{(\eps-\frac{1}{2}\eps)^2} + 4 + \log(\eps_n^{-1})\right) = \frac{1}{c_n\log(\eps_n^{-1})}\left(\frac{4+ 4\eps^2}{\eps^2} + \log(\eps_n^{-1})\right) \ll \frac{1}{c_n}.
\end{align*}

\textbf{Case III: $0 < \alpha < 1$.} In this final case, we can use a linear approximation (at $u=-\frac{\eps_n}{2}$) to the numerator of the integrand to achieve
	\begin{align*}
		\int_{\frac{\eps_n}{2}}^{\eps-\frac{3}{2}\eps_n}\frac{(u+\frac{3}{2}\eps_n)^\alpha}{u^2}\,du &\leq \int_{\frac{\eps_n}{2}}^{\eps-\frac{3}{2}\eps_n}\frac{\eps_n^\alpha + \alpha\eps_n^{\alpha-1}(u + \frac{\eps_n}{2})}{u^2}\,du \\
		&\leq \left(1+\frac{\alpha}{2}\right)\int_{\frac{\eps_n}{2}}^{\eps-\frac{3}{2}\eps_n}\left(\eps_n^\alpha u^{-2} + \alpha\eps_n^{\alpha-1}u^{-1}\right)du\\
		&= \left(1+\frac{\alpha}{2}\right)\left(2\eps_n^{\alpha-1} - \frac{\eps_n^\alpha}{\eps-\frac{3}{2}\eps_n} + \alpha\eps_n^{\alpha-1}\log\abs{\frac{\eps-\frac{3}{2}\eps_n}{\frac{1}{2}\eps_n}}\right)\\
		&< \frac{3}{2}\eps_n^{\alpha-1}(3+\log(\eps_n^{-1})),
	\end{align*}
	which combines with \eqref{eqn:Gbd} and the definition $L_n := \eps_n^{\alpha-1}\log(\eps_n^{-1})$ to establish
\begin{align*}
	\P\left(\sum_{i=1}^n\frac{\sind{\abs{X_i}< 1 -2\eps_n}}{(1-\frac{3}{2}\eps_n-\abs{X_i})^2}\geq nc_nL_n\right)&\leq \frac{1}{c_nL_n}\left(\frac{1}{(\eps-\frac{1}{2}\eps)^2} +\frac{3}{2}\eps_n^{\alpha-1}(3+\log(\eps_n^{-1}))\right)\\
	&\leq \frac{1}{c_n\eps_n^{\alpha-1}\log(\eps_n^{-1})}\left(\frac{4}{\eps^2} + 6\eps_n^{\alpha-1}\log(\eps_n^{-1})\right)\\
	&\ll \frac{1}{c_n},
\end{align*}
where we used $\eps_n< \eps/3 \leq 1/3$ (so $\log(\eps_n^{-1}) > 1$) in the simplification. The proof of Lemma  \ref{lem:GnSmall} is complete.
\end{proof}

\begin{lemma}\label{lem:netProb}
	Fix $\alpha > -1$, suppose $\mathcal{N}_n \subset \mathbb{A}_{{\eps}/{2}} := \set{z \in \C: 1-\frac{\eps}{2} \leq \abs{z} \leq 1}$, and let $\delta_n$ and $c_n$ be positive sequences satisfying $\delta_n^{\min\set{1+\alpha, 1}} = o\left(\frac{1}{c_n}\right) = o(1)$. We have
	\begin{equation}
	\P\left(\max_{z\in \mathcal{N}_n}\abs{\frac{1}{n}\sum_{i=1}^n\frac{\sind{\abs{z-X_i}> \frac{\delta_n}{2}}}{z-X_i}-m_\mu(z)}\geq \frac{1}{c_n}\right) \ll \abs{\mathcal{N}_n}\cdot \frac{c_n^2\log\left(\frac{1}{\delta_n}\right)\delta_n^{\min\set{\alpha, 0}}}{n}.	
	\label{eqn:netConv}
	\end{equation}
\end{lemma}
\begin{proof}
	This follows from Markov's inequality and the union bound (over all $z \in \mathcal{N}_n$). We begin by fixing $z \in \mathcal{N}_n \subset \mathbb{A}_{{\eps}/{2}}$ and finding an asymptotic upper bound for the variance of $\frac{1}{n}\sum_{i=1}^n\frac{1}{z-X_i}\sind{\abs{z-X_i}> \frac{\delta_n}{2}}$. To that end, define $\widetilde{m}_n(z)$ to be the expectation of $\frac{1}{z-X_i}\sind{\abs{z-X_i}> \frac{\delta_n}{2}}$ and observe that for large $n$,
	\begin{equation}
	\begin{aligned}
		&\E\left[\abs{\frac{1}{n}\sum_{i=1}^n\left(\frac{1}{z-X_i}\sind{\abs{z-X_i}> \frac{\delta_n}{2}}-\widetilde{m}_n(z)\right)}^2\right]\\
		&\qquad=\frac{1}{n^2}\E\left[\sum_{i=1}^n\left(\frac{\sind{\abs{z-X_i}> \frac{\delta_n}{2}}}{z-X_i}-\widetilde{m}_n(z)\right)\cdot\sum_{j=1}^n\left(\overline{\frac{\sind{\abs{z-X_j}> \frac{\delta_n}{2}}}{z-X_j}-\widetilde{m}_n(z)}\right)\right]\\
		&\qquad=\frac{1}{n^2}\sum_{i=1}^n\E\left[\abs{\frac{1}{z-X_i}\sind{\abs{z-X_i}> \frac{\delta_n}{2}}-\widetilde{m}_n(z)}^2\right]\\
		&\qquad\leq\frac{1}{n}\cdot \E\left[\abs{\frac{1}{z-X_1}\sind{\abs{z-X_1}> \frac{\delta_n}{2}}}^2\right],
	\end{aligned}
	\label{eqn:netCtrl}
\end{equation}
	where we used the fact that $X_i$ are i.i.d.\ to simplify the expectation of the product of sums, and in the last line, we used that for a complex random variable $Y$, $\E\left[\abs{Y-\E[Y]}^2\right] = \E\abs{Y^2} - \abs{\E[Y]}^2$. 
	
	We now invoke the absolute moment bounds from Lemma \ref{lem:moments}. In the case where $\alpha \geq 0$, apply Lemma \ref{lem:moments} with $\mathfrak{p} = 2-1/\log(1/\delta_n)$ to obtain
	\[
		\E\left[\abs{\frac{1}{z-X_1}\sind{\abs{z-X_1}> \frac{\delta_n}{2}}}^2\right] \leq  \left(\frac{2}{\delta_n}\right)^{\log(1/\delta_n)}\cdot\E{\abs{z-X_1}^{-\mathfrak{p}}} \ll e\cdot \log(1/\delta_n),
	\]
	and if $\alpha \in (-1, 0)$, apply Lemma \ref{lem:moments} with $\mathfrak{p} = 2 + \alpha -1/\log(1/\delta_n)$ to obtain
	\[
	\E\left[\abs{\frac{1}{z-X_1}\sind{\abs{z-X_1}> \frac{\delta_n}{2}}}^2\right] \leq  \left(\frac{2}{\delta_n}\right)^{\log(1/\delta_n)-\alpha}\cdot\E{\abs{z-X_1}^{-\mathfrak{p}}} \ll \delta_n^{\alpha} \cdot \log(1/\delta_n).
	\]
	Combining these with \eqref{eqn:netCtrl} yields 
	\[
	\E\abs{\frac{1}{n}\sum_{i=1}^n\left(\frac{1}{z-X_i}\sind{\abs{z-X_i}> \frac{\delta_n}{2}}-\widetilde{m}_n(z)\right)}^2 \ll \frac{\delta_n^{\min\set{\alpha, 0}}\cdot \log(1/\delta_n)}{n},
	\]
	where the implied constant is independent of $z \in \mathcal{N}_n$.	Inequality \eqref{eqn:netConv} follows after applying Markov's inequality and a union bound over $z\in \mathcal{N}_n$ provided that $\max_{z\in \mathcal{N}_n}\abs{\widetilde{m}_n(z)-m_\mu(z)} = o\left(\frac{1}{c_n}\right)$.
	
	To establish this last fact, we consider two cases based on the sign of $\alpha$. If $\alpha \geq 0$ and $n$ is large enough to guarantee that $\frac{\delta_n}{2} < \frac{\eps}{2}$ we have the following bound which is uniform over all $z\in \mathcal{N}_n \subset \mathbb{A}_{{\eps}/{2}}$:
	\[
	\abs{\widetilde{m}_n(z)-m_\mu(z)} = \abs{\E\left[\frac{\sind{\abs{z-X_1} \leq \frac{\delta_n}{2}}}{z-X_1}\cdot \ind{X_1 \in\mathbb{A}_\eps}\right]} \ll \int_0^{\frac{\delta_n}{2}}\frac{1}{r}\cdot r\,dr \ll \delta_n = o\left(\frac{1}{c_n}\right).
	\]
	On the other hand, when $-1 < \alpha < 0$, $f_\mu(x)$ is no longer bounded, so we will use the tail bounds from Lemma \ref{lem:heavyTail}. In particular,
	\[
	\abs{\widetilde{m}_n(z)-m_\mu(z)} = \abs{\E\left[\frac{\sind{\abs{z-X_1} \leq \frac{\delta_n}{2}}}{z-X_1}\cdot \ind{X_1 \in\mathbb{A}_\eps}\right]} = \int_{2/\delta_n}^\infty\P\left(\abs{\frac{1}{z-X_1}} \geq t\right)\,dt,
	\]
	so Inequality \eqref{eqn:tailUpper} from Lemma \ref{lem:heavyTail} guarantees
	\[
	\abs{\widetilde{m}_n(z)-m_\mu(z)} \ll \int_{2/\delta_n}^\infty t^{-(2+\alpha)}\,dt \ll \delta_n^{1+\alpha} = o\left(\frac{1}{c_n}\right),
	\]
	where the implied constant is independent of $z$. (Note that $\delta_n^{1+\alpha} = o\left(\frac{1}{c_n}\right)$ by the assumptions in the hypothesis.) The proof is complete.
\end{proof}

\begin{lemma}[The largest roots of $p_n$ are near the unit circle with high probability]\label{lem:maxCDF}
	Fix $\alpha > -1$, suppose $\gamma_n \in (0,\eps)$ is a sequence satisfying $\omega(1/n) = \gamma_n^{\alpha+1} = o(1)$, and suppose $l_n = o(n)$ is a sequence of positive integers for which $l_n\cdot\log(n\gamma_n^{\alpha+1}) = o(n\gamma_n^{\alpha+1})$. Then, there is a positive constant $c = c(\alpha, c_\mu)>0$ so that
	\[
	\P\left(\abs{X^{\ua}_{(n-l_n+1)}} \leq 1-\gamma_n\right) \ll e^{-c n\gamma_n^{\alpha+1}}.
	\]
\end{lemma}
\begin{proof}
	By Assumption \ref{ass:ComplexMu}, Inequality \eqref{eqn:RadCdfBds} implies that for $1 -\eps \leq r \leq 1$, 
	\begin{equation}
		F_R(r) \leq 1-\frac{c_\mu}{\alpha+1}(1-r)^{\alpha + 1}\quad \text{and}\quad 1-F_R(r) \leq \frac{C_\mu}{\alpha + 1}(1-r)^{\alpha + 1}.
		\label{eqn:RadCdfBds2}
	\end{equation}
	Now, by standard results for order statistics, see e.g.\ page 9 of \cite{DNag},
	\begin{equation*}
		\P\left(\abs{X^{\ua}_{(n-l_n+1)}} \leq r\right) =\sum_{i={n-l_n+1}}^n\binom{n}{i}\left(F_R(r)\right)^i\left(1-F_R(r)\right)^{n-i},
	\end{equation*}
	so after applying \eqref{eqn:RadCdfBds2}, overestimating the binomial coefficient with $n^{n-i}$, and then substituting $1-\gamma_n$ for $r$, we obtain
	\begin{align*}
		\P\left(\abs{X^{\ua}_{(n-l_n+1)}} \leq r\right)&\leq\sum_{i={n-\ell_n+1}}^nn^{n-i} \left(1-\frac{c_\mu(1-r)^{\alpha+1}}{\alpha+1}\right)^i \left(\frac{C_\mu(1-r)^{\alpha+1}}{\alpha+1}\right)^{n-i}\\
		&\leq\sum_{i={n-l_n+1}}^n\left(\frac{nC_\mu\gamma_n^{\alpha+1}}{\alpha +1}\right)^{n-i}\left(1-\frac{c_\mu\gamma_n^{\alpha+1}}{\alpha+1}\right)^i \\ 
		&\leq l_n\left(C_\mu n\gamma^{\alpha + 1}\right)^{l_n-1} \left(1-\frac{c_\mu\gamma_n^{\alpha+1}}{\alpha+1}\right)^{n-l_n+1},	
	\end{align*}
	where the last two inequalities hold for large $n$. Via the log-exponential trick, it follows that for large $n$,
	\begin{align*}
		&\P\left(\abs{X^{\ua}_{(n-l_n+1)}} \leq 1-\gamma_n\right)\\
		&\qquad\leq \exp\left(\log(l_n) + (l_n-1)\log(C_\mu n \gamma_n^{\alpha+1}) + (n-l_n+1)\log\left[1-\frac{c_\mu\gamma_n^{\alpha+1}}{\alpha+1}\right]\right)\\
		&\qquad\leq \exp\left(l_n\cdot\log(C_\mu n\gamma_n^{\alpha+1}) + (n-l_n +1)\cdot\frac{-c_\mu\gamma_n^{\alpha+1}}{\alpha+1}\right).
	\end{align*}
	By assumption, $n\gamma_n^{\alpha+1} = \omega(1)$ and $l_n\cdot\log(n\gamma_n^{\alpha+1}) = o(n\gamma_n^{\alpha+1})$, so for large $n$,  
	\[
	l_n\cdot\log(C_\mu n\gamma_n^{\alpha+1}) \leq l_n\cdot\log(n\gamma_n^{\alpha+1}\cdot n\gamma_n^{\alpha+1}) = 2l_n\cdot\log(n\gamma_n^{\alpha+1}) = o(n\gamma_n^{\alpha+1}),
	\]
	which gives an asymptotic upper bound on the first term inside the exponential. We also assumed $l_n = o(n)$, so the second term inside the exponential is bounded above by $\frac{n}{2}\cdot \frac{-c_\mu\gamma_n^{\alpha+1}}{\alpha+1}$. Combining these bounds allows us to achieve the desired asymptotic result.
\end{proof}

\subsection{The proof of Theorem \ref{thm:manyPair}}\label{sec:manyPairPf}

In view of Lemma \ref{lem:manyPair}, that we proved in Subsection \ref{sec:manyPairConcHold} above, we conclude the proof of Theorem \ref{thm:manyPair} by showing that the ``bad'' events are negligible in the limit as $n\to \infty$.

\begin{lemma}\label{lem:manyPairProb}
	Suppose $X_1, X_2, \ldots$ are i.i.d.\ draws from a distribution $\mu$ that satisfies Assumption \ref{ass:ComplexMu} with $\alpha \geq 0$, fix $\delta \in \left(\frac{1}{4\alpha + 3}, \frac{1}{\alpha +1}\right)$, and let $c_n=\omega(1)$ be a positive sequence satisfying $\log(c_n) = o(\log(n))$. If we define
	\[
	\eps_n := \frac{1}{n^\delta},\ \delta_n:=\frac{\left(c_n^7\log{n}\right)^{3/4}}{n^{3/4 + \delta/4}},\ \text{and}\ l_n := \frac{n^{1-\delta(\alpha+1)}}{c_n\log{n}}
	\]
	then $\P(E_n \cup F_n \cup G_n \cup H_n \cup \{|X^{\da}_{(l_n)}| \leq 1-\eps_n/2\}) = O(c_n^{-1})$.
\end{lemma}
\begin{proof}
In view of the union bound, it suffices to show that each of the events $E_n$, $F_n$, $G_n$, $H_n$, $\{|X^{\da}_{(l_n)}| \leq 1-\eps_n/2\}$ has probability $O(c_n^{-1})$, which we accomplish as follows using Lemma \ref{lem:fewRts} for $E_n$, Lemma \ref{lem:sepRts} for $F_n$,  Lemma \ref{lem:GnSmall} for $G_n$, Lemma \ref{lem:netProb} for $H_n$, and Lemma \ref{lem:maxCDF} for $\{|X^{\da}_{(l_n)}| \leq 1-\eps_n/2\}$. 

To bound the probability of $E_n$, apply Lemma \ref{lem:fewRts} with $\gamma_n:= 2\eps_n$ and $C=3C_\mu > 2C_\mu \geq \frac{2C_\mu}{\alpha +1}$, to obtain
\[
\P(E_n) \leq \exp\left(-C_\mu\cdot n \cdot (2\eps_n)^{\alpha +1} \right) = \exp\left(-C_\mu\cdot 2^{\alpha +1} \cdot n^{1-\delta(\alpha+1)}\right) =O\left(c_n^{-1}\right)
\]
since $\delta < \frac{1}{\alpha +1}$ and $\log(c_n) = o(\log{n})$. To see that $\P(F_n)$ is negligible in the limit, observe that via Lemma \ref{lem:sepRts} with $\gamma_n := 2\eps_n$ and $d_n := \delta_n$, there exists $c>0$ such that
\[
\P(F_n) \ll \max\set{(c_n^7\log{n})^{3/2}\cdot n^{\frac{1}{2}-\delta(2\alpha +\frac{3}{2})}, \exp\left(-c\cdot n^{1-\delta(\alpha +1)}\right)} = O(c_n^{-1})
\]
since $\delta \in \left(\frac{1}{4\alpha + 3}, \frac{1}{\alpha +1}\right)$ and $\log(c_n) = o(\log(n))$. (Note: to bound the second quantity in the max, use e.g.\ $n^{1-\delta(\alpha+1)} >1/c\cdot\log{c_n}$ for large $n$.) It is clear from Lemma \ref{lem:GnSmall} that $\P(G_n) = O(c_n^{-1})$. Lemma \ref{lem:netProb} guarantees that $\P(H_n) = O(c_n^{-1})$ because $\delta_n = o(c_n^{-1})$ and (via \eqref{eqn:netParams}\footnote{Note: $c_n >1 $ and $\delta_n < \eps_n$ for large $n$ so the hypotheses of Lemma \ref{lem:epsNet} are met.})
\[
\abs{\mathcal{N}_n} \ll \frac{c_n^4\eps^{1/3}}{\delta_n^{4/3}} = \frac{n}{c_n^3\cdot \log{n}}.
\]
Finally, we appeal to Lemma \ref{lem:maxCDF} to show that $\P(\{|X^{\da}_{(l_n)}| \leq 1-\eps_n/2\}) = O(c_n^{-1})$. Indeed, $X^{\da}_{(l_n)} = X^{\ua}_{(n-l_n+1)}$, so this result follows from Lemma \ref{lem:maxCDF} with $\gamma_n = \frac{\eps_n}{2} = \frac{1}{2n^\delta}$, and $l_n = \frac{n^{1-\delta(\alpha+1)}}{c_n\log{n}}$. Note that 
\begin{align*}
l_n \cdot \log\left(n\left(\frac{1}{2n^\delta}\right)^{\alpha + 1}\right) &\leq l_n \cdot (1-\delta(\alpha + 1)) \cdot \log(n)\\
&= \frac{1-\delta(\alpha+1)}{c_n}\cdot n^{1-\delta(\alpha + 1)} =o\left(n\left(\frac{1}{2n^\delta}\right)\right)
\end{align*}
as is necessary in the hypothesis.
\end{proof}

Together, Lemmas \ref{lem:manyPair} and \ref{lem:manyPairProb} imply Theorem \ref{thm:manyPair}. In the next subsection, we conclude our justification of our main $\alpha \geq 0$ results with a proof of Theorem \ref{thm:maxPair}.


\subsection{The proof of Theorem \ref{thm:maxPair}}\label{sec:maxPairPf}
Suppose, as in the hypothesis of Theorem \ref{thm:maxPair}, that $X_1, X_2, \ldots$ are i.i.d.\ draws from a distribution $\mu$ satisfying Assumption \ref{ass:ComplexMu} with $\alpha \geq 0$, that $\ell_n = \omega(1)$ is a sequence of positive integers with $\ell_n = o\left(\sqrt[16]{\log{n}}\right)$, and that the largest $\ell_n$ critical points of $p_n(z) = \prod_{j=1}^n(z-X_j)$, in order of decreasing magnitude, are labeled $W^{\da}_{(1)}, W^{\da}_{(2)}, \ldots, W^{\da}_{(\ell_n)}$. Lemmas \ref{lem:fewRts} and \ref{lem:maxCDF} suggest that the largest in magnitude roots of $p_n$ lie within a distance of about $n^{-1/(\alpha+1)}$ of the unit circle, so our plan is to establish inequalities \eqref{eqn:maxPair} and \eqref{eqn:maxPairSharper} for all pairs of roots and critical points in the slightly wider annulus, $\mathcal{A}_n^*$ defined below. We will use Theorem \ref{thm:manyPair} to obtain, with high probability, a one-to-one pairing between critical points and roots of $p_n$ that lie in $\mathcal{A}_n^*$ and argue, by Lemma \ref{lem:sepRts}, that with probability tending to 1, the largest roots are sufficiently radially separated to guarantee that the largest critical point is paired with the largest root, the second largest critical point is paired with the second largest root, and so on. Then, we will use the results from Subsection \ref{sec:CSBehaved} for carefully chosen parameters $\eps_n, \eps_n^*, \delta_n, \delta_n^*, c_n$ depending on $n$ and $\ell_n$ to obtain Lipschitz control of the discrete Cauchy--Stieltjes sums with a smaller Lipschitz constant (depending on $n$) than was found in the proof of Theorem \ref{thm:manyPair}. Via Theorem \ref{thm:determ}, this will allow us to establish inequality \eqref{eqn:maxPairSharper} that we will use in a series of approximations to show that the variables on the left of \eqref{eqn:maxConv} are close to the scaled sums $\frac{1}{\sqrt{n}}\sum_{j=1}^n\frac{1}{U_i-X_j}$, where $U_1, \ldots, U_L$ are i.i.d.\ draws from the unit circle, independent from $X_j$. We will employ the Cram\'{e}r--Wold technique to establish the convergence in distribution indicated in \eqref{eqn:maxConv}. In the case where $\alpha =0$, we will need to use the heavy-tailed central limit Theorem \ref{thm:CLT} which will involve dividing the sums by an extra factor of $\sqrt{\log{n}}$.

Motivated by the plan above, define the sequences
\begin{align*}
	c_n &:= \ell_n^4 & \eps_n &:= \frac{1}{\ell_n^4}\left(\frac{1}{n^{\frac{1+2\alpha}{2+2\alpha}}}\right)^{{1}/{(\alpha +1)}}& \delta_n &:= \left(\frac{1}{n^{\frac{3+4\alpha}{4 + 4\alpha}}}\right)^{1/(\alpha + 1)}\\
	&&\eps^*_n &:= \left(\frac{\ell_n^2}{n}\right)^{1/(\alpha +1)} & \delta^*_n &:=\frac{1}{\ell_n^3}\left(\frac{1}{\sqrt{n}}\right)^{1/(\alpha + 1)},
\end{align*}
and as in Subsection \ref{sec:param}, define the annuli
\begin{align*}
	\mathcal{A}_n &= \set{z\in \C: 1 - \eps_n \leq \abs{z} \leq 1},\\
	\mathcal{B}_n &= \set{z\in \C: 1 - 2\eps_n \leq \abs{z} < 1-\eps_n},\\
	\mathcal{A}^*_n &=\set{z\in \C: 1 - \eps^*_n \leq \abs{z} \leq 1}.
\end{align*}
Continuing in the spirit of Subsection \ref{sec:param}, use e.g.\ Lemma \ref{lem:epsNet} to find a deterministic net $\mathcal{N}_n$ of $\mathcal{A}_n$ that satisfy
\footnote{Note: For large $n$, $c_n > 1$ and $\delta_n < \eps_n$, so  $\eps_n^{1/3}\delta_n^{2/3}c_n^{-2}$ is less than $\eps_n = 1-(1-\eps_n)$ and the hypotheses of Lemma \ref{lem:epsNet} are met.}
\begin{equation*}
	\max_{z\in \mathcal{A}_n}\left(\min_{w\in \mathcal{N}_n}\abs{z-w}\right) \leq \frac{\eps_n^{{1}/{3}}\delta_n^{{2}/{3}}}{c_n^2}\quad \text{and}\quad \abs{\mathcal{N}_n} \ll \frac{c_n^4}{\eps_n^{{2}/{3}}\delta_n^{{4}/{3}}}\cdot 1 \cdot \eps_n = \frac{c_n^4\eps_n^{{1}/{3}}}{\delta_n^{{4}/{3}}},
\end{equation*}
and in terms of these updated parameters, use the same definitions as in Subsection \ref{sec:param} for the events $E_n, F_n, F_n^*, G_n,$ and $H_n$.  We define two additional ``bad'' events
\begin{align*}
	E^*_n &:= \set{X^{\da}_{(\ell_n + 1)}\notin \mathcal{A}_n^*}\\
	&\phantom{:={}}\textit{``Some of the largest $\ell_n+1$ roots aren't in $\mathcal{A}_n^*$.''}\\[1ex]
	F^\parallel_n &:= \set{\exists\ i,j\in [n],\ i\neq j,\ X_i \in \mathcal{A}^*_n: \abs{\abs{X_{i}}-\abs{X_{j}}} \leq \frac{1}{n\ell_n^3}}\\
&\phantom{:={}}\textit{``The largest $\ell_n$ roots aren't radially separated.''}
\end{align*}
off of which we will be able to show that the largest roots and critical points of $p_n$ ``interlace'' in such a way that the largest root is paired with the largest critical point and so on.

Applying Theorem \ref{thm:manyPair} with $c_n = \ell_n^4$ and $\delta := \frac{1}{4\alpha + 2}$ shows that there is a constant $C'> 0$ and an event $\Omega_n$ whose probability is order $O(\ell_n^{-4})$ so that on the complement of $\Omega_n$
\begin{enumerate}[(I)]
	\item \label{item:maxPair1} within a distance $n^{-11/12}$ (note $\delta < {1}/{2}$) of each root 
	\[
	X_i \in \mathcal{A}_n^* \subset \set{z \in \C: 1-n^{-\frac{1}{4\alpha +2}} \leq \abs{z} \leq 1},
	\]
	there is precisely one critical point $W_i^{(n)}$ of $p_n$, and these critical points satisfy
	\[
	\max_{X_i \in \mathcal{A}^*_n}\abs{W_i^{(n)} - X_i(1-n^{-1})} < \frac{C'}{n\ell_n^4};
	\]
	\item \label{item:maxPair2} there is an injection $\iota$ from the set of critical points of $p_n$ that lie in $\mathcal{A}_n^*$ to the set of indices $[n]$ of the roots of $p_n$ so that
	\[
	\max_{{\rm c.p.}\ W \in \mathcal{A}_n^*}\abs{W - X_{\iota(W)}(1-n^{-1})}< \frac{C'}{n\ell_n^4}.
	\]
\end{enumerate}
It follows from \ref{item:maxPair1} that on the complement of $E_n^* \cup F^\parallel_n \cup \Omega_n$, $p_n$ has at least $\ell_n$ critical points in $\mathcal{A}_n^*$ (note that $X^{\da}_{(\ell_n)}$ is at least $1/(n\ell_n^3)$-far from the inner edge of $\mathcal{A}_n^*$ on the complement of $E_n^* \cup F_n^\parallel$), so via \ref{item:maxPair2}, for large $n$, on the complement of $E_n^* \cup F^\parallel_n \cup \Omega_n$, we have
\begin{equation*}
	\abs{W^{\da}_{(i)} - X^{\da}_{(i)}(1-n^{-1})} < \frac{C'}{n\ell_n^4},\ 1\leq i \leq \ell_n.
\end{equation*}
Indeed, if $W_1,W_2 \in \mathcal{A}_n^*$ with $\abs{W_1} < \abs{W_2}$ are two critical points satisfying
\[
\abs{W_i - X_{\iota(W_i)}(1-n^{-1})} < \frac{C'}{n\ell_n^4},\ i=1,2,
\]
then, using this assumption twice yields
\[
\abs{W_i} \leq \abs{X_{\iota(W_i)}}-\frac{1}{n}\abs{X_{\iota(W_i)}} + \frac{C'}{n\ell_n^4} \leq \abs{X_{\iota(W_i)}}-\frac{1}{n-1}\abs{W_i} + \frac{2C'}{(n-1)\ell_n^4} < \abs{X_{\iota(W_i)}},
\]
for $i=1,2$ (the last equality follows for large $n$ by $\abs{W_i} \geq 1-\eps_n^*$), so $X_{\iota(W_1)}, X_{\iota(W_2)} \in \mathcal{A}^*_n$. Hence, on the complement of $F_n^\parallel$, the roots $X_{\iota(W_1)}$, $X_{\iota(W_2)}$ are radially separated by $(n\ell_n^3)^{-1}$. The triangle inequality implies
\begin{equation*}
0 < \abs{W_2} - \abs{W_1} \leq \frac{2C'}{n\ell_n^4} + \frac{n-1}{n}\abs{X_{\iota(W_2)}} - \frac{n-1}{n}\abs{X_{\iota(W_1)}},
\end{equation*}
from which follows
\[
\abs{X_{\iota(W_1)}} < \frac{2C'}{(n-1)\ell_n^4} + \abs{X_{\iota(W_2)}},
\]
and the condition $\abs{\abs{X_{\iota(W_1)}} - \abs{X_{\iota(W_2)}}} > (n\ell_n^3)^{-1}$ forces $\abs{X_{\iota(W_1)}} < \abs{X_{\iota(W_2)}}$. We have proved the following lemma.
\begin{lemma}
	There is a $C' > 0$ so that for large $n$, on the complement of $E_n^* \cup F^\parallel_n \cup \Omega_n$, we have
	\begin{equation*}
		\abs{W^{\da}_{(i)} - X^{\da}_{(i)}(1-n^{-1})} < \frac{C'}{n\ell_n^4},\ 1 \leq i \leq \ell_n.
	\end{equation*}
\end{lemma}
Next, we justify the following result that is a direct consequence of Lemmas \ref{lem:nearLip} and \ref{lem:nearCS}
\begin{lemma}\label{lem:defBadEventEn} For large $n$, on the complement of 
	\[
	\mathcal{E}_n := E_n \cup E_n^* \cup F_n \cup F_n^* \cup F_n^\parallel \cup G_n \cup H_n \cup \Omega_n,
	\]
	the following are true statements for $1 \leq i \leq \ell_n$:
	\begin{enumerate}[(\thetheorem.i)]
		\item \label{item:maxDeterm1} $\displaystyle \frac{1}{2} \leq \abs{\frac{1}{n-1}\sum_{\substack{j=1\\j\neq i}}^n\frac{1}{X^{\da}_{(i)}-X^{\da}_{(j)}}} \leq \frac{1}{(1-\eps)}$
		\item \label{item:maxDeterm2} The function $z \mapsto \frac{1}{n-1}\sum_{\substack{j=1\\j\neq i}}^n\frac{1}{z-X^{\da}_{(j)}}$ is Lipschitz continuous on the open ball $B(X^{\da}_{(i)}, \frac{\delta_n}{4})$, with Lipschitz constant $k_{\rm Lip} = 2\ell_n^8n^{\frac{1+2\alpha}{2(1+\alpha)^2}} = \begin{cases}o(\sqrt{n}) & \alpha > 0\\ o(\sqrt{n\log{n}}) &\alpha = 0\end{cases}$
		\item \label{item:maxDeterm3} $\displaystyle\min_{j\in [n]\setminus \set{i}}\abs{X^{\da}_{(i)}-X^{\da}_{(j)}} > \delta_n^* = \omega(n^{-1})$,
	\end{enumerate}
	so in particular, 
	\[
	\abs{W^{\da}_{(i)} - X^{\da}_{(i)} + \frac{1}{n} \frac{1}{\frac{1}{n-1}\sum_{\substack{j=1\\ j\neq i}}^n\frac{1}{X^{\da}_{(i)}-X^{\da}_{(j)}}}} < C''\ell_n^8\cdot n^{-\frac{3}{2}-\frac{\alpha^2}{2(1+\alpha)^2}} = \begin{cases}o(n^{-3/2}) & \alpha > 0\\ o\left(\frac{\sqrt{\log{n}}}{n^{3/2}}\right) & \alpha = 0, \end{cases}
	\]
	where $C''>0$ is a deterministic constant that is independent of $n$.
\end{lemma}
\begin{proof}
	Define $\mathcal{E}_n$ as in the hypothesis. Note that for large $n$, on the complement of $\mathcal{E}_n$, all $X^{\da}_{(i)}$, $W^{\da}_{(i)}$, $1 \leq i \leq \ell_n$ are contained in $\mathcal{A}^*_n \subset \mathcal{A}_n$ so \ref{item:maxDeterm1} follows from Lemma \ref{lem:nearCS} since $m_\mu(z) = \frac{F_R(\abs{z})}{z}$ is uniformly bounded away from $0$ and above by $1/(1-\eps)$ on $\mathcal{A}_n^*$  and 
	\[
	\frac{14}{c_n} + \frac{172\log(n)\eps_n^{1/3}}{nc_n^2\delta_n^{4/3}} \ll \frac{1}{\ell_n^4} + \frac{\log(n)}{\ell_n^8}n^{-1/6}
	\]
	based on our definitions of $c_n$, $\eps_n$, and $\delta_n$. To see that \ref{item:maxDeterm2} is true, apply special case \ref{it:spCase1} of Lemma \ref{lem:nearLip} for $\alpha > 0$ and special case \ref{it:spCase2} for $\alpha =0$. (Observe in the latter case, $B(X^\da_{(i)}, \delta_n/2) \subset \mathcal{A}_n \cup B(0,1)^c$ since on the complement of $E^*_n$, each $X^\da_{(i)}$, $1 \leq i \leq \ell_n$ is at least a distance $\eps_n - \eps_n^* = \omega(\delta_n)$ away from the inner edge of $\mathcal{A}_n$. Also note that for $\alpha \geq 0$, we have $n^{-1} < \delta_n = o(\eps_n)$, and for $\alpha = 0$, we have $\delta_n =o(\delta_n^*)$, so the hypotheses of \ref{lem:nearLip} are satisfied.) Finally, on the complement of $F_n^*$, the $X^{\da}_{(i)} \in \mathcal{A}_n^*$ are separated by $\delta_n^*$, so \ref{item:maxDeterm3} is also true. We conclude the proof by applying Theorem \ref{thm:determ} $\ell_n$ times, one application with each $X^{\da}_{(i)}$ taking a turn at the role of $\xi$ (we use the same constants $C_1$, $C_2$, $k_{\rm Lip}$ for every $i$), noting that all of the bounds are uniform in $i$, $1 \leq i \leq \ell_n$.
\end{proof}

Choosing $C = \max\set{C',C''}$, we have achieved \eqref{eqn:maxPair} and \eqref{eqn:maxPairSharper} for large $n$ on the complement of $\mathcal{E}_n$, defined in the statement of Lemma \ref{lem:defBadEventEn}. We now show that this ``bad'' event $\mathcal{E}_n$ tends to $0$ in probability as $n \to \infty$.
\begin{lemma}\label{lem:BadEventsSmallFluct}
	$\P(\mathcal{E}_n) = O(\ell_n^{-1})$, where $\mathcal{E}_n = E_n \cup E_n^* \cup F_n \cup F_n^* \cup F_n^\parallel \cup G_n \cup H_n \cup \Omega_n$ is defined as in the statement of Lemma \ref{lem:defBadEventEn}.
\end{lemma}
\begin{proof}
	To bound $\P(E_n)$, apply Lemma \ref{lem:fewRts} with $\gamma_n :=2\eps_n$ and $C:= 3C_\mu > 2C_\mu \geq \frac{2C_\mu}{(\alpha + 1)}$ to obtain
	\[
	\P(E_n) \leq \exp\left(-C_\mu 2^{\alpha+1}\cdot \ell_n^{-\frac{4}{\alpha +1 }} n^{1-\frac{1+2\alpha}{2+2\alpha}}\right) \leq \exp\left(-C_\mu 2^{\alpha+1}\ell_n^{-\frac{4}{\alpha +1 }}\cdot n^{\frac{1}{2+2\alpha}}\right) \ll \frac{1}{\ell_n}.
	\]
	By Lemma \ref{lem:maxCDF} with $l_n := \ell_n$ and $\gamma_n := \eps_n^*$ we have, for a fixed constant $c>0$,
	\[
	\P(E_n^*) \ll e^{-c\ell_n^2} \ll \frac{1}{\ell_n}.
	\]
	Next, we contend with  $F_n^*$, which we break into two pieces\footnote{We note that in this section, $\alpha \geq 0$, so for the case labeled ``$\alpha \leq 0$,'' the reader can consider $\alpha = 0$. The reason we include the situation $\alpha < 0$ here is so we can reuse part of this proof and avoid repeating nearly identical arguments in Section \ref{sec:proofsNeg} below.  The definition of $F_n^*$ is one place where the case $\alpha = 0$ is more similar to $\alpha <0$ than it is to the case $\alpha > 0$. In the $\alpha >0$ case, there is a decaying root density near the edge of the unit disk, so working on the complement of $E_n$ is already sufficient to rule out ``too many'' roots near each $X_j \in \mathcal{A}_n^*$. We need finer control in the cases $\alpha =0$ and $\alpha < 0$, where the root density doesn't decay near the edge of the unit disk.}:
	\begin{align*}
	F_{n,1}^* &:= \begin{cases} \emptyset & \text{if $\alpha > 0$,}\\
		\begin{aligned}
			\Big\{\exists\ i\in [n]:{}& X_i \in \mathcal{A}^*_n,\\
			&\#\set{j\in [n],\ j \neq i: \abs{X_i-X_j} < \eps_n} > c_n^3n\eps_n^{\alpha+2} \Big\}
		\end{aligned}  &\text{if $\alpha \leq 0$;}\\
	\end{cases}\\
	F_{n,2}^* &:= \set{\exists\ i,j\in [n],\ i\neq j: X_i \in \mathcal{A}^*_n,\ \abs{X_{i}-X_{j}} \leq \delta^*_n}.
	\end{align*}
	We will start by controlling the probability of $F_{n,1}^*$ when $\alpha \leq 0$ (there is nothing to show for $\alpha > 0$). In the current section, we need only bound $\P(F^*_{n,1})$ when $\alpha = 0$, but we will write the argument to also cover the case $-1< \alpha < 0$ to avoid repeating ourselves in Section \ref{sec:proofsNeg}. In this case, for large $n$, we have\footnote{Note that for large $n$, $\eps^*_n < \eps_n$, and to obtain the second inequality, we can apply the result depicted in Figure \ref{fig:avoid1} with $\mathfrak{r} := \eps_n$, $z := e^{\sqrt{-1}\arg{z}}$.}
	\begin{align*}
	\sup_{z \in \mathcal{A}^*_n}\P\left(\abs{z-X_1} < \eps_n\right) &\leq \sup_{z \in \mathcal{A}^*_n}\P\left(\abs{e^{\sqrt{-1}\arg{z}}-X_1} < 2\eps_n\right)\\
	&\leq \P\left(1-2\eps_n \leq \abs{X_1} \leq 1 \text{ and } \arg\left(X_1\right) \in (-4\eps_n, 4\eps_n)\right)\\
	&= \P\left(1-2\eps_n \leq \abs{X_1} \leq 1\right)\cdot \frac{4}{\pi}\eps_n,
	\end{align*}
	where we have used that $\mu$ is radially symmetric (so the radial and angular parts of $X_1$ are independent). 
	Using \eqref{eqn:massOnEdge} from Lemma \ref{lem:fewRts} with $\gamma_n = 2\eps_n$ to bound the first factor on the right yields the following inequality for large $n$:
	\[
	\sup_{z\in \mathcal{A}_n^*}\P\left(\abs{z-X_1} < \eps_n\right) \leq \frac{4C_\mu}{\pi(\alpha + 1)}(2\eps_n)^{\alpha+2}.
	\]
	For each $z \in \mathcal{A}_n^*$ define 
	\begin{align*}
		\rho_{z,n} &:= \P\left(\abs{z-X_j} < \eps_n\right) \leq \frac{4C_\mu}{\pi(\alpha + 1)}(2\eps_n)^{\alpha+2}\\
		\intertext{and the random variables}
		N_{z,n} &:= \#\set{j \in [n-1]: \abs{z - X_j} < \eps_n} \sim {\rm Binomial}(n-1, \rho_{z,n}).
	\end{align*}
	An argument nearly identical to the one in the proof of Lemma \ref{lem:fewRts} (i.e., evaluating the Binomial moment-generating function at $t=1$ to achieve a Chernoff bound and applying the log-exponential trick) yields the following bound for large $n$ and any positive sequence $D_n$:
	\begin{equation}\label{eqn:nznbd}
	\sup_{z\in \mathcal{A}_n^*}\P(N_{z,n} > D_n) \leq \exp\left(-D_n + (n-1)\cdot\frac{8C_\mu(2\eps_n)^{\alpha+2}}{\pi(\alpha+1)}\right).
	\end{equation}
	If $D_n$ is any positive sequence, the independence of $X_1, \ldots, X_n$ implies
	\begin{align*}
		&\P\left(X_n \in \mathcal{A}_n^*\text{ and } \#\set{j\in [n-1] : \abs{X_n-X_j} < \eps_n} > D_n\right)\\
		&\hspace*{2mm}=\int_\mathbb{C}\ind{x_n \in \mathcal{A}_n^*}\left(\int_\mathbb{C}\cdots\int_\mathbb{C}\ind{\#\set{j\in [n-1]\: :\: \abs{x_n-x_j} < \eps_n} > D_n}\,d\mu(x_1)\cdots \,d\mu(x_{n-1})\right)\!d\mu(x_n)\\
		&\hspace*{2mm}=\int_\mathbb{C}\ind{x_n \in \mathcal{A}_n^*}\cdot \P(N_{x_n,n} > D_{n})\,d\mu(x_n),
	\end{align*}
	so in view of \eqref{eqn:nznbd} and the fact that $X_1, \ldots, X_n$ are i.i.d., we have the following for large $n$:
	 \begin{align*}
	 &\P\left(X_i \in \mathcal{A}_n^*\text{ and } \#\set{j\in [n],\ j \neq i : \abs{X_n-X_j} < \eps_n} > D_n\right)\\
	 &\qquad \leq \P(X_1 \in \mathcal{A}_n^*) \cdot \exp\left(-D_n + n\cdot\frac{8C_\mu(2\eps_{n})^{\alpha+2}}{\pi(\alpha+1)}\right).
	 \end{align*}
 	Setting $D_n = c_n^3n\eps_n^{\alpha+2}$ and applying the union bound over $i \in [n]$, establishes that for large $n$,
	 \begin{equation}
 		\P(F_{n,1}^*) \leq n\cdot \P(X_1 \in \mathcal{A}_n^*)\cdot \exp\left(-\left(c_n^3-\frac{C_\mu2^{\alpha+5}}{\pi(\alpha + 1)}\right)  n\eps_{n}^{\alpha+2}\right),\
 		 -1 < \alpha \leq 0, \label{eqn:alphaNegFstarbd}
	 \end{equation}
 	which yields the following simpler asymptotic bound in the case $\alpha=0$, after we have applied \eqref{eqn:massOnEdge} from Lemma \ref{lem:fewRts} to bound $\P(X_1 \in \mathcal{A}_n^*)$ and using the definitions $c_n = \ell_n^4$, $\eps_n = \ell_n^{-4}n^{-1/2}$, and $\eps_n^* = \ell_n^2/n$ from the top of the current Subsection \ref{sec:maxPairPf}:
 	\begin{equation}
	\P(F_{n,1}^*) \ll \ell_n^2e^{-\frac{1}{2}c_n^3n\eps_n^{\alpha+2}} = \ell_n^2e^{-\frac{1}{2}\ell_n^{4}} \leq \frac{\ell_n^2}{1+\frac{1}{2}\ell_n^4} \ll \frac{1}{\ell_n},\ \text{for $\alpha = 0$}.\label{eqn:alpha0Fstarbd}
 	\end{equation}
		
	For the remainder of the proof, we return to the case $\alpha \geq 0$. To see that $F_n, F_{n,2}^*$, and $F_n^\parallel$ hold with high probability, we apply Lemma \ref{lem:sepRts} three times. First, we take $\gamma_n:=2\eps_n$ and $d_n := \delta_n$ to obtain (for some absolute constant $c>0$)
	\begin{align*}
	\P(F_n) &\ll \max\set{\frac{1}{\ell_n^{4(2\alpha+1)}}\cdot n^{2 - \frac{(1+2\alpha)^2}{2(1+\alpha)^2}-\frac{3+4\alpha}{2(1+\alpha)^2}}, \exp\left(-\frac{c}{\ell_n^{4(\alpha+1)}}n^{1-\frac{1+2\alpha}{2+2\alpha}}\right)}\\
	&\ll \max\set{\frac{1}{\ell_n^{4(2\alpha+1)}}, \frac{1}{\ell_n}} = \frac{1}{\ell_n};
	\end{align*}
	second, we set $\gamma_n := \eps_n^*$ and $d_n:= \delta_n^*$, which yields (for some fixed $c>0$)
	\[
	\P(F_{n,2}^*) \ll \max\set{\ell_n^{\frac{2(2\alpha+1)}{\alpha+1}-6}n^{2-\frac{2\alpha +1}{\alpha +1}-\frac{1}{\alpha+1}}, e^{-c\ell_n^2}} = \max\set{\ell_n^{-\frac{2\alpha +4}{\alpha +1}}, e^{-c\ell_n^2}}\ll \frac{1}{\ell_n},
	\]
	and third, we set $\gamma_n := \eps_n^*$ and $R_n:=\frac{1}{n\ell_n^3}$ which shows that
	\[
	\P(F_n^\parallel) \ll \frac{1}{\ell_n}
	\]
	by a nearly identical inequality to the one just prior (note $R_n < (\delta_n^*)^2$). It remains to bound $\P(G_n)$ and $\P(H_n)$ since $\P(\Omega_n) = O(\ell_n^{-1})$ by its definition. Lemma \ref{lem:GnSmall} establishes that $\P(G_n) \ll \frac{1}{\ell_n^4}$ and Lemma \ref{lem:netProb} shows that 
	\[
	\P(H_n) \ll \abs{\mathcal{N}_n}\cdot \frac{\ell_n^8\log{n}}{n} \ll \frac{\ell_n^{16}\eps_n^{1/3}}{\delta_n^{4/3}}\cdot \frac{\ell_n^8\log{n}}{n}= \ell_n^{24-4/3}\log{n}\cdot n^{-1/6}\ll \frac{1}{\ell_n}
	\]
	as is desired (note that we have applied Lemma \ref{lem:epsNet} with $\eps:=\frac{\eps_n^{1/3}\delta_n^{2/3}}{c_n^2} < \eps_n$).
\end{proof}

At this point, the only conclusions of Theorem \ref{thm:maxPair} that we need to prove are the joint fluctuations results \eqref{eqn:maxConv}, namely, that 
\[
	\left(\frac{n^{3/2}}{a_ne^{\sqrt{-1}\arg(X^\da_{(i)})}}\left(W^\da_{(i)} - X^\da_{(i)}(1-n^{-1})\right)\right)_{i=1}^L \longrightarrow( N_1, \ldots,  N_L)
\]
in distribution as $n\to \infty$, where $( N_1, \ldots,  N_L)$ is a vector whose coordinates have complex Gaussian marginal distributions with mean zero and covariance structure given by
\begin{align*}
	\var\left(\Re( N_i)\right) &=
	\begin{cases}
		\frac{\pi f_\mu(1)}{4} & \text{if $\alpha =0$,}\\
		\var\left(\Re\left(\frac{X_1}{1-X_1}\right)\right) & \text{if $\alpha > 0$};
	\end{cases}\\
	\var\left(\Im( N_i)\right) &= 
	\begin{cases}
		\frac{\pi f_\mu(1)}{4} & \text{if $\alpha =0$,}\\
		\var\left(\Im\left(\frac{X_1}{1-X_1}\right)\right) & \text{if $\alpha > 0$;}
	\end{cases}\\
	\cov\left(\Re( N_i),\Im( N_i)\right) &= 0,
\end{align*}
and that the coordinates $(N_1, \ldots, N_L)$ are i.i.d.\ when $\alpha = 0$ and have joint structure described by Lemma \ref{lem:CLT} when $\alpha > 0$. We give a heuristic outline of our argument before providing the technical details in Lemmas \ref{lem:topCSTermsSmall} and \ref{lem:CLT} below.

Toward establishing \eqref{eqn:maxConv}, we seek to understand the joint behavior as $n \to \infty$ of the sums $\sum_{\substack{j=1\\ j\neq i}}^n\frac{1}{X^{\da}_{(i)}-X^{\da}_{(j)}}$, $1 \leq i \leq L$, but the dependence between the order statistics $X^{\da}_{(i)}$ and $X^{\da}_{(j)}$, $j \neq i$ prevents our direct application of classical central limit theorems. To circumvent this obstruction, we will argue that there is no change in the joint distributional limit when we remove the summands involving the largest $L$ of the $X^{\da}_{(j)}$ and replace $X^{\da}_{(i)}$, $1 \leq i \leq L$ with the unit vectors $e^{\sqrt{-1}\arg(X^{\da}_{(i)})}$, which are independent from $X^{\da}_{(L+1)}, \ldots, X^{\da}_{(n)}$ by the radial symmetry of the distribution $\mu$. As a consequence, we will have reduced the problem to understanding the joint limiting behavior of the variables $\sum_{j=L+1}^{n}\frac{X^{\da}_{(j)}}{e^{\sqrt{-1}\arg(X^{\da}_{(i)})}-X^{\da}_{(j)}}$, $1\leq i \leq L$ which have the same distribution as $\sum_{j=L+1}^{n}\frac{X^{\da}_{(j)}}{{U_i}-X^{\da}_{(j)}}$, $1\leq i \leq L$, where $U_1, \ldots, U_L$ are i.i.d.\ uniform draws from the unit circle that are independent from $X_1, X_2, \ldots$ (to center the sums we will need to add the factors $X^{\da}_{(j)}$ in the numerators). We will conclude by adding the ``missing'' summands involving the largest roots $X^{\da}_{(1)},\ldots, X^{\da}_{(L)}$ and applying the classical central limit theorem for the case $\alpha > 0$ and the heavy-tailed result, Corollory \ref{cor:CLT}, for the case $\alpha = 0$. (Note: When $\alpha > 0$, the scaling factor is $1/\sqrt{n}$ as in the classical central limit theorem, but when $\alpha =0$, we need the slightly smaller scaling factor $1/\sqrt{n\log{n}}$. See Lemma  \ref{lem:CLT} and Corollary \ref{cor:CLT} for more details.)

Lemma \ref{lem:topCSTermsSmall} shows that there is no harm in adjusting the summands involving the largest $L$ roots according to our plan above (because we are only perturbing the sums by finitely many terms which contribute negligibly in the limit), and Lemma \ref{lem:CLT} establishes \eqref{eqn:maxConv} in view of Slutsky's Theorem and the Cram\'{e}r--Wold technique.

\begin{lemma}\label{lem:topCSTermsSmall}
	Let $L$ be a fixed positive integer, let $\ell_n$ be a positive sequence satisfying $\omega(1)=\ell_n = o(n)$, and suppose $X_1, X_2, \ldots$ are i.i.d.\ draws from a distribution $\mu$ that satisfies Assumption \ref{ass:ComplexMu} for $\alpha > -1$. Then,
	\begin{align*}
		&\P\left(\exists j \in \set{1,\ldots, L}: \frac{1}{\abs{1-X^{\da}_{(j)}}}>\ell_n\right) = O\left(\frac{1}{\ell_n}\right),
		\intertext{and}	
		&\P\left(\exists i,j \in \set{1,\ldots L},\ i\neq j: \frac{1}{\abs{e^{\sqrt{-1}\arg(X^{\da}_{(i)})}-X^{\da}_{(j)}}}>\ell_n\right) = O\left(\frac{1}{\ell_n}\right),
	\end{align*}
	and if $U_1, \ldots, U_L$ are any collection of i.i.d.\ draws from the unit circle in $\C$ that are jointly independent from $X_1, X_2,\ldots$, we have
	\[
	\P\left(\exists i, j \in \set{1, \ldots, L}: \frac{1}{\abs{U_i-X^{\da}_{(j)}}}>\ell_n\right) = O\left(\frac{1}{\ell_n}\right).
	\]
\end{lemma}
\begin{proof}
	Since $\mu$ is radially symmetric, for a fixed pair $i,j \in \set{1,\ldots, L}$ with $i \neq j$, the variables $\arg\left(X^{\da}_{(i)}\right)$ and $X^{\da}_{(j)}$ are independent and $e^{\sqrt{-1}\arg(X^{\da}_{(i)})}X^{\da}_{(j)}$ has the same distribution as $X^{\da}_{(j)}$. Consequently, via the union bound over at most $L^2$ parings of the form $(\arg(X^{\da}_{(i)}), X^{\da}_{(j)})$ or $(U_i, X^{\da}_{(j)})$, it suffices to show that for each $j \in \set{1, \ldots, L}$,
	\begin{equation}
		\P\left( \frac{1}{\abs{1-X^{\da}_{(j)}}}>\ell_n\right) = O\left(\frac{1}{\ell_n}\right).
		\label{eqn:maxAvoid1}
	\end{equation}
	Applying Lemma \ref{lem:maxCDF} with $\gamma_n = \left(\frac{\ell_n}{n}\right)^{1/(\alpha +1)}$ and $l_n = L$ (note the difference between the parameters $\ell_n$ and $l_n$) implies that there is a positive constant $c$ so that
	\begin{align*}
		&\P\left( \frac{1}{\abs{1-X^{\da}_{(j)}}}>\ell_n\right)\\
		&\qquad= \P\left( \frac{1}{\abs{1-X^{\da}_{(j)}}}>\ell_n,\ \abs{X^{\da}_{(j)}} > 1- \left(\frac{\ell_n}{n}\right)^{1/(\alpha +1)}\right) + O(e^{-c\cdot \ell_n}).
	\end{align*}
Now, $\gamma_n=o(\ell_n^{-1})$, so for $n$ large enough to guarantee that $\ell_n^{-1} < 1$, the conditions $|1-X^{\da}_{(j)}|^{-1}>\ell_n$ and $|X^{\da}_{(j)}| > 1- \gamma_n$ imply $-2\ell_n^{-1} \leq \arg(X^{\da}_{(j)}) \leq 2\ell_n^{-1}$ via a cubic lower approximation to the sine function (see also Figure \ref{fig:avoid1}).
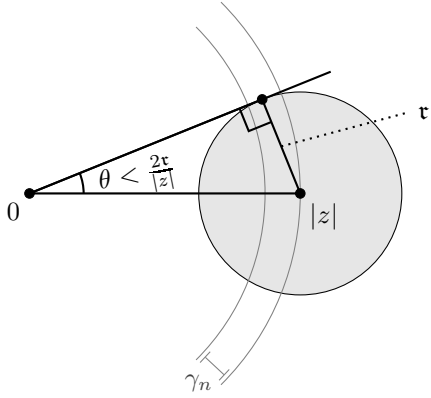
\begin{figure}
	\begin{center}
		\begin{tikzpicture}[scale=1.8]
			\draw[fill=gray!20] (2,0) circle (0.75);
			\draw[draw=gray] (2*0.7071, 2*-0.7071) arc (-45:45:2);
			\draw[draw=gray] (2*0.87*0.7071, 2*0.87*-0.7071) arc (-45:45:2*0.87);
			\draw (0,0) [fill=black] circle (1pt) node[below left] {$0$} -- (2,0) [fill=black] circle (1pt) node[below right]{$\abs{z}$} -- (2,0);
			\draw[|-|, rotate = -43, draw =gray] (2*0.885,0)--(2*0.985,0);
			\draw (1.25,-1.4) node{\textcolor{darkgray}{$\gamma_n$}};
			\draw[thick] (0,0) --(2.4*1.854/2,2.4*3/8);
			\draw[thick] (0,0) --(1.854*1.854/2, 1.854*3/8) --(2,0)--(0,0);
			\draw[rotate=22.02,thick] (1.854, -0.18)--(1.854 - 0.18, -0.18)--(1.854 - 0.18, 0);
			\draw[fill=black] (1.854/2*1.854, 1.854/2*3/4) circle(1pt);
			\draw[thick] (0.4,0) arc (0:22.02:0.4);
			\draw (0.8,0.12) node[rotate=11.01]{$\theta < \frac{2\mathfrak{r}}{\abs{z}}$};
			\draw[dotted, thick] (1.859, 0.348)--(2.8, 0.6) node[right] {$\mathfrak{r}$};
		\end{tikzpicture}
	\end{center}
	\caption{If $\mathfrak{r}<\abs{z}$, then $\theta < \frac{2\mathfrak{r}}{\abs{z}}$. In particular, setting $z = 1$ and $\mathfrak{r} = \ell_n^{-1}$, we have the following: if $\gamma_n= o(\ell_n^{-1}) < 1$ and $|X^{\da}_{(j)}| \geq 1-\gamma_n$, then $|1-X^{\da}_{(j)}| < \ell_n^{-1}$ implies $\arg(X^{\da}_{(j)}) \in (-2\ell_n^{-1}, 2\ell_n^{-1})$. \label{fig:avoid1}}
\end{figure}
Equation \eqref{eqn:maxAvoid1} follows because $e^{-c\cdot \ell_n} \ll \ell_n^{-1}$ and $\arg(X^{\da}_{(j)})$ is uniformly distributed on $(-\pi, \pi]$ independently of $|X^{\da}_{(j)}|$.
\end{proof}

\begin{lemma}\label{lem:CLT}
	Fix a positive integer $L$, and define\footnote{Note: When $\alpha > 0$, $|1-X_1|^{-1}$ has a finite second moment, so the variance of $\sum_{j=1}^n{|1-X_j|}^{-1}$ is order $O(n)$, but when $\alpha = 0$,  ${|1-X_1|}^{-1}$ has an infinite second moment and  $\sum_{j=1}^n|1-X_j|^{-1}$ is order $O(\sqrt{n\log{n}})$. See also Lemma \ref{lem:moments} and Theorem \ref{thm:CLT}.} $a_n := \sqrt{\log{n}}$ if $\alpha = 0$ and $a_n := 1$ if $\alpha > 0$. Then, for each $i$, $1 \leq i \leq L$,
	\begin{equation}\label{eqn:circApprox}
	\abs{\frac{1}{a_n\sqrt{n}}\sum_{\substack{j=1\\ j\neq i}}^n \frac{X^{\da}_{(j)}}{X^{\da}_{(i)}-X^{\da}_{(j)}}- \frac{1}{a_n\sqrt{n}}\sum_{j=L+1}^{n}\frac{X^{\da}_{(j)}}{e^{\sqrt{-1}\arg(X^{\da}_{(i)})}-X^{\da}_{(j)}}}  \to 0
	\end{equation}
	in distribution as $n\to \infty$. Additionally, for any fixed $t_1, \ldots, t_L \in \C$, the (real-valued) random variable
	\[
	\Re\left(\frac{1}{a_n\sqrt{n}}\sum_{k=1}^L\sum_{j=L+1}^{n}\frac{t_k\cdot X^{\da}_{(j)}}{e^{\sqrt{-1}\arg(X^{\da}_{(k)})}-X^{\da}_{(j)}}\right)
	\]
	converges in distribution as $n\to \infty$ to a random variable $\mathfrak{L}$ with mean zero.  When $\alpha = 0$, $\mathfrak{L}$ is a Gaussian random variable with mean zero and variance 
	\[ \sum_{k=1}^L \frac{\pi |t_k|^2 f_\mu(1)}{4}. \]
	When $\alpha > 0$, $\mathfrak{L}$ has the same distribution as the product $N \sigma(U_1, \ldots, U_L)$, where $N$ and $\sigma(U_1, \ldots, U_L)$ are independent random variables, $N$ is a standard normal random variable, 
	\begin{equation} \label{eq:def:sigma}
		\sigma^2(e^{\sqrt{-1}\theta_1}, \ldots, e^{\sqrt{-1}\theta_L}) = \var\left(\sum_{k=1}^L\Re\left(\frac{t_kX_1}{e^{\sqrt{-1}\theta_k} - X_1}\right)\right) 
	\end{equation}
	for deterministic $\theta_1, \ldots, \theta_L \in [0, 2 \pi)$, and $U_1, \ldots, U_L$ are i.i.d.\ draws of a uniform random variable on the unit circle in the complex plane, independent of all other sources of randomness. In particular, when $L = 1$, by the rotational invariance of $X_1$, $\sigma(U_1)$ is a constant and $\mathfrak{L}$ is a mean zero Gaussian random variable with variance 
	\[ \var \left(\Re\left(\frac{t_1X_1}{1 - X_1}\right)\right). \]
\end{lemma}
\begin{remark}
The distribution of $\mathfrak{L}$ when $\alpha > 0$ is known as a compound Gaussian distribution; we can view the distribution as a mean zero normal distribution with random variance $\sigma^2(U_1, \ldots, U_L)$.  
\end{remark}
\begin{proof}[Proof of Lemma \ref{lem:CLT}]
	Fix $i \in \set{1,\ldots, L}$ and observe that 
	\begin{align*}
		&\sum_{\substack{j=1\\ j\neq i}}^n \frac{X^{\da}_{(j)}}{X^{\da}_{(i)}-X^{\da}_{(j)}}-\sum_{j=L+1}^{n}\frac{X^{\da}_{(j)}}{e^{\sqrt{-1}\arg(X^{\da}_{(i)})}-X^{\da}_{(j)}}\\
		&\quad= \sum_{\substack{j=1\\ j\neq i}}^L \frac{X^{\da}_{(j)}}{e^{\sqrt{-1}\arg(X^{\da}_{(i)})}-X^{\da}_{(j)}} + \sum_{\substack{j=1\\ j\neq i}}^n \frac{X^{\da}_{(j)}(e^{\sqrt{-1}\arg(X^{\da}_{(i)})}-X^{\da}_{(i)})}{(X^{\da}_{(i)}-X^{\da}_{(j)})(e^{\sqrt{-1}\arg(X^{\da}_{(i)})}-X^{\da}_{(j)})}.
	\end{align*}

Via Lemma \ref{lem:topCSTermsSmall}, with probability at least $1-O(\ell_n^{-1})$, each of the $L-1$ terms in the first sum on the right is bounded above by $\ell_n$. Consequently, with probability at least $1-O(\ell_n^{-1})$, 
\[
\abs{\frac{1}{a_n\sqrt{n}}\sum_{\substack{j=1\\ j\neq i}}^n \frac{X^{\da}_{(j)}}{X^{\da}_{(i)}-X^{\da}_{(j)}}- \frac{1}{a_n\sqrt{n}}\sum_{j=L+1}^{n}\frac{X^{\da}_{(j)}}{e^{\sqrt{-1}\arg(X^{\da}_{(i)})}-X^{\da}_{(j)}}} \leq \frac{L\cdot\ell_n}{a_n\sqrt{n}} + \frac{1}{a_n\sqrt{n}}\abs{\sum_{\substack{j=1\\ j\neq i}}^nQ_{i,j,n}},
\]	
where,
\[
Q_{i,j,n} := \frac{X^{\da}_{(j)}(e^{\sqrt{-1}\arg(X^{\da}_{(i)})}-X^{\da}_{(i)})}{(X^{\da}_{(i)}-X^{\da}_{(j)})(e^{\sqrt{-1}\arg(X^{\da}_{(i)})}-X^{\da}_{(j)})},\ \text{for $j \in [n]$, $j \neq i$}.
\]
Since $\frac{L\cdot\ell_n}{a_n\sqrt{n}} = o(1)$, we can establish the convergence in Equation \eqref{eqn:circApprox} by showing that as $n\to \infty$,  $\frac{1}{a_n\sqrt{n}}\abs{\sum_{\substack{j=1\\ j\neq i}}^nQ_{i,j,n}}$ converges to $0$ in distribution. We proceed differently in the cases $\alpha = 0$ and $\alpha > 0$.

\textbf{Case I: $\alpha = 0$.} On the complement of $\mathcal{E}_n = E_n \cup E_n^* \cup F_n \cup F_n^* \cup F_n^\parallel \cup G_n \cup H_n \cup \Omega_n$, as defined in Lemma \ref{lem:defBadEventEn}, we know, $X^{\da}_{(i)} \in \mathcal{A}_n^*$ and $X^{\da}_{(i)}$ is separated by $\delta_n^*$ from all other $X^{\da}_{(j)}$. Lemma \ref{lem:BadEventsSmallFluct} guarantees that $\P(\mathcal{E}_n) = O(\ell_n^{-1})$, and Lemma \ref{lem:topCSTermsSmall} guarantees that with probability at least $1-O(\ell_n^{-1})$, each of $\abs{e^{\sqrt{-1}\arg(X^{\da}_{(i)})}-X^{\da}_{(j)}}^{-1}$, $j < L$, are bounded by $\ell_n$. Thus, with probability at least $1- O(\ell_n^{-1})$,
\begin{align*}
\frac{1}{a_n\sqrt{n}}\abs{\sum_{\substack{j=1\\ j\neq i}}^nQ_{i,j,n}} &\leq \frac{1}{\sqrt{\log{n}}\cdot \sqrt{n}}\sum_{\substack{j=1\\ j\neq i}}^n\frac{1\cdot \eps^*_n}{\delta_n^*\cdot \abs{e^{\sqrt{-1}\arg(X^{\da}_{(i)})}-X^{\da}_{(j)}}}\\
&\leq \frac{\eps_n^*\cdot L\cdot\ell_n}{\delta^*_n\sqrt{n\log{n}}} +  \frac{\eps_n^*\sqrt{n}}{\delta_n^*\sqrt{\log{n}}}S_{(i,n)},
\end{align*}
	where 
	\[
	S_{(i,n)} := \frac{1}{n}\sum_{j=L+1}^{n}\frac{1}{\abs{e^{\sqrt{-1}\arg(X^{\da}_{(i)})}-X^{\da}_{(j)}}}.
	\]
	Now, $X_1, X_2,\ldots$ are i.i.d.\ and radially symmetric, so the random vector $(X^{\da}_{(1)},\ldots, X^{\da}_{(n)})$ is equal in distribution to the random vector
	\[
	\left(X^{\da}_{(1)}, \ldots, X^{\da}_{(L)}, e^{\sqrt{-1}\arg(X^{\da}_{(i)})}X^{\da}_{(L+1)},\ldots, e^{\sqrt{-1}\arg(X^{\da}_{(i)})}X^{\da}_{(n)}\right).
	\]
	Specifically, the angles remain i.i.d.\ uniform and independent of $\set{X^{\da}_{(j)}}_{j > L}$. It follows that $S_{(i,n)}$ is equal in distribution to 
	\[
	S_{(i,n)}' := \frac{1}{n}\sum_{j=L+1}^{n}\frac{1}{\abs{1-X^{\da}_{(j)}}} = \frac{1}{n}\sum_{j=1}^{n}\frac{1}{\abs{1-X_{j}}} - \frac{1}{n}\sum_{j=1}^L\frac{1}{\abs{1-X^{\da}_{(j)}}},
	\]
	which, in view of Lemma \ref{lem:topCSTermsSmall}, converges in probability to a constant by the law of large numbers and Slutsky's theorem. We conclude that \eqref{eqn:circApprox} is true for each $i \in \set{1,\ldots L}$ by another application of Slutsky's theorem since 
	\[
	\frac{\eps_n^*\cdot L\ell_n}{\delta^*_n\cdot \sqrt{n\log{n}}} = \frac{\ell_n^5}{\sqrt{n}}\cdot \frac{L\ell_n }{\sqrt{n\log{n}}} = o(1)\quad \text{and} \quad \frac{\eps_n^*\cdot \sqrt{n}}{\delta_n^*\cdot \sqrt{\log{n}}} \leq \frac{\ell_n^5}{\sqrt{n}}\cdot \frac{\sqrt{n}}{\sqrt{\log{n}}} = o(1).
	\]
	Recall here that via the definitions in Subsection \ref{sec:maxPairPf}, we have $\ell_n = o(\sqrt[16]{\log{n}})$, $\eps^*_n = \ell_n^2/n$ and $\delta_n^* = 1/(\ell_n^3\sqrt{n})$. (For the last two of these we set $\alpha =0$.) 
	
	\textbf{Case II: $\alpha > 0$.} We will use Markov's inequality with a second moment bound to show that in probability (and hence in distribution), as $n \to \infty$,
	\[
	\frac{1}{a_n\sqrt{n}}\abs{\sum_{\substack{j=1\\ j\neq i}}^nQ_{i,j,n}} = \frac{1}{\sqrt{n}}\abs{\sum_{\substack{j=1\\ j\neq i}}^n\frac{X^{\da}_{(j)}(e^{\sqrt{-1}\arg(X^{\da}_{(i)})}-X^{\da}_{(i)})}{(X^{\da}_{(i)}-X^{\da}_{(j)})(e^{\sqrt{-1}\arg(X^{\da}_{(i)})}-X^{\da}_{(j)})}} \to 0.
	\]
	Before we invoke the moment bound, we first remove some large terms from the sum by showing that the $X^\da_{(j)}$ which are the very closest to $X^\da_{(i)}$ contribute negligibly with high probability. In view of Lemmas \ref{lem:BadEventsSmallFluct} and \ref{lem:topCSTermsSmall}, we have with probability at least $1-O(\ell_n^{-1})$ that
	$X^{\da}_{(i)} \in \mathcal{A}_n^*$,  $X^{\da}_{(i)}$ is separated by $\delta_n^*$ from all other $X^{\da}_{(j)}$, and $\abs{e^{\sqrt{-1}\arg(X^{\da}_{(i)})}-X^{\da}_{(j)}}^{-1}$, $j \leq L$, are bounded by $\ell_n$. Consequently, with probability tending to $1$, 
	\begin{equation}\label{eqn:rev:largeL}
	\frac{1}{\sqrt{n}}\abs{\sum_{\substack{j=1\\ j\neq i}}^LQ_{i,j,n}} \leq \frac{1}{\sqrt{n}}\sum_{\substack{j=1\\ j\neq i}}^L\frac{1 \cdot \eps^*_n \cdot \ell_n}{\delta_n^*}\leq \frac{L\cdot \ell_n}{\sqrt{n}} \cdot \frac{\eps^*_n}{\delta_n^*} = \frac{L\cdot\ell_n}{\sqrt{n}}\cdot \ell_n^3 \left(\frac{\ell_n^2\sqrt{n}}{n}\right)^{1/(\alpha + 1)} = o(1),
	\end{equation}
	where we have used the definitions of $\eps_n^*$ and $\delta_n^*$ in Subsection \ref{sec:maxPairPf} (note $\alpha > 0$). In addition, by Equation \eqref{eqn:annulusNumBd} from Lemma \ref{lem:fewRts} with $\gamma_n := 3\eps_n^*$, it follows that with probability at least $1 - O(e^{-\ell_n^2}) = 1-O(\ell_n^{-1})$, there are at most
	\[
	\ell_n \cdot n ( 3\eps_n^*)^{\alpha+1} = 3^{\alpha+1}\ell_n^3
	\]
	roots $X^\da_{(j)}$, $j \in [n]$, satisfying $1-3\eps_n^* \leq \abs{X^\da_{(j)}}\leq 1$. This means, with probability approaching 1, 
	\begin{equation}\label{eqn:rev:largeXj}
	\frac{1}{\sqrt{n}}\abs{\sum_{\substack{j=1\\ j\neq i}}^nQ_{i,j,n}\ind{\{|X^\da_{(j)}|\geq 1- 3 \eps_n^*\}}} \leq \frac{3^{\alpha+1}\ell_n^3}{\sqrt{n}}\cdot \frac{1}{\delta_n^*} = 3^{\alpha+1}\ell_n^6n^{\frac{-\alpha}{2(\alpha+1)}} = o(1),
	\end{equation}
	where for each of the terms with $\ind{\{|X^\da_{(j)}|\geq 1- 3 \eps_n^*\}}$, we've bounded $\abs{X^\da_{(i)} - X^\da_{(j)}}$ below by $\delta_n^*$, and we've used
	\[
	\abs{e^{\sqrt{-1}\arg(X^{\da}_{(i)})}-X^{\da}_{(i)}} \leq \abs{e^{\sqrt{-1}\arg(X^{\da}_{(i)})}-X^{\da}_{(j)}}
	\]
	to bound the factor $\abs{e^{\sqrt{-1}\arg(X^{\da}_{(i)})}-X^{\da}_{(i)}}$ in the numerator of $\abs{Q_{i,j,n}}$ (among $z$ in the unit disk satisfying, $\abs{z} \leq \abs{X^{\da}_{(i)}}$,  $X^{\da}_{(i)}$ is nearest to $e^{\sqrt{-1}\arg(X^{\da}_{(i)})}$). With Equations \eqref{eqn:rev:largeL} and \eqref{eqn:rev:largeXj} in hand, to obtain Equation \eqref{eqn:circApprox} in the $\alpha > 0$ case, it suffices to show  
	\[
	\frac{1}{\sqrt{n}}\abs{\sum_{j=L+1}^nQ_{i,j,n}\ind{\{|X^\da_{(j)}| < 1-3\eps_n^*\}}} \to 0
	\]
	in probability as $n\to \infty$.
	
	In preparation for invoking a second-moment Markov bound, we establish
	\[
	\E[Q_{i,j,n}\cdot\ind{E^*_{i,j}}] = 0, \quad \text{and} \quad \E[Q_{i,j,n}\cdot\ind{E^*_{i,j}}\cdot \overline{Q_{i,k,n}\cdot\ind{E^*_{i,k}}}] =0,
	\]
	for $j, k > i$ with $j \neq k$, where $E_{i,j}^*$ are defined to be the events
	\[
	E_{i,j}^* := \set{\abs{X^\da_{(i)}} \geq 1-\eps_n^*} \cup \set{\abs{X^\da_{(j)}} < 1-3\eps_n^* }.
	\]
	By Fubini's theorem (which can be justified when we bound $\E|Q_{i,j,n}^2\ind{E^*_{i,j}}|$ later on) and by the independence\footnote{The original $X_1, \ldots, X_n$ are i.i.d. with a rotationally symmetric distribution, so the angular parts of the order statistics $X^\da_{(1)}, \ldots, X^\da_{(n)}$ remain i.i.d.\ and independent of the radial parts.} of $e^{\sqrt{-1}\arg(X^{\da}_{(j)})}$ from the $\sigma$-algebra generated by
	\[\left(e^{\sqrt{-1}\arg(X^\da_{(i)})}, \abs{X^\da_{(i)}}, \abs{X^\da_{(j)}}\right),\] we can compute $\E[Q_{i,j,n}\ind{E^*_{i,j}}]$ as an iterated integral where we integrate with respect to the argument of $X^{\da}_{(j)}$ in the innermost integral. In particular,
	\begin{align*}
		&Q_{i,j,n}\ind{E^*_{i,j}}\\
		&\ =\frac{\ind{\{|X^\da_{(i)}| \geq 1-\eps_n^*\} }\ind{\{|X^\da_{(j)}| < 1-3\eps_n^*\} }\abs{X^\da_{(j)}}(e^{\sqrt{-1}\arg(X^{\da}_{(i)})}-X^{\da}_{(i)}) \cdot e^{-\sqrt{-1}\arg(X^{\da}_{(j)})}}{X^\da_{(i)}e^{\sqrt{-1}\arg(X^{\da}_{(i)})}\cdot \left(e^{-\sqrt{-1}\arg(X^{\da}_{(j)})}- \frac{\abs{X^\da_{(j)}}}{X^\da_{(i)}}\right)\left(e^{-\sqrt{-1}\arg(X^{\da}_{(j)})}-\frac{\abs{X^\da_{(j)}}}{e^{\sqrt{-1}\arg(X^{\da}_{(i)})}}\right)},
	\end{align*}
	so, letting $\nu$ denote the probability measure on $\mathbb{R}^3$ associated to the joint distribution of $(\arg(X^\da_{(i)}), |X^\da_{(i)}|, |X^\da_{(j)}|)$, we have
	\begin{align*}
		&\E[Q_{i,j,n}\ind{E_{i,j}^*}]\\
		&\ =\int_{\mathbb{R}^3}g\left(\theta_i, r_i, r_j\right) \frac{1}{2\pi}\int_{-\pi}^{\pi}\frac{e^{-\sqrt{-1}\theta_j}}{(e^{-\sqrt{-1}\theta_j}-\frac{r_j}{r_ie^{\sqrt{-1}\theta_i}})(e^{-\sqrt{-1}\theta_j}-\frac{r_j}{e^{\sqrt{-1}\theta_i}})}\,d\theta_j\,d\nu(\theta_i, r_i, r_j),\\
		&\ = \int_{\mathbb{R}^3}g\left(\theta_i, r_i, r_j\right) \frac{-1}{2\pi\sqrt{-1}}\oint_{\gamma}\frac{1}{(z-\frac{r_j}{r_ie^{\sqrt{-1}\theta_i}})(z-\frac{r_j}{e^{\sqrt{-1}\theta_i}})}\,dz\,d\nu(\theta_i, r_i, r_j),
	\end{align*}
	where we have re-written the integral with respect to $\theta_j$ as a complex line integral over the unit circle parameterized by $\gamma(\theta_j) = e^{-\sqrt{-1}\theta_j}$, and
	\[
		g(\theta_i, r_i, r_j) := \frac{r_j(e^{\sqrt{-1}\theta_i}-r_ie^{\sqrt{-1}\theta_i})}{r_ie^{\sqrt{-1}\theta_i}e^{\sqrt{-1}\theta_i}}\cdot \sind{r_i\geq \eps_n^*}\sind{r_j < 1-3\eps_n^*}.
	\]
	Since $\frac{r_j}{r_ie^{\sqrt{-1}\theta_i}}$ and $\frac{r_j}{e^{\sqrt{-1}\theta_i}}$ are both in the interior of the unit disk (almost surely, $|X^\da_{(j)}| < |X^\da_{(i)}| < 1$), Cauchy's residue theorem shows that the line integral evaluates to 0. Hence $\E[Q_{i,j,n}\ind{E_{i,j}^*}] = 0$ for $j > i$. A nearly identical argument establishes that $\E[Q_{i,j,n}\ind{E_{i,j}^*}\cdot\overline{Q_{i,k,n}\ind{E_{i,k}^*}}] = 0$ for $j\neq k$, $j, k > i$, with the only differences being that in the iterated integral,  $\nu$ is a measure on $\mathbb{R}^5$ instead of on $\mathbb{R}^3$, and in addition to $g(\theta_i, r_i, r_j)$, we factor from the innermost integral all of the factors corresponding to $\overline{Q_{i,k,n}\ind{E_{i,k}^*}}$.
	
	At this point, in the case $\alpha > 0$, we have shown that for a fixed $i \in [L]$,
	\begin{itemize}
		\item $\frac{1}{\sqrt{n}}|\sum_{\substack{j=1\\ j\neq i}}^LQ_{i,j,n}| \to 0$ in probability as $n\to \infty$, 
		\item $\frac{1}{\sqrt{n}}\abs{\sum_{\substack{j=1\\ j\neq i}}^nQ_{i,j,n}\ind{\{|X^\da_{(j)}|\geq 1- 3 \eps_n^*\}}} \to 0$ in probability as $n\to \infty$,
		\item $\E[Q_{i,j,n}\ind{E_{i,j}^*}] = 0$ for $j > i$, and
		\item $\E[Q_{i,j,n}\ind{E_{i,j}^*}\cdot\overline{Q_{i,k,n}\ind{E_{i,k}^*}}] =0$ for $i > j, k$ with $j \neq k$.
	\end{itemize}
	In view of these facts, we establish Equation \eqref{eqn:circApprox} by proving
	\[
	\frac{1}{\sqrt{n}}\abs{\sum_{j=L+1}^nQ_{i,j,n}\ind{\{|X^\da_{(j)}|< 1- 3 \eps_n^*\}}} \to 0
	\]
	in probability as $n\to \infty$ using Markov's inequality and control of the second moments of $Q_{i,j,n}$. To that end, let $\eta > 0$ be arbitrary, and consider
	\begin{equation}\label{eqn:rev:Aha}
	\begin{aligned}
		&\P\left(\frac{1}{\sqrt{n}}\abs{\sum_{j=L+1}^nQ_{i,j,n}\ind{\{|X^\da_{(j)}|< 1- 3 \eps_n^*\}}} > \eta\right)\\
		&\qquad\leq \P\left(\frac{1}{\sqrt{n}}\abs{\sum_{j=L+1}^nQ_{i,j,n}\ind{E_{i,j}^*}} > \eta\right) + \P\left(X^\da_{(i)} \notin \mathcal{A}_n^*\right) \\
		&\qquad\leq\frac{1}{\eta^2 n}\E\left[\abs{\sum_{j=L+1}^nQ_{i,j,n}\ind{E_{i,j}^*}}^2\right] + O(\ell_n^{-1})\\
		&\qquad=\frac{1}{\eta^2 n}\E\left[\sum_{j=L+1}^n\abs{Q_{i,j,n}\ind{E_{i,j}^*}}^2\right] + o(1),
	\end{aligned}
\end{equation}
	where we've used Lemma \ref{lem:BadEventsSmallFluct} to bound $\P(X^\da_{(i)} \notin \mathcal{A}_n^*)$ and that \[\E\left[Q_{i,j,n}\ind{E_{i,j}^*}\cdot\overline{Q_{i,k,n}\ind{E_{i,k}^*}}\right] =0\] to eliminate the covariance terms in the expectation. We will control the expectation of the last sum by approximating $X^\da_{(i)}$ with $e^{\sqrt{-1}\arg(X^\da_{(i)})}$, and adding in the ``missing'' terms so that this sum resembles a sum of i.i.d.\ variables to which we can apply the absolute moment bound from Lemma \ref{lem:moments}. On $E^*_{i,j}$, we have $\abs{X^\da_{(j)}} < 1-3\eps_n^*$, so $\abs{e^{\sqrt{-1}\arg(X^\da_{(i)})}-X^\da_{(j)}} > 3\eps_n^*$, and	by the reverse triangle inequality, we have
	\begin{align*}
	\abs{X^\da_{(i)} - X^\da_{(j)}}^2\ind{E_{i,j}^*} &\geq	\left(\abs{e^{\sqrt{-1}\arg(X^\da_{(i)})} - X^\da_{(j)}} - 	\abs{X^\da_{(i)}-e^{\sqrt{-1}\arg(X^\da_{(i)})}}\right)^2\ind{E_{i,j}^*}\\
	&\geq \left(\abs{e^{\sqrt{-1}\arg(X^\da_{(i)})} - X^\da_{(j)}} - \eps^*_n\right)^2\ind{E_{i,j}^*}\\
	&\geq \left(\abs{e^{\sqrt{-1}\arg(X^\da_{(i)})} - X^\da_{(j)}}^2 - 2\eps_n^*\abs{e^{\sqrt{-1}\arg(X^\da_{(i)})} - X^\da_{(j)}}\right)\ind{E_{i,j}^*}\\
	&\geq \frac{1}{3}\abs{e^{\sqrt{-1}\arg(X^\da_{(i)})} - X^\da_{(j)}}^2,
	\end{align*}
where we've used that $\frac{1}{3}|X^\da_{(j)}-e^{\sqrt{-1}\arg(X^\da_{(i)})}| > \eps_n^*$ on the event $E_{i,j}^*$ to achieve the last bound. Now,
\begin{align*}
		\E\left[\sum_{j=L+1}^n\abs{Q_{i,j,n}\ind{E_{i,j}^*}}^2\right] &\leq 3\cdot\E\left[\sum_{j =L+1}^n\frac{1 \cdot \abs{e^{\sqrt{-1}\arg(X^{\da}_{(i)})}-X^{\da}_{(i)}}^2\cdot \ind{\{|X^\da_{(i)}| \geq 1-\eps_n^*\}} }{\abs{e^{\sqrt{-1}\arg(X^\da_{(i)})} - X^\da_{(j)}}^4}\right]\\
		&\leq 3(\eps^*_n)^{\alpha/2} \cdot \E\left[\sum_{j = 1}^n\frac{1}{\abs{e^{\sqrt{-1}\arg(X^\da_{(i)})} - X_j}^{2+\alpha/2}}\right]\\
		& \leq 3(\eps^*_n)^{\alpha/2} \cdot \frac{2C_{\ref{lem:moments}}}{\alpha} \cdot n,
\end{align*} 
where we've used $|X^\da_{(i)}| \geq 1-\eps_n^*$ to bound $\abs{e^{\sqrt{-1}\arg(X^{\da}_{(i)})}-X^{\da}_{(i)}}^{\alpha/2}$ by $(\eps_n^*)^{\alpha/2}$,
\[
\abs{e^{\sqrt{-1}\arg(X^{\da}_{(i)})}-X^{\da}_{(i)}} \leq \abs{e^{\sqrt{-1}\arg(X^{\da}_{(i)})}-X^{\da}_{(j)}}
\]
to bound the remaining factor $\abs{e^{\sqrt{-1}\arg(X^{\da}_{(i)})}-X^{\da}_{(i)}}^{2-\alpha/2}$ in the numerator\footnote{Among $z$ in the unit disk satisfying, $\abs{z} \leq \abs{X^{\da}_{(i)}}$,  $X^{\da}_{(i)}$ is nearest to $e^{\sqrt{-1}\arg(X^{\da}_{(i)})}$.}, and Equation \eqref{eqn:pMomentCirc} from Lemma \ref{lem:moments} to bound the $2+\alpha/2$-th moment of  $\abs{e^{\sqrt{-1}\arg(X^\da_{(i)})} - X_j}^{-1}$ by the constant $\frac{2C_{\ref{lem:moments}}}{\alpha}$ which depends only on $\mu$ (and hence on $\alpha$ and $\eps$). Combining this second moment bound with Inequality \eqref{eqn:rev:Aha} yields 
\[
\P\left(\frac{1}{\sqrt{n}}\abs{\sum_{j=L+1}^nQ_{i,j,n}\ind{\{|X^\da_{(j)}|< 1- 3 \eps_n^*\}}} > \eta\right) \leq \frac{6C_{\ref{lem:moments}}}{\eta^2\alpha}\cdot (\eps_n^*)^{\alpha/2} + o(1) = o(1),
\]
from which we conclude \eqref{eqn:circApprox} in the case $\alpha > 0$, since $\eps_n^* \to 0$ and $\eta > 0$ was arbitrary.
	
We now prove the second conclusion of Lemma \ref{lem:CLT} that gives the joint fluctuations of the terms \[\frac{1}{a_n\sqrt{n}}\sum_{j=L+1}^{n}\frac{X^{\da}_{(j)}}{e^{\sqrt{-1}\arg(X^{\da}_{(i)})}-X^{\da}_{(j)}}, i \in \set{1, 2, \ldots, L}.\]	
To see that the second part of the conclusion holds, fix $t_1, \ldots, t_L \in \mathbb{C}$ and consider that due to the fact that $X_1, X_2,\ldots$ are i.i.d.\ and rotationally symmetric, the random vector
	\[
	\left(e^{\sqrt{-1}\arg(X^{\da}_{(1)})}, \ldots, e^{\sqrt{-1}\arg(X^{\da}_{(L)})}, X^{\da}_{(L+1)}, \ldots, X^{\da}_{(n)}\right)
	\]
	has the same distribution as the random vector
	\[
	(U_1, \ldots, U_L, X^{\da}_{(L+1)}, \ldots, X^{\da}_{(n)}),
	\]
	where $U_1, \ldots, U_L$ are i.i.d.\ draws from the unit circle that are jointly independent from $X_1, X_2, \ldots$. Consequently, the random variable
	\[
	Y_{n} := \Re\left(\frac{1}{a_n\sqrt{n}}\sum_{k=1}^L\sum_{j=L+1}^{n}\frac{t_k\cdot X^{\da}_{(j)}}{e^{\sqrt{-1}\arg(X^{\da}_{(k)})}-X^{\da}_{(j)}}\right)
	\]
	has the same distribution as
	\[
	\Re\left(\frac{1}{a_n\sqrt{n}}\sum_{k=1}^L\sum_{j=L+1}^{n}\frac{t_k\cdot X^{\da}_{(j)}}{U_k-X^{\da}_{(j)}}\right) = \Re\left(\frac{1}{a_n\sqrt{n}}\sum_{j=1}^{n}\left(\sum_{k=1}^L\frac{t_k\cdot X_{j}}{U_k-X_{j}}\right)\right) + O\left(\frac{L^2\ell_n}{a_n\sqrt{n}}\right),
	\]
	where the asymptotic bound is true with high probability due to Lemma \ref{lem:topCSTermsSmall}. In the case where $\alpha = 0$, applying the heavy-tailed central limit theorem Corollary \ref{cor:CLT} (and Slutsky's theorem) implies $Y_n \to \mathfrak{L}$ in distribution as $n \to \infty$, where $\mathfrak{L}$ is the desired Gaussian random variable. 
	
	When $\alpha > 0$, we apply a slightly more involved argument.  Define 
	\[ \mathfrak S_n(\mathbf u) = \Re\left(\frac{1}{\sqrt{n}}\sum_{j=1}^{n}\left(\sum_{k=1}^L\frac{t_k\cdot X_{j}}{u_k -X_{j}}\right)\right) \]
	for deterministic $\mathbf{u} = (u_1, \ldots u_L) \in \mathbb{T}^L$, where $\mathbb{T}$ is the unit circle in the complex plane.  We view $\mathfrak S_n$ as a random process on the space $\mathbb{T}^L$.  We will show that $\mathfrak{S}_n$ converges in distribution as $n \to \infty$ to a Gaussian process.  We will do so by establishing the convergence of the finite dimensional distributions of $\mathfrak{S}_n$ and then establishing a bound which implies tightness of $\mathfrak{S}_n$ on the space of continuous functions on $\mathbb{T}^L$; these arguments are fairly standard so we will only provide a sketch of the proof. 
	Note that $\mathfrak S_n(\mathbf u)$ has mean zero since
	\begin{equation}\label{eqn:centeredMeanZero}
	\E\left[\frac{t_kX_j}{u_k-X_j}\right] = t_k\E\left[\frac{1}{1-X_j/u_k} - 1\right] = t_k(m_\mu(1) - 1) = 0,
	\end{equation}
	and 
	\[ \frac{1}{\sqrt{n}}\sum_{j=1}^{n}\left(\sum_{k=1}^L\frac{t_k\cdot X_{j}}{u_k-X_{j}}\right) \]
	is a sum of $n$ i.i.d.\ random variables (scaled by $1/\sqrt{n}$).  Therefore, it follows from the classical central limit theorem that the finite dimensional distributions of $\mathfrak S_n$ converge to a mean zero Gaussian random variable whose variance can be expressed in a form similar to the definition of $\sigma^2(u_1, \ldots, u_L)$ in \eqref{eq:def:sigma}.  Moreover, since the random variables $\frac{t_k\cdot X_{j}}{u_k-X_{j}}$ have finite second moment, a computation shows that 
	\[ \var(\mathfrak{S}_n(\mathbf u) - \mathfrak{S}_n(\mathbf u')) \leq C \| \mathbf u - \mathbf u'\|_2^2, \]
	where $\mathbf u, \mathbf u' \in \mathbb{T}^L$, $C > 0$ is a constant (depending on $t_1, \ldots, t_L$), and $\| \cdot \|_2$ is the $L^2$-norm on $\mathbb{T}^L$.  
	By Theorems 7.5 and 12.3 from \cite{MR233396}, it follows that $(\mathfrak{S}_n(\mathbf u))_{\mathbf u \in \mathbb{T}^L}$ converges in distribution to a Gaussian process $(\mathfrak{S}(\mathbf u))_{\mathbf u \in \mathbb{T}^L}$, where, for fixed $\mathbf{u} = (u_1, \ldots, u_L) \in \mathbb{T}^L$, the marginal distribution of $\mathfrak S(\mathbf u)$ is a mean-zero Gaussian distribution with variance $\sigma^2(u_1, \ldots, u_L)$.  It then follows that $\mathfrak{S}_n(U_1, \ldots, U_L)$ converges in distribution to $\mathfrak{S}(U_1, \ldots, U_L) =: \mathfrak{L}$, as desired.  
\end{proof}

In view of the Cram\'{e}r--Wold technique, we conclude the proof of Theorem \ref{thm:maxPair} by combining \eqref{eqn:maxPairSharper} with the conclusion of Lemma \ref{lem:CLT} and applying Slutsky's theorem several times. In particular, we argue that with high probability (indeed, on the complement of $\mathcal{E}_n$ defined in Lemma \ref{lem:defBadEventEn}; see also Lemma \ref{lem:BadEventsSmallFluct}), \eqref{eqn:maxPairSharper} holds, so for each $i \in \set{1, \ldots, L}$,
\begin{equation}
\begin{aligned}
&\frac{n^{3/2}}{a_ne^{\sqrt{-1}\arg(X^{\da}_{(i)})}}\left(W^{\da}_{(i)} - X^{\da}_{(i)}(1-n^{-1})\right)\\
&\qquad = o(1) + \frac{\sqrt{n}}{a_ne^{\sqrt{-1}\arg(X^{\da}_{(i)})}}\left(X^\da_{(i)} - \frac{1}{\frac{1}{n-1}\sum_{\substack{j=1\\j\neq i}}^n\frac{1}{X^{\da}_{(i)}-X^{\da}_{(j)}}}\right)\\
&\qquad = o(1) + \frac{ \sqrt{n}\cdot e^{-\sqrt{-1}\arg(X^{\da}_{(i)})}}{a_n\cdot \frac{1}{n-1}\sum_{\substack{j=1\\j\neq i}}^n\frac{1}{X^{\da}_{(i)}-X^{\da}_{(j)}}}\cdot\frac{1}{n-1}\sum_{\substack{j=1\\j\neq i}}^n\left(\frac{X^\da_{(i)}}{X^{\da}_{(i)}-X^{\da}_{(j)}} -1\right)\\
&\qquad = o(1) + \frac{e^{-\sqrt{-1}\arg(X^{\da}_{(i)})}}{\frac{1}{n-1}\sum_{\substack{j=1\\j\neq i}}^n\frac{1}{X^{\da}_{(i)}-X^{\da}_{(j)}}}\cdot \frac{n}{n-1}\left(\frac{1}{a_n\sqrt{n}}\sum_{\substack{j=1\\ j\neq i}}^n \frac{X^{\da}_{(j)}}{X^{\da}_{(i)}-X^{\da}_{(j)}}\right),
\end{aligned}\label{eqn:fluctEqn}
\end{equation}
where the asymptotic is uniform over all $i \in \set{1,\ldots, L}$. Lemma \ref{lem:CLT} describes the joint limiting behavior of the rightmost factor\footnote{See Lemma \ref{lem:CovStructurePos} in Appendix \ref{sec:appendix} for details on why this implies the covariance structure \eqref{eqn:CovStructurePos}.}, so taking Slutsky's theorem into consideration, the only thing left to do is to show that for each $i \in \set{1,\ldots, L}$, 
\[
e^{\sqrt{-1}\arg(X^{\da}_{(i)})}\cdot\frac{1}{n-1}\sum_{\substack{j=1\\j\neq i}}^n\frac{1}{X^{\da}_{(i)}-X^{\da}_{(j)}} \to 1
\]
in distribution as $n\to \infty$. Indeed, on the complement event $\mathcal{E}_n$, defined as in Lemma \ref{lem:defBadEventEn} (note: $\P(\mathcal{E}_n) = o(1)$ by Lemma \ref{lem:BadEventsSmallFluct}), Lemma \ref{lem:nearCS} shows that
\[
\max_{1 \leq i \leq L}\abs{\frac{1}{n-1}\sum_{\substack{j=1\\j\neq i}}^n\frac{1}{X^{\da}_{(i)}-X^{\da}_{(j)}}-m_\mu(X^{\da}_{(i)})} = o(1),
\]
and by Lemma \ref{lem:StielCalc},
\[
e^{\sqrt{-1}\arg(X^{\da}_{(i)})}\cdot m_\mu(X^{\da}_{(i)}) = \frac{F_R(|X^{\da}_{(i)}|)}{|X^{\da}_{(i)}|}\to 1,
\]
where the convergence is uniform over $i \in \set{1,\ldots, L}$ since for such $i$, $X^{\da}_{(i)} \in \mathcal{A}_n^*$ which is an annulus converging to the unit circle as $n \to \infty$.


\section{Proof of results when $\alpha < 0$}\label{sec:proofsNeg}

In the case where $\alpha < 0$, we construct our arguments in the same spirit as those in Section \ref{sec:proofsPos} above, but several challenges arise because the radial density $f_R(r)$ is unbounded near the unit circle. First, the largest in magnitude roots of $p_n$ are located within a distance $o(n^{-1})$ of the unit circle, but the largest in magnitude critical points are at a distance from the unit circle that is on the order $\Theta(n^{-1})$. This means that no critical points lie within $\mathcal{A}_n^*$ as it was defined in Section \ref{sec:proofsPos}, so we introduce the annulus $\mathcal{A}_n^-$ of width $\ell_n/n$ in which we will analyze the largest in magnitude critical points. Second, in Section \ref{sec:proofsPos}, we used the radial separation of the largest in magnitude $X_i$ to prove that $X^\da_{(1)}$ pairs with $W^\da_{(1)}$, $X^\da_{(2)}$ pairs with $W^\da_{(2)}$ and so on. When $\alpha < 0$, the radial ``gaps'' between the largest roots are on the order $n^{-1/(1+\alpha)}$ (up to logs) which is smaller than the distance $1/(n\ell_n^3)$ we used in our definition of $F_n^\parallel$ when $\alpha > 0$. To contend with this difference when $\alpha < 0$, we will need tighter control of $\abs{\overline{M}_\mu(z) - m_\mu(z)}$ motivating us to change our definition of the event $H_n$ to the event $H_n^*$ below, off of which $\overline{M}_\mu(z)$ is closer to $m_\mu(z)$ for $z$ in deterministic nets $\mathcal{N}_n^*$ that necessarily have fewer points then the nets $\mathcal{N}_n$ did in Section \ref{sec:proofsPos}. Accordingly, we define the nets $\mathcal{N}_n^*$ below to be nets of $\mathcal{A}_n^*$ rather than $\mathcal{A}_n$ as was done in Section \ref{sec:proofsPos}. Finally, since $m_z(z)$ is not quite Lipschitz continuous when $\alpha < 0$, we define the annulus $\mathcal{A}_n^*$ to have outer radius $1-(n\ell_n)^{-1/(\alpha+1)}$ in order to employ Corollary \ref{cor:CSnice} for $x, y \in \mathcal{A}_n^*$.

As we did in Section \ref{sec:proofsPos}, we will start the proofs by introducing some ``bad'' events and parameters in Subsection \ref{sec:paramsNeg}, and then we will establish control of $\overline{M}_\mu(z)$ off of the ``bad'' events in Subsection \ref{sec:regCSNeg}. We prove Theorem \ref{thm:manyPairNeg} in Subsection \ref{sec:manyPairNegProof} by showing its conclusions hold off of the ``bad'' events whose probabilities tend to zero as $n\to \infty$. We conclude by proving Theorem \ref{thm:manyPair} in Subsection \ref{sec:fluctNegProof}.

\subsection{Important parameters and events for proof of Theorem \ref{thm:manyPairNeg}}\label{sec:paramsNeg}
As in the hypothesis of Theorem \ref{thm:manyPairNeg}, suppose $X_1, X_2, \ldots$ are i.i.d.\ draws from a distribution $\mu$ that satisfies Assumption \ref{ass:ComplexMu} with $-0.095\leq\alpha < 0$, let $\ell_n = \omega(1)$ be a sequence of positive integers with $\ell_n \leq \sqrt{\log{n}}$, and in order of decreasing magnitude, denote the largest $\ell_n$ critical points of $p_n:= \prod_{j=1}^n(z-X_j)$ by $W^{\da}_{(1)}, \ldots, W^{\da}_{(\ell_n)}$. In the style of Subsection \ref{sec:maxPairPf}, define
\[
\eps_n := \left(\frac{1}{n}\right)^{\frac{1}{2(\alpha + 1)}}, \quad \delta_n := \log{n}\cdot\left(\frac{1}{n}\right)^{\frac{\alpha+3}{4(\alpha+1)}}, \quad \text{and}\quad \eps^*_n := \left(\frac{\ell_n^2}{n}\right)^{\frac{1}{\alpha +1}},
\]
and the annuli
\begin{align*}
	\mathcal{A}_n &:= \set{z\in \C: 1 - \eps_n \leq \abs{z} \leq 1},\\
	\mathcal{B}_n &:= \set{z\in \C: 1 - 2\eps_n \leq \abs{z} < 1-\eps_n},\\
	\intertext{and (since $m_\mu(z)$ is not Lipschitz continuous in an annulus that includes the unit circle [see Corollary \ref{cor:CSnice}]) the (asymptotically) narrower annuli}
	\mathcal{A}^-_n &:=\set{z\in \C: 1 - \frac{\ell_n}{n} \leq \abs{z} \leq 1},
	\intertext{where we expect to find the largest critical points of $p_n$, and }
	\mathcal{A}^*_n &:=\set{z\in \C: 1 - \eps^*_n \leq \abs{z} \leq 1-\eps_n^*\ell_n^{-3/(\alpha + 1)}},
\end{align*}
 where we expect to find $X^{\da}_{(1)}, \ldots, X^{\da}_{(\ell_n)}$. (Note that $\mathcal{A}_n^*$ is defined slightly differently here than in the proofs of the $\alpha \geq 0$ results because we will eventually employ Corollary \ref{cor:CSnice} for $x, y \in \mathcal{A}_n^*$, and if $x$ or $y$ is on the unit circle, the bound in \eqref{eqn:CSniceNeg} is meaningless.)

Using strategies similar to those we employed in earlier sections, we will show that off of some ``bad'' events defined below, $\overline{M}_\mu(z)$ is close to $m_\mu(z)$ on a net of deterministic points $\mathcal{N}^*_n$ of $\mathcal{A}^*_n$ that grows with $n$. To this end, use e.g.\ Lemma \ref{lem:epsNet} to find nets $\mathcal{N}^*_n$ of $\mathcal{A}^*_n$, depending on $n$, which satisfy\footnote{Note: Our choice of $\eps_n^* > 0$ implies   $\ell_n^{-6}\cdot n^{-\frac{1-3\alpha}{2(\alpha + 1)}} > \eps_n^*$ for $-0.095 \leq \alpha < 0$, so the hypotheses of Lemma \ref{lem:epsNet} are met. (e.g.\ Apply Lemma \ref{lem:epsNet} with  $r_1 = 1- \eps_n^*$, $r_2 = 1$, $\eps = \ell_n^{-6}\cdot n^{-\frac{1-3\alpha}{2(\alpha + 1)}}$, and use the $18\eps^{-1}r_2$ bound.) }
\begin{equation}
	\max_{z\in \mathcal{A}^*_n}\left(\min_{w\in \mathcal{N}^*_n}\abs{z-w}\right) \leq \frac{1}{\ell_n^6}\left(\frac{1}{n}\right)^{\frac{1-3\alpha}{2(\alpha + 1)}}\quad \text{and}\quad \abs{\mathcal{N}^*_n} \ll \ell_n^6\cdot n^{\frac{1-3\alpha}{2(\alpha + 1)}}.
	\label{eqn:netParamsNeg}
\end{equation}

We now specify the ``bad'' events off of which $\overline{M}_\mu(z)$ is close to $m_\mu(z)$ and $z\mapsto \overline{M}_\mu(z)$ has a manageable Lipschitz constant:\footnote{We note that $E_n$, $F_n$, and $G_n$ are defined exactly as in Subsection \ref{sec:param} (with the caveat that in the definition of $G_n$, the sequence $c_n$ has been replaced by $\ell_n$), $E_n^*$ and $F^\parallel_n$ are defined in the same spirit as their counterparts in Subsection \ref{sec:maxPairPf} (the definition of $\mathcal{A}^*_n$ has been modified as mentioned just above, and $F_n^\parallel$ has been adjusted to account for the unbounded radial density $f_R(r)$), and $H_n$ has been replaced by $H_n^*$ since $\mathcal{N}_n^*$ are nets of $\mathcal{A}_n^*$ instead of $\mathcal{A}_n$.}
{\allowdisplaybreaks
	\begin{align*}
		E_n &:= \set{\#\set{i \in [n]: X_i \in \mathcal{A}_n\cup \mathcal{B}_n} > 3C_\mu n(2\eps_n)^{\alpha+1}}\\ 
		&\phantom{:={}}\textit{``There are too many $X_i$ in $\mathcal{A}_n \cup \mathcal{B}_n$''}\\[1ex]		
		E^*_n &:= \set{\exists i \in \set{1, \ldots, \ell_n + 1} : X^{\da}_{(i)}\notin \mathcal{A}_n^*}\\
		&\phantom{:={}}\textit{``Some of the largest $\ell_n+1$ roots aren't in $\mathcal{A}_n^*$.''}\\[1ex]				
		F_n &:= \set{\exists\ i,j\in [n],\ i\neq j: X_i \in \mathcal{A}_n \cup \mathcal{B}_n,\ \abs{X_i-X_j} \leq \delta_n}\\
		&\phantom{:={}}\textit{``The $X_i$ in $\mathcal{A}_n \cup \mathcal{B}_n$ aren't sufficiently isolated.''}\\[1ex]
		F^\parallel_n &:= \set{\exists\ i,j\in [n],\ i\neq j,\ X_i \in \mathcal{A}^*_n: \abs{\abs{X_{i}}-\abs{X_{j}}} \leq \left(\frac{1}{n\ell_n^3}\right)^{1/(\alpha + 1)}}\\
		&\phantom{:={}}\textit{``The $X_i$ in $\mathcal{A}^*_n$ aren't radially separated.''}\\[1ex]
		G_n &:= \set{\sum_{i=1}^n\frac{1}{(1-\frac{3}{2}\eps_n -\abs{X_i})^2}\cdot \sind{\abs{X_i}< 1 -2\eps_n}\geq n\ell_n\eps_n^{\alpha-1}}\\
		&\phantom{:={}}\begin{minipage}{.85\linewidth}\textit{``The  $X_i \notin \mathcal{A}_n \cup \mathcal{B}_n$ are too big (too close to the $X_j \in \mathcal{A}_n$) on average.''}\end{minipage}\\[1ex]
		H_n^* &:= \set{\max_{z\in \mathcal{N}^*_n}\abs{\frac{1}{n}\sum_{i=1}^n\frac{1}{z-X_i}\sind{\abs{z-X_i}> \frac{\delta_n}{2}}-m_\mu(z)}\geq \frac{1}{\ell_n^{4/(\alpha +1)}}\left(\frac{1}{n}\right)^{\frac{-\alpha}{\alpha+1}}}\\
		&\phantom{:={}}\begin{minipage}{.85\linewidth}\textit{``$\overline{M}_\mu(z)$ is insufficiently close to $m_\mu(z)$ for points in the net $\mathcal{N}^*_n$.''}\end{minipage}
	\end{align*}
}


\subsection{Regularity of $\overline{M}_\mu(z)$ off of the ``bad'' events when $\alpha < 0$.}\label{sec:regCSNeg}

This section is devoted to the proof of the following two lemmas concerning the behavior of $\overline{M}_\mu(z)$. We note that the results in Lemmas \ref{lem:nearLipNeg} and \ref{lem:nearCSNeg} hold for $\alpha$ even more negative than $\alpha \in (-0.095, 0)$, which is the regime we ultimately consider in Subsection \ref{sec:manyPairNegProof} and beyond. The place in our argument that requires the strictest bound on $\alpha$, namely $\alpha > -0.095$, is in the proof of Lemma \ref{lem:badEventSmallNeg}.

\begin{lemma}\label{lem:nearLipNeg}
	Suppose $-\frac{1}{3} < \alpha < 0$. For large $n$, on the complement of $E_n \cup F_n \cup G_n$, for all $z, w \in \mathcal{A}_n \cup B(0,1)^c$, the differences $\abs{\overline{M}_\mu(z) - \overline{M}_\mu(w)}$ are finite, and
	\begin{equation}
	\abs{\overline{M}_\mu(z) - \overline{M}_\mu(w)} < \abs{z-w}\left(2 \ell_n n^{\frac{1-\alpha}{2(\alpha + 1)}}\right) + \frac{8}{n\delta_n},
	\label{eqn:nearLipNeg}
	\end{equation}
and when $z,w$ are in the open ball $B(X_j, \frac{\delta_n}{2})$ for some $X_j \in \mathcal{A}_n$, $j\in [n]$, the last term, $\frac{8}{n\delta_n}$, can be replaced with $0$.
\end{lemma}
\begin{proof}
	This follows from Lemma \ref{lem:nearLip} with $c_n := \ell_n$, and with $\eps_n$, $\delta_n$, and $\mathcal{A}_n$ as defined in Subsection \ref{sec:paramsNeg}. (Note that after setting $c_n := \ell_n$,  the ``bad'' events $E_n$, $F_n$, and $G_n$ have the same definitions in terms of $\eps_n, \delta_n, \ell_n$ in Section \ref{sec:proofsPos} as they do in Subsection \ref{sec:paramsNeg}.) When $-1/3 < \alpha < 0$, the condition $1/n \leq \delta_n  \leq \frac{\eps_n}{2} \leq \frac{1}{4}$ from the hypothesis of Lemma \ref{lem:nearLip} is met for large $n$ because (comparing exponents on factors of $1/n$ in the definitions of $\delta_n$ and $\eps_n$)
	\[
	1 > \frac{\alpha + 3}{4(\alpha +1)} = \frac{\alpha + 3}{2} \cdot \frac{1}{2(\alpha + 1)} > \frac{1}{2(\alpha + 1)}>0.
	\]
	To achieve the simplified bounds in \eqref{eqn:nearLipNeg}, recall $\ell_n = o(\log{n})$, and observe that
	\begin{align}
	c_n\cdot \eps_n^{\alpha-1} &= \ell_n \left(\frac{1}{n}\right)^{\frac{\alpha-1}{2(\alpha+1)}} = \ell_n\cdot n^{\frac{1-\alpha}{2(\alpha+1)}}, \label{eqn:nearLipNegfirstBd}\\
	\frac{172\log{n}}{n\delta_n^2} &= \frac{172}{\log{n}}\cdot n^{\frac{\alpha+3}{2(\alpha+1)}-1} = \frac{172}{\log{n}}\cdot n^{\frac{1-\alpha}{2(\alpha+1)}}.\notag
	\end{align}
\end{proof}

\begin{lemma}\label{lem:nearCSNeg}
	Suppose $-1/7 < \alpha < 0$. For large $n$, on the complement of $E_n \cup F_n \cup G_n \cup H_n^*$, 
	\begin{equation}\label{eqn:nearCSNeg}
		\sup_{1-\frac{2\ell_n}{n} < \abs{z}\leq 1-\eps_n^*\ell_n^{-3/(\alpha + 1)}}\abs{\overline{M}_\mu(z) - m_\mu(z)} < \frac{2}{\ell_n^{4/(\alpha +1)}}\left(\frac{1}{n}\right)^{\frac{-\alpha}{\alpha+1}}.
	\end{equation}
\end{lemma}
\begin{proof}
	Suppose the complement of $E_n \cup F_n \cup G_n \cup H_n^*$ occurs and $n$ is large enough for the conclusions of Lemma \ref{lem:nearLipNeg} to hold and for the net $\mathcal{N}_n^*$ of $\mathcal{A}_n^*$ to satisfy \eqref{eqn:netParamsNeg}. When $-1/5 < \alpha <0$, we have $\frac{2\ell_n}{n} = o\left(\ell_n^{-6}\left(\frac{1}{n}\right)^{\frac{1-3\alpha}{2(\alpha+1)}}\right)$, so for any fixed $z$ with $1-2\ell_n/n < \abs{z} \leq 1-\eps_n^*\ell_n^{-3/(\alpha + 1)}$, there is a $\xi_z \in \mathcal{N}_n^* \subset \mathcal{A}_n$ for which 
	\begin{equation}\label{eqn:netAndHBds}
	\abs{z-\xi_z} \leq \frac{2}{\ell_n^6}\cdot \left(\frac{1}{n}\right)^{\frac{1-3\alpha}{2(\alpha+1)}} \text{ and }   \abs{\frac{1}{n}\sum_{i=1}^n\frac{\sind{\abs{\xi_z-X_i}> \frac{\delta_n}{2}}}{\xi_z-X_i}-m_\mu(\xi_z)} < \ell_n^{\frac{-4}{\alpha +1}}\left(\frac{1}{n}\right)^{\frac{-\alpha}{\alpha+1}}.
	\end{equation}
	(Note the second of these is true on the complement of $H_n^*$.) In view of the triangle inequality, we will overestimate $\abs{\overline{M}_\mu(z) - m_\mu(z)}$ by comparing $\overline{M}_\mu(z)$ to $\overline{M}_\mu(\xi_z)$, $\overline{M}_\mu(\xi_z)$ to $m_\mu(\xi_z)$, and $m_\mu(\xi_z)$ to $m_\mu(z)$. Applying \eqref{eqn:onlyOneClose} with $\xi_z$ in place of $z$ to the second inequality in \eqref{eqn:netAndHBds} shows that 
	\begin{equation*}
	\abs{\overline{M}_\mu(\xi_z) - m_\mu(\xi_z)} = \abs{\frac{1}{n}\sum_{j \neq i_{\xi_z}}\frac{1}{\xi_z - X_j} - m_\mu(\xi_z)} < \frac{1}{\ell_n^{4/(\alpha +1)}}\left(\frac{1}{n}\right)^{\frac{-\alpha}{\alpha+1}} + \frac{2}{n\delta_n},
	\end{equation*}
	and (for $-1 < \alpha < 0$) Corollary \ref{cor:CSnice} implies
	\begin{align*}
	\abs{m_\mu(\xi_z) - m_\mu(z)} &< \kappa_\mu(1-\max{\abs{z},\abs{\xi_z}})^\alpha\abs{\xi_z - z} <\kappa_\mu\cdot2\ell_n^{\frac{-\alpha}{\alpha+1}-6}\left(\frac{1}{n}\right)^{\frac{\alpha}{\alpha +1}+\frac{1-3\alpha}{2(\alpha+1)}}\\
	&=2\kappa_\mu\cdot\ell_n^{\frac{-\alpha}{\alpha+1}-6}\left(\frac{1}{n}\right)^{\frac{1-\alpha}{2(\alpha+1)}} = o\left(\frac{1}{\ell_n^{4/(\alpha +1)}}\left(\frac{1}{n}\right)^{\frac{-\alpha}{\alpha+1}}\right),
	\end{align*}
	where we used that $|z|, |\xi_z| \leq 1- \eps_n^*\ell_n^{-3/(\alpha + 1)}$, and the last asymptotic holds for $-1 < \alpha < 0$. Via \eqref{eqn:nearLipNeg} from Lemma \ref{lem:nearLipNeg}, we have 
	\begin{align*}
	\abs{\overline{M}_\mu(z) - \overline{M}_\mu(\xi_z)} &\leq \frac{2}{\ell_n^6}\cdot\left(\frac{1}{n}\right)^{\frac{1-3\alpha}{2(1+\alpha)}}\cdot\left(2 \ell_n n^{\frac{1-\alpha}{2(\alpha + 1)}}\right) + \frac{8}{n\delta_n}\\
	&= \frac{4}{\ell_n^5}\cdot \left(\frac{1}{n}\right)^{\frac{-\alpha}{1+\alpha}} + \frac{8}{n\delta_n} = o\left(\frac{1}{\ell_n^{4/(\alpha +1)}}\left(\frac{1}{n}\right)^{\frac{-\alpha}{\alpha+1}}\right) + \frac{8}{n\delta_n},
	\end{align*}
	where the last asymptotic holds for $-1/5 < \alpha < 0$. Combining the last three inequalities yields
	\begin{align*}
		\abs{\overline{M}_\mu(z) - m_\mu(z)} &\leq \abs{\overline{M}_\mu(z) - \overline{M}_\mu(\xi_z)} + \abs{\overline{M}_\mu(\xi_z) - m_\mu(\xi_z)} + \abs{m_\mu(\xi_z)-m_\mu(z)}\\
		 &< \frac{1}{\ell_n^{4/(\alpha +1)}}\left(\frac{1}{n}\right)^{\frac{-\alpha}{\alpha+1}} + \frac{10}{n\delta_n} + o\left(\frac{1}{\ell_n^{4/(\alpha +1)}}\left(\frac{1}{n}\right)^{\frac{-\alpha}{\alpha+1}}\right)
	\end{align*}
	where 
\[
\frac{10}{n\delta_n} = 10 \cdot \log{n} \cdot n^{\frac{\alpha +3}{4(\alpha + 1)}-1} = 10\cdot \log{n} \left(\frac{1}{n}\right)^{\frac{1+3\alpha}{4(\alpha+1)}} =  o\left(\frac{1}{\ell_n^{4/(\alpha +1)}}\left(\frac{1}{n}\right)^{\frac{-\alpha}{\alpha+1}} \right),
\]
and the asymptotic holds for $-1/7 < \alpha < 0$. Since these bounds are uniform over $z$ with $1-\frac{2\ell_n}{n} < \abs{z}\leq 1-\eps_n^*\ell_n^{-3/(\alpha + 1)}$, inequality \eqref{eqn:nearCSNeg} follows.
\end{proof}


\subsection{The proof of Theorem \ref{thm:manyPairNeg}}\label{sec:manyPairNegProof}

To prove Theorem \ref{thm:manyPairNeg}, we follow a similar strategy to the one used in Subsections \ref{sec:manyPairConcHold}, \ref{sec:badEventsSmall}, and \ref{sec:manyPairPf}. First, we will establish in Lemma \ref{lem:manyPairConcNeg} that for large $n$, the conclusions of Theorem \ref{thm:manyPairNeg} hold on the complement of
\[
\mathcal{E}^-_n := E_n \cup E_n^* \cup F_n \cup F^\parallel_n \cup G_n \cup H_n^*\cup \set{\abs{X^\da_{(n^\delta)}} \leq 1-\frac{\ell_n-3}{n}},
\]
where the constituent ``bad'' events are defined as in Subsection \ref{sec:paramsNeg} above. Then, we will show in Lemma \ref{lem:badEventSmallNeg} that $\P(\mathcal{E}^-_n) = O(\ell_n^{-1})$. Intuitively, off of the event $\mathcal{E}^-$:
\begin{itemize}
	\item the annuli $\mathcal{A}_n, \mathcal{B}_n$ each contain $O(n\eps_n^{\alpha+1})$ roots $X_i$;
	\item the largest $\ell_n+1$ roots are in $\mathcal{A}_n^*$ and at most $n^\delta-1$ roots are within a distance $(\ell_n-3)/n$ of the unit circle;
	\item each root in $\mathcal{A}_n\cup \mathcal{B}_n$ is $\delta_n$-far from all other roots, and each root in $\mathcal{A}_n^*$ is radially separated from all other roots by a distance $(n\ell_n^3)^{-1/(\alpha +1)}$;
	\item the sums $n^{-1}\sum_{i=1}^n (1-3/2\eps_n-|X_i|)^{-2}\sind{|X_i|<1-2\eps_n}$ are less than $\ell_n\eps_n^{\alpha-1}$;
	\item on the net $\mathcal{N}_n^*$, the Cauchy--Stieltjes transform $m_\mu(z)$ approximates the sums $n^{-1}\sum_{i=1}^n|z-X_i|^{-1}\sind{|z-X_i|> \delta_n/2}$ with error at most $n^{-4/(\alpha+1)}n^{\alpha/(\alpha+1)}$. 
\end{itemize} 

\begin{lemma}\label{lem:manyPairConcNeg}
		As in the hypothesis of Theorem \ref{thm:manyPairNeg}, suppose $X_1, X_2, \ldots$ are i.i.d.\ draws from a distribution $\mu$ that satisfies Assumption \ref{ass:ComplexMu} with $-0.095<\alpha < 0$, let $\ell_n = \omega(1)$ be a sequence of positive integers with $\ell_n \leq \sqrt{\log{n}}$, and fix $\delta \in (0, -\alpha)$. Additionally, in order of decreasing magnitude, denote the largest $\ell_n$ critical points of $p_n:= \prod_{j=1}^n(z-X_j)$ by $W^{\da}_{(1)}, \ldots, W^{\da}_{(\ell_n)}$. If we define
		\[
		\eps_n := \left(\frac{1}{n}\right)^{\frac{1}{2(\alpha + 1)}}, \quad \delta_n := \log{n}\cdot\left(\frac{1}{n}\right)^{\frac{\alpha+3}{4(\alpha+1)}}, \quad \text{and}\quad \eps^*_n := \left(\frac{\ell_n^2}{n}\right)^{\frac{1}{\alpha +1}},
		\]
		then there is a positive constant $C$ such that for large $n$, on the complement of $\mathcal{E}^-_n$, the conclusions of Theorem \ref{thm:manyPairNeg} hold.
\end{lemma}
\begin{proof}Fix $\alpha \in (-0.095, 0)$ and $\delta \in (0, -\alpha)$, and suppose the complement of $\mathcal{E}_n^-$ holds. We begin with the remark that due to the Gauss--Lucas theorem, on the complement of $E_n^*$, all critical points and roots of $p_n$ have modulus less than $|X^{\da}_{(1)}| \leq 1-\eps_n^*\ell_n^{-3/(\alpha+1)}$. For large $n$, on the complement of the ``bad'' event $\mathcal{E}^-_n$, we will establish the conclusions of Theorem \ref{thm:manyPairNeg} in several stages:
	\begin{enumerate}[(I)]
		\item \label{stage:rtsToCpts} First, we will show that each root $X_i$, $i \in [n]$, satisfying $\abs{X_i} > 1-\frac{2\ell_n}{n}$ has a unique critical point $W^{(n)}_i$ within a distance $\frac{3}{n}$, and we have
		 \[
		 \max_{i\in [n] :\abs{X_{i}} > 1-\frac{2\ell_n}{n} } \abs{W^{(n)}_i - X_i + \frac{1}{n}\frac{1}{\frac{1}{n-1}\sum_{j \neq i}\frac{1}{X_i-X_j}}} \ll  \ell_nn^{-\frac{5\alpha +3}{2(\alpha+1)}}.
		 \] In view of Lemma \ref{lem:nearCSNeg}, this will imply conclusions \ref{item:enoughCptsNeg} and \ref{item:rtsToCpNeg} of Theorem \ref{thm:manyPairNeg} .
		 
		\item \label{stage:iota} Second, we will show that each critical point $W$ of $p_n$ that satisfies $W \in \mathcal{A}_n^-$ is within a distance $2/n$ of a nearest root $X_{i_W}$, where $\abs{X_{i_W}} > 1-\frac{2\ell_n}{n}$. By the results in Stage \ref{stage:rtsToCpts} above, we will have established conclusion \ref{item:iotaNeg} of Theorem \ref{thm:manyPairNeg}.
		 
		\item \label{stage:LgCp} Denote by $W^{(n)}_{(\ell_n)}$ the critical point of $p_n$ that is paired with $X^{\da}_{(\ell_n)}$. We will use the results of Stages \ref{stage:rtsToCpts} and \ref{stage:iota} above to show that any critical point $W$ of $p_n$, that satisfies $\abs{W} > \abs{W^{(n)}_{(\ell_n)}}$ must be the unique critical point within a distance $\frac{3}{n}$ of $X^{\da}_{(i)}$ for some $i < \ell_n$. 
		
		\item \label{stage:LgOrder} Finally, we will use the radial separation of the largest $\ell_n + 1$ roots of $p_n$ (note: on the complement of $E_n^*$, these are in $\mathcal{A}_n^*$) to justify \eqref{eqn:maxPairNeg} and \eqref{eqn:maxPairSharperNeg}. Hence we will have verified conclusion \ref{item:maxPairNeg} of Theorem \ref{thm:manyPairNeg}.
	\end{enumerate}

	We begin Stage \ref{stage:rtsToCpts} by invoking Theorem \ref{thm:determ} many times at once, one for each $i \in [n]$ such that $\abs{X_i} > 1-\frac{2\ell_n}{n}$. In particular, for each such $i$, apply Theorem \ref{thm:determ} with $n-1$ in place of $n$, with
	\[
	\xi_i:= X_i,\ \vec{\zeta}_i:=(X_1, \ldots, X_{i-1},  X_{i+1}, \ldots,  X_{n})
	\]
	and with
	\[
	C_1: = \frac{1}{2},\ C_2:= 2,\ \text{and } k_{\rm Lip} := 2\ell_nn^{\frac{1-\alpha}{2(\alpha+1)}}.
	\]
	We argue that for large $n$, on the complement of $E_n \cup E_n^* \cup F_n \cup G_n \cup H_n^*$, the three main conditions in the hypothesis of Theorem \ref{thm:determ} hold with these parameter choices that are uniform in $i$.   In view of Lemma \ref{lem:nearCSNeg} and the triangle inequality, for large $n$, on the complement of $E_n \cup E_n^* \cup F_n \cup G_n \cup H_n^*$, condition \ref{item:det1} holds for each $i \in [n]$ with $\abs{X_i} > 1-\frac{2\ell_n}{n}$ because
	\[
	\abs{\frac{1}{n-1} \sum_{j \neq i}\frac{1}{X_i-X_j}} = \frac{n}{n-1}\abs{\overline{M}_\mu(X_i)}
	\]
	and (via Lemma \ref{lem:StielCalc})
	\begin{equation}\label{eqn:CSBdsLgX}
	\frac{F_R\left(1-\frac{2\ell_n}{n}\right)}{1} \leq \inf_{z: \abs{z} > 1-\frac{2\ell_n}{n}}\abs{m_\mu(z)} \leq \sup_{z: \abs{z} > 1-\frac{2\ell_n}{n}}\abs{m_\mu(z)} \leq \frac{F_R(1)}{1-\frac{2\ell_n}{n}},
	\end{equation}
	where $F_R(r)$ denotes the radial c.d.f.\ of $\mu$. (Note that $\frac{2\ell_n}{n}= o(1)$ and $F_R(r)$ is continuous at $r=1$ with $F_R(1) = 1$.) Lemma \ref{lem:nearLipNeg} implies condition \ref{item:det2} for large $n$ on the complement of $E_n \cup F_n \cup G_n$. In particular, we note that for $\alpha > -1/3$, $\frac{2}{C_1(n-1)} = o(\delta_n/2)$ and when $\abs{X_i} > 1-\frac{2\ell_n}{n}$, we have $B(X_i, \frac{\delta_n}{2}) \subset \mathcal{A}_n \cup B(0,1)^c$. Lastly, for large $n$, on the complement of $F_n$, condition \ref{item:det3} is true for each $i \in[n]$ with $\abs{X_i} > 1-\frac{2\ell_n}{n}$ because such $X_i$ are separated by a distance $\delta_n = \omega(3/(C_1(n-1)))$ from the other roots of $p_n$. (We note that the asymptotic holds for $\alpha > -1/3$.) Now, since $k_{\rm Lip} = 2\ell_nn^{\frac{1-\alpha}{2(\alpha+1)}} = o(n)$ for $\alpha > -1/3$, the conclusions of Theorem \ref{thm:determ} imply that for large $n$, on the complement of $E_n \cup E_n^* \cup F_n\cup G_n \cup H_n^*$, for each $i\in [n]$ with $\abs{X_i} > 1-\frac{2\ell_n}{n}$, the polynomial $p_n$ has exactly one critical point $W^{(n)}_i$ that is within a distance of $3/n$ of $X_i$, and 
	\begin{equation}\label{eqn:AminusRtsBd}
	\max_{i\in[n]: \abs{X_i} > 1-\frac{2\ell_n}{n}} \abs{W^{(n)}_i - X_i + \frac{1}{n}\frac{1}{\frac{1}{n-1}\sum_{j \neq i}\frac{1}{X_i-X_j}}} \ll \frac{\ell_nn^\frac{1-\alpha}{2(\alpha+1)}}{n^2} = \ell_nn^{-\frac{5\alpha +3}{2(\alpha+1)}},
	\end{equation}
	which implies \eqref{eqn:rtsToCptsBdNeg}. To see that conclusion \ref{item:enoughCptsNeg} and Inequality \eqref{eqn:rtsToCptsBd2Neg} from Theorem \ref{thm:manyPairNeg} hold, we appeal to a sequence of inequalities like \eqref{eqn:CSnearX} above. Indeed, for any $i$,
		\begin{align*}
			\abs{\frac{1}{\frac{1}{n-1}\sum_{j\neq i} \frac{1}{X_i - X_j}}-X_i} &= \frac{\abs{X_i}}{\abs{\frac{1}{n-1}\sum_{j\neq i} \frac{1}{X_i - X_j}}}\cdot \abs{\frac{1}{n-1}\sum_{j\neq i} \frac{1}{X_i - X_j}-\frac{1}{X_i}}\\
			&=\frac{\abs{X_i}}{\frac{n}{n-1}\abs{\overline{M}_\mu(X_i)}}\cdot\abs{\frac{n}{n-1}\overline{M}_\mu(X_i) - \frac{1}{X_i}}\\
			&\leq  \frac{1}{\abs{\overline{M}_\mu(X_i)}}\cdot\left(\abs{\frac{1}{n-1}\overline{M}_\mu(X_i)} + \abs{\overline{M}_\mu(X_i)-\frac{1}{X_i}}\right),
		\end{align*}
	and if we use Lemma \ref{lem:nearCSNeg} to interpolate between $\overline{M}_\mu(X_i)$ and $m_\mu(X_i)$, which is bounded as in \eqref{eqn:CSBdsLgX}, we obtain that for large $n$, on the complement of $E_n\cup E_n^* \cup F_n \cup G_n \cup H_n^*$, when $\abs{X_i} > 1 - \frac{2\ell_n}{n}$, 
		\begin{align*}
			\abs{\frac{1}{\frac{1}{n-1}\sum_{j\neq i} \frac{1}{X_i - X_j}}-X_i} &\leq \frac{1}{n-1} +  \frac{\abs{m_\mu(X_i)-\frac{1}{X_i}}+\frac{2}{\ell_n^{4/(\alpha +1)}}\left(\frac{1}{n}\right)^{\frac{-\alpha}{\alpha+1}}}{\abs{m_\mu(X_i)}-\frac{2}{\ell_n^{4/(\alpha +1)}}\left(\frac{1}{n}\right)^{\frac{-\alpha}{\alpha+1}}} \\
			&\leq \frac{1}{n-1} + \frac{\abs{\frac{F_R(\abs{X_i})}{X_i}-\frac{1}{X_i}} + \frac{2}{\ell_n^{4/(\alpha +1)}}\left(\frac{1}{n}\right)^{\frac{-\alpha}{\alpha+1}}}{F_R(1-\frac{2\ell_n}{n}) - \frac{2}{\ell_n^{4/(\alpha +1)}}\left(\frac{1}{n}\right)^{\frac{-\alpha}{\alpha+1}}}\\
			&<2 \left(1-F_R(\abs{X_i})\right) + O\left(\frac{1}{\ell_n^{4/(\alpha +1)}}\left(\frac{1}{n}\right)^{\frac{-\alpha}{\alpha+1}}\right)\\
			&\ll \frac{1}{\ell_n^{4/(\alpha +1)}}\left(\frac{1}{n}\right)^{\frac{-\alpha}{\alpha+1}},
		\end{align*}
	where we used \eqref{eqn:RadCdfBds} to bound $1-F_R(\abs{X_i})$ in the last line. Combining this inequality with \eqref{eqn:AminusRtsBd} yields \eqref{eqn:rtsToCptsBd2Neg}, so we have established conclusions \ref{item:enoughCptsNeg} and \ref{item:rtsToCpNeg} of Theorem \ref{thm:manyPairNeg}. (Note that first conclusion follows from \eqref{eqn:rtsToCptsBd2Neg} on the complement of $\{|X^\da_{(n^\delta)}| \leq 1-(\ell_n-3)/n\}$.)
	
	We proceed with Stage \ref{stage:iota} in which we will show that for large $n$, off of the ``bad'' event $\mathcal{E}^-_n$, each critical point of $p_n$ that lies in $\mathcal{A}_n^-$ is within a distance $\frac{2}{n}$ of a nearest root $X_i$ with $\abs{X_i} > 1 - \frac{2\ell_n}{n}$. In view of \eqref{eqn:GLPair} and Lemmas \ref{lem:nearCSNeg} and \ref{lem:StielCalc}, for large $n$, on the complement of $E_n \cup E_n^* \cup F_n \cup G_n \cup H_n^*$, any critical point $W$ of $p_n$ that lies in $\mathcal{A}_n^-$ satisfies
	\begin{align*}
		\abs{W - X_{i_W}} &= \frac{1}{n}\frac{1}{\abs{\frac{1}{n}\sum_{j\neq i_W}\frac{1}{W-X_j}}} = \frac{1}{n\abs{\overline{M}_\mu(W)}}\\
		&\leq \frac{1}{n\abs{m_\mu(W)} - n\abs{\overline{M}_\mu(W)-m_\mu(W)}}\\
		&\leq\frac{1}{n\frac{F_R(\abs{W})}{\abs{W}} + o(n)}\\
		&\leq \frac{1}{n\cdot F_R\left(1-\frac{\ell_n}{n}\right) + o(n)}\\
		&\leq \frac{2}{n},
	\end{align*}
	where $F_R(r)$ denotes the radial c.d.f.\ associated with $\mu$, and to achieve the last inequality, we have used that $F_R(r)$ is continuous at $r=1$, with  $F_R(1) = 1$. We have shown that for large $n$, on the complement of $\mathcal{E}^-_n$, each critical point $W$ of $p_n$ that lies in $\mathcal{A}_n^-$ is within a distance $2/n$ of the nearest root $X_{i_W}$, $1\leq i_W \leq n$. By the triangle inequality, such a root $X_{i_W}$ must satisfy $\abs{X_{i_W}}> 1-\frac{2\ell_n}{n}$, and so by the Stage \ref{stage:rtsToCpts} arguments, we obtain \eqref{eqn:cptsToRtsBdNeg} and \eqref{eqn:cptsToRtsUniformBdNeg}, and conclusion \ref{item:iotaNeg} of Theorem \ref{thm:manyPairNeg} is true for large $n$ on the complement of $\mathcal{E}^-_n$. (Note that \eqref{eqn:cptsToRtsBdNeg} and \eqref{eqn:cptsToRtsUniformBdNeg} imply $\abs{W} < \abs{X_{i_{W}}}$ so $X_{i_W} \in \mathcal{A}^-_n$ as is claimed in conclusion \ref{item:iotaNeg}.)
	
	We now tackle Stage \ref{stage:LgCp} in which we will argue that for large $n$, on the complement of $\mathcal{E}^-_n$, any critical point $W$ of $p_n$ that is larger in magnitude than the critical point paired with $X^\da_{(\ell_n)}$ must be the unique critical point of $p_n$ that is within a distance $3/n$ of $X^\da_{(i)}$ for some $i < \ell_n$. By the reasoning in Stages \ref{stage:rtsToCpts} and \ref{stage:iota}, for large $n$, on the complement of $E_n \cup E_n^* \cup F_n \cup G_n \cup H_n^*$, conclusions \ref{item:rtsToCpNeg} and \ref{item:iotaNeg} hold, and there is a unique critical point $W^{(n)}_{(\ell_n)}$ that is within a distance $3/n$ of $X^\da_{(\ell_n)}$. Consequently, for large $n$, off of this ``bad'' event, if $W$ is any critical point of $p_n$ with $|W| > \abs{W^{(n)}_{(\ell_n)}}$, then $W, X_{i_W} \in \mathcal{A}^-_n$ and we have
	\[
	\max\set{\abs{W^{(n)}_{(\ell_n)} - X^\da_{(\ell_n)}(1-n^{-1})}, \abs{W - X_{i_W}(1-n^{-1})}} < C\left(\frac{1}{\ell_n^4n}\right)^{1/(\alpha+1)}.
	\]
	By the triangle inequality, it follows that
	\begin{align*}
		\abs{X^\da_{(\ell_n)}}(1-n^{-1}) - C\left(\frac{1}{\ell_n^4n}\right)^{1/(\alpha+1)} &\leq \abs{W^{(n)}_{(\ell_n)}} < \abs{W}\\
		 &< \abs{X_{i_W}}(1-n^{-1}) + C\left(\frac{1}{\ell_n^4n}\right)^{1/(\alpha+1)},
	\end{align*}
	so 
	\[
	\abs{X^\da_{(\ell_n)}} - \abs{X_{i_W}} < \frac{2C(n-1)}{n}\left(\frac{1}{\ell_n^4n}\right)^{1/(\alpha+1)} = o\left( \left(\frac{1}{n\ell_n^3}\right)^{1/(\alpha + 1)}\right).
	\]
	On the complement of $E_n^* \cup F_n^\parallel$, this last inequality precludes $|X^\da_{(\ell_n)}| > |X_{i_W}|$ since the largest $\ell_n + 1$ roots of $p_n$ are radially separated by a distance $\left(n\ell_n^3\right)^{-1/(\alpha + 1)}$. We conclude that for large $n$, on the complement of $E_n \cup E_n^* \cup F_n \cup F_n^\parallel \cup G_n \cup H_n^*$, any critical point $W$ of $p_n$ that is larger in magnitude than $W^{(n)}_{(\ell_n)}$ must be paired within a distance $3/n$ of some $X^\da_{(i)}$ for $1 \leq i < \ell_n$. Stage \ref{stage:LgCp} is complete. 
	
	We conclude the proof of Lemma \ref{lem:manyPairConcNeg} with Stage \ref{stage:LgOrder} in which we will use the result of Stage \ref{stage:LgCp} and the already verified conclusions \ref{item:rtsToCpNeg} and \ref{item:iotaNeg} to establish conclusion \ref{item:maxPairNeg} of Theorem \ref{thm:manyPairNeg} for large $n$ off of the ``bad'' event $\mathcal{E}_n^-$. Suppose $n$ is large enough to guarantee these conclusions on the complement of $\mathcal{E}^-_n$, and suppose the complement of $\mathcal{E}^-_n$ holds. Inequality \eqref{eqn:maxPairRtsInNeg} is clearly true on the complement of $E_n^*$, and \eqref{eqn:maxPairNeg}, \eqref{eqn:maxPairSharperNeg} follow from \eqref{eqn:rtsToCptsBdNeg}, \eqref{eqn:cptsToRtsUniformBdNeg} provided we can show that the nearest critical point to $X^\da_{(i)}$ is $W^\da_{(i)}$ for $1 \leq i \leq \ell_n$. We will verify this now. By conclusion \ref{item:rtsToCpNeg}, its constituent inequality \eqref{eqn:rtsToCptsBd2Neg}, and the triangle inequality, if $W^{(n)}_{(i)}$ and $W^{(n)}_{(j)}$ are the nearest critical points to $X^\da_{(i)}$ and $X^\da_{(j)}$, respectively, where $i,j \in \set{1, \ldots, \ell_n}$ with $i \neq j$, and we have $\abs{W^{(n)}_{(i)}} \leq \abs{W^{(n)}_{(j)}}$, then
	\begin{align*}
		\abs{X^\da_{(i)}}(1-n^{-1}) - C\left(\frac{1}{\ell_n^4n}\right)^{1/(\alpha+1)} &\leq \abs{W^{(n)}_{(i)}} \leq \abs{W^{(n)}_{(j)}}\\
		&\leq \abs{X^\da_{(j)}}(1-n^{-1}) + C\left(\frac{1}{\ell_n^4n}\right)^{1/(\alpha+1)}.
	\end{align*}
	It follows that
	\[
	\abs{X^\da_{(i)}} - \abs{X^\da_{(j)}} \leq \frac{2C(n-1)}{n}\left(\frac{1}{\ell_n^4n}\right)^{1/(\alpha+1)} = o\left( \left(\frac{1}{n\ell_n^3}\right)^{1/(\alpha + 1)}\right),
	\]
	so on the complement of $E_n^* \cup F_n^\parallel$, $\abs{X^\da_{(i)}} \leq \abs{X^\da_{(j)}}$, or in other words, $1 \leq j < i \leq \ell_n$. By contrapositive logic, we have established that 
	\[
	\abs{W^{(n)}_{(\ell_n)}} < \abs{W^{(n)}_{(\ell_n -1)}} < \cdots < \abs{W^{(n)}_{(1)}}.
	\]
	In view of Stage \ref{stage:LgCp}, $W^{(n)}_{(i)}$, $1 \leq i < \ell_n$ are the only critical points of $p_n$ that are larger in magnitude than $W^{(n)}_{(\ell_n)}$, so $W^{(n)}_{(i)} = W^\da_{(i)}$ for $1 \leq i \leq \ell_n$. We have established conclusion \ref{item:maxPairNeg} for large $n$ on the complement of $\mathcal{E}^-_n$. 
\end{proof}

We will complete the proof of Theorem \ref{thm:manyPairNeg} by showing that $\P(\mathcal{E}^-_n) = O(\ell_n^{-1})$. This is the content of Lemma \ref{lem:badEventSmallNeg}. As an intermediate step, we prove Lemma \ref{lem:sepRtsNeg} that we will use to control the probabilities of $F_n$ and $F^\parallel_n$ in Lemma \ref{lem:badEventSmallNeg}.

\begin{lemma}[The largest roots are separated with high probability, $\alpha < 0$ case]\label{lem:sepRtsNeg}
	Fix $\alpha \in (-1, 0)$ and $\eps>0$ as in Assumption \ref{ass:ComplexMu}; suppose $\gamma_n, d_n \in (0,\eps/2)$ are sequences satisfying $n\gamma_n^{\alpha+1} = \omega(1)$ and $n^2\gamma_n^{\alpha + 1}d_n^{\alpha+2} = o(1)$; and define
	\[
	A_n := \set{z\in \C: 1-\gamma_n \leq \abs{z} \leq 1}\subset \mathbb{A}_\eps.
	\]
	There is a constant $c>0$ such that for large $n$, 
	\begin{equation}
		\P\left(\exists i,j \in [n], i\neq j: X_i\in A_n, \abs{X_i-X_j}\leq d_n\right)\ll\max\set{n^2\gamma_n^{\alpha+1}d_n^{\alpha+2}, e^{-cn\gamma_n^{\alpha+1}}}.
		\label{eqn:circSepNeg}
	\end{equation}
	To control the radial separation between the largest roots of $p_n$, let $\gamma_n, R_n \in (0,\eps/2)$ be sequences satisfying $n\gamma_n^{\alpha +1}=\omega(1)$ and $n^2\gamma_n^{\alpha+1}R_n^{\alpha+1} = o(1)$. There is a constant $c>0$ such that for large $n$,
	\begin{equation}
		\P\left(\exists i,j \in [n], i\neq j: X_i\in A_n, \abs{\abs{X_i}-\abs{X_j}}\leq R_n\right)\ll\max\set{n^2\gamma_n^{\alpha+1}R_n^{\alpha+1}, e^{-cn\gamma_n^{\alpha+1}}}.
		\label{eqn:radSepNeg}
	\end{equation}
\end{lemma}
\begin{proof}
	The proof is nearly identical to the proof of Lemma \ref{lem:sepRts} above. To establish \eqref{eqn:circSepNeg}, make the following change to that proof. When conditioning on $X_1, \ldots, X_n$ to be at least $d_n$-far from other $X_i \in A_n$ in \eqref{eqn:avoidRts}, replace all instances of $M_nd_n^2$ with $\frac{2^{\alpha+2}C_\mu}{\pi(\alpha+1)(1-\eps)}d_n^{\alpha + 2}$. This can be justified because $X_i$ are i.i.d., so Lemma \ref{lem:heavyTail} (with $t= d_n^{-1} = \omega(1)$) implies the following for large $n$ and $j < i \leq n$:
	\begin{align*}
		\P\left(\abs{X_i - X_j} \leq d_n \mid X_j \in A_n \right) &= \P\left(\abs{X_j - X_i}^{-1} \geq d_n^{-1} \mid  X_j \in A_n\right)\\
		&\leq \frac{2^{\alpha+2}C_\mu}{\pi(\alpha+1)(1-\eps)}d_n^{\alpha + 2}.
	\end{align*}	
	
	The asymptotic bound \eqref{eqn:radSepNeg} can be achieved by making the following change when conditioning on $X_1, \ldots, X_n$ to be radially separated by at least $R_n$ from other $X_i \in A_n$ in the style of \eqref{eqn:avoidRts}: replace all instances of $M_nd_n^2$ with $\frac{C_\mu2^{\alpha+1}}{\alpha+1}R_n^{\alpha + 1}$. This can be justified because
	\begin{align*}
		\P\left(\abs{\abs{X_i} - \abs{X_j}} \leq R_n \mid X_j \in A_n \right) &= \int_{\abs{X_j}-R_n}^{\min\set{1, \abs{X_j}+R_n}}f_R(r)\,dr\\
		&\leq C_\mu\int_{1-2R_n}^1 (1-r)^{\alpha} =\frac{C_\mu2^{\alpha+1}}{\alpha+1}R_n^{\alpha + 1}.
	\end{align*}
	(Note: that $(1-r)^\alpha$ is increasing in $r$ for $\alpha < 0$, so we can replace the original limits on the integral with $1-2R_n$ and $1$.)
\end{proof}

\begin{lemma}\label{lem:badEventSmallNeg}
As in the hypothesis of Theorem \ref{thm:manyPairNeg}, suppose $X_1, X_2, \ldots$ are i.i.d.\ draws from a distribution $\mu$ that satisfies Assumption \ref{ass:ComplexMu} with $-0.095<\alpha < 0$, let $\ell_n = \omega(1)$ be a sequence of positive integers with $\ell_n \leq \sqrt{\log{n}}$, and fix $\delta \in (0, -\alpha)$. If we define
\begin{align*}
\eps_n &:= \left(\frac{1}{n}\right)^{\frac{1}{2(\alpha + 1)}}, \quad \delta_n := \log{n}\cdot\left(\frac{1}{n}\right)^{\frac{\alpha+3}{4(\alpha+1)}}, \text{and}\quad \eps^*_n := \left(\frac{\ell_n^2}{n}\right)^{\frac{1}{\alpha +1}},
\intertext{and}
\mathcal{E}^-_n &:= E_n \cup E_n^* \cup F_n \cup F^\parallel_n \cup G_n \cup H_n^*\cup \set{\abs{X^\da_{(n^\delta)}} \leq 1-\frac{\ell_n-3}{n}}
\end{align*}
as in Lemma \ref{lem:manyPairConcNeg}, then $\P(\mathcal{E}^-_n) = O(\ell_n^{-1})$. 
\end{lemma}
\begin{proof}	
	This proof is similar in flavor to the proof of Lemma \ref{lem:manyPairProb} above, and accordingly, we will use many of the lemmas from Subsection \ref{sec:badEventsSmall}, which were written to cover both cases $\alpha \geq 0$ and $\alpha < 0$. In view of the union bound, we will show that each of the events $E_n$, $E_n^*$, $F_n$, $F^\parallel_n$, $G_n$, $H_n^*$ and $\{|X^\da_{(n^\delta)}| \leq 1-(\ell_n-3)/n\}$ has probability on the order $O(\ell_n^{-1})$. 
	
	To see that $\P(E_n) = O(\ell_n^{-1})$, apply Lemma \ref{lem:fewRts} with $\gamma_n:= 2\eps_n =o(1)$ and $C := 3C\mu > 2.5C_\mu > \frac{2C_\mu}{\alpha + 1}$ to obtain 
	\[
	\P(E_n) \leq \exp\left(-0.5\cdot C_\mu n (2\eps_n)^{\alpha + 1}\right) = \exp\left(-2^{\alpha}C_\mu \sqrt{n}\right)\ll \frac{1}{\sqrt{\log{n}}} \leq \frac{1}{\ell_n}.
	\]
	We will use Lemmas \ref{lem:fewRts} and \ref{lem:maxCDF} to show that $\P(E_n^*) = O(\ell_n^{-1})$. Observe that
	\begin{equation}\label{eqn:EStarSmall}
	\begin{aligned}
	E_n^* &= \set{\exists i \in \set{1, \ldots, \ell_n+1}: X^\da_{(i)} \notin \mathcal{A}_n^*}\\
	&= \set{\abs{X^\da_{(1)}} > 1- \eps_n^*\ell_n^{-3/(\alpha+1)}} \cup \set{\abs{X^\da_{(\ell_n+1)}} < 1- \eps_n^*}.
	\end{aligned}
	\end{equation}
	Using the independence of the $X_i$ and Inequality \eqref{eqn:massOnEdge} from Lemma \ref{lem:fewRts} with $\gamma_n:=\eps_n^*\ell_n^{-3/(\alpha+1)}$, we see that the first event on the right has probability
	\begin{align*}
	1-\P\left(\forall i \leq n, \abs{X_i} \leq 1- \gamma_n \right)
	&= 1- \big(1-\P\left(1-\gamma_n < \abs{X_1} \leq 1 \right)\big)^n\\
	&\leq 1-\left(1-\frac{C_\mu}{\alpha + 1}\gamma_n^{\alpha+1}\right)^n\\
	&\leq n\frac{C_\mu}{\alpha + 1}\gamma_n^{\alpha+1}\\
	&= \frac{C_\mu}{\alpha + 1}\frac{1}{\ell_n},
	\end{align*}
	where the penultimate line follows from the first order approximation
	\[
	1-(1-x)^n \leq nx \text{ for $x \in (0,1)$, $n \in \mathbb{Z}^+$.}
	\]
	An upper bound on the probability of the second event on the right of \eqref{eqn:EStarSmall} can be achieved using Lemma \ref{lem:maxCDF} with $\gamma_n:= \eps_n^*$ and $l_n := \ell_n+1$ (note $n(\eps_n^*)^{\alpha+1} = \ell_n^2$ so the hypotheses are satisfied). In particular, there is a positive constant $c$ so that 
	\[
	\P\left(\abs{X^\da_{(\ell_n+1)}} < 1- \eps_n^*\right) \ll e^{-cn\gamma_n^{\alpha +1}} = e^{-c\ell_n^2} \leq \frac{1}{1 + c\ell_n^2} \ll \ell_n^{-1}.
	\]
	We conclude that $\P(E_n^*) = O(\ell_n^{-1})$ as is desired.
	
	To find upper bounds on $\P(F_n)$ and $\P(F_n^\parallel)$, we appeal to Lemma \ref{lem:sepRtsNeg}. Setting
	\[
	\gamma_n := 2\eps_n = 2\left(\frac{1}{n}\right)^{\frac{1}{2(\alpha +1)}}\quad \text{and}\quad d_n:= \delta_n = \log{n} \cdot \left(\frac{1}{n}\right)^{\frac{\alpha +3}{4(\alpha +1)}},
	\]
	we see that for large $n$, $\gamma_n, d_n \in (0,\eps/2)$, 
	\[
	n\gamma_n^{\alpha+1} = 2^{\alpha+1}\sqrt{n} = \omega(1),
	\]
	and
	\[
		n^2\gamma_n^{\alpha+1}d_n^{\alpha+2} = 2^{\alpha+1}\left(\log{n}\right)^{\alpha+2}\cdot n^{\frac{3}{2} - \frac{(\alpha +3)(\alpha+2)}{4(\alpha +1)}} \ll \frac{1}{\ell_n} = o(1),
	\]
	so \eqref{eqn:circSepNeg} implies (for some fixed positive constant $c$),
	\[
	\P(F_n) \ll \max \set{\frac{1}{\ell_n}, e^{-c2^{\alpha+1}\sqrt{n}}} \ll \frac{1}{\ell_n}.
	\]
	By the second part of Lemma \ref{lem:sepRtsNeg} with 
	\[
	\gamma_n := \eps_n^* = \left(\frac{\ell_n^2}{n}\right)^{\frac{1}{\alpha+1}}\quad\text{and}\quad R_n:=\left(\frac{1}{n\ell_n^3}\right)^{\frac{1}{\alpha+1}},
	\]
	which are in $(0,\eps/2)$ for large $n$ and that satisfy
	\[
	n\gamma_n^{\alpha+1} = \ell_n^2 = \omega(1)\quad\text{and}\quad n^2\gamma_n^{\alpha+1}R^{\alpha+1} = \frac{1}{\ell_n} = o(1), 
	\]
	 we have (for some fixed positive constant $c$)
	\[
	\P(F_n^\parallel) \ll \max\set{\frac{1}{\ell_n}, e^{-c\ell_n^2}} \ll \frac{1}{\ell_n}.
	\]
	
	Next, consider that $\P(G_n) \ll \ell_n^{-1}$ by Lemma \ref{lem:GnSmall} with 
	\[
	\eps_n:=\left(\frac{1}{n}\right)^{\frac{1}{2(\alpha+1)}} = o(1)\quad\text{and}\quad c_n:=\ell_n.
	\]
	
	To contend with $\P(H_n^*)$, we apply Lemma \ref{lem:netProb} with the nets $\mathcal{N}_n := \mathcal{N}_n^*$ of $\mathcal{A}_n^*$ that satisfy \eqref{eqn:netParamsNeg} and with
	\[
	\delta_n:= \log{n}\cdot \left(\frac{1}{n}\right)^{\frac{\alpha+3}{4(\alpha+1)}}\quad \text{and} \quad c_n:= \ell_n^{4/(\alpha+1)}n^{\frac{-\alpha}{\alpha +1}}.
	\]
	We note that for $\alpha \in (-0.095, 0)$, 
	\[
	\delta_n^{\alpha+1} = \left(\log{n}\right)^{\alpha+1}\left(\frac{1}{n}\right)^{(\alpha+3)/4} = o\left(\ell_n^{-4/(\alpha+1)}\left(\frac{1}{n}\right)^{\frac{-\alpha}{\alpha +1}}\right), 
	\]
	so the hypotheses of Lemma \ref{lem:netProb} are satisfied, and via \eqref{eqn:netConv} (and the conditions \eqref{eqn:netParamsNeg}, namely $|\mathcal{N}_n^*| \ll \ell_n^6n^{\frac{1-3\alpha}{2(\alpha+1)}}$), 
	\begin{align*}
	\P(H_n^*) &\ll \abs{\mathcal{N}^*_n} \cdot \frac{c_n^2\log\left(\frac{1}{\delta_n}\right)\delta_n^\alpha}{n}\\
	& \ll \underbrace{\ell_n^6\cdot n^{\frac{1-3\alpha}{2(\alpha+1)}}}_{\abs{\mathcal{N}_n^*}\ \rm bound} \cdot \frac{1}{n} \cdot \underbrace{\ell_n^{\frac{8}{\alpha+1}}\cdot n^{\frac{-2\alpha}{\alpha +1}}}_{c_n^2\ \rm bound} \cdot \log\left(\frac{1}{\delta_n}\right) \cdot \underbrace{\left(\log{n}\right)^\alpha \cdot n^{\frac{-\alpha(\alpha+3)}{4(\alpha+1)}}}_{\delta_n^\alpha\ \rm bound}\\
	&\ll \left(\log{n}\right)^{C_1}\cdot n^{-\frac{\alpha^2 + 21\alpha +2}{4(\alpha+1)}},
	\end{align*}
	where $C_1> 0$ is a constant. As long as the exponent on $n$ is negative, $\P(H_n^*) = O(\ell_n^{-1})$ and this is true for $\alpha \in (-0.095, 0)$.\footnote{We note that $x=-0.095$ is slightly larger than the largest root of $x^2 + 21x + 2 = 0$.}
	
	We conclude the proof of Lemma \ref{lem:badEventSmallNeg} by using Lemma \ref{lem:maxCDF} to obtain 
	\begin{equation}\label{eqn:adHocProb}
	\P\left(\abs{X^\da_{(n^\delta)}} \leq 1-\frac{\ell_n-3}{n}\right) \ll\frac{1}{\ell_n},
	\end{equation}
	for a fixed $\delta \in (0, -\alpha)$. Indeed, set $\gamma_n := \frac{\ell_n-3}{n}$ and $l_n := n^\delta$ and observe that
	\[
	l_n \cdot \log\left(n\gamma_n^{\alpha+1}\right) \ll n^\delta \cdot \log{n} = o\big(\underbrace{n^{-\alpha}\cdot (\ell_n-3)^{\alpha+1}}_{n\gamma_n^{\alpha+1}}\big),
	\]
	so for large $n$, the hypotheses of Lemma \ref{lem:maxCDF} are satisfied, and \eqref{eqn:adHocProb} follows. Note that for $\alpha \in(-1,0)$ and any constant $c>0$,
	\[
	\exp\left(-cn\gamma_n^{\alpha+1}\right) = \exp\left(-c(\ell_n-3)^{\alpha+1}n^{-\alpha}\right) \leq \frac{1}{1+ c(\ell_n-3)^{\alpha+1}n^{-\alpha}} \ll \frac{1}{\ell_n}.
	\]
\end{proof}

Together, Lemmas \ref{lem:manyPairConcNeg} and \ref{lem:badEventSmallNeg} imply Theorem \ref{thm:manyPairNeg}.


\subsection{The proof of Theorem \ref{thm:fluctNeg}}\label{sec:fluctNegProof}
We would like to use \eqref{eqn:maxPairSharperNeg} to establish an equation like \eqref{eqn:fluctEqn} and then analyze the asymptotic behavior of the sums $\sum_{j=2}^{n}\frac{X^{\da}_{(j)}}{X^{\da}_{(1)}-X^{\da}_{(j)}}$. Unfortunately, the 
appropriate scaling (after which these sums converge to a $(2+\alpha)$-stable distribution) requires a slightly sharper asymptotic than \eqref{eqn:maxPairSharperNeg}, namely that the left-hand side of  \eqref{eqn:maxPairSharperNeg} is $o\left(n^{-\frac{3+2\alpha}{2+\alpha}}\right)$. With some care, we can achieve such a bound with high probability via Theorem \ref{thm:determ} and special case \ref{it:spCase2} of Lemma \ref{lem:nearLip} using the same types of arguments we have employed earlier in the paper. In particular, fix $\alpha \in (-0.095, 0)$, let $c_n = \omega(1)$ be a sequence of positive integers with $c_n < \sqrt{\log{n}}$, define 
\[
\eps_n := \left(\frac{1}{n}\right)^{\frac{(\alpha +1)^2}{(1-\alpha)(\alpha +2)}},\quad \delta_n := \frac{c_n}{n},\quad \eps^*_n := \left(\frac{c_n^2}{n}\right)^{1/(\alpha +1)},\quad \delta^*_n :=\left(\frac{1}{nc_n^3}\right)^{1/(\alpha + 2)},
\]
and in terms of these parameters, define as in Subsection \ref{sec:param} the annuli $\mathcal{A}_n$, $\mathcal{B}_n$, $\mathcal{A}_n^*$, and the events $E_n$, $F_n$, $F_n^*$, $G_n$. (Recall that some of the events in Subsection \ref{sec:param} concern the case $\alpha < 0$ so that Lemma \ref{lem:nearLip} applies to cases when $\alpha \geq 0$ and $\alpha <0$. This has allowed us to avoid repeated arguments in places like the present.) Finally let $\Omega_n$ be an event having probability $O(c_n^{-1})$ on whose complement conclusion \ref{item:maxPairNeg} of Theorem \ref{thm:manyPairNeg} holds for the constant $C_{\ref{thm:manyPairNeg}}> 0$ and the sequence $\ell_n:=c_n$. We note that for large $n$, our definitions of $c_n$, $\eps_n$, $\delta_n$, $\delta_n^*$ satisfy
\begin{equation}\label{eqn:negFluctParamsProps}
c_n > 0,\quad \frac{1}{n} < \delta_n \leq \min\set{\frac{\eps_n}{2}, \delta_n^*}<\frac{1}{4},\quad \eps_n^* = o(\eps_n)
\end{equation}
so special case \ref{it:spCase2} of Lemma \ref{lem:nearLip} implies the following lemma.
\begin{lemma}\label{lem:nearLipMaxNeg}
	Fix $\alpha \in (-0.095, 0)$. For large $n$, on the complement of $E_n \cup F_n \cup F_n^* \cup G_n \cup \Omega_n$, when $z, w \in B(X^{\da}_{(1)}, \frac{c_n}{2n})$,
	\begin{equation}\label{eqn:nearLipMaxNeg}
	\abs{\overline{M}_\mu(z) - \overline{M}_\mu(w)} < \abs{z-w} \left(\left(c_n+24C_\mu\right)\eps_n^{\alpha-1} + \frac{4c_n^3\eps_n^{\alpha + 2}}{\left(\delta^*_n\right)^2}\right) < 2c_n\eps_n^{\alpha-1}\abs{z-w} .
	\end{equation}
\end{lemma}
\begin{proof}
	In view of \eqref{eqn:negFluctParamsProps}, for large $n$, on the complement of $E_n \cup F_n \cup F_n^* \cup G_n$ special case \ref{it:spCase2} of Lemma \ref{lem:nearLip} applies. We note that for large $n$, on the complement of $\Omega_n$,
	\[
	\abs{X^\da_{(1)}} - \frac{2c_n}{n} \geq 1-\left(\frac{c_n^2}{n}\right)^{\frac{1}{\alpha+1}} - \frac{2c_n}{n} > 1-\eps_n,
	\]
	so $B(X^{\da}_{(1)}, \frac{c_n}{2n}) \subset \mathcal{A}_n \cup B(0,1)^c$. To achieve the rightmost inequality in \eqref{eqn:nearLipMaxNeg},  observe that
	\[
	\frac{4c_n^3\eps_n^{\alpha + 2}}{\left(\delta^*_n\right)^2} = \eps_n^{\alpha-1}\cdot 4c_n^{3 + \frac{6}{\alpha+2}}\cdot n^{\frac{3\alpha^2 + 8\alpha + 1}{(\alpha-1)(\alpha+2)}},
	\]
	and the exponent on $n$ in the last factor is negative for $\alpha > \frac{1}{3}(\sqrt{13}-4)\approx -0.13$.
\end{proof}

The next lemma gives a sharper version of \eqref{eqn:maxPairSharperNeg} for large $n$ on the complement of $E_n \cup F_n \cup F_n^* \cup G_n \cup \Omega_n$.

\begin{lemma}\label{lem:maxPairSharpestNeg}
Fix $\alpha \in (-0.095, 0)$. There is a constant $C'> 0$ so that for large $n$, on the complement of $E_n \cup F_n \cup F_n^* \cup G_n \cup \Omega_n$,
\begin{equation}\label{eqn:maxPairSharpestNeg}
	\abs{W^\da_{(1)} - X^\da_{(1)} + \frac{1}{n} \frac{1}{\frac{1}{n-1}\sum_{\substack{j=2}}^n\frac{1}{X^\da_{(1)} -X^\da_{(j)}}}} < C' c_n n^{\frac{\alpha^2-3}{\alpha+2}}.
\end{equation}
Furthermore, $\P(E_n \cup F_n \cup F_n^* \cup G_n \cup \Omega_n) = O(1/c_n)$.
\end{lemma}
\begin{proof}
	We first use Theorem \ref{thm:determ} to prove \eqref{eqn:maxPairSharpestNeg} on the complement of $E_n \cup F_n \cup F_n^* \cup G_n \cup \Omega_n$, and then, we will show that the probability of this ``bad'' event tends to zero as $n\to \infty$. Toward the first of these ends, choose $C_1, C_2 > 0$ such that $C_1 < 1 < C_2$, define $k_{Lip} := 2c_n\eps_n^{\alpha-1}\frac{n}{n-1}$, and apply Theorem \ref{thm:determ} with $\xi := X^\da_{(1)}$ and $\vec{\zeta} := (X^\da_{(2)}, \ldots, X^\da_{(n)})$. We justify the hypotheses for large $n$ on the complement of $E_n \cup F_n \cup F_n^* \cup G_n \cup \Omega_n$ as follows:
	\begin{itemize}
		\item Hypothesis \ref{item:det1} holds for large $n$ on the complement of $\Omega_n$ by combining \eqref{eqn:maxPairNeg} and \eqref{eqn:GLPair}. In particular, via the triangle inequality, \eqref{eqn:maxPairNeg} implies for large $n$
		\[
		\abs{W^\da_{(1)} - X^\da_{(1)}} < \frac{1}{n}\abs{X^\da_{(1)}} + C_{\ref{thm:manyPairNeg}}\left(\frac{1}{c_n^4n}\right)^{1/(\alpha+1)} < \frac{1}{(n-1)C_1}
		\]
		and 
		\[
		\abs{W^\da_{(1)} - X^\da_{(1)}} > \frac{1}{n}\abs{X^\da_{(1)}} - C_{\ref{thm:manyPairNeg}}\left(\frac{1}{c_n^4n}\right)^{1/(\alpha+1)} > \frac{1}{(n-1)C_2},
		\]
		so applying \eqref{eqn:GLPair} with $W = W^\da_{(1)}$ and $X_i = X^\da_{(1)}$ establishes \ref{item:det1} for large $n$, on the complement of $\Omega_n$.
		
		\item Hypothesis \ref{item:det2} holds for large $n$ on the complement of $E_n \cup F_n \cup F_n^* \cup G_n \cup \Omega_n$ by Lemma \ref{lem:nearLipMaxNeg}. (Note: For large $n$, $\frac{c_n}{2n} < \frac{2}{C_1(n-1)}$).
		
		\item Hypothesis \ref{item:det3} holds for large $n$ on the complement of $\Omega_n \cup F_n^*$ since \eqref{eqn:maxPairRtsInNeg} implies $X^\da_{(1)}  \in \mathcal{A}^*_n$ and (for $\alpha > -1$), we have $\delta_n^*>\frac{1}{nc_n^3} > \frac{3}{C_1(n-1)}$.
	\end{itemize}	
	Now, for $\alpha \in (-1, 0)$, $k_{\rm Lip} = O\left(c_n\cdot n^{\frac{(\alpha+1)^2}{(\alpha+2)}}\right) = o(n)$, so if we choose $C' > 2\cdot\frac{8(1+2C_2^2)}{C_1^3}$, the conclusion of Theorem \ref{thm:determ} yields \eqref{eqn:maxPairSharpestNeg} for large $n$, on the complement of the ``bad'' event $E_n \cup F_n \cup F_n^* \cup G_n \cup \Omega_n$.
	
	It remains to show that the ``bad'' event $E_n \cup F_n \cup F_n^* \cup G_n \cup \Omega_n$ is asymptotically negligible, which we do using a similar strategy to what we have done several times already earlier in the paper. In view of the union bound and the fact that $\P(\Omega_n) = O(c_n^{-1})$ by definition, it suffices to show that each of $E_n$, $F_n$, $F_n^*$, and $G_n$ has probability $O(c_n^{-1})$.  	
	
	To contend with $E_n$, we apply Lemma \ref{lem:fewRts}  with $\gamma_n := 2\eps_n$ and $C := 3C_\mu > 2C_\mu/(\alpha+1)$ and observe that for large $n$,
	\begin{equation}\label{eqn:nGammaStuffSmall}
	n\gamma_n^{\alpha+1} = n(2\eps_n)^{\alpha+1} \geq n^{\frac{(\alpha+1)^3}{(\alpha-1)(\alpha+2)}+1} \geq \sqrt{n}
	\end{equation}
	for $\alpha \in (-1,0)$, so
	\[
	\P(E_n) \leq e^{-(3C_\mu - 2C_\mu/(\alpha+1)) \sqrt{n}} \leq \frac{1}{1+(3C_\mu - 2C_\mu/(\alpha+1)) \sqrt{n}} \ll \frac{1}{c_n}.
	\]
	
	To bound $\P(F_n)$ and $\P(F_n^*)$, we use nearly identical reasoning to that in the proof Lemma \ref{lem:BadEventsSmallFluct} above. In particular, we start by breaking $F_n^*$ into two parts:
	\begin{align*}
		F_{n,1}^* &:= \Big\{\exists\ i\in [n]: X_i \in \mathcal{A}^*_n,\ \#\set{j\in [n],\ j \neq i: \abs{X_i-X_j} < \eps_n} > c_n^3n\eps_n^{\alpha+2} \Big\}
			\\
		F_{n,2}^* &:= \set{\exists\ i,j\in [n],\ i\neq j: X_i \in \mathcal{A}^*_n,\ \abs{X_{i}-X_{j}} \leq \delta^*_n}.
	\end{align*}
	By an identical argument to the one in the proof of Lemma \ref{lem:BadEventsSmallFluct} above, \eqref{eqn:alphaNegFstarbd} implies that for large $n$ (and a fixed $\alpha \in (-1,0)$), 
	\[
		\P(F_{n,1}^*) \leq n\cdot \P(X_1 \in \mathcal{A}_n^*) \cdot \exp\left(-c_n^4\cdot n\cdot\eps_n^{\alpha+2}\right).
	\]
	Now, $n\eps_n^{\alpha+2} = n^{\frac{-x(x+3)}{1-x}} = \omega(1)$ for $\alpha \in (-1, 0)$, and \eqref{eqn:massOnEdge} from Lemma \ref{lem:fewRts} implies
	\[
	\P(X_1 \in \mathcal{A}_n^*) \leq \frac{C_\mu}{\alpha+1}(\eps_n^*)^{1/(\alpha+1)}  = \frac{C_\mu}{\alpha+1} \cdot \frac{c_n^2}{n},
	\]
	so, continuing from above,
	\[
	\P(F_{n,1}^*) \leq n\cdot \frac{C_\mu}{\alpha+1} \cdot \frac{c_n^2}{n}\cdot \frac{1}{1+ c_n^4} \ll \frac{1}{c_n}.
	\]
	We can control $\P(F_n)$ and $\P(F_{n,2}^*)$ by invoking Lemma \ref{lem:sepRtsNeg} twice:
	\begin{itemize}
		\item First, to bound $\P(F_n)$, set $\gamma_n:=2\eps_n$ and $d_n:= \delta_n$, and observe that $n\gamma_n^{\alpha+1}  = \omega(\sqrt{n})$ (see \eqref{eqn:nGammaStuffSmall}) and that (for $-0.22\approx (\sqrt{17}-5)/4 < \alpha < 0$ which guarantees that the exponent on $n$ be negative),
		\[
		n^2\gamma_n^{\alpha+1}d_n^{\alpha + 2} = 2^{\alpha+1}c_n^{\alpha+2}n^{\frac{2\alpha^2+5\alpha+1}{(\alpha-1)(\alpha+2)}} = O(c_n^{-1}).
		\]
		
		\item Second, set  $\gamma_n:= \eps_n^*$ and $d_n := \delta_n*$, and observe that $n\gamma_n^{\alpha + 1} = c_n^2 = \omega(1)$ and
		\[
		n^2\gamma_n^{\alpha+1}d_n^{\alpha + 2} = \frac{1}{c_n},
		\]
		so \eqref{eqn:circSepNeg} implies $\P(F_{n,2}^*) = O(c_n^{-1})$ as is desired.
	\end{itemize}
	We conclude the proof by appealing to Lemma \ref{lem:GnSmall} to establish that  $\P(G_n) = O(c_n^{-1})$.		
\end{proof}

With Lemma \ref{lem:maxPairSharpestNeg} in hand, we work to demonstrate the convergence in \eqref{eqn:maxConvNeg} by way of several intermediate approximations. Our argument closely follows the strategy we used in the proof of Lemma \ref{lem:CLT} above. We start by using Lemma \ref{lem:maxPairSharpestNeg} to argue (see \eqref{eqn:fluctEqn}) that with high probability,
\begin{align*}
	&\frac{n^{\frac{3+2\alpha}{2+\alpha}}}{e^{\sqrt{-1}\arg(X^{\da}_{(1)})}}\left(W^{\da}_{(1)} - X^{\da}_{(1)}(1-n^{-1})\right)\\
		&\qquad =  \frac{e^{-\sqrt{-1}\arg(X^{\da}_{(1)})}}{\frac{1}{n-1}\sum_{j=2}^n\frac{1}{X^{\da}_{(1)}-X^{\da}_{(j)}}}\cdot \frac{n}{n-1}\left(\frac{1}{n^{1/(2+\alpha)}}\sum_{j=2}^n \frac{X^{\da}_{(j)}}{X^{\da}_{(1)}-X^{\da}_{(j)}}\right) + {\rm err}_n,
\end{align*}
where $\abs{{\rm err}_n} \leq C' c_n n^{\alpha} = o(1)$. In view of Slutsky's theorem, the convergence in \eqref{eqn:maxConvNeg} follows from the next lemma that highlights a law of large numbers and central limit theorem for the scaled sums $\frac{1}{n}\sum_{j=2}^n \frac{1}{X^{\da}_{(1)}-X^{\da}_{(j)}}$.
\begin{lemma}\label{lem:LLNCLTNeg} Fix $\alpha \in (-0.095, 0)$. We have 
	\begin{equation}\label{eqn:LLNCLT}
	e^{\sqrt{-1}\arg(X^{\da}_{(1)})}\cdot \frac{1}{n}\sum_{j=2}^n \frac{1}{X^{\da}_{(1)}-X^{\da}_{(j)}} \to  1\ \text{and}\ \frac{1}{n^{1/(2+\alpha)}}\sum_{j=2}^n \frac{X^{\da}_{(j)}}{X^{\da}_{(1)}-X^{\da}_{(j)}} \to \mathcal{H}_{2+\alpha}
	\end{equation}
	in distribution as $n\to \infty$, where $\mathcal{H}_{2+\alpha}$ is the random variable mentioned in the conclusion of Theorem \ref{thm:fluctNeg}.
\end{lemma}
\begin{proof}
	We tackle the law of large numbers first. Consider that
	\begin{align*}
		&e^{\sqrt{-1}\arg(X^{\da}_{(1)})}\cdot \frac{1}{n}\sum_{j=2}^n \frac{1}{X^{\da}_{(1)}-X^{\da}_{(j)}}\\
		&\qquad =  \frac{1}{n}\sum_{j=2}^n \frac{e^{\sqrt{-1}\arg(X^{\da}_{(1)})}}{e^{\sqrt{-1}\arg(X^{\da}_{(1)})}-X^{\da}_{(j)}} + e^{\sqrt{-1}\arg(X^{\da}_{(1)})}\left(\overline{M}_\mu(X^\da_{(1)})-\overline{M}_\mu(e^{\sqrt{-1}\arg(X^{\da}_{(1)})})\right).
	\end{align*}
	For large $n$, on the complement of $\Omega_n$, $1-\abs{X^\da_{(1)}} \leq \eps_n^* < \frac{c_n}{2n}$, so $e^{\sqrt{-1}\arg(X^{\da}_{(1)})} \in B(X^\da_{(1)}, \frac{c_n}{2n})$, so by Lemma \ref{lem:nearLipMaxNeg}, it follows that for large $n$, on the complement of $E_n \cup F_n\cup F_n^* \cup G_n \cup \Omega_n$, 
	\[
	\abs{e^{\sqrt{-1}\arg(X^{\da}_{(1)})}\left(\overline{M}_\mu(X^\da_{(1)})-\overline{M}_\mu(e^{\sqrt{-1}\arg(X^{\da}_{(1)})})\right)} \leq 2c_n\eps_n^{\alpha-1}\eps_n^* = o(1),
	\]
	so continuing from above, we have that
	\[
	e^{\sqrt{-1}\arg(X^{\da}_{(1)})}\cdot \frac{1}{n}\sum_{j=2}^n \frac{1}{X^{\da}_{(1)}-X^{\da}_{(j)}} - \frac{1}{n}\sum_{j=2}^n \frac{1}{1-e^{\sqrt{-1}\arg(X^{\da}_{(1)})}\cdot X^{\da}_{(j)}} \to 0
	\]
	in distribution as $n\to \infty$ (note $\P(E_n \cup F_n\cup F_n^* \cup G_n \cup \Omega_n) = o(1)$ by Lemma \ref{lem:maxPairSharpestNeg}). By the radial symmetry of $\mu$, $\frac{1}{n}\sum_{j=2}^n \left(1-e^{\sqrt{-1}\arg(X^{\da}_{(1)})}\cdot X^{\da}_{(j)}\right)^{-1}$ has the same distribution as 
	\[
	\frac{1}{n}\sum_{j=2}^n \frac{1}{1- X^{\da}_{(j)}} = \frac{1}{n}\sum_{j=1}^n \frac{1}{1- X_j} - \frac{1}{n}\frac{1}{1- X^{\da}_{(1)}},
	\]
	which converges in distribution to $m_\mu(1) =1$ as $n\to \infty$ by the usual law of large numbers, Lemma \ref{lem:topCSTermsSmall}, and Slutsky's theorem. One more application of Slutsky's theorem establishes the first distributional limit in \eqref{eqn:LLNCLT}. 
	
	It remains to show that $n^{-1/(2+\alpha)}\sum_{j=2}^n \frac{X^{\da}_{(j)}}{X^{\da}_{(1)}-X^{\da}_{(j)}} \to \mathcal{H}_{2+\alpha}$ in distribution as $n \to \infty$. We have
	\begin{align*}
		&\abs{\frac{1}{n^{1/(2+\alpha)}}\sum_{j=2}^n \frac{X^{\da}_{(j)}}{X^{\da}_{(1)}-X^{\da}_{(j)}} - \frac{1}{n^{1/(2+\alpha)}}\sum_{j=2}^n \frac{X^{\da}_{(j)}}{e^{\sqrt{-1}\arg(X^{\da}_{(1)})}-X^{\da}_{(j)}}}\\
		&\qquad\leq \frac{1}{n^{1/(2+\alpha)}}\sum_{j=2}^n\abs{ \frac{X^{\da}_{(j)}\left(e^{\sqrt{-1}\arg(X^{\da}_{(1)})}-X^\da_{(1)}\right)}{\left(X^{\da}_{(1)}-X^{\da}_{(j)}\right)\left(e^{\sqrt{-1}\arg(X^{\da}_{(1)})}-X^{\da}_{(j)}\right)}}\\
		&\qquad\leq  \frac{1}{n^{1/(2+\alpha)}}\cdot \frac{\eps_n^*}{\delta_n^*} \cdot \sum_{j=2}^n \frac{1}{\abs{e^{\sqrt{-1}\arg(X^{\da}_{(1)})}-X^{\da}_{(j)}}},
	\end{align*}
where the last inequality holds for large $n$ on the complement of $\Omega_n \cup F_n^*$. (Off of $\Omega_n$, $X^\da_{(1)} \in \mathcal{A}_n^*$. Hence, off of $\Omega_n \cup F_n^*$, $X^\da_{(1)}$ is within $\eps_n^*$ of $e^{\sqrt{-1}\arg(X^{\da}_{(1)})}$ and $X^\da_{(1)}$ is $\delta_n^*$-far from all other $X^\da_{(j)}$, $j \in [n]$.) Now,
\[
\frac{n}{n^{1/(2+\alpha)}}\cdot \frac{\eps_n^*}{\delta_n^*} = c_n^{5\alpha+7} \cdot n^{\frac{\alpha}{\alpha+1}} = o(1),
\]
and $n^{-1}\sum_{j=2}^n \left(\abs{e^{\sqrt{-1}\arg(X^{\da}_{(1)})}-X^{\da}_{(j)}}\right)^{-1}$ has the same distribution as
\[
\frac{1}{n}\sum_{j=2}^n \frac{1}{\abs{1-X^{\da}_{(j)}}} = \frac{1}{n}\sum_{j=1}^n \frac{1}{\abs{1-X_{j}}} - \frac{1}{n}\frac{1}{\abs{1-X^{\da}_{(j)}}},
\]
which converges in probability to $\E\left[\abs{1-X_1}^{-1}\right] < \infty$ by the law of large numbers (see Lemma \eqref{lem:moments}), Lemma \ref{lem:topCSTermsSmall}, and Slutsky's Theorem. Continuing from above with another application of Slutsky's theorem shows that
\[
\frac{1}{n^{1/(2+\alpha)}}\sum_{j=2}^n \frac{X^{\da}_{(j)}}{X^{\da}_{(1)}-X^{\da}_{(j)}} - \frac{1}{n^{1/(2+\alpha)}}\sum_{j=2}^n \frac{X^{\da}_{(j)}}{e^{\sqrt{-1}\arg(X^{\da}_{(1)})}-X^{\da}_{(j)}} \to 0
\]
in distribution as $n\to \infty$. By the rotational symmetry of $\mu$, the scaled sum $n^{-1/(2+\alpha)}\sum_{j=2}^n \frac{X^{\da}_{(j)}}{e^{\sqrt{-1}\arg(X^{\da}_{(1)})}-X^{\da}_{(j)}}$ has the same distribution as
\[
\frac{1}{n^{1/(2+\alpha)}}\sum_{j=2}^n \frac{X^{\da}_{(j)}}{1-X^{\da}_{(j)}} = \frac{1}{n^{1/(2+\alpha)}}\sum_{j=1}^n \frac{X_j}{1-X_j} - \frac{1}{n}\frac{X^{\da}_{(1)}}{1-X^{\da}_{(1)}},
\]
which converges in distribution, as $n\to \infty$ to the desired variable $\mathcal{H}_{2+\alpha}$ by the heavy-tailed CLT Corollary \ref{cor:CLTNeg}, Lemma \ref{lem:topCSTermsSmall}, and Slutsky's Theorem. One final application of Slutsky's theorem yields the second distributional limit in \eqref{eqn:LLNCLT}.
\end{proof}
The proof of Theorem \ref{thm:fluctNeg} is complete.


\appendix

\section{Additional tools and results} \label{sec:appendix}

\begin{lemma}[Computation of $m_\mu(z)$ for radially symmetric distributions] \label{lem:StielCalc}
	Suppose $\mu$ is radially symmetric about the origin (i.e., suppose $\mu(A) = \mu(e^{\sqrt{-1}\theta}A)$ for any measurable $A \subset \mathbb{C}$ and any $\theta \in [0, 2\pi)$), and let $F_R(r)$ denote the associated radial c.d.f. Then, for any nonzero $z\in \C$ such that $\abs{z}$ is a continuity point of $F_R$, 
	\[
	m_\mu(z) = \frac{\mu(\set{x\in \C: \abs{x} \leq \abs{z}})}{z} = \frac{F_R(\abs{z})}{z}.
	\]
	In the case where $0$ is a continuity point of $F_R$, $m_\mu(0)= 0$.
\end{lemma}
\begin{proof}
	 Let $R = \abs{X}$ and $\Theta = \arg(X) \in [0, 2\pi)$ denote the radial and angular parts, respectively, of $X \sim \mu$. (Note: In this proof, it will be convenient to choose the range of $\arg(\cdot)$ to be $[0,2\pi)$ rather than $(-\pi,\pi]$.) Since $\mu$ is radially symmetric, $\Theta \sim {\rm Unif}(0, 2\pi)$ is independent of $R$. In the case where $z \neq 0$ and $\abs{z}$ is a continuity point of $F_R$, we can use polar coordinates and Laurent series to obtain
	\begin{align*}
		m_\mu(z) &= \E\left[\frac{1}{2\pi}\int_0^{2\pi}\frac{1}{z - Re^{\sqrt{-1}\theta}} \,d\theta\right]\\
		&= \frac{1}{z}\E\left[\frac{\sind{R< \abs{z}}}{2\pi}\int_0^{2\pi}\frac{1}{1 - \frac{R}{z}e^{\sqrt{-1}\theta}}\,d\theta\right] -  \frac{1}{z}\E\left[\frac{\sind{R > \abs{z}}}{2\pi}\int_0^{2\pi}\frac{\frac{z}{R}e^{-\sqrt{-1}\theta}}{1-\frac{z}{R}e^{-\sqrt{-1}\theta}}\,d\theta\right]\\
		&= \frac{1}{z}\E\left[\sind{R< \abs{z}} \sum_{j=0}^\infty\frac{1}{2\pi}\int_0^{2\pi}\left(\frac{R}{z}e^{\sqrt{-1}\theta}\right)^j\,d\theta\right]\\
		&\qquad\qquad -  \frac{1}{z}\E\left[\sind{R> \abs{z}} \sum_{j=0}^\infty\frac{z}{R}\frac{1}{2\pi}\int_0^{2\pi}e^{-\sqrt{-1}\theta}\left(\frac{z}{R}e^{-\sqrt{-1}\theta}\right)^j\,d\theta\right].
	\end{align*}
	The only nonzero integral occurs when the power on the exponential is $0$, so we obtain
	\[
	m_\mu(z) = \frac{1}{z}\E\left[\sind{R< \abs{z}}\right] 
	\] 
	as desired. Finally, observe that if $0$ is a continuity point of $F_R$, 
	\[
	m_\mu(0) = \E\left[\frac{1}{2\pi}\int_0^{2\pi} \frac{-1}{Re^{\sqrt{-1}\theta}} \,d\theta\right] =\E\left[\frac{-1}{R} \frac{1}{2\pi}\int_0^{2\pi} e^{-\sqrt{-1}\theta} \,d\theta\right] = 0.
	\]
\end{proof}
The following result establishes that when $\mu$ satisfies Assumption \ref{ass:ComplexMu}, the Cauchy--Stieltjes transform is Lipschitz continuous when $\alpha \geq 0$ and nearly Lipschitz when $-1 < \alpha < 0$ on the region $\mathbb{A}_\eps$. We direct the reader to Lemma 5.7 in \cite{KS} and Lemma 3.5 in \cite{OW2} for more general statements about the Lipschitz continuity of $m_\mu(z)$.
\begin{corollary}[Lipschitz continuity of $m_\mu(z)$ on $\mathbb{A}_\eps$]
	\label{cor:CSnice}
	Suppose $\mu$ satisfies Assumption \ref{ass:ComplexMu}, and let $F_R(r)$ denote the associated radial c.d.f. If $\alpha \geq 0$, there is a positive constant $\kappa_\mu$ so that
	\begin{equation}
		\abs{m_\mu(x)-m_\mu(y)} \leq \kappa_\mu\abs{x-y},\ \text{for $x,y \in \mathbb{A}_\eps$}.
		\label{eqn:CSnice}
	\end{equation}
	If, on the other hand, $-1 < \alpha < 0$, there is a positive constant $\kappa_\mu$ so that
	\begin{equation}
		\abs{m_\mu(x)-m_\mu(y)} \leq \kappa_\mu(1-\max{\set{\abs{x}, \abs{y}}})^\alpha\abs{x-y},\ \text{for $x,y \in \mathbb{A}_\eps$}.
		\label{eqn:CSniceNeg}
	\end{equation}
\end{corollary}

\begin{proof}
	Fix $x,y\in \mathbb{A}_\eps$. Since $\mu$ is radially symmetric, we have $m_\mu(z) = \frac{1}{z}F_R(\abs{z})$, so 
	\begin{equation}
		\begin{aligned}
			\abs{m_\mu(x)-m_\mu(y)} &\leq \frac{\abs{yF_R(\abs{x}) - xF_R(\abs{y})}}{\abs{xy}}\\
			&\leq \frac{1}{(1-\eps)^2}\left(\abs{x-y} + \abs{F_R(\abs{x}) - F_R(\abs{y})}\right),
		\end{aligned}
		\label{eqn:mLip}
	\end{equation}
	where we have used that \[\abs{x_1y_1 - x_2y_2} \leq \abs{x_1-x_2} + \abs{y_1-y_2}\ \text{for $\abs{x_i}, \abs{y_i} \leq 1$}.\] Without loss of generality, assume that $\abs{x} \geq \abs{y}$. Then, by Assumption \ref{ass:ComplexMu},
	\begin{equation*}
		\big|{F_R(\abs{x}) - F_R(\abs{y})\big|} = \int_{\abs{y}}^{\abs{x}}f_R(r)\,dr \leq C_\mu\int_{\abs{y}}^{\abs{x}}(1-r)^\alpha\,dr. 
	\end{equation*}
	In the case where $\alpha \geq 0$, the last integral is bounded above by
	$\abs{\abs{x} - \abs{y}} \leq \abs{x-y}$, and in the case where $\alpha \in (-1, 0)$, this integral is bounded above by
	\[
	(1-\abs{x})^\alpha \left(\abs{\abs{x} - \abs{y}}\right) \leq (1-\abs{x})^\alpha\abs{x-y}.
	\]
	Setting $\kappa_\mu = (1-\eps)^{-2} + C_\mu$ therefore establishes \eqref{eqn:CSnice} and \eqref{eqn:CSniceNeg} by way of \eqref{eqn:mLip}.
\end{proof}

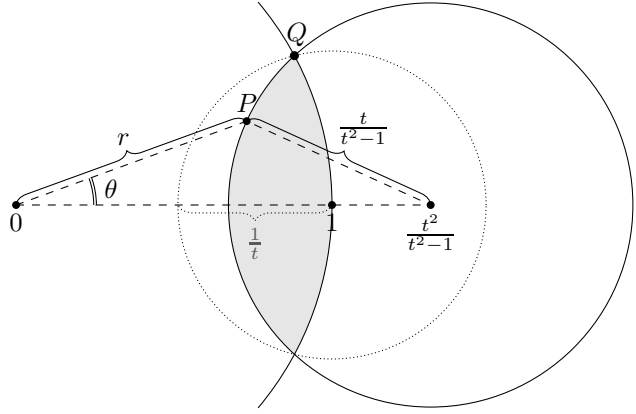
\begin{figure}
	\centering
	\begin{tikzpicture}[x=4.2cm, y=4.2cm]
		\pgfmathsetmacro{\myT}{2.05}
		\pgfmathsetmacro{\circC}{\myT*\myT/(\myT*\myT - 1)} 
		\pgfmathsetmacro{\circR}{\myT/(\myT*\myT - 1)} 
		\pgfmathsetmacro{\maxTheta}{acos(1-1/(2*\myT*\myT))};
		\pgfmathsetmacro{\weirdTheta}{asin((1-1/(\myT*\myT))*sqrt(1-1/(4*\myT*\myT)))};
		
		\fill[color=gray!20] (0.1, 0) -- ({cos(-\maxTheta)},{sin(-\maxTheta)}) arc (-\maxTheta:\maxTheta:1) -- cycle;
		\fill[color=white] (0,0)-- ({cos(\maxTheta)},{sin(\maxTheta)}) arc ({180-\weirdTheta}:{180+\weirdTheta}:\circR) -- cycle;

		\draw ({cos(-40)},{sin(-40)}) arc (-40:40:1);
		\draw[dashed] (0,0) -- (\circC,0);
		\draw[densely dotted] (1,0) circle[radius = {1/\myT}];
		\draw (\circC,0) circle[radius = \circR];


		
		\pgfmathsetmacro{\intAngle}{20};
		\pgfmathsetmacro{\intR}{(\myT*\myT*cos(\intAngle)/(\myT*\myT-1)) - (\myT/(\myT*\myT-1)*sqrt(1-\myT*\myT*sin(\intAngle)*sin(\intAngle)))};		
		
		\draw[dashed] (0,0) -- ({\intR*cos(\intAngle)},{\intR*sin(\intAngle)})-- (\circC,0);
		\draw [decorate,decoration={brace,amplitude=5pt}] (0,0) -- ({\intR*cos(\intAngle)}, {\intR*sin(\intAngle)}) node [midway, xshift=-3pt, yshift=9pt]{$r$};
		\draw [decorate,decoration={brace,amplitude=5pt}] ({\intR*cos(\intAngle)}, {\intR*sin(\intAngle)})--(\circC,0) node [midway, above right, xshift=-4pt, yshift=3pt]{$\frac{t}{t^2-1}$};
		\draw [double] (.25,0) arc (0:\intAngle:.25);
		\draw (.25, .05) node[right] {$\theta$};


		\fill ({\intR*cos(\intAngle)},{\intR*sin(\intAngle)}) circle[radius =1.5pt] node[above]{$P$};
		
		\fill (0,0) circle[radius = 1.5pt] node[below]{$0$};
		\draw [decorate,decoration={brace,amplitude=6pt}, densely dotted] (1,0)--({1-1/\myT},0) node [midway, below, yshift=-5pt, color=black!65]{$\frac{1}{t}$};
		\fill (1,0) circle[radius = 1.5pt] node[below]{$1$};		
		\fill (\circC,0) circle[radius = 1.5pt] node[below]{$\frac{t^2}{t^2-1}$};
		\draw[fill=black] ({cos(\maxTheta)}, {sin(\maxTheta)}) circle[radius=1.5pt] node[above, xshift=1pt]{$Q$};
	\end{tikzpicture}
	\caption{For $t> 1$, this graphic depicts $\mathcal{R}_t(r,\theta)$, the region of integration involved our justification of \eqref{eqn:tailSharp}. When $X$ lies in the shaded subset of the unit disk, we have $|X^{-1} -1| \leq t^{-1}$. The solid circle with center ${t^2}/(t^2-1)$ and radius ${t}/(t^2-1)$ is the image under $z\mapsto z^{-1}$ of the dotted circle with center $1$ and radius $t^{-1}$. This means $X^{-1}$ lies interior to the dotted circle precisely when $X$ lies interior to the solid circle. One can apply the law of cosines to the dashed triangle to relate the distance $r$ to the angle $\theta$ for any point $P$ on the solid circle.\label{fig:tailRegion}}
\end{figure}

\begin{lemma}\label{lem:heavyTail}
	Suppose $\mu$ satisfies Assumption \ref{ass:ComplexMu} with $-1 < \alpha < 0$, let $X \sim \mu$, and define
	\[
	\mathbb{A}_{{\eps}/{2}} := \set{z\in \mathbb{C}: 1-\frac{\eps}{2} \leq \abs{z} \leq 1}.
	\]
	Then, if $t > \max\set{2/\eps, 1/(1-\eps)}$, we have
	\begin{equation}\label{eqn:tailUpper}
 \sup_{z \in \mathbb{A}_{{\eps}/{2}}}\P\left(\abs{\frac{1}{z-X}} \geq t\right) \leq \frac{2^{\alpha+2}C_\mu}{\pi(\alpha+1)(1-\eps)}t^{-(2 + \alpha)}.
	\end{equation}
	Furthermore, if  the radial density $f_R(r)$ that is guaranteed by Assumption \ref{ass:ComplexMu} satisfies $\lim_{r\to1^-}\frac{f_R(r)}{(1-r)^\alpha} = C$ for a constant $C>0$, then
	\begin{equation}\label{eqn:tailSharp}
		\begin{aligned}
	\lim_{t\to \infty}\left[t^{2 + \alpha}\cdot \P\left(\abs{\frac{X}{1-X}} \geq t\right)\right] &= \frac{C}{\pi(\alpha + 1)}\int_0^1(1-u^2)^{(1+\alpha)/2}\,du\\
	&= \frac{C}{2\pi(\alpha+2)}\int_{-\pi/2}^{\pi/2}\left(\cos\theta\right)^\alpha\,d\theta,
	\end{aligned}
	\end{equation}
	and for any Borel set $D \subset \mathbb{S}^1 := \set{z \in \mathbb{C}: \abs{z} = 1}$, 
	\begin{equation}\label{eqn:tailAngle}
	\lim_{t\to \infty}\P\left(\left.\frac{X}{1-X}\middle/\abs{\frac{X}{1-X}}\right.\in D\ \Bigg\vert\ \abs{\frac{X}{1-X}} \geq t\right) = \frac{\int_{\arg(D) \cap (-\pi/2, \pi/2)}  \left(\cos\theta\right)^\alpha d\theta}{\int_{-\pi/2}^{\pi/2}\left(\cos\theta\right)^\alpha d\theta}.
	\end{equation}
	(Note: Here, we use the convention that $\arg(D) := \set{\arg(z): z \in D}$ is the image of $D$ under $\arg(\cdot)$.)
\end{lemma}
\begin{proof}
	We first establish \eqref{eqn:tailUpper} in which $z$ is allowed to range over all of $\mathbb{A}_{\eps/2}$, and then we refine our approximations in the case $\lim_{r\to1^-}\frac{f_R(r)}{(1-r)^\alpha} = C$ to achieve \eqref{eqn:tailSharp} and \eqref{eqn:tailAngle} for $z = 1$. At present, fix $z \in \mathbb{A}_{\eps/2}$. Since $\mu$ is radially symmetric, we have $\P\left(\abs{z-X} \leq t^{-1}\right) = \P\left(\abs{\abs{z}-X} \leq t^{-1}\right)$, and in view of the estimate in Figure \ref{fig:avoid1}, for $t > (1-\eps)^{-1}$, 
	\[
	\P\left(\abs{\frac{1}{z-X}} \geq t\right) \leq \P\left(-\frac{2}{t\abs{z}} \leq \theta \leq \frac{2}{t\abs{z}}\ \text{and}\ \abs{z} - \frac{1}{t} \leq \abs{X} \leq \abs{z} + \frac{1}{t} \right), 
	\]
	where $\theta = \arg(X)$. 
	After passing to polar coordinates, and appealing to symmetry across the $x$-axis, we obtain for $t > \max\set{2/\eps,\ (1-\eps)^{-1}}$,
	\begin{align*}
		\P\left(\abs{\frac{1}{z-X}} \geq t\right) &\leq 2\int_0^{\frac{2}{t\abs{z}}}\int_{\abs{z}-1/t}^{\min\set{\abs{z}+1/t,1}}\frac{1}{2\pi}f_R(r)\,dr\,d\theta\\
		&\leq \frac{C_\mu}{\pi}\int_0^{\frac{2}{t\abs{z}}}\int_{\abs{z}-1/t}^{\min\set{\abs{z}+1/t,1}}(1-r)^\alpha\,dr\,d\theta.
	\end{align*}
	(Note that for $z \in \mathbb{A}_{\eps/2}$ and $t > 2/\eps$, $\abs{z-X}\leq 1/t$ implies $X \in \mathbb{A}_\eps$, so we can use the density $f_R(r)$ in the integrand.) Since $\alpha < 0$, the integrand is increasing in $r$, so for $t>\max\set{2/\eps, (1-\eps)^{-1}}$, we can set the bounds of integration on the inner integral to $1 - 2/t$ and $1$, which yields
	\[
		\P\left(\abs{\frac{1}{z-X}} \geq t\right) \leq \frac{C_\mu}{\pi}\int_0^{\frac{2}{t\abs{z}}}\int_{1-2/t}^{1}(1-r)^\alpha\,dr\,d\theta = \frac{2^{\alpha +2}C_\mu}{\pi(\alpha + 1)\abs{z}}\cdot t^{-(2+\alpha)}.
	\]
	Equation \eqref{eqn:tailUpper} follows since this bound holds for any $z \in \mathbb{A}_{\eps/2}$. 
	
	We now compute the limit in \ref{eqn:tailSharp}.  Observe that for $t > 1/\eps$, 
	\begin{equation}\label{eqn:tailInt}
	\P\left(\abs{\frac{X}{1-X}} \geq t\right) = \P\left(\abs{X^{-1} -1} \leq t^{-1}\right) = \iint_{\mathcal{R}_t(r,\theta)}\frac{1}{2\pi}f_R(r)\,dr\,d\theta,
	\end{equation}
	where $\mathcal{R}_t(r,\theta)$ is the shaded region in Figure \ref{fig:tailRegion}. One can find the polar representation of any point $P$ on the ``inner'' boundary of $\mathcal{R}_t(r,\theta)$ by applying the law of cosines to the dashed triangle in Figure \ref{fig:tailRegion} and using the quadratic formula to solve for $r$. In particular, after substituting $1-\cos^2\theta$ for $\sin^2\theta$ under the radical,  
	\[
		r = \frac{t^2\cos{\theta}}{t^2-1} - \frac{t}{t^2-1}\sqrt{1-t^2\sin^2{\theta}},
	\]
	so for $t>1/\eps$, the region $\mathcal{R}_t(r,\theta)$ can be described in polar coordinates by the inequalities
	\begin{equation}\label{eqn:Rbds}
	\begin{aligned}
		\abs{\theta} \leq \theta_{\rm max} &:=\arccos\left(1-\frac{1}{2t^2}\right)\ \text{and} \\
		r_{\rm min}(\theta)&:=\frac{t^2\cos{\theta}}{t^2-1} - \frac{t}{t^2-1}\sqrt{1-t^2\sin^2{\theta}} \leq r \leq 1.
	\end{aligned}
	\end{equation}
	(To easily obtain $\theta$ when $P=Q$ in Figure \ref{fig:tailRegion}, one can apply the law of cosines to the isosceles triangle formed with corners $0$, $Q$, and $1$.)
	
	We will approximate the integral in \eqref{eqn:tailInt} above and below using the hypothesis $\lim_{r\to1^-}\frac{f_R(r)}{(1-r)^\alpha} = C$. Let $\eta \in (0, C)$ be given, and define $T_\eta > 2/\eps$ such that $t > T_\eta$ implies $\abs{\frac{f_R(r)}{(1-r)^\alpha}-C} < \eta$ for $(r,\theta) \in \mathcal{R}_t(r,\theta)$. Finding such a $T_\eta$ is possible because $\mathcal{R}_t(r,\theta) \subset \set{(r,\theta): 1-t^{-1} \leq r \leq 1,\ \theta \in (-\pi,\pi]}$ for $t > 1$ (see also Figure \ref{fig:tailRegion}). Then, for $t > T_\eta$, \eqref{eqn:tailInt} yields
	\begin{equation}\label{eqn:tailUpperEqns}
	\begin{aligned}
		\P\left(\abs{\frac{X}{1-X}} \geq t\right) &=\int_{-\theta_{\rm max}}^{\theta_{\rm max}} \int_{r_{\rm min}(\theta)}^1\frac{1}{2\pi}f_R(r)\,dr\,d\theta\\
		& \leq \frac{(C + \eta)}{2\pi}\int_{-\theta_{\rm max}}^{\theta_{\rm max}} \int_{r_{\rm min}(\theta)}^1(1-r)^\alpha\,dr\,d\theta\\
		& = \frac{(C + \eta)}{\pi(\alpha+1 )}\int_{0}^{\theta_{\rm max}} (1-r_{\rm min}(\theta))^{\alpha+1}\,d\theta\\
		& = \frac{(C + \eta)}{\pi(\alpha+1 )}\frac{1}{t^{2+\alpha}}\int_{0}^{t\theta_{\rm max}} (t-t\cdot r_{\rm min}(u/t))^{\alpha+1}\,du
	\end{aligned}
	\end{equation}
	where the penultimate equality follows since $r_{\rm min}(\theta)$ is an even function of $\theta$ and to achieve the last line, we have employed the substitution $u = t\theta$. For $t > T_\eta$ and $u \in [0, t\theta_{\rm max}]$, we have\footnote{We have used that $1-\cos(x) \leq x^2/2$ and $\sin(x) \leq x$ for $x \geq 0$.}
	\begin{equation}\label{eqn:tailIntegrandBd}
	0\leq t(1-r_{\rm min}(u/t)) = \frac{t^3(1-\cos(u/t))-t}{t^2-1} + \frac{t^2}{t^2-1}\sqrt{1-t^2\sin^2(u/t)}\leq \frac{\frac{tu^2}{2} + t^2}{t^2-1},
	\end{equation}
	which is bounded by an absolute constant (e.g.\ $2$) for large $t$. In view of the fact that $t\theta_{\rm max} \to 1$ as $t\to \infty$ and $t(1-r_{\rm min}(u/t)) \to \sqrt{1-u^2}$ pointwise as as $t \to \infty$, Lebesgue's dominated convergence theorem implies
	\begin{equation}\label{eqn:limsupUpperTail}
		\limsup_{n\to \infty}\left[t_n^{2 + \alpha}\cdot \P\left(\abs{\frac{X}{1-X}} \geq t_n\right)\right] \leq  \frac{C + \eta}{\pi(\alpha + 1)}\int_0^1(1-u^2)^{(\alpha+1)/2}\,du,
	\end{equation}
	for any positive real sequence $t_n$ that tends to $\infty$ as $n\to \infty$.

	On the other hand, if we approximate $f_R(r)$ below by $(C-\eta)(1-r)^\alpha$ for $t > T_\eta$, and otherwise employ the same steps as we did in \eqref{eqn:tailUpperEqns}, Lebesgue's dominated convergence theorem implies 
	\[
	\liminf_{n\to \infty}\left[t_n^{2 + \alpha}\cdot \P\left(\abs{\frac{X}{1-X}} \geq t_n\right)\right] \geq  \frac{C - \eta}{\pi(\alpha + 1)}\int_0^1(1-u^2)^{(\alpha+1)/2}\,du,
	\]
	for any positive real sequence $t_n$ that grows to $\infty$ with $n$. Combining this inequality with \eqref{eqn:limsupUpperTail} establishes \eqref{eqn:tailSharp} since $\eta \in (0, C)$ and the sequence $t_n$ are arbitrary.
	
	It remains to prove \eqref{eqn:tailAngle}, which gives the limiting angular density of $\frac{X}{1-X}$ conditioned on $\abs{\frac{X}{1-X}} \geq t$ as $t\to \infty$. We first establish the result for Borel sets $D \subset \mathbb{S}^1$ such that $\arg(D) = [0, \beta]$, where $0 \leq \beta \leq \frac{\pi}{2}$. Then, we'll prove it for arbitrary open subsets of $\mathbb{S}^1$ by approximating these with countable disjoint unions $\cup_{i=1}^\infty\mathcal{O}_i \subset \mathbb{S}^1$, where $\arg\left(\mathcal{O}_i\right)$ are open intervals. Finally, we'll extend \eqref{eqn:tailAngle} to arbitrary Borel subsets of the unit circle by approximating them above (resp.\ below) by open (resp.\ closed) subsets of $\mathbb{S}^1$.
	
	To start, fix $\beta\in [0, \pi/2]$ and consider $D\subset\mathbb{S}^1$ for which $\arg(D) = [0, \beta]$. Observe that for $t > 1/\eps$,
	\begin{equation}\label{eqn:tailAngleInt}
		\P\left(\left.\frac{X}{1-X}\middle/\abs{\frac{X}{1-X}}\right. \in D\ \text{and}\ \abs{\frac{X}{1-X}} \geq t\right) = \iint_{\mathcal{R}_{t,\beta}(r,\theta)}\frac{1}{2\pi}f_R(r)\,dr\,d\theta,
	\end{equation}
	where $\mathcal{R}_{t,\beta}(r,\theta)$ is the dark-gray region in Figure \ref{fig:tailAngleRegion}.
	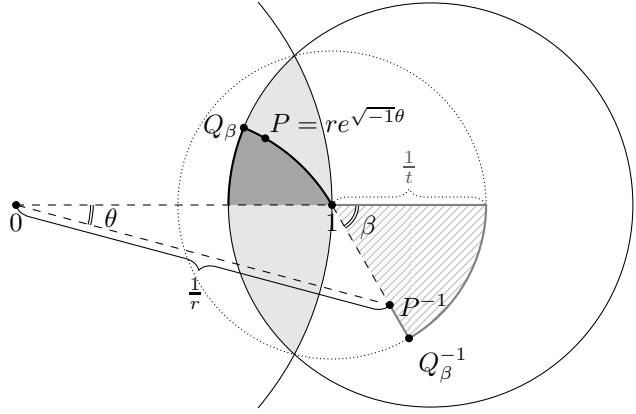
\begin{figure}
		\centering
		\begin{tikzpicture}[x=4.2cm, y=4.2cm]
			\pgfmathsetmacro{\myT}{2.05}
			\pgfmathsetmacro{\circC}{\myT*\myT/(\myT*\myT - 1)} 
			\pgfmathsetmacro{\circR}{\myT/(\myT*\myT - 1)} 
			\pgfmathsetmacro{\maxTheta}{acos(1-1/(2*\myT*\myT))};
			\pgfmathsetmacro{\weirdTheta}{asin((1-1/(\myT*\myT))*sqrt(1-1/(4*\myT*\myT)))};
			
			\fill[color=gray!20] (0.1, 0) -- ({cos(-\maxTheta)},{sin(-\maxTheta)}) arc (-\maxTheta:\maxTheta:1) -- cycle;
			\fill[color=white] (0,0)-- ({cos(\maxTheta)},{sin(\maxTheta)}) arc ({180-\weirdTheta}:{180+\weirdTheta}:\circR) -- cycle;

			\draw ({cos(-40)},{sin(-40)}) arc (-40:40:1);
			\draw[densely dotted] (1,0) circle[radius = {1/\myT}];
			\draw (\circC,0) circle[radius = \circR];

			
			\pgfmathsetmacro{\myBeta}{60} 
			\pgfmathsetmacro{\IntPInv}{0.75} 
			\pgfmathsetmacro{\IntTheta}{asin(\IntPInv/\myT*sin(\myBeta)/sqrt(1+\IntPInv*\IntPInv/(\myT*\myT) + 2*\IntPInv/\myT*cos(\myBeta)))} 
			\pgfmathsetmacro{\ThetaBeta}{asin(sin(\myBeta)/\myT/sqrt(1+1/(\myT*\myT)+ 2*cos(\myBeta)/\myT))}  

			\fill[pattern = north east lines, pattern color =gray!50] (1,0) -- ({1+1/(\myT)},0) arc (0:-\myBeta:{1/(\myT)}) --  cycle;
			\draw [double] (1.08,0) arc (0:-\myBeta:0.08);
			\draw (1.06, -.07) node[right] {$\beta$};			
			\draw[thick, color=gray] (1,0) -- ({1+1/\myT},0) arc (0:-\myBeta:{1/\myT})--({1+\IntPInv/\myT*cos(-\myBeta)}, {\IntPInv/\myT*sin(-\myBeta)}) ;			
			\fill ({1+cos(\myBeta)/\myT}, {-sin(\myBeta)/\myT}) circle[radius = 1.5pt] node[below right]{$Q_\beta^{-1}$};	
					
			\fill[domain=0:\ThetaBeta, color=gray!70] plot({\x}:{cos(\x) + cot(-\myBeta)*sin(\x)}) -- (\ThetaBeta:{cos(\ThetaBeta) + cot(-\myBeta)*sin(\ThetaBeta)}) node[xshift = 1.5pt, yshift= 0pt, left, color=black]{$Q_\beta$} arc ({180-asin( sin(\ThetaBeta)*(\myT*\myT-1)/\myT*(cos(\ThetaBeta) + cot(-\myBeta)*sin(\ThetaBeta)))}:180:{\myT/(\myT*\myT-1)})--(1,0);
			\fill (\ThetaBeta:{cos(\ThetaBeta) + cot(-\myBeta)*sin(\ThetaBeta)}) circle[radius = 1.5pt];
			\draw[domain=0:\ThetaBeta, thick, color=black] plot({\x}:{cos(\x) + cot(-\myBeta)*sin(\x)}) -- (\ThetaBeta:{cos(\ThetaBeta) + cot(-\myBeta)*sin(\ThetaBeta)}) arc  ({180-asin( sin(\ThetaBeta)*(\myT*\myT-1)/\myT*(cos(\ThetaBeta) + cot(-\myBeta)*sin(\ThetaBeta)))}:180:{\myT/(\myT*\myT-1)});			
			\fill ({\IntTheta}:{cos(\IntTheta) + cot(-\myBeta)*sin(\IntTheta)}) circle[radius = 1.5pt] node[above right, xshift=-2pt, yshift = -1pt]{$P = re^{\sqrt{-1}\theta}$};
			
			
			\fill (0,0) circle[radius = 1.5pt] node[below]{$0$};
			\draw [decorate,decoration={brace,amplitude=6pt}, densely dotted] (1,0)--({1+1/\myT},0) node [midway, above, yshift=5pt, color=black!65]{$\frac{1}{t}$};
			\fill (1,0) circle[radius = 1.5pt] node[below]{$1$};		
		
			\draw[dashed] (0,0) -- (1,0) -- ({1+\IntPInv*cos(\myBeta)/\myT}, {-\IntPInv*sin(\myBeta)/\myT}) -- cycle;
			\draw [double] (.24,0) arc (0:-\IntTheta:0.24);
			\draw (0.25, -.04) node[right] {$\theta$};			
		
			\draw [decorate,decoration={brace,amplitude=6pt}] ({1+\IntPInv*cos(\myBeta)/\myT}, {-\IntPInv*sin(\myBeta)/\myT})  --(0,0) node [midway, below, xshift=-3pt, yshift = -5pt]{$\frac{1}{r}$};
			
			\fill ({1+\IntPInv*cos(\myBeta)/\myT}, {-\IntPInv*sin(\myBeta)/\myT}) circle[radius = 1.5pt] node[right, xshift=0pt, yshift = 0pt]{$P^{-1}$};			
		\end{tikzpicture}
		\caption{For $t> 1$, this graphic depicts $\mathcal{R}_{t,\beta}(r,\theta)$, the region of integration involved our justification of \eqref{eqn:tailAngle}. The shaded intersection of the three circles is the same as in Figure \ref{fig:tailRegion} above, and the dark gray subset is the image under $z \mapsto 1/z$ of the sector with parallel hatching. For a fixed $\beta \in[0, \pi/2]$, when $X$ lies in the dark gray subset, we have $|X^{-1} -1| \leq t^{-1}$ and $\arg(X^{-1}-1) \in [-\beta,0]$, which means that $|X/(1-X)| \geq t$ and $\arg(X/(1-X)) \in [0,\beta]$. One can apply the law of sines to the dashed triangle to relate the modulus $r$ to the angle $\theta$ for any point $P=re^{\sqrt{-1}\theta}$ on the upper-right boundary of the dark gray subset. \label{fig:tailAngleRegion}}
	\end{figure}
	By applying the law of sines to the dashed triangle in Figure \ref{fig:tailAngleRegion}, we obtain the following polar equation for the upper-right boundary of $\mathcal{R}_{t,\beta}(r, \theta)$ when $0 < \beta \leq \pi/2$:
	\begin{equation}\label{eqn:rBetaDef}
	r_\beta(\theta) := \frac{\sin(\beta-\theta)}{\sin(\pi-\beta)} = \cos(\theta)-\cot(\beta)\sin(\theta),
	\end{equation}
	where
	\begin{equation}\label{eqn:thetaBetaDef}
	0 \leq \theta \leq \theta_ \beta := \arcsin\left(\frac{\sin(\beta)}{t\sqrt{1+t^{-2} + 2t^{-1}\cos(\beta)}}\right).
	\end{equation}
	The upper bound on $\theta$ follows after applying the laws of cosines and sines to the dashed triangle in Figure \ref{fig:tailAngleRegion} when $P^{-1} = Q_\beta^{-1}$ to obtain
	\[
		\frac{1}{r^2} = 1 + \frac{1}{t^2} - \frac{2}{t}\cos(\pi - \beta)\quad \text{and}\quad \frac{\sin(\theta)}{\frac{1}{t}} = \frac{\sin(\pi-\beta)}{\frac{1}{r}}
	\]
	and solving the rightmost equation for $\theta$.
	
	When $\beta = 0$, $\mathcal{R}_{t,\beta}(r,\theta)$ is a line segment that has Lebesgue measure zero, so \eqref{eqn:tailAngle} holds when $\arg(D)=[0,0] = \set{0}$. We now find upper and lower bounds on the integral in \eqref{eqn:tailAngleInt} when $\beta \in (0,\pi/2]$. Making use of $\lim_{r \to 1^-}\frac{f_R(r)}{(1-r)^\alpha} = C$, let $\eta \in (0, C)$ be given, and define $T_\eta > 2/\eps$ so that $t > T_\eta$ implies $\abs{\frac{f_R(r)}{(1-r)^\alpha} - C} < \eta$ for $(r,\theta) \in \mathcal{R}_{t,\beta}(r, \theta)$. Then, for $t > T_\eta$, we have the following upper bound on \eqref{eqn:tailAngleInt}:
	 \begin{align}
	 		&\P\left(\left.\frac{X}{1-X}\middle/\abs{\frac{X}{1-X}}\right. \in D\ \text{and}\ \abs{\frac{X}{1-X}} \geq t\right) \notag\\
	 		&\qquad = \int_0^{\theta_\beta}\int_{r_{\rm min}(\theta)}^{r_\beta(\theta)}\frac{1}{2\pi}f_R(r)\,dr\,d\theta\notag\\
	 		&\qquad \leq \frac{C + \eta}{2\pi(\alpha +1)}\int_0^{\theta_\beta}(1-r_{\rm min}(\theta))^{\alpha+1} - (1-r_\beta(\theta))^{\alpha+1}\,d\theta \notag\\
	 		&\qquad = \frac{C+\eta}{2\pi(\alpha +1) t^{2+\alpha}}\int_0^{t\theta_\beta}(t-t\cdot r_{\rm min}(u/t))^{\alpha+1} - (t-t\cdot r_\beta(u/t))^{\alpha + 1}\,du, \label{eqn:tailAngleUpInt1}
	 \end{align}
 where $r_{\rm min}(\theta)$ is defined as in \eqref{eqn:Rbds}, and the last equality follows after the substitution $u = t\theta$. For $u \in [0, t\theta_\beta]$ and $t > 2$, we have\footnote{We have used that $1-\cos(x) \leq \frac{x^2}{2}$ and $\sin(x) \leq x$ for $x \geq 0$.}
 \begin{equation*}
 0 \leq t(1-r_\beta(u/t)) = t(1-\cos(u/t)) + \cot(\beta)\cdot t\sin(u/t) \leq \frac{u^2}{t} + \cot{\beta}\cdot u
 \end{equation*}
 and (see \eqref{eqn:tailIntegrandBd})
 \[
 0\leq t(1-r_{\rm min}(u/t)) \leq \frac{t\cdot \frac{u^2}{2} + t^2}{t^2-1},
 \]
 and we also have $t\theta_\beta \to \sin(\beta)$ as $t\to \infty$.  It follows that for large $t$, the integrand of the integral in \eqref{eqn:tailAngleUpInt1} is bounded by an absolute constant, and Lebesgue's dominated convergence theorem implies
\begin{equation}\label{eqn:tailAngLimSup}
\begin{aligned}
&\limsup_{n\to \infty}\left(t_n^{2+\alpha}\cdot\P\left(\left.\frac{X}{1-X}\middle/\abs{\frac{X}{1-X}}\right. \in D\ \text{and}\ \abs{\frac{X}{1-X}} \geq t_n\right)\right)\\
&\quad\leq \frac{C+\eta}{2\pi (\alpha+1)}\int_0^{\sin(\beta)}\lim_{n\to\infty}\left((t_n-t_n r_{\rm min}(u/t_n))^{\alpha+1} - (t_n-t_n r_\beta(u/t_n))^{\alpha + 1}\right)du\\
&\quad= \frac{C+\eta}{2\pi (\alpha+1)}\int_0^{\sin(\beta)}(1-u^2)^{(1+\alpha)/2} - \left(\cot(\beta)u\right)^{1+\alpha}\,du\\
&\quad= \frac{C+\eta}{2\pi (\alpha+2)}\int_0^\beta (\cos{\theta})^\alpha\,d\theta,
\end{aligned}
\end{equation}
where $t_n$ is any positive real sequence that grows to $\infty$ with $n$. The last equality follows by substituting $u = \sin\theta$ and using integration by parts to establish
\[\int_0^\beta (\cos{\theta})^{\alpha + 2}\,d\theta = \frac{\alpha+1}{\alpha+2}\int_0^\beta (\cos{\theta})^\alpha\,d\theta + \frac{1}{\alpha +2} \sin{\beta}(\cos{\beta})^{\alpha+1}.
\]
We note that the rightmost term equals $\int_0^\beta (\cot(\beta)\sin{\theta})^{1+\alpha}\cos{\theta}\,d\theta$. 

Using the lower bound $f_R(r) \geq (C-\eta)(1-r)^\alpha$ for $(r,
\theta) \in \mathcal{R}_{t,\beta}(r,\theta)$, we can employ reasoning nearly identical to that just above to obtain
\begin{align*}
	\liminf_{n\to \infty}\left(t_n^{2+\alpha}\cdot\P\left(\left.\frac{X}{1-X}\middle/\abs{\frac{X}{1-X}}\right. \in D,\ \abs{\frac{X}{1-X}} \geq t_n\right)\right) \geq \frac{(C-\eta)\int_0^\beta (\cos{\theta})^\alpha\, d\theta}{2\pi (\alpha+2)},
\end{align*}
for any positive sequence $t_n$ growing to infinity with $n$. Combining this with \eqref{eqn:tailAngLimSup} establishes
\[
\lim_{t\to \infty}\left(t^{2+\alpha}\cdot\P\left(\left.\frac{X}{1-X}\middle/\abs{\frac{X}{1-X}}\right. \in D,\ \abs{\frac{X}{1-X}} \geq t\right)\right) = \frac{C\int_0^\beta (\cos{\theta})^\alpha\,d\theta}{2\pi (\alpha+2)}
\]
since $\eta \in(0, C)$ and the sequence $t_n$ are arbitrary. In view of \eqref{eqn:tailSharp} and the definition of conditional probability, \eqref{eqn:tailAngle} holds for sets $D \subset \mathbb{S}^1$ such that $\arg(D)=[0,\beta]$ for a fixed $\beta \in (0,\pi/2]$.

So far, we have established \eqref{eqn:tailAngle} when $\arg(D) = [0, \beta]$ for $\beta \in [0, \pi/2]$. We note that if $\arg(D) = \set{x}$ for some $x \in [0, \pi/2]$, then the limit in \eqref{eqn:tailAngle} is $0$ because $\mu$ has a density with respect to the Lebesgue measure on $\mathcal{R}_t(r,\theta)$ for large enough $t$. By the additivity of $\P(\cdot)$ and the limit sum law, this means \eqref{eqn:tailAngle} also holds for $D \subset \mathbb{S}^1$ such that $\arg(D) = (\beta, \gamma)$, where $0 \leq \beta < \gamma \leq \pi/2$. We now extend the result to the situation where $\arg(D) \subset (0, \pi/2)$ is an arbitrary open set in the standard topology on $\mathbb{R}$. If $\arg(D) \subset (0, \pi/2)$ is the disjoint union of finitely many open intervals, the result follows by the additivity of $\P(\cdot)$ and the limit sum law, so it suffices to consider the case where there is a sequence of disjoint open intervals $(\beta_n, \gamma_n) \subset (0, \pi/2)$, $n \geq 1$, such that $\arg(D) = \cup_{n=1}^\infty (\beta_n, \gamma_n)$. By countable additivity of $\P(\cdot)$, we have
\begin{equation}\label{eqn:tailAngleOpenApprox}
\begin{aligned}
&\lim_{t\to \infty}\left(t^{2+\alpha}\cdot\P\left(\left.\frac{X}{1-X}\middle/\abs{\frac{X}{1-X}}\right. \in D\ \text{and}\ \abs{\frac{X}{1-X}} \geq t\right)\right)\\
&\hspace*{1cm}=\lim_{t\to \infty}\sum_{n=1}^\infty t^{2+\alpha}\cdot\P\left(\arg\left(\frac{X}{1-X}\right) \in (\beta_n, \gamma_n),\ \text{and}\ \abs{\frac{X}{1-X}} \geq t\right),
\end{aligned}
\end{equation}
so the desired result follows if we can justify switching the sum and the limit as $t \to \infty$. To do so, we will show that independently of $t$, the summands are bounded above by a non-negative sequence $g_n$ satisfying $\sum_{n=1}^\infty g_n < \infty$ and employ the Lebesgue dominated convergence theorem. To that end, fix $t > 1/\eps$ and $n \geq 1$, and consider that
\begin{equation}\label{eqn:LebDomStart}
\begin{aligned}
	&t^{2+\alpha}\cdot\P\left(\arg\left(\frac{X}{1-X}\right) \in (\beta_n, \gamma_n),\ \text{and}\ \abs{\frac{X}{1-X}} \geq t\right)\\
	&\quad= t^{2+\alpha}\left(\iint_{\mathcal{R}_{t,\beta_n}(r,\theta)}\frac{1}{2\pi}f_R(r)\,dr\,d\theta-\iint_{\mathcal{R}_{t,\gamma_n}(r,\theta)}\frac{1}{2\pi}f_R(r)\,dr\,d\theta\right)\\
	&\quad\leq \int_{\theta_{\beta_n}}^{\theta_{\gamma_n}}\int_{r_{\rm min}(\theta)}^1\frac{C_\mu t^{2+\alpha}}{2\pi}(1-r)^\alpha\,dr\,d\theta + \int_{0}^{\theta_{\beta_n}}\int_{r_{\beta_n}(\theta)}^{r_{\gamma_n}(\theta)}\frac{C_\mu t^{2+\alpha}}{2\pi}(1-r)^\alpha\,dr\,d\theta,
\end{aligned}
\end{equation}
where $\mathcal{R}_{t, \beta_n}(r, \theta)$, $\mathcal{R}_{t, \gamma_n}(r, \theta)$, $\theta_{\beta_n}$, $\theta_{\gamma_n}$, $r_{\rm min}(\theta)$, $r_{\beta_n}(\theta)$, $r_{\gamma_n}(\theta)$ are defined in the obvious way as above (see Figures \ref{fig:tailRegion} and \ref{fig:tailAngleRegion} and Equations \eqref{eqn:Rbds},\eqref{eqn:rBetaDef}, and \eqref{eqn:thetaBetaDef}). We bound the two integrals on the right separately. For $t > 1/\eps$, and $n \geq 1$, we have 
\[
	\int_{\theta_{\beta_n}}^{\theta_{\gamma_n}}\int_{r_{\rm min}(\theta)}^1\frac{C_\mu t^{2+\alpha}}{2\pi}(1-r)^\alpha\,dr\,d\theta = \frac{C_\mu}{2\pi(1+\alpha)}	\int_{t\theta_{\beta_n}}^{t\theta_{\gamma_n}}(t-t r_{\rm min}(u/t))^{1+\alpha}\,du,
\]
where we have evaluated the inner integral with respect to $r$ and then employed the substitution $u = t\theta$. In view of \eqref{eqn:tailIntegrandBd}, there is an absolute constant $\widetilde{C} > 0$ and a $T>0$ so that if $t > T$, the integrand is bounded by $\widetilde{C}$, so for $t > \max\set{1/\eps, T}$ and $n \geq 1$, we have
\[
\int_{\theta_{\beta_n}}^{\theta_{\gamma_n}}\int_{r_{\rm min}(\theta)}^1\frac{C_\mu t^{2+\alpha}}{2\pi}(1-r)^\alpha\,dr\,d\theta \leq \frac{C_\mu\widetilde{C}}{2\pi(1+\alpha)}\cdot (t\theta_{\gamma_n}-t\theta_{\beta_n}).
\]
Since $x \mapsto \arcsin(x)$ is Lipschitz continuous on $[0, 1/2]$ and for $t \geq 2$, 
\[
\sup_{\beta \in [0,\pi/2]}\left(\frac{\sin(\beta)}{t\sqrt{1 + t^{-2} + 2t^{-1}\cos(\beta)}}\right) \leq \frac{1}{2}, 
\] 
there is a constant $C'_\mu > 0$ so that $t> \max\set{\eps^{-1}, T, 2}$ implies
\begin{align*}
&\int_{\theta_{\beta_n}}^{\theta_{\gamma_n}}\int_{r_{\rm min}(\theta)}^1\frac{C_\mu t^{2+\alpha}}{2\pi}(1-r)^\alpha\,dr\,d\theta\\
&\qquad\leq C'_\mu \left(\frac{\sin(\gamma_n)}{\sqrt{1 + t^{-2} + 2t^{-1}\cos(\gamma_n)}}- \frac{\sin(\beta_n)}{\sqrt{1 + t^{-2} + 2t^{-1}\cos(\beta_n)}}\right)\\
&\qquad \leq C'_\mu\abs{\sin(\gamma_n) - \sin(\beta_n)}\\
&\hspace*{2cm} + C'_\mu\abs{\frac{1}{\sqrt{1 + t^{-2} + 2t^{-1}\cos(\gamma_n)}}-\frac{1}{\sqrt{1 + t^{-2} + 2t^{-1}\cos(\beta_n)}}},
\end{align*}
where the last inequality follows because \[\abs{w_1z_1 - w_2z_2} \leq \abs{w_1-w_2} + \abs{z_1 -z_2}\] for complex numbers $w_1,w_2, z_1, z_2$ of modulus at most $1$. Now, $x \mapsto 1/\sqrt{x}$ is Lipschitz continuous with constant $1/2$ on $[1, \infty)$, so for  $t> \max\set{\eps^{-1}, T, 2}$, we have
\begin{equation}\label{eqn:DCTBound1}
\begin{aligned}
	&\int_{\theta_{\beta_n}}^{\theta_{\gamma_n}}\int_{r_{\rm min}(\theta)}^1\frac{C_\mu t^{2+\alpha}}{2\pi}(1-r)^\alpha\,dr\,d\theta\\
	&\qquad \leq C'_\mu\abs{\sin(\gamma_n) - \sin(\beta_n)} + \frac{C'_\mu}{2}\cdot \frac{2}{t}\abs{\cos(\gamma_n) - \cos(\beta_n)}\\
	&\qquad\leq 2C_\mu'\big(\gamma_n-\beta_n\big),
\end{aligned}
\end{equation}
where the last line follows because both of $x\mapsto \sin(x)$ and $x\mapsto \cos(x)$ are Lipschitz continuous with constant $1$.

We now turn our attention to the second integral in the last line of \eqref{eqn:LebDomStart}. For $t > 1/\eps$ and $n \geq 1$, we have
\begin{align*}
&\int_{0}^{\theta_{\beta_n}}\int_{r_{\beta_n}(\theta)}^{r_{\gamma_n}(\theta)}\frac{C_\mu t^{2+\alpha}}{2\pi}(1-r)^\alpha\,dr\,d\theta\\
&\quad= \frac{C_\mu}{2\pi(1+\alpha)} \int_{0}^{t\theta_{\beta_n}}u^{1+\alpha}\left(\left(\frac{t-t\cdot r_{\beta_n}(u/t)}{u}\right)^{1+\alpha} - \left(\frac{t-t\cdot r_{\gamma_n}(u/t)}{u}\right)^{1+\alpha}\right)du,
\end{align*}
where we have evaluated the inner integral with respect to $r$ and then employed the substitution $u= t\theta$. For $0 \leq u \leq t\theta_{\beta_n}$ (note $t\theta_{\beta_n} \leq t\theta_{\gamma_n} \leq t\arcsin(1/2) = t\pi/6$),
\[
\min\set{\frac{t-t\cdot r_{\beta_n}(u/t)}{u},\ \frac{t-t\cdot r_{\gamma_n}(u/t)}{u}} \geq \cot(\beta_n)\cdot\frac{\sin(u/t)}{u/t} \geq \frac{1}{2}\cot(\beta_n),
\]
so in view of the fact that $x \mapsto x^{1+\alpha}$ is Lipschitz continuous with constant $(\alpha + 1)\left(\frac{1}{2}\cot(\beta_n)\right)^\alpha$ on the interval $[\cot(\beta_n)/2, \infty)$, we have
\begin{align*}
	&\int_{0}^{\theta_{\beta_n}}\int_{r_{\beta_n}(\theta)}^{r_{\gamma_n}(\theta)}\frac{C_\mu t^{2+\alpha}}{2\pi}(1-r)^\alpha\,dr\,d\theta\\
	&\quad\leq \frac{C_\mu(\cot(\beta_n))^\alpha}{2^{1+\alpha}\pi} \int_{0}^{t\theta_{\beta_n}}u^{1+\alpha}\left(\frac{t-t\cdot r_{\beta_n}(u/t)}{u} - \frac{t-t\cdot r_{\gamma_n}(u/t)}{u}\right)du\\
	&\quad= \frac{C_\mu(\cot(\beta_n))^\alpha}{2^{1+\alpha}\pi} \int_{0}^{t\theta_{\beta_n}}u^{1+\alpha}\frac{\sin(u/t)}{u/t}\left(\cot(\beta_n) - \cot(\gamma_n)\right)du\\
	&\quad\leq \frac{C_\mu(\cot(\beta_n))^\alpha}{2^{1+\alpha}\pi(2+\alpha)}\left(\cot(\beta_n) - \cot(\gamma_n)\right)\left(t\theta_{\beta_n}\right)^{2+\alpha}.
\end{align*}
Based on the definition of $\theta_\beta$ (see \eqref{eqn:thetaBetaDef}), there is an $S > 0$ so that for any $\beta \in [0, \pi/2]$, $t > S$ implies $t\theta_\beta \leq 2\sin(\beta)$. It is also true that on $[\beta_n, \gamma_n]$, $x\mapsto \cot(x)$ is Lipschitz continuous with constant $\csc^2(\beta_n)$, so continuing from above, we have that for $t > \max\set{S, \eps^{-1}}$ and $n \geq 1$,
\begin{align*}
	\int_{0}^{\theta_{\beta_n}}\int_{r_{\beta_n}(\theta)}^{r_{\gamma_n}(\theta)}\frac{C_\mu t^{2+\alpha}}{2\pi}(1-r)^\alpha\,dr\,d\theta &\leq \frac{C_\mu(\cot(\beta_n))^\alpha}{2^{1+\alpha}\pi(2+\alpha)}\cdot\frac{(\gamma_n-\beta_n)}{\sin^2(\beta_n)}\cdot (2\sin(\beta_n))^{2+\alpha}\\
	&\leq C''_\mu(\gamma_n-\beta_n),
\end{align*}
where $C''_\mu$ is a constant depending only on $\mu$. Combining this inequality with \eqref{eqn:LebDomStart} and \eqref{eqn:DCTBound1} establishes that for $t > \max\set{T, S, \eps^{-1}, 2}$ and $n\geq 1$,
\[
t^{2+\alpha}\cdot\P\left(\arg\left(\frac{X}{1-X}\right) \in (\beta_n, \gamma_n),\ \text{and}\ \abs{\frac{X}{1-X}} \geq t\right) \leq (2C_\mu' + C_\mu'')(\gamma_n - \beta_n).
\]
Since $(\beta_n, \gamma_n) \subset (0, \pi/2)$, $n \geq 1$, are disjoint intervals, $\sum_{n=1}^\infty (2C_\mu' + C_\mu'')(\gamma_n - \beta_n) < \infty$, and via Lebesgue's dominated convergence theorem, we can interchange the limit and sum in \eqref{eqn:tailAngleOpenApprox} to obtain \eqref{eqn:tailAngle} for $D\subset \mathbb{S}^1$ where $\arg(D) \subset (0, \pi/2)$ is open. By symmetry, the  additivity of $\P(\cdot)$, and the limit sum law, it follows that \eqref{eqn:tailAngle} holds for $D \subset \mathbb{S}^1$ where $\arg(D) \subset (-\pi/2, \pi/2)$ is open in the standard topology on $\mathbb{R}$. (Recall that \eqref{eqn:tailAngle} also holds when $\arg(D) = \set{0}$.)

Now, if $D$ is an arbitrary open subset of the unit circle (here, we assume $\mathbb{S}^1$ inherits the subspace topology from $\mathbb{C}$), we can write $D$ as the disjoint union $D = \mathcal{O}_D \cup D_0$ where\footnote{Here, $\arg^{-1}(\cdot)$ denotes the preimage of $\cdot$ under $\arg$.}
\begin{align*}
	\mathcal{O}_D &:= D \cap \arg^{-1}\big((-\pi/2, \pi/2)\big)\\
	D_0 &:= D \cap \arg^{-1}\big((-\pi, -\pi/2] \cup [\pi/2, \pi]\big).
\end{align*}
Since $z \mapsto \arg(z)$ is a topological isomorphism\footnote{Here, we assume $\{e^{\sqrt{-1}\theta}: -\pi/2 < \theta < \pi/2\}$, respectively  $(-\pi/2, \pi/2)$, has the subspace topology inherited from $\mathbb{C}$, respectively $\mathbb{R}$.} from the open set $\{e^{\sqrt{-1}\theta}: -\pi/2 < \theta < \pi/2\} \subset \mathbb{S}^1$ to the open set $(-\pi/2, \pi/2) \subset \mathbb{R}$, it follows that $\mathcal{O}_D$ is open in $\mathbb{S}^1$ and its image, $\arg(\mathcal{O}_D) \subset (-\pi/2, \pi/2)$, is open in $\mathbb{R}$. Consequently, the reasoning in the preceding paragraphs implies that \eqref{eqn:tailAngle} holds with $\mathcal{O}_D$ in place of $D$. By the monotonicity and additivity of $\P(\cdot)$, we have for $t > \eps^{-1}$,
\begin{equation}\label{eqn:D0Bound}
\begin{split}
0 &\leq\P\left(\left.\frac{X}{1-X}\middle/\abs{\frac{X}{1-X}}\right. \in D_0,\ \abs{\frac{X}{1-X}} \geq t\right)\\
&\leq \P\left(\abs{\frac{X}{1-X}} \geq t\right) - \P\left(\arg\left(\frac{X}{1-X}\right) \in (-\pi/2, \pi/2),\ \abs{\frac{X}{1-X}} \geq t\right),
\end{split}
\end{equation}
so dividing by $\P\left(\abs{\frac{X}{1-X}} \geq t\right)$ and taking $t \to \infty$ establishes
\eqref{eqn:tailAngle} for $D_0$ in place of $D$. (Both sides of \eqref{eqn:tailAngle} are $0$ when evaluated at $D_0$; note that $\arg(D_0) \cap (-\pi/2, \pi/2) = \emptyset$.) It follows from the additivity of $\P(\cdot)$ and the limit sum law that \eqref{eqn:tailAngle} is true for arbitrary open (in the subspace topology on $\mathbb{S}^1$) subsets $D \subset \mathbb{S}^1$. 

We conclude the proof of Lemma \ref{lem:heavyTail} by establishing \eqref{eqn:tailAngle} for closed subsets of $\mathbb{S}^1$ and then by approximating arbitrary Borel subsets of $\mathbb{S}^1$ by open and closed subsets of $\mathbb{S}^1$. To that end, first observe that if $\mathcal{C}$ is closed in $\mathbb{S}^1$, then we can write $\mathcal{C} = \mathbb{S}^1\setminus\mathcal{C}^c$, where $\mathbb{S}^1$ and $\mathcal{C}^c$ are open in $\mathbb{S}^1$, and we can apply \eqref{eqn:tailAngle} with each of $\mathbb{S}^1$ and $\mathbb{C}^c$ in place of $D$. In view of the additivity of $\P(\cdot)$, the limit sum law, and the fact that $\arg(\mathcal{C}) = \arg(\mathbb{S}^1) \setminus \arg(\mathcal{C}^c)$, we conclude that \eqref{eqn:tailAngle} is also true with the arbitrary closed set $\mathcal{C}\subset \mathbb{S}^1$ in place of $D$. We are now ready to tackle the case where $D \subset \mathbb{S}^1$ is an arbitrary Borel set.

Suppose $D \subset \mathbb{S}^1$ is an arbitrary Borel set, and let $\set{t_n}_{n=1}^\infty$ be an arbitrary positive sequence tending to infinity. We can write $D$ as the disjoint union $D =  D_+ \cup D_0$, where 
\begin{align*}
	D_+ &:= D \cap \arg^{-1}\big((-\pi/2, \pi/2)\big)\\
	D_0 &:= D \cap \arg^{-1}\big((-\pi, -\pi/2] \cup [\pi/2, \pi]\big).
\end{align*}
Since $z \mapsto \arg(z)$ is a topological isomorphism from the open set $\{e^{\sqrt{-1}\theta}: -\pi/2 < \theta < \pi/2\} \subset \mathbb{S}^1$ to the open set $(-\pi/2, \pi/2) \subset \mathbb{R}$, it follows that $\arg(D_+) \subset (-\pi/2, \pi/2)$ is a Borel subset of $\mathbb{R}$.  Consequently, for any $\eta > 0$, we can find a closed set $\mathcal{C}_D$ and an open set $\mathcal{O}_D$ for which $\mathcal{C}_D \subset \arg(D_+) \subset \mathcal{O}_D \subset (-\pi/2, \pi/2)$ and the Lebesgue measure of $\mathcal{O}_D \setminus \mathcal{C}_D$ is less than $\eta$ (see e.g.\ Theorem 12.3 on page 184 of \cite{Bill}). We have 
\[
\arg^{-1}(\mathcal{C}_D) \subset D_+ \subset \arg^{-1}(\mathcal{O}_D),
\]
where $\arg^{-1}(\mathcal{C}_D)$ is closed in $\mathbb{S}^1$ and $\arg^{-1}(\mathcal{O}_D)$ is open in $\mathbb{S}^1$. Applying \eqref{eqn:tailAngle} with each of $\arg^{-1}(\mathcal{C}_D)$ and $\arg^{-1}(\mathcal{O}_D)$ in place of $D$, and appealing to the monotonicity of $\mathbb{P}(\cdot)$ establishes that
\begin{equation}\label{eqn:DPlusProb}
\begin{split}
0 &\leq \abs{\limsup_{n\to \infty}\P\left(\left.\frac{X}{1-X}\middle/\abs{\frac{X}{1-X}}\right.\in D_+\ \Bigg\vert\ \abs{\frac{X}{1-X}} \geq t_n\right) - \frac{\int_{\arg(D_+)}  \left(\cos\theta\right)^\alpha d\theta}{\int_{-\pi/2}^{\pi/2}\left(\cos\theta\right)^\alpha d\theta}}\\
&\leq \frac{\int_{\mathcal{O}_D \setminus \mathcal{C}_D}  \left(\cos\theta\right)^\alpha d\theta}{\int_{-\pi/2}^{\pi/2}\left(\cos\theta\right)^\alpha d\theta}.
\end{split}
\end{equation}
Now, $-1 < \alpha < 0$, so we can find $p_\alpha, q_\alpha \in (1, \infty)$ such that $-1 < \alpha\cdot p_\alpha < 0$ and $1/p_\alpha + 1/q_\alpha = 1$. It follows from H\"{o}lder's inequality that 
\[
\int_{\mathcal{O}_D \setminus \mathcal{C}_D}  \left(\cos\theta\right)^\alpha d\theta \leq  \left(\int_{-\pi/2}^{\pi/2}  \left(\cos\theta\right)^{\alpha p_\alpha} d\theta\right)^{1/p_\alpha} \cdot \left(\lambda_{\mathbb{R}}(\mathcal{O}_D \setminus \mathcal{C}_D)\right)^{1/q_\alpha}  < C_\alpha \cdot \eta^{1/q_\alpha},
\]
where $\lambda_{\mathbb{R}}$ denotes the Lebesque measure on $\mathbb{R}$, and $C_\alpha$ is a positive constant depending on $\alpha$. Since $\eta > 0$ and $t_n \to \infty$ are arbitrary and the $\limsup_{n\to \infty}$ may be replaced with $\liminf_{n\to \infty}$ in \eqref{eqn:DPlusProb}, we conclude that \eqref{eqn:tailAngle} holds with $D_+$ in place of $D$. By the monotonicity and additivity of $\P(\cdot)$, we have \eqref{eqn:D0Bound} for $t > 1/\eps$, so dividing by $\P\left(\abs{\frac{X}{1-X}} \geq t\right)$ and taking $t \to \infty$ establishes \eqref{eqn:tailAngle} for $D_0$ in place of $D$. By the additivity of $\P(\cdot)$ and the limit sum law, we conclude that \eqref{eqn:tailAngle} is true for arbitrary Borel sets $D \subset \mathbb{S}^1$.
\end{proof}

\begin{corollary}
	\label{cor:CLTNeg}
	Suppose $\mu$ satisfies Assumption \ref{ass:ComplexMu} with $-1 < \alpha < 0$ and let $X_1, X_2, \ldots$ be independent complex-valued random variables with common distribution $\mu$. Then,
	\[
	\frac{1}{n^{1/(2+\alpha)}}\sum_{j=1}^n \frac{X_j}{1-X_j} \to \mathcal{H}_{2+\alpha},
	\]
	in distribution as $n\to \infty$, where $\mathcal{H}_{2+\alpha}$ is a complex-valued random variable, which, by identifying $\mathbb{C}$ with $\mathbb{R}^2$, is a multivariate $(2+\alpha)$-stable random vector with spectral measure given by the right-hand side of \eqref{eqn:tailAngle}. 
\end{corollary}
\begin{proof}
	The claim follows from standard results on domains of attractions for stable distributions; we refer the reader to \cite{MR1280932,MR150795,MR270403,MR1652283,MR4230105,MR62975} and references therein for more details concerning stable distributions and convergence theorems.  Here, we use Theorem 4.2 from \cite{MR150795}.  The assumptions of Theorem 4.2 from \cite{MR150795} are verified in \eqref{eqn:tailSharp} and \eqref{eqn:tailAngle}, and it follows that there exists a deterministic sequence $b_n$ so that 
	\[ \frac{1}{n^{1/(2+\alpha)}}\sum_{j=1}^n \frac{X_j}{1-X_j} - b_n \to \mathcal{H}_{2+\alpha} \]
	in distribution as $n \to \infty$, where $\mathcal{H}_{2+\alpha}$ is defined in the statement of the corollary. In particular, by the continuous mapping theorem, this implies that the real and imaginary parts of 
	\[ \frac{1}{n^{1/(2+\alpha)}}\sum_{j=1}^n \frac{X_j}{1-X_j} - b_n \]
	converge in distribution to the real and imaginary parts of $\mathcal{H}_{2+\alpha}$, which are univariate $(2 + \alpha)$-stable random variables.  It follows from standard results on univariate domains of attraction (see, for instance, \cite{MR150795,MR270403,MR62975}) and from Remark 3 in \cite{MR624688} that $b_n \to 0$ as $n \to \infty$.\footnote{We note that $\E[X_j/(1-X_j)] = 0$ (see e.g.\ \eqref{eqn:centeredMeanZero}, which holds for $\alpha > -1$).} The result follows. 
\end{proof}

\begin{lemma}[Absolute moments of $(z-X)^{-1}$ for $\abs{z} \geq 1-\eps/2$]\label{lem:moments}
	Suppose $\mu$ is a radially symmetric measure supported on the unit disk that satisfies Assumption \ref{ass:ComplexMu}, and define $\mathbb{A}_{\eps/2} := \set{z\in \mathbb{C}: 1 -\frac{\eps}{2} \leq \abs{z} \leq 1}.$
\begin{enumerate}[(\thetheorem.i)]
\item When $\alpha \geq 0$, there is a constant $C >0$ (depending on $\mu, \alpha, \eps$) so that 
\begin{align}
	&\sup_{z \in \mathbb{A}_{\eps/2}}\E\abs{z-X}^{-\mathfrak{p}} \leq \frac{C}{2-\mathfrak{p}}\ \text{for $0 \leq \mathfrak{p} < 2$, and} \label{eqn:pMomentA}\\
	&\sup_{\abs{z} = 1}\E\abs{z-X}^{-\mathfrak{p}} \leq \frac{C}{(\alpha + 2)-\mathfrak{p}}\ \text{for $0\leq \mathfrak{p} < 2 + \alpha$}.\label{eqn:pMomentCirc}
\end{align}
\item  When $-1 < \alpha < 0$, there is a constant $C>0$ (depending on $\mu, \alpha, \eps$) so that
\begin{equation}\label{eqn:pMomentNeg}
\sup_{z \in \mathbb{A}_{\eps/2}}\E\abs{z-X}^{-\mathfrak{p}} \leq \frac{C}{(\alpha + 2)-\mathfrak{p}}\ \text{for $0\leq \mathfrak{p} < 2 + \alpha$.}
\end{equation}
\end{enumerate}
\end{lemma}
\begin{proof}
	First, suppose $\alpha \geq 0$ and $0 \leq \mathfrak{p} < 2$, and fix $z\in \mathbb{A}_{\eps/2}$. After employing a truncation and using the fact that on $\mathbb{A}_\eps$, $\mu$ has a density $f_\mu$ satisfying $f_\mu(x) \leq \frac{C_\mu(1-\abs{x})^\alpha}{2\pi\abs{x}}$, we have
	\begin{equation}\label{eqn:pMomentFirstControl}
		\begin{aligned}
	\E\abs{\frac{1}{(z-X)^\mathfrak{p}}} &\leq \E\abs{\frac{1}{(z-X)^\mathfrak{p}}\cdot \ind{X \in \mathbb{A}_\eps}} + \left(\frac{\eps}{2}
	\right)^{-\mathfrak{p}}\\
	&\leq \frac{C_\mu}{2\pi(1-\eps)}\int_{\mathbb{A}_\eps}\frac{(1-\abs{x})^\alpha}{\abs{z-x}^\mathfrak{p}}\,d\lambda(x)+ \left(\frac{\eps}{2}
	\right)^{-\mathfrak{p}}.
		\end{aligned}
	\end{equation}
	Using that $1-\abs{x} < \eps$ for $x \in \mathbb{A}_\eps$ and switching to polar coordinates yields
	\begin{align*}
	\E\abs{z-X}^{-\mathfrak{p}} &\leq \frac{C_\mu\eps^\alpha}{2\pi(1-\eps)}\int_{\mathbb{A}_\eps}\abs{z-x}^{-\mathfrak{p}}\,d\lambda(x)+ \left(\frac{\eps}{2}
	\right)^{-\mathfrak{p}}\\
	&\leq \frac{C_\mu\eps^\alpha}{(1-\eps)}\int_0^2r^{-\mathfrak{p}}\cdot r\,dr+ \left(\frac{\eps}{2}
	\right)^{-\mathfrak{p}}\cdot\frac{2}{2-\mathfrak{p}},
	\end{align*}
	for $\mathfrak{p} < 2$. (Note: We can multiply $({\eps}/{2}
	)^{-\mathfrak{p}}$ by $2/(2-\mathfrak{p}) \geq 1$ since $\mathfrak{p} \geq 0$.) This bound is independent of $z \in \mathbb{A}_{\eps/2}$, so \eqref{eqn:pMomentA} follows. In the special case where $\abs{z} = 1$, the triangle inequality implies $1-\abs{x} =\abs{z}-\abs{x} \leq \abs{z-x}$ for $x\in\mathbb{A}_\eps$, so after switching to polar coordinates \eqref{eqn:pMomentFirstControl} establishes
	\begin{align*}
	\E\abs{z-X}^{-\mathfrak{p}} &\leq \frac{C_\mu}{2\pi(1-\eps)}\int_{\mathbb{A}_\eps}\abs{z-x}^{\alpha-\mathfrak{p}}\,d\lambda(x) +  \left(\frac{\eps}{2}
	\right)^{-\mathfrak{p}}\\
	&\leq \frac{C_\mu}{(1-\eps)}\int_{0}^2r^{\alpha-\mathfrak{p}}\cdot r\,dr + \left(\frac{\eps}{2}
	\right)^{-\mathfrak{p}}\cdot\frac{\alpha + 2}{(\alpha + 2) - \mathfrak{p}},
	\end{align*} from which \eqref{eqn:pMomentCirc} follows as is desired. 

	In the case where $-1 < \alpha < 0$, we appeal to \eqref{eqn:tailUpper} from Lemma \ref{lem:heavyTail}. In particular, for $0 \leq \mathfrak{p} < 2 + \alpha$ and a fixed $z\in \mathbb{A}_{\eps/2}$, we have
	\[
		\E\abs{z-X}^{-\mathfrak{p}} = \int_0^\infty \P\left(\abs{z-X}^{-\mathfrak{p}} \geq t\right)\,dt =  \int_0^\infty \P\left(\abs{z-X}^{-1} \geq t^{1/\mathfrak{p}}\right)\,dt,
	\]
	so via \eqref{eqn:tailUpper}, 
	\begin{align*}
		\E\abs{z-X}^{-\mathfrak{p}} &\leq \int_0^{\left(\frac{2}{\eps(1-\eps)}\right)^{\mathfrak{p}}}1\,dt + \int_{\left(\frac{2}{\eps(1-\eps)}\right)^{\mathfrak{p}}}^\infty\frac{2^{\alpha+2}C_\mu}{\pi(\alpha+1)(1-\eps)}t^{-(2 + \alpha)/\mathfrak{p}}\,dt\\
		&= \left(\frac{2}{\eps(1-\eps)}\right)^{\mathfrak{p}} + \frac{2^{\alpha+2}C_\mu}{\pi(\alpha+1)(1-\eps)} \cdot \frac{\mathfrak{p}}{2+\alpha - \mathfrak{p}}\left(\frac{2}{\eps(1-\eps)}\right)^{\mathfrak{p}-(2+\alpha)}.
	\end{align*}
	This bound is uniform over all $z\in \mathbb{A}_{\eps/2}$, so we have established \eqref{eqn:pMomentNeg}. (Note that the leftmost term achieves its maximum when $\mathfrak{p} = 2+\alpha$.)
\end{proof}

\begin{lemma}[$\eps$-nets for annuli]
	Given an annulus
	\[
	\mathcal{A}:=\set{z\in \C : r_1 \leq \abs{z} \leq r_2},\ 0<r_1<r_2
	\]
	and $\eps \in (0,\frac{r_2}{2})$, there is an $\frac{\eps r_1}{2r_2}$-separated $\eps$-net of $\mathcal{A}$ of size at most \[\begin{cases}22\eps^{-2}r_2(r_2-r_1) &\text{if $\eps \leq r_2-r_1$,}\\18\eps^{-1}r_2&\text{if $r_2 -r_1 < \eps < \frac{r_2}{2}$}.\end{cases}\] More concretely, there exist $z_1,\ldots, z_m \in \mathcal{A}$ which satisfy
	\begin{enumerate}[(\thetheorem.i)]
		\item $\mathcal{A} \subset \cup_{i=1}^mB(z_i, \eps)$,
		\item if $1 \leq i < j \leq m$, then $\abs{z_i-z_j} > \frac{\eps r_1}{2r_2}$,
		\item $m \leq 22\eps^{-2}r_2(r_2-r_1)$ if $\eps \leq r_2-r_1$, and $m \leq 18\eps^{-1}r_2$ otherwise.
	\end{enumerate}
\label{lem:epsNet}
\end{lemma}

\begin{proof}
	First, fix $\delta \in (0, {r_2}/{2})$, and consider the $\sqrt{2}\delta$-net of points
	\begin{align*}
	\mathcal{N}_\delta&:=\set{(r_1 + j\cdot\delta)e^{\sqrt{-1}\cdot k\cdot\frac{\delta}{r_2}} :\
	\begin{aligned}
		&0 \leq j \leq \max\set{0,\left\lfloor (r_2-r_1)/\delta\right\rfloor -1},\\
		&0 \leq k \leq \left\lfloor 2\pi r_2/\delta\right\rfloor -1
	\end{aligned}}\\
	&\quad\cup\Bigg\{r_2e^{\sqrt{-1}\cdot k\cdot\frac{\delta}{r_2}}: 0 \leq k \leq \left\lfloor 2\pi r_2/\delta\right\rfloor -1\Bigg\},
	\end{align*}
	which are at the intersections of the radial grid formed by the circles
	\[
	\set{z\in \C: \abs{z} = r_2},\ \set{z\in \C: \abs{z} = r_1 + j\cdot\delta},\ 0 \leq j \leq \max\set{0,\left\lfloor (r_2-r_1)/\delta\right\rfloor -1}
	\]
	and rays\footnote{Here, we define $\arg(z)$ to have range $[0, 2\pi)$.}
	\[
	\set{z\in \C\setminus\set{0}: \arg(z)= k\cdot\frac{\delta}{r_2}},\ 0 \leq k \leq \left\lfloor 2\pi r_2/\delta\right\rfloor -1.
	\]
	Using Figure \ref{fig:epsNet} as a guide, it is easy to see that any point in $\mathcal{A}$ is at most $\sqrt{2}\delta$-far from some $z_i \in \mathcal{N}_\delta$. (For example, the maximum distance between a point of $\mathcal{A}$ and $\mathcal{N}_\delta$ is less than the distance between the center and corner of a square of side length $2\delta$.) Furthermore, by construction, the pairs of points in $\mathcal{N}_\delta$ that lie closest to one another are on the circle $\set{z\in \C: \abs{z} = r_1}$ at endpoints of arcs whose angle measures are $\delta/r_2$. By the law of cosines and a fourth-degree Taylor approximation to the cosine function, one obtains that for any $z,w \in \mathcal{N}_\delta$
	\[
	\abs{z-w} > \sqrt{2}r_1\sqrt{1-\cos(\delta/r_2)} \geq \sqrt{2}r_1\sqrt{\frac{\delta^2}{2r_2^2} - \frac{\delta^4}{24r_2^4}} > \frac{r_1\delta}{r_2}\sqrt{\frac{11}{12}}
	\]
	(recall that we assumed $\delta < {r_2}/{2}$).
	
	If we set $\delta = \eps/\sqrt{2}$, the collection $\mathcal{N}_{\eps/\sqrt{2}}$ is an $\eps$-net satisfying the conclusion of the lemma. Note that the number of points in $\mathcal{N}_{\eps/\sqrt{2}}$ is at most
	\[
	\left((r_2-r_1)\sqrt{2}/\eps +1\right)\left( 2\pi r_2\sqrt{2}/\eps\right) \leq \frac{\left((r_2-r_1)\sqrt{2} + \eps\right)\left(2\sqrt{2}\pi r_2\right)}{\eps^2}\leq \frac{22r_2(r_2-r_1)}{\eps^2}
	\]
	in the case where $\eps \leq r_2-r_1$. If $\eps > r_2 - r_1$, then \[
	\mathcal{N}_{\eps/\sqrt{2}} = \Bigg\{re^{\sqrt{-1}\cdot k\cdot\frac{\eps}{r_2\sqrt{2}}}: 0 \leq k \leq \left\lfloor 2\pi r_2\sqrt{2}/\eps\right\rfloor -1,\ r = r_1, r_2\Bigg\},
	\] so $\abs{\mathcal{N}_{\eps/\sqrt{2}}}$ is at most $2\cdot \left( 2\pi r_2\sqrt{2}/\eps\right) < 18 r_2 /\eps$.
	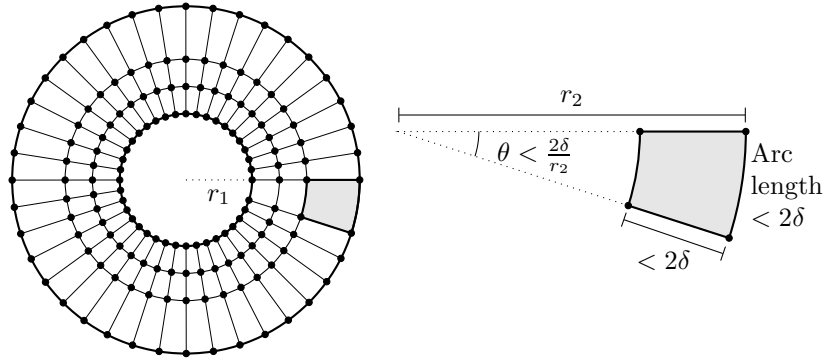
\begin{figure}
		\centering
		\pgfmathsetmacro{\myRa}{.8} 
		\pgfmathsetmacro{\myRb}{2.1} 
		\pgfmathsetmacro{\myD}{0.33} 
		\pgfmathsetmacro{\upbdj}{floor((\myRb-\myRa)/\myD)-1} 
		\pgfmathsetmacro{\upbdk}{floor(6.2832*\myRb/\myD)-1} 
		\pgfmathsetmacro{\bigTheta}{360-deg(\upbdk*\myD/\myRb)} 
		\begin{minipage}{.4\linewidth}
		\begin{tikzpicture}[x=1.1cm,y=1.1cm]
		%
		\draw[thick,fill=gray!20] (\myRa+\upbdj*\myD,0) -- (\myRb,0) arc(0:-\bigTheta:\myRb) -- ({(\myRa+\upbdj*\myD)*cos(\bigTheta)},{-(\myRa+\upbdj*\myD)*sin(\bigTheta)}) arc(-\bigTheta:0:{\myRa+\upbdj*\myD});
		\draw[thick] (0,0) circle (\myRb);
		\foreach \k in {0,...,\upbdk}{
			\draw[line width = 0.2pt] (0,0)--({\myRb*cos(deg(\k*\myD/\myRb))},{\myRb*sin(deg(\k*\myD/\myRb))});
			}
		\draw[fill=white, thick] (0,0) circle (\myRa);
		%
		\foreach \j in {0,...,\upbdj}{
			\draw (0,0) circle (\myRa+\j*\myD);
			\foreach \k in {0,...,\upbdk}{
				\draw[fill=black, line width = 0.2pt] ({(\myRa+\j*\myD)*cos(deg(\k*\myD/\myRb))},{(\myRa+\j*\myD)*sin(deg(\k*\myD/\myRb))}) circle(1.2pt);
			}
		}
		\foreach \k in {0,...,\upbdk}{
			\draw[fill=black] ({\myRb*cos(deg(\k*\myD/\myRb))},{\myRb*sin(deg(\k*\myD/\myRb))}) circle(1.2pt);
		}
		\draw[dotted](0,0)--(\myRa/2,0) node[below]{$r_1$} -- (\myRa,0);
		\end{tikzpicture}
	\end{minipage}
	\begin{minipage}{.5\linewidth}
		\begin{tikzpicture}[x=2.2cm, y=2.2cm]
			\draw[thick,fill=gray!20] (\myRa+\upbdj*\myD,0) --(\myRb,0) arc(0:-\bigTheta:\myRb) -- ({(\myRa+\upbdj*\myD)*cos(\bigTheta)},{-(\myRa+\upbdj*\myD)*sin(\bigTheta)}) arc(-\bigTheta:0:{\myRa+\upbdj*\myD});
			\draw[dotted](\myRa+\upbdj*\myD,0)--(0,0)--({\myRb*cos(\bigTheta)},{-\myRb*sin(\bigTheta)});
			\draw ({(\myRa+\upbdj*\myD)/3},0) arc (0:-\bigTheta:{(\myRa+\upbdj*\myD)/3});
			\draw[|-|] (\myRb,0.1)--({\myRb/2},0.1) node[above]{$r_2$}--(0,0.1);
			\draw[|-|]({(\myRa+\upbdj*\myD)*cos(\bigTheta+3.5)},{-(\myRa+\upbdj*\myD)*sin(\bigTheta+3.5)})--({\myRb*cos(\bigTheta+2.75)},{-\myRb*sin(\bigTheta+2.75)});
			\draw ({(\myRa+\upbdj*\myD+\myRb)/2*cos(\bigTheta+8)},{-(\myRa+\upbdj*\myD+\myRb)/2*sin(\bigTheta+8)})node{$<2\delta$};
			\draw ({((\myRa+\upbdj*\myD)/3+0.35)*cos(\bigTheta/2)},{-((\myRa+\upbdj*\myD)/3 + 0.45)*sin(\bigTheta/2)}) node{$\theta < \frac{2\delta}{r_2}$};
			
			\draw ({\myRb*cos(\bigTheta/2)},{-\myRb*sin(\bigTheta/2)}) node[right]{\begin{minipage}{0.6in}\flushleft Arc length $<2\delta$\end{minipage}};
				
			\draw[fill=black] (\myRb,0) circle(1.2pt);
			\draw[fill=black] ({\myRa+\upbdj*\myD},0) circle(1.2pt);
			\draw[fill=black] ({\myRb*cos(\bigTheta)},{-\myRb*sin(\bigTheta)}) circle(1.2pt);
			\draw[fill=black] ({(\myRa+\upbdj*\myD)*cos(\bigTheta)},{-(\myRa+\upbdj*\myD)*sin(\bigTheta)}) circle(1.2pt);
		\end{tikzpicture}
		\end{minipage}
		\caption{All $z$ interior to largest possible ``wedge'' are within $\sqrt{2}\delta$ of at least one corner.\label{fig:epsNet}}
	\end{figure}
\end{proof}

\begin{theorem}[Heavy-tailed CLT, $\alpha = 0$ case]\label{thm:CLT}
	Let $X_1, X_2, \ldots$ be i.i.d. complex-valued random variables with common distribution $\mu$, fix $s \in \mathbb{N}$, and suppose $\xi_1, \ldots, \xi_s, t_1, \ldots, t_s \in \mathbb{C}$ are deterministic values with $\xi_1, \ldots, \xi_s$ distinct and having magnitude one.  In addition assume $\mu$ satisfies Assumption \ref{ass:ComplexMu} with $\alpha = 0$ and that the density $f_\mu: \mathbb{A}_\eps \to [0, \infty)$ is continuous at each $\xi_l$, $1 \leq l \leq s$.  Then
\[ \frac{1}{\sqrt{n \log n}} \sum_{j=1}^n \sum_{k=1}^s t_k \left[ \frac{1}{\xi_k - X_j} - m_{\mu}(\xi_k) \right] \longrightarrow N \]
in distribution as $n \to \infty$, where $N$ is a complex-valued random variable with mean zero whose real and imaginary parts have a joint Gaussian distribution with covariance matrix 
\begin{equation} \label{def:Sigma}
	\Sigma := \sum_{k=1}^s \frac{\pi |t_k|^2 f_\mu(\xi_k)}{4} I, 
\end{equation}
where $I$ is the $2 \times 2$ identity matrix.  
\end{theorem}
\begin{proof}
The proof follows the proof of Theorem B.1 given in \cite{OW2} nearly exactly.  Only the following changes need to be made:
\begin{itemize}
\item The random variables $N_1, N_2, \ldots$ need to be defined as i.i.d. complex-valued random variables, independent of $\{X_j\}_{j \geq 1}$, whose real and imaginary parts are jointly Gaussian random variables with mean zero and covariance matrix $\Sigma$, defined in \eqref{def:Sigma} above.  
\item In the computation of the second moments $\E \left[ \Re^2 (\widetilde\zeta_1) \right]$ and $\E \left[ \Im^2 (\widetilde\zeta_1) \right]$, one needs to take into account that the points $\xi_l$, $1 \leq l \leq s$ are on the boundary of the unit disk (rather than in the interior of the disk as they are in Theorem B.1 from \cite{OW2}).  This difference only affects two integrals in the proof, where we only need to adjust the bounds of integration to take into account that the centering term $\xi_l$ is on the boundary.  For example, if $\xi_1 = 1$, then following the proof in \cite{OW2}, one ends up with an integral of the form
\[ \int_{\pi / 2 + \eps_n}^{3\pi/2 - \eps_n} \int_{1/ (|t_k| \eps \sqrt{n \log n})}^{\delta/|t_k|} \frac{\cos^2 \theta}{r} \ dr \ d\theta \]
for some sequence $\eps_n$ that tends to zero as $n$ tends to infinity.  Here, only the bounds for $\theta$ have been adjusted (no changes to the $r$ bounds are required);  this adjustment to the bounds of integration is the reason why the definition of $\Sigma$ given in \eqref{def:Sigma} is off by a factor of $2$ from the definition given in \cite{OW2}.  
\end{itemize}
\end{proof}

\begin{corollary}\label{cor:CLT}
	Let $X_1, X_2, \ldots$ be i.i.d.\ random variables having a common distribution $\mu$ that satisfies Assumption \ref{ass:ComplexMu} with $\alpha = 0$ and a radial density $f_R(r)$ that is continuous from the left at $r=1$. In addition, fix $L \in \mathbb{N}$, and suppose $t_1, \ldots, t_L \in \mathbb{C}$ are deterministic values and that $U_1, \ldots, U_L$ are independent draws from the uniform distribution on the unit circle. Then
	\[ \frac{1}{\sqrt{n \log n}} \sum_{j=1}^n \sum_{k=1}^L\frac{t_k\cdot X_j}{U_k - X_j} \longrightarrow N \]
	in distribution as $n \to \infty$, where $N$ is a complex-valued random variable with mean zero whose real and imaginary parts have a joint Gaussian distribution with covariance matrix 
	\begin{equation} \label{def:SigmaCor}
		\Sigma := \sum_{k=1}^L \frac{\pi |t_k|^2 f_\mu(1)}{4} I ,
	\end{equation}
	where $I$ is the $2 \times 2$ identity matrix. 
\end{corollary}
\begin{proof}
	Fix $t_1, \ldots, t_L \in \mathbb{C}$, and suppose $X_1, X_2,\ldots \sim \mu$ and $U_1, \ldots, U_L$ satisfy the hypotheses. Assumption \ref{ass:ComplexMu} holds and $\abs{U_k} = 1$, $1 \leq k \leq L$, so we can apply Lemma \ref{lem:StielCalc} to obtain
	\[
	\frac{1}{\sqrt{n \log n}} \sum_{j=1}^n \sum_{k=1}^L\frac{t_k\cdot X_j}{U_k - X_j} = \frac{1}{\sqrt{n \log n}} \sum_{j=1}^n \sum_{k=1}^L t_kU_k \left[ \frac{1}{U_k - X_j} - m_{\mu}(U_k) \right],
	\]
	where the sums at right resemble those under consideration in Theorem \ref{thm:CLT} above. For $\xi_1, \dots, \xi_L \in \mathbb{C}$ of unit length, define
	\[
	S_n(\xi_1, \ldots, \xi_L) = \frac{1}{\sqrt{n \log n}} \sum_{j=1}^n \sum_{k=1}^L t_k\xi_k \left[ \frac{1}{\xi_k - X_j} - m_{\mu}(\xi_k) \right].
	\]
	We seek to understand the limiting distribution of $S_n(U_1, \ldots, U_L)$. Toward that end, let $B \subset \mathbb{C}$ be an arbitrary borel set and apply the total law of probability to obtain
	\begin{align*}
	&\P(S_n(U_1, \ldots, U_L) \in B)\\
	&\quad= \int\cdots\int\P\left(S_n(U_1, \ldots, U_L) \in B \mid U_1 = \xi_1, \ldots, U_L=\xi_L\right)\,d\nu(\xi_1)\cdots d\nu(\xi_L),
	\end{align*}
	where $\nu$ is the uniform probability distribution on the unit circle in the complex plane. By Theorem \ref{thm:CLT} with $t_k := t_k\xi_k$ and $\xi_k := \xi_k$, if $\xi_1, \ldots, \xi_L$ are distinct, we have that 
	\begin{align*}
	\P\left(S_n(U_1, \ldots, U_L) \in B \mid U_1 = \xi_1, \ldots, U_L=\xi_L\right)
	&= \P(S_n(\xi_1, \ldots, \xi_L) \in B)\\
	& \rightarrow \P(N\in B),
	\end{align*}
	pointwise as $n \to \infty$, where $N$ is a complex-valued Gaussian with mean zero and covariance matrix given in \eqref{def:SigmaCor}. (Note that since $\abs{\xi_k} =1$ and $\mu$ is radially symmetric, we have $\abs{t_k\xi_k}^2 = \abs{t_k}^2$ and $f_\mu(\xi_k) = f_\mu(1)$ for $1 \leq k \leq L$.) Thus, by the dominated convergence theorem (since $\xi_1, \ldots, \xi_L$ are distinct almost everywhere), 
	\begin{align*}
		&\P(S_n(U_1, \ldots, U_L) \in B)\\
		&\quad= \int\cdots\int\P\left(S_n(U_1, \ldots, U_L) \in B \mid U_1 = \xi_1, \ldots, U_L=\xi_L\right)\,d\nu(\xi_1)\cdots d\nu(\xi_L)\\
		&\quad\longrightarrow \int\cdots\int\P(N \in B)\,d\nu(\xi_1)\cdots d\nu(\xi_L) = \P(N\in B).
	\end{align*}
	The dominated convergence theorem is applicable here since $\P(S_n(U_1, \ldots, U_L) \in B \mid U_1 = \xi_1, \ldots, U_L=\xi_L)$ is non-negative and uniformly upper-bounded by 1. We conclude that $S_n(U_1, \ldots, U_L) \to N$ in distribution as $n\to \infty$, as desired.
\end{proof}

\begin{lemma}\label{lem:CovStructurePos}
Suppose $X$ is a complex-valued random variable whose distribution $\mu$ satisfies Assumption \ref{ass:ComplexMu} with $\alpha > 0$, and let $U_1, U_2$ be uniform draws from the unit circle that are jointly independent from each other and from $X$. Then,
\begin{align}
	\E\left[\Re\left(\frac{X}{U_1-X}\right)\Re\left(\frac{X}{U_2-X}\right)\right] &=0,\label{eqn:covReRe}\\
	\E\left[\Im\left(\frac{X}{U_1-X}\right)\Re\left(\frac{X}{U_2-X}\right)\right] &=0,\label{eqn:covImRe}\\
	\E\left[\Im\left(\frac{X}{U_1-X}\right)\Im\left(\frac{X}{U_2-X}\right)\right] &=0,\label{eqn:covImIm}\\
	\E\left[\Re\left(\frac{X}{U_1-X}\right)\Im\left(\frac{X}{U_1-X}\right)\right] &=0.\label{eqn:covSameU}
\end{align}
\end{lemma}
\begin{proof}
	To establish \eqref{eqn:covReRe} and \eqref{eqn:covImRe}, it suffices to verify  
	\begin{equation}\label{eqn:covReInt}
	\E\left[\left(\frac{X}{U_1-X}\right)\Re\left(\frac{X}{U_2-X}\right)\right] =0.
	\end{equation}
	For a fixed $x \in \C$ with $\abs{x}<1$, we have
	\[
	\E\left[\frac{x}{U_1-x}\right] = \frac{1}{2\pi}\int_0^{2\pi}\frac{x}{e^{\sqrt{-1}\theta}- x}\,d\theta = \frac{1}{2\pi\sqrt{-1}}\oint_\gamma \frac{x}{(z-x)z}\,dz,
	\]
	where the last integral is a complex line integral over the unit circle parameterized by $\gamma(\theta) = e^{\sqrt{-1}\theta}.$ By Cauchy's integral formula, we obtain
	\[
	\E\left[\frac{x}{U_1-x}\right] = \frac{1}{2\pi\sqrt{-1}}\oint_\gamma \frac{1}{z-x} - \frac{1}{z}\,dz = 1-1 = 0.
	\]
	Equation \eqref{eqn:covReInt} follows by conditioning on $X$ and $U_2$. 
	
	We can use a nearly identical argument to the one just made to show that \eqref{eqn:covImIm} is true, so it remains to establish \eqref{eqn:covSameU}. To that end, first consider that $X/U_1$ has the same distribution as $X$, and then pass to polar coordinates to obtain:
	\begin{equation}\label{eqn:covSameUInt}
	\begin{aligned}
	&\E\left[\Re\left(\frac{X/U_1}{1-X/U_1}\right)\Im\left(\frac{X/U_1}{1-X/U_1}\right)\right]\\
	&\qquad= \frac{1}{2\pi}\int_0^1f_R(r)\int_{-\pi}^\pi\Re\left(\frac{re^{\sqrt{-1}\theta}}{1-re^{\sqrt{-1}\theta}}\right)\Im\left(\frac{re^{\sqrt{-1}\theta}}{1-re^{\sqrt{-1}\theta}}\right)\,d\theta\,dr.
	\end{aligned}
	\end{equation}
	Now
	\[
	\Re\left(\frac{re^{\sqrt{-1}\theta}}{1-re^{\sqrt{-1}\theta}}\right) = \frac{r\cos{\theta}(1-r\cos\theta) - r^2\sin^2\theta}{(1-r\cos\theta)^2 + r^2\sin^2\theta}
	\]
	is an even function in the argument $\theta$ and 
	\[
	\Im\left(\frac{re^{\sqrt{-1}\theta}}{1-re^{\sqrt{-1}\theta}}\right) = \frac{r\sin{\theta}}{(1-r\cos\theta)^2 + r^2\sin^2\theta}
	\]
	is an odd function in the argument $\theta$, so their product is an odd function in the argument $\theta$, from which it follows that the inner integral with respect to $\theta$ on the right side of \eqref{eqn:covSameUInt} equals $0$. The proof is complete.
\end{proof}


\bibliography{maxCpFluct}{}
\bibliographystyle{abbrv}

\end{document}